\documentclass{amsart}

\usepackage[english]{babel} 
\usepackage[utf8]{inputenc} 
\usepackage{amsmath} 
\usepackage{amsthm}
\usepackage{amssymb}
\usepackage{mathrsfs}
\usepackage{amsfonts} %
\usepackage{bm} 
\usepackage{fancyvrb}
\usepackage[linesnumbered,algoruled,boxed]{algorithm2e}
\usepackage{graphicx}
\usepackage{psfrag}
\usepackage{subfig}
\usepackage[square,sort,comma,numbers]{natbib} 
\usepackage[colorlinks=false, linkcolor=blue,  citecolor=OliveGreen,  pdfstartview={}{}]{hyperref} 
\usepackage{lineno}
\usepackage{stmaryrd}
\usepackage[font=scriptsize]{caption}

\DeclareFontEncoding{FMS}{}{}
\DeclareFontSubstitution{FMS}{futm}{m}{n}
\DeclareFontEncoding{FMX}{}{}
\DeclareFontSubstitution{FMX}{futm}{m}{n}
\DeclareSymbolFont{fouriersymbols}{FMS}{futm}{m}{n}
\DeclareSymbolFont{fourierlargesymbols}{FMX}{futm}{m}{n}
\DeclareMathDelimiter{\VERT}{\mathord}{fouriersymbols}{152}{fourierlargesymbols}{147}

\newtheorem{theorem}{Theorem}[section] 
\newtheorem{corollary}[theorem]{Corollary} 
\newtheorem{lemma}[theorem]{Lemma}
\newtheorem{proposition}{Proposition}[section] 
\theoremstyle{remark}
\newtheorem{remark}[theorem]{Remark} 
\numberwithin{equation}{section}
\newcommand{\T}{\mathscr{T}}
\newcommand{\M}{\mathscr{M}}
\newcommand{\Sides}{\mathscr{S}}
\newcommand{\diff}{\, \mbox{\rm d}}

\usepackage[usenames,dvipsnames]{color}
\newcommand{\EO}[1]{{\color{black}#1}}

\newcommand{\FF}[1]{{\color{black}#1}}

\usepackage{pgf,tikz}
\usetikzlibrary{arrows}

\begin{document}
\title[AFEM for sparse PDE--constrained optimization]{Adaptive finite element methods for sparse PDE--constrained optimization}



%
%
%

\author{Alejandro Allendes\textsuperscript{\textdagger}}
\address{\textdagger Departamento de Matemática, Universidad Técnica Federico Santa María, Valparaíso, Chile.}
\email{alejandro.allendes@usm.cl}

\author{Francisco Fuica\textsuperscript{\textdaggerdbl}}
\address{\textdaggerdbl Departamento de Matemática, Universidad Técnica Federico Santa María, Valparaíso, Chile.}
\email{francisco.fuica@sansano.usm.cl}

\author{Enrique Otárola\textsuperscript{\textsection}}
\address{\textsection Departamento de Matemática, Universidad Técnica Federico Santa María, Valparaíso, Chile.}
\email{enrique.otarola@usm.cl}

\maketitle

\begin{abstract}
{We propose and analyze reliable and efficient a posteriori error estimators for an optimal control problem that involves a nondifferentiable cost functional, the Poisson problem as state equation and control constraints. To approximate the solution to the state and adjoint equations we consider a piecewise linear finite element method whereas \emph{three different strategies} are used to approximate the control variable: piecewise constant discretization, piecewise linear discretization and the so--called variational discretization approach. For the first two aforementioned solution techniques we devise an error estimator that can be decomposed as the sum of four contributions: two contributions that account for the discretization of the control variable and the associated subgradient, and two contributions related to the discretization of the state and adjoint equations. The error estimator for the variational discretization approach is decomposed only in two contributions that are related to the discretization of the state and adjoint equations. On the basis of the devised a posteriori error estimators, we design simple adaptive strategies that yield optimal rates of convergence for the numerical examples that we perform.}
{PDE--constrained optimization, nondifferentiable objectives, sparse controls, a posteriori error analysis, adaptive finite elements.}
\end{abstract}

\section{Introduction.}
\label{section1} 
In this work we shall be interested in the design and analysis of a posteriori error estimators for a nondifferentiable optimal control problem; control constraints are also considered. To make matters precise, we let $\Omega \subset \mathbb{R}^d$, with $d \in \{2,3\}$, be an open and bounded polytopal domain with Lipschitz boundary $\partial \Omega$. Given $f \in L^2(\Omega)$, a desired state $y^{}_{\Omega} \in L^2(\Omega)$, a regularization parameter $\alpha >0$, and a sparsity parameter $\beta  > 0$, we define the nondifferentiable cost functional
\begin{equation}
\label{def:functional}
J(y,u):=\frac{1}{2}\left\|y-y^{}_{\Omega}\right\|_{L^2(\Omega)}^2+\frac{\alpha}{2} \left\|u \right\|_{L^2(\Omega)}^2+\beta\left\|u\right\|_{L^1(\Omega)}.
\end{equation}
We shall thus be concerned with the nonsmooth optimal control problem: Find
\begin{equation}
\label{def:min_eq}
\text{min} \:J(y,u)
\end{equation}
subject to the state equation
\begin{equation}
\label{eq:state_eq} 
-\Delta y=u+f  \text{ in }  \Omega, \qquad
 y=0  \text{ on }  \partial\Omega,    
\end{equation}
and the control constraints
\begin{equation}
\label{def:box_constraints}
u\in \mathbb{U}_{ad}, \quad \mathbb{U}_{ad}:=\{v\in L^2(\Omega): a \leq v(x) \leq b \textrm{ a.e. } x \in \Omega\}.
\end{equation}
We immediately comment that, since we are interested in the nondifferentiable scenario, we assume that  $a,b\in\mathbb{R}$ are such that $a< 0 < b$. We refer the reader to \cite[Remark 2.1]{CHW:12} for a discussion.

The design and analysis of solution techniques for problem \eqref{def:min_eq}--\eqref{def:box_constraints} are motivated by the following two observations:
\begin{itemize}
\item The cost functional $J$ involves the $L^1(\Omega)$--norm of the control variable. This term, that is a natural measure of the control cost, leads to sparsely supported optimal controls \cite{CHW:12,MR2556849,WW:11}, i.e., \emph{optimal controls that are not zero only in a small region of the considered domain}. This is a desirable feature in applications, for instance, in the optimal placement of discrete actuators \cite{MR2556849}.
\item The cost functional $J$ is nondifferentiable. As a consequence, the study of solution techniques for \eqref{def:min_eq}--\eqref{def:box_constraints} present some extra mathematical difficulties compared with the standard case $\alpha > 0$ and $\beta=0$ that is presented, for instance, in \cite{Troltzsch}. Fortunately, these difficulties can be overcame with elements from convex analysis \cite{CHW:12,WW:11}.
\end{itemize}

The analysis and finite element discretization of PDE--constrained optimization problems that involve a cost functional containing a $L^1(\Omega)$--control cost term have been considered in a number of works. To the best of our knowledge, the first work that provides an analysis when the state equation is a linear and elliptic PDE is \cite{MR2556849}. In this work, the author utilizes a regularization technique that involves a $L^2(\Omega)$--control cost term, analyzes optimality conditions, and studies the convergence properties of a proposed semismooth Newton method. Later, these results were complemented with rates of convergence with respect to the regularization parameter $\alpha$ in \cite{WW:11}. Subsequently, the authors of \cite{CHW:12} consider a nonlinear version of \eqref{def:min_eq}--\eqref{def:box_constraints} where the state equation is a semilinear elliptic PDE and analyze second order optimality conditions. We refer the reader to the recent work of \cite{Casas2017} for a complete overview of the results available in the literature. Simultaneously with these advances, discretization techniques based on finite element methods and their corresponding \emph{a priori error analyses} have been considered. We refer the reader to \cite{WW:11}, when the state equation is a linear elliptic PDE, to \cite{CHW:12b,CHW:12} for extensions to the semilinear case, and to \cite{OS2} when the state equation \eqref{eq:state_eq} involves the spectral fractional powers of elliptic operators. We also mention references \cite{MR3601024,MR3612174} for extensions of the aforementioned developments to evolution problems.

As opposed to the available a priori error analysis for finite element approximations of sparse PDE--constrained optimization, the design and analysis of a posteriori error estimators are rather scarce. An a posteriori error estimator is a computable quantity that depends on the discrete solution and data, and provides information about the local quality of the approximate solution. It is an essential ingredient of  adaptive finite element methods (AFEMs). The theory for linear second--order elliptic boundary value problems is well--established \cite{MR1885308,NSV:09,NV,Verfurth}. In contrast, the theory for constrained optimal control problems is not as developed. The main source of difficulty is its inherent nonlinear feature, which appears as a result of the control constraints. To the best of our knowledge, the first work that provides an advance is \cite{LiuYan} where the authors propose an estimator and derive a reliability estimate \cite[Theorem 3.1]{LiuYan}. Subsequently, the analysis was improved in \cite{HHIK} by providing efficiency estimates involving oscillation terms \cite[Theorems 5.1 and 6.1]{HHIK}. An attempt to unify the available results in the literature was later presented in \cite{KRS}: on the basis of an important error equivalence the analysis is simplified to provide reliable and efficient estimators for the state and adjoint equations. The analysis is based on the energy norm inherited by the state and adjoint equations. Recently, the authors of \cite{SW:16} provided a general framework that complements the one developed in \cite{KRS}, and measures the error in a norm that is motivated by the objective. The analysis relies on the convexity of $\Omega$. The common feature in all the previous cited references is that, in contrast to  \eqref{def:min_eq}--\eqref{def:box_constraints}, $\beta = 0$. For different approaches based on \emph{weighted residual} and \emph{goal--oriented} methods and advances in the semilinear and nonlinear case, the reader is referred to \cite{BBMRV,HH:08,MRW:15,VW:08}.

To the best of our knowledge, the only work that provides an advance concerning a posteriori error analysis for \eqref{def:min_eq}--\eqref{def:box_constraints} is \citep{WW:11}. In this work, the authors consider a piecewise constant discretization for the control variable, propose a residual--type a posteriori error estimator. In Theorem 6.2, it is proved that the devised error estimator yields an upper bound for the approximation errors of the state and control variables (the errors committed in the approximation of the associated subgradient and the adjoint variable are not considered). However, no efficiency analysis is provided in \cite{WW:11}. 
In this work we complement and extend the results presented in \cite[Section 6]{WW:11} as follows:  We consider \emph{three discretization schemes} for \eqref{def:min_eq}--\eqref{def:box_constraints} that rely on the discretization of the state and adjoint equations with piecewise linear functions. The schemes differ on the type of discretization considered for the control variable: piecewise constant, piecewise linear or variational discretization. For the first two schemes, we design an a posteriori error estimator that accounts for the discretization of the optimal control variable, its associated subgradient, and the state and adjoint variables. The a posteriori error estimator designed for the variational discretization approach only needs to account for the discretization of the state and adjoint variables. We measure the total error in energy--norms and $L^2(\Omega)$--norms and derive, for each scheme, global reliability and local efficiency results in a unified manner. With these estimators at hand, we also design simple adaptive strategies that yield optimal rates of convergence for the numerical examples that we perform.

We organize our exposition as follows. We set notation in Section \ref{section2} and briefly recall elements from convex analysis. In Section \ref{section3} we present existence and uniqueness results together with first--order necessary and sufficient optimality conditions. In Section \ref{section4} we present three finite element discretizations for the optimal control problem \eqref{def:min_eq}--\eqref{def:box_constraints}; all of them rely on the discretization of the state and adjoint equations by using piecewise linear functions. To approximate the control variable three strategies are considered: piecewise constant, piecewise linear and variational discretization. The core of our work is Section \ref{section5} where, for each discretization presented in Section \ref{section4}, we design an a  posteriori error estimator and derive reliability and local efficiency results.  We conclude, in Section \ref{section6}, with a series of numerical examples that illustrate our theory.


\section{Notation and Preliminaries}\label{section2}

Let us fix notation and the functional setting in which we will operate. Throughout this work $d \in \{ 2,3\}$ and $\Omega \subset \mathbb{R}^d$ is an open and bounded polytopal domain with Lipschitz boundary. For a bounded domain $G \subset \mathbb{R}^d$, $L^{2}(G)$ and $H^{1}(G)$ denote the standard Lebesgue and Sobolev spaces, respectively, and $H_{0}^{1}(G)$ is the subspace of $H^{1}(G)$ consisting of functions whose trace is zero on $\partial G$. Let $(\cdot,\cdot)_{L^{2}(G)}$ and $\| \cdot \|_{L^2(G)}$ denote, respectively, the inner product and norm in $L^{2}(G)$.
The seminorm in $H^1(G)$ is denoted by $|\cdot|_{H^1(G)}$.

If $\mathcal{X}$ and $\mathcal{Y}$ are normed vector spaces, we write $\mathcal{X}  \hookrightarrow \mathcal{Y}$ to denote that $\mathcal{X}$ is continuously embedded in $\mathcal{Y}$. We denote by $\mathcal{X}^{\star}$ the dual of $\mathcal{X}$. The relation $a \lesssim b$ indicates that $a \leq C b$, with a nonessential constant $C$ that might change at each occurrence. Finally, throughout the manuscript we will frequently make use of the following Poincar\'e inequality
\begin{equation}
 \label{eq:Poincare}
\| w \|_{L^2(\Omega)}\leq \mathfrak{C} \| \nabla w\|_{L^2(\Omega)} \quad \forall w \in H_0^1(\Omega).
\end{equation}


\subsection{Convex functions and subdifferentials}
In this section we recall some elements from convex analysis that will be essential for the analysis that we will perform.

Let $E$ be a real normed vector space. Let $\eta: E \rightarrow \mathbb{R} \cup \{\infty\}$ be convex and proper, and let $v\in E$ with $\eta(v)<\infty$. A subgradient of $\eta$ at $v$ is a continuous linear functional $v^\star$ on $E$ that satisfies 
\begin{equation}
\label{subdif}
\langle v^\star,w-v\rangle \leq \eta(w)-\eta(v)\hspace{0.5cm}\forall~ w\in E,
\end{equation} 
where $\langle \cdot,\cdot\rangle$ denotes the duality pairing between $E^{\star}$ and $E$. We immediately remark that a function may admit many subgradients at a point of nondifferentiability. The set of all subgradients of $\eta$ at $v$ is called subdifferential of $\eta$ at $v$ and is denoted by $\partial\eta(v).$ By convexity, the subdifferential $\partial\eta(v)\not=\emptyset$ for all points $v$ in the interior of the effective domain of $\eta$. Finally, we mention that the subdifferential is monotone, i.e.,
\begin{equation}
\label{eq:subdiff_monotone}
\langle v^\star - w^\star, v - w \rangle \geq 0 \quad \forall v^\star \in \partial \eta(v),\ \forall w^\star \in \partial\eta(w).
\end{equation}
We refer the reader to \cite{MR1058436,non} for a thorough discussion on convex analysis.


\section{Sparse PDE--constrained optimization.}\label{section3}
In this section we briefly review the analysis of the nondifferentiable optimal control problem \eqref{def:min_eq}--\eqref{def:box_constraints}. We recall existence and uniqueness results together with first--order necessary and sufficient optimality conditions. 

For $J$ defined as in \eqref{def:functional}, the nondifferentiable optimal control problem reads: 
\begin{equation}\label{eq:weak_functional}
\min_{H_{0}^{1}(\Omega)\times \mathbb{U}_{ad}}J(y,u)
\end{equation}
subject to the linear and elliptic PDE
\begin{equation}\label{eq:weak_state_eq}
\left.
\begin{array}{c}
(\nabla y, \nabla v)_{L^2(\Omega)} = (u+f,v)_{L^2(\Omega)} \quad \forall v\in H_0^1(\Omega).
\end{array}
\right.
\end{equation}
We must immediately notice that the set $\mathbb{U}_{ad}$, defined as in \eqref{def:box_constraints}, is a nonempty, bounded, closed and convex subset of $L^2(\Omega)$.

We define the control-to-state map $\mathcal{Z}$ as follows: given $u \,\in L^2(\Omega)$, $\mathcal{Z}$ associates to it a unique state $y \in H_0^1(\Omega)$ that solves \eqref{eq:weak_state_eq}. Since $H_0^1(\Omega) \hookrightarrow L^2(\Omega)$, we may also consider $\mathcal{Z}$ acting from $L^2(\Omega)$ into itself. An immediate application of Lax-Milgram Lemma implies that $\mathcal{Z}$ is a linear and continuous map. With this operator at hand we define the reduced cost functional
\begin{equation*}\label{def:reduced_functional}
j(u)=J(\mathcal{Z}u,u) :=\frac{1}{2}\left\|\mathcal{Z}u-y^{}_{\Omega}\right\|_{L^2(\Omega)}^2+\frac{\alpha}{2}\left\|u\right\|_{L^2(\Omega)}^2+\beta\left\|u\right\|_{L^1(\Omega)},
\end{equation*}
and present the following result; see also \cite[Lemma 2.1]{WW:11}.
\begin{lemma}[well--posedness]
The sparse PDE--constrained optimization problem \eqref{eq:weak_functional}--\eqref{eq:weak_state_eq} has a unique optimal solution $(\bar{y},\bar{u}) \in H_0^1(\Omega)\times L^2(\Omega)$. 
\end{lemma}
\begin{proof} 
Since $\mathcal{Z}$ is injective and continuous, $j$ is strictly convex and weakly lower semicontinuous. The fact that $\mathbb{U}_{ad}$ is weakly sequentially compact allows us to conclude.
\end{proof}

In order to obtain optimality conditions for \eqref{eq:weak_functional}--\eqref{eq:weak_state_eq} we introduce the following ingredients. First, we define the so--called adjoint state $p$ as follows:
\begin{equation}\label{eq:adjoint_equation}
p \in H_0^1(\Omega): \quad (\nabla w, \nabla p)_{L^2(\Omega)} = (y-y^{}_{\Omega},w)_{L^2(\Omega)} \quad \forall w \in H_0^1(\Omega).
\end{equation} 

We define the convex and Lipschitz function $\psi: L^1(\Omega)\mapsto \mathbb{R}$ by $\psi(u):=\left\|u\right\|_{L^1(\Omega)}$; it corresponds to the nondifferentiable component of the reduced cost functional $j$. The \FF{differentiable counterpart} of the latter is defined by 
\begin{equation*}\label{dif}
\varphi:L^2(\Omega)\rightarrow\mathbb{R},
\quad 
u\mapsto \varphi(u):=\frac{1}{2}\left\|\mathcal{Z}u-y^{}_{\Omega}\right\|_{L^2(\Omega)}^2+\frac{\alpha}{2}\left\|u\right\|_{L^2(\Omega)}^2.
\end{equation*}
Standard arguments reveal that $\varphi$ is Fr\'echet differentiable with $\varphi'(u)=p+\alpha u$ (see \cite[Theorem 2.20]{Troltzsch}.

With these ingredients at hand, we present necessary and sufficient optimality conditions for our sparse PDE--constrained optimization problem.

\begin{theorem}[optimality conditions]
\label{optimality_cond}
The pair $(\bar{y},\bar{u})\in H^1_0(\Omega)\times \mathbb{U}_{ad}$ is optimal for problem \eqref{eq:weak_functional}--\eqref{eq:weak_state_eq} if and only if $\bar{y}=\mathcal{Z}\bar{u}$ and $\bar{u}$ satisfies the variational inequality
\begin{equation}\label{eq:control_ineq}
(\bar{p}+\alpha \bar{u}+\beta\bar{\lambda},u-\bar{u})_{L^2(\Omega)}\geq 0 \quad \forall u\in \mathbb{U}_{ad},
\end{equation}
where $\bar p$ denotes the solution to \eqref{eq:adjoint_equation} with $y$ replaced by $\bar{y}$ and $\bar{\lambda}\in \partial\psi(\bar{u})$.
\end{theorem}
\begin{proof}
See \cite[Lemma 2.2]{WW:11}.
\end{proof}

To present the following result we introduce, for  $\mathfrak{a},\mathfrak{b} \in \mathbb{R}$ the projection formula 
\begin{equation}
\label{def:projection}
\Pi_{[\mathfrak{a},\mathfrak{b}]}\left(v(x)\right)=\text{min}\left\{\mathfrak{b},\text{max}\left\{\mathfrak{a},v(x)\right\}\right\}.
\end{equation}
\begin{corollary}[projection formula]
\label{projection_formula}
 Let $\bar{y},\bar{p},\bar{u}$ and $\bar{\lambda}$ be as in Theorem \ref{optimality_cond}. Then, we have that
\begin{equation}
\label{proju}
\bar{u}(x)=\Pi_{[a,b]}\left(-\frac{1}{\alpha}\big(\bar{p}(x)+\beta\bar{\lambda}(x)\big)\right),
\end{equation}
and
\begin{equation}
\label{projection_lambda}
\bar{u}(x)=0 \quad \Leftrightarrow \quad |\bar{p}(x)|\leq \beta, 
\qquad 
\bar{\lambda}(x)=\Pi_{[-1,1]}\left(-\frac{1}{\beta}\bar{p}(x)\right).
\end{equation}
Consequently, $\bar{u},\bar{\lambda} \in H_0^1(\Omega)$, and $\bar{\lambda}$ is uniquely determined.
\end{corollary}
\begin{proof}
See \citep[Corollary 3.2]{CHW:12}. 
\end{proof} 

We immediately comment that the projection formula \eqref{projection_lambda} guarantees the uniqueness of the subgradient $\bar{\lambda}$ \cite[Corollary 3.2]{CHW:12}; this property is not usually obtained in non--differentiable optimization problems.

To summarize, the pair $(\bar{y},\bar{u})$ is optimal for \eqref{eq:weak_functional}--\eqref{eq:weak_state_eq} if and only if $(\bar{y},\bar{p},\bar{u}) \in H_{0}^{1}(\Omega) \times H_{0}^{1}(\Omega) \times  \mathbb{U}_{ad}$ solves
\begin{equation}\label{eq:optimal_control_system}
\left\{
\begin{array}{cl}
(\nabla \bar{y}, \nabla v)_{L^2(\Omega)}  =  (\bar u+f,v)^{}_{L^2(\Omega)}
& \forall v\in H_{0}^{1}(\Omega),\\
(\nabla w, \nabla \bar{p} )_{L^2(\Omega)}  =  (\bar{y}-y^{}_{\Omega},w)^{}_{L^2(\Omega)}
& \forall w\in H_{0}^{1}(\Omega) ,\\
\multicolumn{1}{c}{(\bar{p}+\alpha \bar{u}+\beta\bar{\lambda},u-
\bar{u})^{}_{L^2(\Omega)}\geq 0} &  \forall u\in \mathbb{U}_{ad},
\end{array}
\right.
\end{equation}  
where $\bar \lambda \in \partial \psi (\bar u)$.

\section{Finite element discretization.}\label{section4} 
We present three finite element solution techniques for the nondifferentiable optimal control problem \eqref{eq:weak_functional}--\eqref{eq:weak_state_eq}. All the techniques discretize the state and adjoint equations with piecewise linear functions. However, they differ on the type of discretization technique used for the optimal control variable. In Section \ref{subsec:pcd}, we consider a piecewise constant discretization, in Section \ref{subsec:pld}, an scheme based on piecewise linear functions and, in Section, \ref{subsec:vd} we consider the so--called variational discretization approach.

We begin this section by introducing some finite element notation \citep{CiarletBook,Guermond-Ern}. Let $ \mathscr{T} = \{K\}$ be a conforming partition of $\bar \Omega$ into simplices $K$ with size $h_K := \textrm{diam}(K)$, and set $h_{\mathscr{T}}:= \max_{ K \in \mathscr{T}} h_K$. We denote by $\mathbb{T}$ the collection of conforming and shape regular meshes that are refinements of an initial mesh $\mathscr{T}_0$. 

Given a mesh $\mathscr{T} \in \mathbb{T}$, we define the finite element space of continuous piecewise polynomials of degree one as
\begin{equation}
\label{def:piecewise_linear_discrete}
\mathbb{V}(\mathscr{T})=\{v^{}_\mathscr{T}\in C(\bar{\Omega}): v^{}_{\mathscr{T}}|^{}_{K}\in \mathbb{P}_1(K) \, \forall K\in \mathscr{T}, \, v^{}_\T|^{}_{\partial \Omega} = 0 \}.
\end{equation}

In what follows we will describe the three solution techniques that we will consider for our optimal control problem \eqref{eq:weak_functional}--\eqref{eq:weak_state_eq}.


\subsection{Piecewise constant discretization.}
\label{subsec:pcd}

We define $\mathbb{U}_0(\mathscr{T}):=\{u^{}_\mathscr{T}\in L^2(\Omega): u^{}_\mathscr{T}|^{}_{K}\in\mathbb{P}_0(K) \; \forall K\in\mathscr{T}\}$. The discrete admissible set $\mathbb{U}_{ad,0}(\mathscr{T}$) is thus defined as
\begin{equation*}\label{discrete_box_constant}
\mathbb{U}_{ad,0}(\mathscr{T}):=\mathbb{U}_0(\mathscr{T})\cap \mathbb{U}_{ad}.
\end{equation*}
With these discrete spaces at hand, we propose the following finite element discretization of the optimality system \eqref{eq:optimal_control_system}: Find $(\bar{y}^{}_{\mathscr{T}},
\bar{p}^{}_\mathscr{T},\bar{u}^{}_{\mathscr{T}})\in \mathbb{V}(\mathscr{T})\times \mathbb{V}(\mathscr{T})\times \mathbb{U}_{ad,0}(\mathscr{T})$ such that
\begin{equation}\label{eq:fem_optimal_control}
\left\{
\begin{array}{cl}
(\nabla \bar{y}^{}_\mathscr{T}, \nabla v^{}_\mathscr{T})  =  (\bar{u}^{}_\mathscr{T}+f,v^{}_\mathscr{T})^{}_{L^2(\Omega)}
& \forall v^{}_\mathscr{T}\in \mathbb{V}(\mathscr{T}),\\
( \nabla w^{}_\mathscr{T}, \nabla \bar{p}^{}_\mathscr{T})  =  (\bar{y}^{}_\mathscr{T}-y^{}_{\Omega},w^{}_\mathscr{T})^{}_{L^2(\Omega)}
& \forall w^{}_\mathscr{T}\in \mathbb{V}(\mathscr{T}) ,\\
\multicolumn{1}{c}{(\bar{p}^{}_\mathscr{T}+\alpha \bar{u}^{}_\mathscr{T}+\beta\bar{\lambda}^{}_\mathscr{T},u^{}_\mathscr{T}-
\bar{u}^{}_\mathscr{T})^{}_{L^2(\Omega)}\geq 0} &  \forall u_\mathscr{T}\in \mathbb{U}_{ad,0}(\mathscr{T}),
\end{array}
\right.
\end{equation}  
where $\bar{\lambda}^{}_\mathscr{T}\in \partial \psi(\bar{u}^{}_\mathscr{T})$ and 
\[
 \psi: \mathbb{U}_0(\mathscr{T}) \rightarrow \mathbb{R}, \quad u^{}_\T \mapsto \psi(u^{}_{\T}) = \int_{\Omega} | u^{}_{\T} | \diff x = \sum_{K \in \T} |u^{}_{\T}| |K|.
\]
The next result states discrete projection formulas for $\bar{u}^{}_{\mathscr{T}}$ and $\bar{\lambda}^{}_{\mathscr{T}}$.
\begin{lemma}[discrete projection formulas in $\mathbb{U}_{0}(\mathscr{T})$]
\label{lem:projection_Uad0}
Let $(\bar{y}^{}_{\mathscr{T}},\bar{p}^{}_\mathscr{T},\bar{u}^{}_{\mathscr{T}})\!\in \mathbb{V}(\mathscr{T}) \times \mathbb{V}(\mathscr{T})\times \mathbb{U}_{ad,0}(\mathscr{T})$ be the solution to \eqref{eq:fem_optimal_control}. Then, we have
\begin{equation}
\label{eq:discrete_projection_0}
\bar{u}^{}_{\mathscr{T}}|^{}_K=\Pi_{[a,b]}\left(-\frac{1}{\alpha}\Bigg(\frac{1}{|K|}\int_K \bar{p}^{}_{\mathscr{T}}\:\diff x+\beta\bar{\lambda}^{}_{\mathscr{T}}|^{}_{K}\Bigg)\right),
\end{equation}
and 
\begin{equation}
\label{eq:discrete_subgradient_proj_0}
\bar{u}^{}_{\mathscr{T}}|^{}_K=0  \Leftrightarrow  \frac{1}{|K|}\left|\int_K \bar{p}^{}_{\mathscr{T}}\:\diff x\right|\leq \beta,
\quad
\bar{\lambda}^{}_{\mathscr{T}}|^{}_K=\Pi_{[-1,1]}\left(-\frac{1}{\beta|K|}\int_K \bar{p}^{}_{\mathscr{T}}\:\diff x\right).
\end{equation}
Consequently, the discrete subgradient $\bar \lambda_{\T}$ is unique.

\end{lemma}
\begin{proof}  
We begin by noticing that $\partial \psi(\bar{u}^{}_\mathscr{T}) \subset  \mathbb{U}_{0}(\mathscr{T})^*$. Consequently, standard arguments, on the basis of \eqref{subdif}, allow us to identify $\bar \lambda^{}_{\mathscr{T}} \in \partial \psi(\bar{u}^{}_\mathscr{T})$ with an element of $\mathbb{U}_0(\T)$ that satisfies
\begin{equation}
\label{def:discrete_subgradient_0}
\bar{\lambda}^{}_\mathscr{T}=
\sum_{K\in\mathscr{T}}\chi^{}_K\bar{\lambda}^{}_{\mathscr{T}}|^{}_{K},\qquad \begin{cases}
\bar{\lambda}^{}_{\mathscr{T}}|^{}_K=+1, \quad & \bar{u}^{}_{\mathscr{T}}|^{}_{K}>0,\\
\bar{\lambda}^{}_{\mathscr{T}}|^{}_K=-1, \quad & \bar{u}^{}_{\mathscr{T}}|^{}_{K}<0,\\
\bar{\lambda}^{}_{\mathscr{T}}|^{}_K\in [-1,1], \quad & \bar{u}^{}_{\mathscr{T}}|^{}_{K}=0,
\end{cases}
\end{equation}
where $\chi^{}_K$ corresponds to the characteristic function of the element $K \in \mathscr{T}$. Thus, since $\bar u_{\T} \in \mathbb{U}_{ad,0}(\T)$, the variational inequality in the optimality system \eqref{eq:fem_optimal_control} reads
\begin{equation*}      
\label{eq:constant_optimality_condition}
\sum_{K\in \mathscr{T}}
\left(\int_K \bar{p}^{}_{\mathscr{T}}\: \diff x+|K|\big(\alpha \bar{u}^{}_\mathscr{T}|^{}_K+\beta\bar{\lambda}^{}_\mathscr{T}|^{}_K\big)\right)(u^{}_\mathscr{T}|^{}_K-\bar{u}^{}_\mathscr{T}|^{}_K)\geq 0,
\end{equation*}
where, for every $K \in \T$, $u_{\T}|_{K} \in \mathbb{P}_0(K)$ is such that $a \leq u_{\T}|_{K} \leq b$. We thus invoke similar arguments to the ones used in the proof of \citep[Lemma 2.26]{Troltzsch} to obtain the projection formula \eqref{eq:discrete_projection_0}.

The proof of \eqref{eq:discrete_subgradient_proj_0} follows from \eqref{def:projection}, \eqref{eq:discrete_projection_0}, and \eqref{def:discrete_subgradient_0}; see \citep[Section 4]{CHW:12} for details. This concludes the proof.
\end{proof}


\subsection{Piecewise linear discretization.} 
\label{subsec:pld}
Let us define $\mathbb{U}_1(\mathscr{T}) :=\{u^{}_\mathscr{T}\in C(\bar{\Omega}): u^{}_\mathscr{T}|^{}_{K}\in\mathbb{P}_1(K) \;\forall K\in\mathscr{T}\}$. The discrete admissible set is thus defined as
\begin{equation*}\label{discrete_box_linear}
\mathbb{U}_{ad,1}(\mathscr{T}):=\mathbb{U}_1(\mathscr{T})\cap \mathbb{U}_{ad}.
\end{equation*}

We denote the set of all vertices of the mesh $\T$ by $\mathcal{V}(\T)$, and, for $\textsc{v} \in \mathcal{V}(\T)$, we introduce the function $\phi^{}_{\textsc{v}} \in \mathbb{U}_{1}(\T)$ which is such that $\phi^{}_\textsc{v}(\textsc{v}') = \delta_{\textsc{v} \textsc{v}'}$ for all $\textsc{v}' \in \mathcal{V}(\T)$. The set $\{ \phi^{}_\textsc{v}: \textsc{v} \in \mathcal{V}(\T) \}$ is the so--called Courant basis of the space $\mathbb{U}_{1}(\mathscr{T})$ \cite{Guermond-Ern,NSV:09}.
We notice that every element $u^{}_\mathscr{T}\in \mathbb{U}_{1}(\mathscr{T})$ can be written as
\[
u^{}_\mathscr{T}=\sum_{\textsc{v}\in\mathcal{V}(\T)} u^{}_{\mathscr{T}}(\textsc{v}) \phi^{}_\textsc{v}.
\] 

We follow \cite[Section 3]{CHW:12b} and define, on the space $\mathbb{U}_{1}(\mathscr{T})$, the  discrete inner product $(\cdot,\cdot)^{}_\T$ and norm $\| \cdot \|_{\T}$ by
\begin{equation}\label{def:discrete_versions_l1}
(u^{}_\T,v^{}_\T)^{}_\T=\sum_{\textsc{v} \in \mathcal{V}(\T)}u^{}_\T(\textsc{v}) v^{}_\T(\textsc{v}) \int_\Omega \phi^{}_\textsc{v}\:\diff x, \quad \| u_{\T} \|^2_{\T} = (u^{}_\T,u^{}_\T)^{}_\T,
\end{equation}
respectively. We also define the discrete nondifferentiable component $\psi_{\T}: \mathbb{U}_1(\T) \rightarrow \mathbb{R}$ as $\psi_\T (u^{}_\T)=\sum_{\textsc{v} \in \mathcal{V}(\T)}|u_\T(\textsc{v})|\int_\Omega \phi^{}_\textsc{v}\diff x$ and the discrete cost functional
\[
J^{}_{\T}(y^{}_{\T},u^{}_{\T}) = \frac{1}{2} \| y^{}_{\T} - y^{}_{\Omega} \|^2_{L^2(\Omega)} + \frac{\alpha}{2} \| u^{}_{\T}\|^2_\T + \beta\psi^{}_{\T}(u^{}_{\T}).
\]   
\EO{The following optimality system corresponds to the discrete first--order optimality condition of the problem $\min J_{\T}(y_{\T},u_{\T})$ subject to $(\nabla y_\mathscr{T}, \nabla v_\mathscr{T})  =  (u_\mathscr{T}+f,v_\mathscr{T})_{L^2(\Omega)}$, for all $v^{}_\mathscr{T} \in \mathbb{V}(\mathscr{T})$, and $u_{\T} \in  \mathbb{U}_{ad,1}(\mathscr{T})$; see}
\cite[Theorem 3.3]{CHW:12b}: Find $(\bar{y}_{\mathscr{T}},
\bar{p}_\mathscr{T},\bar{u}_{\mathscr{T}})\in \mathbb{V}(\mathscr{T})\times \mathbb{V}(\mathscr{T})\times \mathbb{U}_{ad,1}(\mathscr{T})$ such that
\begin{align}
\label{eq:fem_optimal_control_1}
\begin{cases}
\qquad \qquad (\nabla \bar{y}^{}_\mathscr{T}, \nabla v^{}_\mathscr{T})  =  (\bar{u}^{}_\mathscr{T}+f,v^{}_\mathscr{T})^{}_{L^2(\Omega)}
& \forall v^{}_\mathscr{T}\in \mathbb{V}(\mathscr{T}),\\
\qquad \qquad ( \nabla w^{}_\mathscr{T}, \nabla \bar{p}^{}_\mathscr{T})  =  (\bar{y}^{}_\mathscr{T}-y^{}_{\Omega},w^{}_\mathscr{T})^{}_{L^2(\Omega)}
& \forall w^{}_\mathscr{T}\in \mathbb{V}(\mathscr{T}) ,\\
\multicolumn{1}{c}{(\bar{p}^{}_\mathscr{T},u^{}_\mathscr{T}-
\bar{u}^{}_\mathscr{T})_{L^2(\Omega)}+(\alpha \bar{u}^{}_\mathscr{T}+\beta\bar{\lambda}^{}_\mathscr{T},u^{}_\mathscr{T}-
\bar{u}^{}_\mathscr{T})^{}_{\T}\geq 0} &  \forall u_\mathscr{T}\in \mathbb{U}_{ad,1}(\mathscr{T}),
\end{cases}
\end{align}  
where $\bar{\lambda}^{}_\mathscr{T}\in \partial \psi_{\T}(\bar{u}^{}_\mathscr{T})$. To present the following result we introduce the quasi--interpolation operator $\Theta_\T:L^1(\Omega)\rightarrow \mathbb{U}_{1}(\T)$, that is defined as follows:   
\begin{equation}
\label{eq:Theta}
\Theta_\T(w)=\sum_{\textsc{v}\in \mathcal{V}(\T)}\theta^{}_{\textsc{v}}(w)\phi^{}_{\textsc{v}},
\qquad
\theta^{}_{\textsc{v}}(w):=\frac{\int_\Omega w \phi^{}_{\textsc{v}}\:\diff x}{\int_\Omega \phi^{}_{\textsc{v}}\:\diff x}.
\end{equation}
\begin{lemma}[discrete projection formulas in $\mathbb{U}_{ad,1}(\mathscr{T})$] 
\label{lem:projection_Uad1}
Let $(\bar{y}^{}_{\mathscr{T}},\bar{p}^{}_\mathscr{T},\bar{u}^{}_{\mathscr{T}})\in \mathbb{V}(\mathscr{T})\times \mathbb{V}(\mathscr{T})\times \mathbb{U}_{ad,1}(\mathscr{T})$ be the solution to \eqref{eq:fem_optimal_control_1}. Then, for every $\textsc{v}\in\mathcal{V}(\T)$, we have
\begin{equation}\label{eq:discrete_projection_1}
\bar{u}^{}_{\mathscr{T}}(\textsc{v})=\Pi_{[a,b]}\left(-\frac{1}{\alpha}\Bigg( \theta^{}_{\textsc{v}}(\bar{p}^{}_{\mathscr{T}})+\beta\bar{\lambda}^{}_{\mathscr{T}}(\textsc{v})\Bigg)\right),
\end{equation}
and
\begin{equation}\label{eq:discrete_subgradient_proj_1}
\bar{u}^{}_{\mathscr{T}}(\textsc{v})=0 \quad \Leftrightarrow \quad \left| \theta^{}_{\textsc{v}}(\bar{p}^{}_{\mathscr{T}})\right|\leq \beta,    
\qquad
\bar{\lambda}^{}_{\mathscr{T}}(\textsc{v})=\Pi_{[-1,1]}\left(-\frac{1}{\beta} \theta^{}_{\textsc{v}}(\bar{p}^{}_{\mathscr{T}})\right).
\end{equation}
In particular, the discrete subgradient $\bar{\lambda}_\T$ is unique.
\end{lemma}
\begin{proof}
See \citep[Lemma 3.4]{CHW:12b}.
\end{proof}

\begin{remark}[projection formulas]
The projection formulas obtained in Lemmas \ref{lem:projection_Uad0} and \ref{lem:projection_Uad1} are essential ingredients in the numerical resolution of system \eqref{eq:fem_optimal_control} and \eqref{eq:fem_optimal_control_1}, respectively; see \textbf{Algorithm 1} in Section \ref{section6}.
\end{remark}  

\begin{proposition}[a priori error estimates] 
\label{error_estimates}
\rm
If \eqref{eq:fem_optimal_control} and \eqref{eq:fem_optimal_control_1} approximate the optimal control problem  (\ref{eq:weak_functional})--(\ref{eq:weak_state_eq}) when $\mathbb{U}_{ad}(\T) = \mathbb{U}_{ad,0}(\T)$ and $\mathbb{U}_{ad}(\T) = \mathbb{U}_{ad,1}(\T)$, respectively and $\Omega$ is convex, then, the following a priori error estimates can be derived: For every $h_0>0$, there is a constant $C>0$ such that for all $h_{\mathscr{T}}\leq h_0$
\begin{equation}
\label{eq:a_priori}
\left\|\bar{u}-\bar{u}^{}_\mathscr{T} \right\|_{L^2(\Omega)}\leq Ch_{\mathscr{T}},
\end{equation}
where $C$ is independent of $h_{\mathscr{T}}$.
\end{proposition}
\begin{proof}
 For a proof of this result we refer the reader to \citep[Proposition 4.5]{WW:11} when $\mathbb{U}_{ad}(\T) = \mathbb{U}_{ad,0}(\T)$ and \citep[Theorem 3.13]{CHW:12b} when $\mathbb{U}_{ad}(\T) = \mathbb{U}_{ad,1}(\T)$.
\end{proof}
\subsection{Variational discretization}
\label{subsec:vd} 
In what follows we will consider the so--called variational discretization approach introduced by Hinze in \citep{hinze}. We discretize the state equation with the help of the discrete space \eqref{def:piecewise_linear_discrete}; the admissible set of controls $\mathbb{U}_{ad}$ is not discretized. In spite of this fact, the proposed semidiscrete scheme will induce a discretization of the optimal control and its unique associated subgradient on the basis of projection formulas; see Lemma \ref{lem:projection_variational} below.

With the aforementioned semidiscrete setting at hand, we propose the following finite element discretization of the optimality system \eqref{eq:optimal_control_system} \cite[Section 5]{CHW:12}: Find $(\bar{y}^{}_{\mathscr{T}},
\bar{p}^{}_\mathscr{T},\bar{u}^{}_{\mathscr{T}})\in \mathbb{V}(\mathscr{T})\times \mathbb{V}(\mathscr{T})\times \mathbb{U}_{ad}$ such that
\begin{equation}\label{eq:variational_optimal_control}
\begin{cases}
\quad (\nabla \bar{y}^{}_\mathscr{T}, \nabla v^{}_\mathscr{T})  =  (\bar{u}^{}_\mathscr{T}+f,v^{}_\mathscr{T})^{}_{L^2(\Omega)}
& \forall v^{}_\mathscr{T}\in \mathbb{V}(\mathscr{T}),\\
\quad ( \nabla w^{}_\mathscr{T}, \nabla \bar{p}^{}_\mathscr{T})  =  (\bar{y}^{}_\mathscr{T}-y^{}_{\Omega},w^{}_\mathscr{T})^{}_{L^2(\Omega)}
& \forall w^{}_\mathscr{T}\in \mathbb{V}(\mathscr{T}) ,\\
\multicolumn{1}{c}{(\bar{p}^{}_\mathscr{T}+\alpha \bar{u}^{}_\mathscr{T}+\beta\bar{\lambda}^{}_\mathscr{T},u^{}_\mathscr{T}-
\bar{u}^{}_\mathscr{T})^{}_{L^2(\Omega)}\geq 0} &  \forall u_\mathscr{T}\in \mathbb{U}_{ad},
\end{cases}
\end{equation}  
where $\bar{\lambda}^{}_\mathscr{T}\in \partial \psi(\bar{u}^{}_\mathscr{T})$.
We now present projection formulas for the variables $\bar u_{\T}$ and $\bar \lambda_{\T}$ that make evident how they are implicitly discretized by the semidiscrete scheme \eqref{eq:variational_optimal_control}.

\begin{lemma}[variational discrete projections]
\label{lem:projection_variational}
Let $(\bar{y}^{}_{\mathscr{T}},\bar{p}^{}_\mathscr{T},\bar{u}^{}_{\mathscr{T}})\!\in \mathbb{V}(\mathscr{T}) \times \mathbb{V}(\mathscr{T})\times \mathbb{U}_{ad}$ be the solution to \eqref{eq:variational_optimal_control}. Then, for all $x \in \Omega$, we have that
\begin{equation}
\label{variational_proju}
\bar{u}^{}_{\T}(x)=\Pi_{[a,b]}\left(-\frac{1}{\alpha}\big(\bar{p}^{}_{\T}(x)+\beta\bar{\lambda}^{}_{\T}(x)\big)\right),
\end{equation}
and
\begin{equation}
\label{variational_projection_lambda}
\bar{u}^{}_{\T}(x)=0 \quad \Leftrightarrow \quad |\bar{p}^{}_{\T}(x)|\leq \beta, 
\qquad 
\bar{\lambda}^{}_{\T}(x)=\Pi_{[-1,1]}\left(-\frac{1}{\beta}\bar{p}^{}_{\T}(x)\right).
\end{equation}
In addition, the discrete subgradient $\bar{\lambda}^{}_\T$ is unique.
\end{lemma}
\begin{proof}
See \cite[Section 5]{CHW:12}.
\end{proof}
\begin{proposition}[a priori error estimate] 
If \eqref{eq:variational_optimal_control} approximate the optimal control problem  (\ref{eq:weak_functional})--(\ref{eq:weak_state_eq}) and $\Omega$ is convex, then the following a priori error estimate can be derived: For every $h_0>0$, there exits $C>0$ such that for all $h_{\mathscr{T}}\leq h_0$,
\begin{equation}
\label{eq:a_priori2}
\left\|\bar{u}-\bar{u}^{}_\mathscr{T} \right\|_{L^2(\Omega)}\leq Ch_{\mathscr{T}}^2,
\end{equation}
where $C$ is independent of $h_{\mathscr{T}}$.
\end{proposition}
\begin{proof}
See \citep[Corollary 4.7]{WW:11}
\end{proof}


\section{A posteriori error estimation.}\label{section5} 
The design and analysis of AFEMs to solve the optimal control problem \eqref{eq:weak_functional}--\eqref{eq:weak_state_eq} are motivated by the following considerations: 
\begin{enumerate}
\item[$\bullet$] the a priori error estimates obtained in \cite{CHW:12,WW:11} require $\T$ to be quasi--uniform and $\Omega$ to be \emph{convex}. In addition, such estimates are valid under the assumption that $h_{\T}$ is sufficiently small. If the condition that $\Omega$ is convex is violated, the optimal variables may have singularities and thus exhibit fractional regularity. As a consequence, quasi--uniform refinement of $\Omega$ would not result in an efficient solution technique.
\item[$\bullet$] the sparsity term $\psi(u)$ in the cost functional yield an optimal control $\bar u$ that is nonzero only in sets of small support in $\Omega$.
\end{enumerate}
It is then natural, to efficiently resolve such a behavior on the optimal control variable and recover optimal rates of convergence when $\Omega$ is not convex, to propose AFEMs.

In the next section we will construct three types of a posteriori error estimators; two of them will be based on the following four contributions: two contributions that account for the discretization of the control variable and the associated subgradient, and two contributions related to the discretization of the state and adjoint variables. Instead, the a posteriori error estimator for the variational discretization approach is based only in two contributions: one related to the discretization of the state variable, and another one related to the discretization of the adjoint variable.


\subsection{A posteriori error analysis for the Laplacian} 
Since the error estimators that we will propose involve contributions that account for the discretization of the state and adjoint variables, in what follows we summarize some classical a posteriori error estimates for the Laplacian.

Let $g\in L^{2}(\Omega)$ and consider the following problem: Find $z\in H_0^1(\Omega)$ such that
\begin{equation}
\label{weak_aux_eq}
(\nabla z,\nabla v)_{L^2(\Omega)}=(g,v)^{}_{L^2(\Omega)} \quad \forall v\in H_0^1(\Omega).
\end{equation}

We define the Galerkin approximation to \eqref{weak_aux_eq} as the solution to: 
Find $z^{}_\mathscr{T}\in \mathbb{V}(\mathscr{T})$ such that 
\begin{equation*}\label{discrete_aux_eq}
(\nabla z^{}_\mathscr{T},\nabla v^{}_\mathscr{T})_{L^2(\Omega)}=(g,v^{}_\mathscr{T})^{}_{L^2(\Omega)}\quad\forall\:v^{}_\mathscr{T}\in \mathbb{V}(\mathscr{T}).
\end{equation*}

We define $\Sides$ as the set of internal $(d-1)$--dimensional interelement boundaries $\gamma $ of $\T$. For $K \in \T$, let $\Sides^{}_K$ denote the subset of $\Sides$ that contains the sides in $\Sides$ which are sides of $K$. We also denote by $\Omega_{\gamma}$ the subset of $\T$ that contains the two elements that have $\gamma$
as a side. In addition, we define the patch associated with an element $K \in \T$ as
\begin{equation}
\label{eq:OmegaK}\Omega_K:= \bigcup_{K' \in \T: \Sides_K \cap \Sides_{K'} \neq \emptyset} K'.
\end{equation}

Given a discrete function $z^{}_{\T} \in \mathbb{V}(\T)$, we define, for any internal side $\gamma \in \Sides$, the jump or interelement residual  $[\![ \nabla z^{}_\mathscr{T}\cdot \nu ]\!]_{\gamma}$ by
\[
[\![ \nabla z^{}_\mathscr{T}\cdot \nu ]\!]_{\gamma}= \nu^{+} \cdot \nabla z^{}_{\mathscr{T}}|^{}_{K^{+}} + \nu^{-} \cdot \nabla z^{}_{\mathscr{T}}|^{}_{K^{-}},
\]
where $\nu^{+}, \nu^{-}$ denote the unit normals to $\gamma$ pointing outward $K^{+}$, $K^{-} \in \T$, respectively, which are such that $K^{+} \neq K^{-}$ and $\partial K^{+} \cap \partial K^{-} = \gamma$.

With these ingredients at hand, we introduce the following a posteriori error indicators and error estimator
\begin{equation}
\label{def:estimator_z_h1}
\mathcal{E}_{z,K}^{2}
=
h_K^{2}\|g\|_{L^2(K)}^2+
h_K^{}\left\|[\![ \nabla z^{}_\mathscr{T}\cdot \nu ]\!]_{\gamma}\right\|_{L^2(\partial K \setminus \partial \Omega)}^2,
\quad
\mathcal{E}^{}_{z}:=\left(\sum_{K\in\mathscr{T}}\mathcal{E}_{z,K}^{2}\right)^{\frac{1}{2}},
\end{equation}
respectively. 
It is well--known that there exists a positive constant $\FF{\mathscr{C}}$ such that the following global reliability result holds:
\begin{eqnarray}\label{eq:H1estimation}
|z-z^{}_{\mathscr{T}}|_{H^{1}(\Omega)} \leq  \FF{\mathscr{C}}\mathcal{E}_{z}.
\end{eqnarray} 
We refer the reader to \cite[Section 2.2]{MR1885308} and \cite[Section 1.4]{Verfurth} for details.

Let us now define the following a posteriori error indicators and error estimator
\begin{equation}
\label{def:estimator_z_l2}
E_{z,K}^{2}
=
h_K^{4}\|g\|_{L^2(K)}^2+
h_K^{3}\left\|[\![ \nabla z^{}_\mathscr{T}\cdot \nu ]\!]_{\gamma}\right\|_{L^2(\partial K \setminus \partial \Omega)}^2,
\quad
E^{}_{z}:=\left(\sum_{K\in\mathscr{T}}E_{z,K}^{2}\right)^{\frac{1}{2}},
\end{equation}
respectively. If $\Omega$ is convex, a duality argument reveals that there exist a positive constant $\FF{C}$ such that
\begin{eqnarray}\label{eq:L2estimation}
\|z-z^{}_{\mathscr{T}}\|_{L^{2}(\Omega)} \leq  \FF{C}E^{}_{z}.
\end{eqnarray} 
We refer the reader to \cite[Section 2.4]{MR1885308} for details.

\begin{remark}[data oscillation]
As it is customary in a posteriori error analysis, global reliability properties for residual--type error estimators do not involve oscillation terms \cite{MR1885308,NV,NSV:09,Verfurth}. Such terms appear when analyzing the asymptotic sharpness of the a posteriori upper bounds \eqref{eq:H1estimation} and \eqref{eq:L2estimation} \cite{MR1885308,NV,NSV:09,Verfurth}. One may think that the issue of oscillation is specific to standard a posteriori error estimation. However all estimators we are aware of suffer from oscillations of the data that are finer than the mesh--size \cite{NV,NSV:09}.
\end{remark}

\subsection{Error estimators for sparse PDE--constrained optimization: reliability}

The upper bounds for the errors that we will obtain in our work are constructed using upper bounds on the error between the solution to the discretization \eqref{eq:fem_optimal_control}, \eqref{eq:fem_optimal_control_1} or \eqref{eq:variational_optimal_control} and auxiliary variables that we define in what follows. 

Let  $(\bar{y}^{}_{\mathscr{T}},\bar{p}^{}_\mathscr{T},\bar{u}^{}_\mathscr{T}) \in \mathbb{V}(\mathscr{T})\times\mathbb{V}(\mathscr{T})\times \mathbb{U}_{ad}(\mathscr{T})$ be the solution to  \eqref{eq:fem_optimal_control}, \eqref{eq:fem_optimal_control_1} or \eqref{eq:variational_optimal_control}; $\mathbb{U}_{ad}(\mathscr{T}) = \mathbb{U}_{ad,0}(\mathscr{T})$ for \eqref{eq:fem_optimal_control}, $\mathbb{U}_{ad}(\mathscr{T}) = \mathbb{U}_{ad,1}(\mathscr{T})$ for \eqref{eq:fem_optimal_control_1}, and $\mathbb{U}_{ad}(\mathscr{T}) = \mathbb{U}_{ad}$ for \eqref{eq:variational_optimal_control}. We define $(\hat{y},\hat{p}) \in H_{0}^{1}(\Omega)\times H_{0}^{1}(\Omega)$ as the solution to
\begin{equation}\label{eq:hat_functions}
\left\{
\begin{array}{rcll}
(\nabla \hat{y},\nabla v)_{L^2(\Omega)} & = & (\bar{u}^{}_\mathscr{T}+f,v)_{L^2(\Omega)} &  \forall v\in H_0^1(\Omega), \\
(\nabla w,\nabla\hat{p})_{L^2(\Omega)} & = & (\bar{y}^{}_\mathscr{T}-y^{}_{\Omega},w)_{L^2(\Omega)} & \forall w\in H_0^1(\Omega).
\end{array}
\right.
\end{equation}
We notice that $(\bar{y}^{}_{\mathscr{T}},\bar{p}^{}_{\mathscr{T}})$ can be seen as a finite element approximation of $(\hat{y},\hat{p})$. This property motivates the following definitions. First, we define 
\begin{align}
\label{def:estimators_y_and_p_H1}\mathcal{E}_{y,K}^2 & =  \displaystyle
h_K^{2}\|\bar{u}^{}_{\mathscr{T}}+f\|_{L^2(K)}^2
+
h_K\left\|[\![ \nabla \bar{y}^{}_\mathscr{T}\cdot \nu ]\!]_{\gamma}\right\|_{L^2(\partial K \setminus \partial \Omega)}^2,
\, \,
\mathcal{E}_{y}^2 :=\! \sum_{K\in\mathscr{T}} \mathcal{E}_{y,K}^{2},
\\ \label{def:estimators_y_and_p_H12}
\mathcal{E}_{p,K}^2  & =  \displaystyle
h_K^{2}\|\bar{y}^{}_{\mathscr{T}}-y^{}_{\Omega}\|_{L^2(K)}^2+
h_K\left\|[\![ \nabla \bar{p}^{}_\mathscr{T}\cdot \nu ]\!]_{\gamma}\right\|_{L^2(\partial K \setminus \partial \Omega)}^2, 
\, \,
\mathcal{E}_{p}^2 :=\! \sum_{K\in\mathscr{T}}\mathcal{E}_{p,K}^{2}.
\end{align}
In view of the results of the previous section we conclude  from \eqref{eq:H1estimation} that there exist constants $\FF{\mathscr{C}_{1}}$ and $\FF{\mathscr{C}_{2}}$ such that
\begin{equation}\label{eq:hat_discreto_H1}
|\hat{y}-\bar{y}^{}_{\mathscr{T}}|_{H^{1}(\Omega)} \leq  \FF{\mathscr{C}_{1}} \mathcal{E}_{y},
\qquad
|\hat{p}-\bar{p}^{}_{\mathscr{T}}|_{H^{1}(\Omega)} \leq  \FF{\mathscr{C}_{2}} \mathcal{E}_{p}.
\end{equation}

Secondly, we define the $L^2(\Omega)$--based a posteriori error indicators and estimators
\begin{align}\label{def:estimators_y_and_p_L2}
E_{y,K}^2 & = 
h_K^{4}\|\bar{u}^{}_{\mathscr{T}}+f\|_{L^2(K)}^2
+ \!
h_K^3\left\|[\![ \nabla \bar{y}^{}_\mathscr{T}\cdot \nu ]\!]_{\gamma}\right\|_{L^2(\partial K \setminus \partial \Omega)}^2, 
\, \,
E_{y}^2 :=\! \sum_{K\in\mathscr{T}}E_{y,K}^{2},
\\ 
\label{def:estimators_y_and_p_L22}
E_{p,K}^2  & =  
h_K^{4}\|\bar{y}^{}_{\mathscr{T}}-y^{}_{\Omega}\|_{L^2(K)}^2 \!+ \!
h_K^3\left\|[\![ \nabla \bar{p}^{}_\mathscr{T}\cdot \nu ]\!]_{\gamma}\right\|_{L^2(\partial K \setminus \partial \Omega)}^2, 
\,
E_{p}^2 :=\! \sum_{K\in\mathscr{T}}E_{p,K}^{2}.
\end{align}
If, in addition, $\Omega$ is convex, we thus have from \eqref{eq:L2estimation} that there exist constants $\FF{C_1}$ and $\FF{C_2}$ such that
\begin{equation}\label{eq:hat_discreto_L2}
\|\hat{y}-\bar{y}^{}_{\mathscr{T}}\|_{L^{2}(\Omega)} \leq  \FF{C_1} E_{y},
\qquad
\|\hat{p}-\bar{p}^{}_{\mathscr{T}}\|_{L^{2}(\Omega)} \leq  \FF{C_2} E_p.
\end{equation}

We now define
\begin{equation}
\label{eq:tilde{u}}
\tilde{\lambda}:=\Pi_{[-1,1]}\left(-\frac{1}{\beta}\bar{p}^{}_\mathscr{T}\right), \qquad 
\tilde{u}:=\Pi^{}_{[a,b]}\left(-\frac{1}{\alpha} \left( \bar{p}^{}_\mathscr{T}+\beta\tilde{\lambda} \right) \right).
\end{equation}

The following remark is thus necessary.
\begin{remark}[properties of $\tilde u$ and $\tilde \lambda$]
We notice two properties which are consequences of definition \eqref{eq:tilde{u}}. First, $\tilde \lambda \in \partial \psi(\tilde u)$. This will be crucial in the a posteriori error analysis that we will perform in Section \ref{section5}. Second, if the variational approach is considered, we thus have that
\begin{equation}
\label{eq:var_approac_equal}
\tilde{u} = \bar{u}^{}_\T, \qquad\tilde{\lambda} =\bar{\lambda}^{}_\T.
\end{equation}
\end{remark}
With these ingredients at hand, we define the following a posteriori error indicators and estimators for the optimal control variable and the associated subgradient:
\begin{align}\label{def:control_estimator} 
{E}_{u,K}^2&:=\left\|\tilde{u} - \bar{u}^{}_\mathscr{T} \right\|_{L^2(K)}^2, 
\qquad
{E}_u:=\left(\sum_{K\in\mathscr{T}}E_{u,K}^2\right)^{\frac{1}{2}},
\\
\label{def:lambda_estimator}
{E}_{\lambda,K}^2&:=\|\tilde{\lambda} - \bar{\lambda}^{}_\mathscr{T}\|_{L^2(K)}^2,
\qquad
{E}_\lambda:=\left(\sum_{K\in\mathscr{T}}{E}_{\lambda,K}^2\right)^{\frac{1}{2}}.  
\end{align}

Related to new variable $\tilde{u}\in L^2(\Omega)$, we set $(\tilde{y},\tilde{p})\in H_{0}^{1}(\Omega)\times H_{0}^{1}(\Omega)$ to be such that
\begin{equation}\label{eq:tilde_y_p}
\left\{
\begin{array}{rcll}
(\nabla\tilde{y},\nabla v)_{L^2(\Omega)} & = & (\tilde{u}+f,v)^{}_{L^2(\Omega)} & \forall v\in H_0^1(\Omega),\\
(\nabla w,\nabla\tilde{p})_{L^2(\Omega)} & = & (\tilde{y}-y^{}_{\Omega},w)^{}_{L^2(\Omega)} & \forall w\in H_0^1(\Omega).
\end{array}
\right.
\end{equation}

Finally, we define the errors $e^{}_y = \bar y - \bar y^{}_{\T}$, $e^{}_p = \bar p - \bar p^{}_{\T}$, $e^{}_{\lambda} = \bar \lambda - \bar \lambda^{}_{\T}$, and $e^{}_u = \bar u- \bar u^{}_{\T}$, and, for $e := (e^{}_y,e^{}_p,e^{}_u,e^{}_{\lambda})$, the norms
\begin{equation}
 \label{eq:energy_error}
 \VERT e \VERT_{\Omega}^2:= 
 |  e_y |^2_{H^1(\Omega)} + | e_p |^2_{H^1(\Omega)} 
 + \| e_u \|^2_{L^2(\Omega)} + \|e_\lambda  \|^2_{L^2(\Omega)},
\end{equation}     
and
\begin{equation}
  \label{eq:L2_error}
 \| e \|_{\Omega}^2:= 
\| e_y \|^2_{L^2(\Omega)} + \| e_p\|^2_{L^2(\Omega)} 
+  \| e_u \|^2_{L^2(\Omega)} + \| e_\lambda \|^2_{L^2(\Omega)}.
\end{equation}

We thus have all the ingredients at hand to develop our a posteriori error analysis.

\begin{theorem}[global reliability of $E$]
\label{global_reliability_l2}
Let $(\bar{y},\bar{p},\bar{u})\!\in H_0^1(\Omega)\times H_0^1(\Omega)\times \mathbb{U}_{ad}$ be the solution to the optimality system \eqref{eq:optimal_control_system}, and $(\bar{y}^{}_\mathscr{T},\bar{p}^{}_\mathscr{T},\bar{u}^{}_\mathscr{T})\in \mathbb{V}(\mathscr{T})\times \mathbb{V}(\mathscr{T})\times\mathbb{U}_{ad,1}(\mathscr{T})$ its numerical approximation obtained as the solution to \eqref{eq:fem_optimal_control_1}. If $\Omega$ is convex, then
\begin{align}
\label{bigestimator2}
\| e \|_{\Omega}
 \leq
E,
\end{align}
where $\| e \|_{\Omega}$ is defined as in \eqref{eq:L2_error} and
\begin{equation}
\label{eq:E}
E=\left(\sum_{K\in\mathscr{T}}E_K^2\right)^{\frac{1}{2}},
\qquad
E_K^2=C_{st}E_{y,K}^2+C_{ad}E_{p,K}^2+C_{ct}E_{u,K}^2+C_{sg}E_{\lambda,K}^2.
\end{equation}
The constants $C_{st},C_{ad},C_{ct}$ and $C_{sg}$ are independent of the continuous and discrete optimal variables, the size of the elements of the mesh $\T$ and $\# \T$.
\end{theorem}
\begin{proof}
We proceed in five steps.

\underline{Step 1.}  The goal of this step is to control the error $\|\bar{u}-\bar{u}^{}_\mathscr{T}\|_{L^2(\Omega)}$. We begin by invoking definitions \eqref{eq:tilde{u}} and \eqref{def:control_estimator} to immediately arrive at the estimate
\begin{equation}\label{erroru}
\left\|\bar{u}-\bar{u}^{}_\mathscr{T}\right\|_{L^2(\Omega)}^2
\leq 
2\|\bar{u}-\tilde{u}\|_{L^2(\Omega)}^2+2E_u^2.
\end{equation}
It thus suffices to control the term $\|\bar{u}-\tilde{u}\|_{L^2(\Omega)}$. To accomplish this task, we first notice that $\tilde{u}$, defined as in \eqref{eq:tilde{u}}, can be equivalently characterized by
\[
(\bar{p}^{}_\mathscr{T}+\alpha\tilde{u}+\beta\tilde{\lambda},u-\tilde{u})_{L^2(\Omega)} \geq 0 \quad \forall u \in \mathbb{U}_{ad}.
\]
Consequently, by setting $u = \bar u$ in the previous variational inequality and  $u=\tilde{u}$ in \eqref{eq:control_ineq}, we arrive at
\begin{equation*}\label{aproxinequ}
(\bar{p}^{}_\mathscr{T}+\alpha\tilde{u}+\beta\tilde{\lambda},\bar{u}-\tilde{u})_{L^2(\Omega)}\geq 0,
\qquad
(\bar{p}+\alpha \bar{u}+\beta\bar{\lambda},\tilde{u}-\bar{u})_{L^2(\Omega)}\geq 0.
\end{equation*}
Adding these variational inequalities we thus obtain the following basic estimate
\begin{equation*}
\alpha \|\bar{u}-\tilde{u} \|_{L^2(\Omega)}^2\leq (\bar{p}-\bar{p}^{}_\mathscr{T},\tilde{u}-\bar{u})_{L^2(\Omega)}+\beta(\bar{\lambda}-\tilde{\lambda},\tilde{u}-\bar{u})_{L^2(\Omega)}.
\end{equation*}
Now, since $ \bar \lambda \in \partial \psi (\bar u)$ and $\tilde \lambda \in \partial \psi (\tilde u)$, an application of \eqref{eq:subdiff_monotone} yields
\begin{equation*}
\beta(\bar{\lambda}-\tilde{\lambda},\tilde{u}-\bar{u})^{}_{L^2(\Omega)} \leq 0.
\end{equation*}
Consequently,
\begin{equation*}
\alpha
\|\bar{u}-\tilde{u}\|_{L^2(\Omega)}^2
\leq 
(\bar{p}-\bar{p}^{}_\mathscr{T},\tilde{u}-\bar{u})_{L^2(\Omega)}.
\end{equation*}

We now invoke the auxiliary states $\hat p$ and $\tilde p$, defined as the solution to problems \eqref{eq:hat_functions} and \eqref{eq:tilde_y_p}, respectively, to rewrite the previous expression as follows:
\begin{equation*}
\label{errorutilde1}
\begin{array}{l}
\alpha
\|\tilde{u}-\bar{u}\|_{L^2(\Omega)}^2 \leq 
(\bar{p}-\tilde{p},\tilde{u}-\bar{u})^{}_{L^2(\Omega)}
\!+(\tilde{p}-\hat{p},\tilde{u}-\bar{u})^{}_{L^2(\Omega)}
\!+ (\hat{p}-\bar{p}^{}_\mathscr{T},\tilde{u}-\bar{u})^{}_{L^2(\Omega)}.
\end{array}
\end{equation*}
We proceed to bound $(\bar{p}-\tilde{p},\tilde{u}-\bar{u})_{L^2(\Omega)}$. To accomplish this task, we notice that $\tilde{y}-\bar{y} \in H_0^1(\Omega)$ and $\bar{p}-\tilde{p} \in H_0^1(\Omega)$ solve, for all $v \in H_0^1(\Omega)$ and $w \in H_0^1(\Omega)$,
\[
  (\nabla(\tilde{y}-\bar{y}),\nabla v)_{L^2(\Omega)} = (\tilde u - \bar u,v)_{L^2(\Omega)}, \qquad (\nabla w,\nabla(\bar{p}-\tilde{p}))_{L^2(\Omega)} = (\bar y - \tilde y,w)_{L^2(\Omega)},
\]
respectively.
Set $v=\bar{p}-\tilde{p}$ and $w = \tilde{y}-\bar{y}$ and conclude that
\begin{equation*}
(\bar{p}-\tilde{p},\tilde{u}-\bar{u})_{L^2(\Omega)}
=(\nabla(\tilde{y}-\bar{y}),\nabla(\bar{p}-\tilde{p}))_{L^2(\Omega)}
=-\|\tilde{y}-\bar{y}\|_{L^2(\Omega)}^2\leq 0.
\end{equation*}
This result allows us to derive that
\begin{equation*}
\label{errorutilde3aux}
\alpha\left\|\tilde{u}-\bar{u}\right\|_{L^2(\Omega)}^2 \leq (\tilde{p}-\hat{p},\tilde{u}-\bar{u})^{}_{L^2(\Omega)} + (\hat{p}-\bar{p}^{}_\mathscr{T},\tilde{u}-\bar{u})^{}_{L^2(\Omega)},
\end{equation*}
which implies the bounds
\begin{align}
\begin{split}
\left \|\tilde{u}-\bar{u}\right\|_{L^2(\Omega)}^2 
& \leq 
\frac{2}{\alpha^2}  \| \tilde{p}-\hat{p}  \|^2_{L^2(\Omega)}  
+ 
\frac{2}{\alpha^2}  \| \hat{p}-\bar{p}^{}_\mathscr{T} \|^2_{L^2(\Omega)} 
\\
&
\leq 
\frac{2}{\alpha^2}  \| \tilde{p}-\hat{p}  \|^2_{L^2(\Omega)}  + \frac{2}{\alpha^2}\FF{C_2}^{2}E_p^2,
\end{split}
\label{errorutilde3}
\end{align}
where, in the last inequality, we have used the a posteriori error estimate \eqref{eq:hat_discreto_L2}.

To control $\|\tilde{p}-\hat{p}\|_{L^2(\Omega)}$, we notice that $(\nabla w,\nabla(\tilde{p}-\hat{p}))_{L^2(\Omega)} = (\tilde{y} - \bar{y}^{}_\mathscr{T},w)_{L^2(\Omega)}$ for all $w \in H_0^1(\Omega)$. An application of the Poincar\'e inequality \eqref{eq:Poincare} thus reveals that
\begin{equation*}
\begin{array}{c}
\mathfrak{C}^{-2}\|\tilde{p}-\hat{p}\|_{L^2(\Omega)}^2
\leq
(\nabla(\tilde{p}-\hat{p}),\nabla(\tilde{p}-\hat{p}))_{L^2(\Omega)}
\leq 
\|\tilde{p}-\hat{p}\|_{L^2(\Omega)}\|\tilde{y}-\bar{y}^{}_\mathscr{T}\|_{L^2(\Omega)},
\end{array}
\end{equation*}
which, in view of the a posteriori estimate \eqref{eq:hat_discreto_L2}, implies the bound
\begin{equation}
\label{ptildehat1}
\begin{array}{c}
\mathfrak{C}^{-4}\left\|\tilde{p}-\hat{p}\right\|_{L^2(\Omega)}^{2}
\leq 2\left\|\tilde{y}-\hat{y}\right\|_{L^2(\Omega)}^2+2\FF{C_1}^{2}E_y^2.
\end{array}
\end{equation}

Similarly, since $\tilde{y}-\hat{y}$ solves $(\nabla(\tilde{y}-\hat{y}),v)_{L^2(\Omega)} = (\tilde{u} - \bar{u}^{}_\mathscr{T},v)^{}_{L^2(\Omega)}$ for all $v \in H_0^1(\Omega)$, we can conclude that
\begin{equation*}
\begin{array}{c}
\mathfrak{C}^{-2} \|\tilde{y}-\hat{y}\|_{L^2(\Omega)}^2\leq (\nabla(\tilde{y}-\hat{y}),\nabla(\tilde{y}-\hat{y}))^{}_{L^2(\Omega)}\leq\|\tilde{y}-\hat{y}\|_{L^2(\Omega)}\|\tilde{u}-\bar{u}^{}_\mathscr{T}\|_{L^2(\Omega)},
\end{array}
\end{equation*}
and thus, invoking definition \eqref{def:control_estimator}, that
\begin{equation*}
\left\|\tilde{y}-\hat{y}\right\|_{L^2(\Omega)}^2\leq \mathfrak{C}^{4} E_u^2.
\end{equation*}
Replacing this estimate into \eqref{ptildehat1} we arrive at
\begin{equation*}
\label{ptildehat2}
\left\|\tilde{p}-\hat{p}\right\|_{L^2(\Omega)}^2\leq 2\mathfrak{C}^{8}E_u^2+2\mathfrak{C}^{4}\FF{C_1}^{2}E_y^2.
\end{equation*}
On the basis of \eqref{errorutilde3}, the collection of our previous findings yields the estimate
\begin{equation*}
\label{errorutilde4}
\left\|\tilde{u}-\bar{u}\right\|_{L^2(\Omega)}^2\leq \frac{2}{\alpha^2}\left(2\mathfrak{C}^8E_u^2
+2\mathfrak{C}^4\FF{C_1}^{2}E_y^2+\FF{C_2}^{2}E_p^2 \right),
\end{equation*}
which, in view of \eqref{erroru}, allows us to conclude the a posteriori error estimate
\begin{equation}
\label{finalerroru_linear}
\|\bar{u}-\bar{u}^{}_\mathscr{T}\|_{L^2(\Omega)}^2\leq  \frac{4}{\alpha^2}\left[\Big\{2\mathfrak{C}^8+\frac{\alpha^2}{2}\Big\}E_u^2+2\mathfrak{C}^4\FF{C_1}^{2}E_y^2+\FF{C_2}^{2}E_p^2 \right].
\end{equation}

\underline{Step 2.} The goal of this step is to bound the error $\|\bar{y}-\bar{y}^{}_\mathscr{T}\|_{L^2(\Omega)}$. We begin with
\begin{equation}
\label{erroryl2one}
\begin{array}{c}
\|\bar{y}-\bar{y}^{}_\mathscr{T}\|_{L^2(\Omega)}^2 
\leq 2\|\bar{y}-\hat{y}\|_{L^2(\Omega)}^2+2\FF{C_1}^{2}E_y^2,
\end{array}
\end{equation}
which follows from \eqref{eq:hat_discreto_L2}. Since $\bar{y}-\hat{y}$ solves $(\nabla(\bar{y}-\hat{y}),\nabla v)_{L^2(\Omega)} = (\bar{u}-\bar{u}^{}_\mathscr{T},v)_{L^2(\Omega)}$ for all $v\in H_0^1(\Omega)$, by setting $v=\bar{y}-\hat{y}$ we can conclude that
\begin{equation*}
\begin{array}{c}
\displaystyle\mathfrak{C}^{-2}\left\|\bar{y}-\hat{y}\right\|_{L^2(\Omega)}^2
\leq  (\nabla(\bar{y}-\hat{y}),\nabla(\bar{y}-\hat{y}))^{}_{L^2(\Omega)}
\leq \|\bar{y}-\hat{y}\|_{L^2(\Omega)}\|\bar{u}-\bar{u}^{}_\mathscr{T}\|_{L^2(\Omega)},
\end{array}
\end{equation*}
which yields the bound
$
\|\bar{y}-\hat{y} \|_{L^2(\Omega)} \leq \mathfrak{C}^{2}\|\bar{u}-\bar{u}^{}_\mathscr{T}\|_{L^2(\Omega)}.
$
This estimate combined with \eqref{finalerroru_linear} and \eqref{erroryl2one} imply that
\begin{equation}
\label{erroryl2two}
\|\bar{y}-\bar{y}^{}_\mathscr{T}\|_{L^2(\Omega)}^2
\leq 
\displaystyle 
\frac{8}{\alpha^2}\mathfrak{C}^4
\bigg[
\left\{2\mathfrak{C}^{8}+\frac{\alpha^2}{2}\right\}E_u^2 
+ 
\FF{C_1}^{2}\left\{2\mathfrak{C}^4 + \frac{\alpha^2}{4\mathfrak{C}^4}\right\}E_y^2 
+ 
\FF{C_2}^{2} E_p^2 
\bigg].
\end{equation}

\underline{Step 3.}  We control the term $\|\bar{p}-\bar{p}^{}_\mathscr{T}\|_{L^2(\Omega)}$. A simple application of the triangle inequality and the estimate \eqref{eq:hat_discreto_L2} reveal that
\begin{equation}
\label{errorpl2one}
\begin{array}{c}
\|\bar{p}-\bar{p}^{}_\mathscr{T}\|_{L^2(\Omega)}^2\leq  2\|\bar{p}-\hat{p}\|_{L^2(\Omega)}^2+2\FF{C_2}^{2}E_p^2.
\end{array}
\end{equation}
To estimate the term $\|\bar{p}-\hat{p}\|_{L^2(\Omega)}$, we notice that $\bar{p}-\hat{p}$ solves $(\nabla w,\nabla(\bar{p}-\hat{p}))_{L^2(\Omega)} = (\bar{y}-\bar{y}^{}_\mathscr{T},w)_{L^2(\Omega)}$ for all $w \in H_0^1(\Omega)$. Set $w=\bar{p}-\hat{p}$ and conclude that
\begin{equation*}
\begin{array}{c}
\displaystyle\mathfrak{C}^{-2}\|\bar{p}-\hat{p}\|_{L^2(\Omega)}^2\leq (\nabla(\bar{p}-\hat{p}),\nabla(\bar{p}-\hat{p}))^{}_{L^2(\Omega)}
\leq\left\|\bar{p}-\hat{p}\right\|_{L^2(\Omega)}\left\|\bar{y}-\bar{y}^{}_\mathscr{T}\right\|_{L^2(\Omega)}.
\end{array}
\end{equation*}
This estimate implies that
$
\|\bar{p}-\hat{p}\|_{L^2(\Omega)}^2\leq \mathfrak{C}^{4}\|\bar{y}-\bar{y}^{}_\mathscr{T}\|_{L^2(\Omega)}^2.
$
Therefore, \eqref{erroryl2two}, \eqref{errorpl2one} and the previous estimate allow us to deduce the a posteriori error estimate
\begin{multline}\label{errorpl2two}
\|\bar{p}-\bar{p}^{}_\mathscr{T}\|_{L^2(\Omega)}^2 \\
 \leq
\displaystyle\frac{16}{\alpha^2} \mathfrak{C}^8
\bigg[
\left\{2\mathfrak{C}^{8}+\frac{\alpha^2}{2}\right\}E_u^2 
+ 
\FF{C_1}^{2}\Big\{2\mathfrak{C}^{4}+\frac{\alpha^2}{4\mathfrak{C}^4}\Big\}E_y^2 +
\FF{C_2}^{2}\left\{ 1 +\frac{\alpha^{2}}{8 \mathfrak{C}^8}\right\}E_{p}^{2}
\bigg].
\end{multline}

\underline{Step 4.} The objective of this step is to bound $\|\bar{\lambda}-\bar{\lambda}_{\T}\|_{L^2(\Omega)}$. To accomplish this task we utilize the auxiliary variable $\tilde \lambda$, defined as in \eqref{eq:tilde{u}}, and proceed as follows:
\begin{align}
\nonumber
\|\bar{\lambda}-\bar{\lambda}^{}_{\T}\|_{L^2(\Omega)}^2
& \leq
2\|\bar{\lambda}-\tilde{\lambda}\|_{L^2(\Omega)}^2
+
2\| \tilde{\lambda}-\bar{\lambda}^{}_\T\|_{L^2(\Omega)}^2 \\
\label{eq:lambda_est_l2}
& \leq 2 \beta^{-2} \|\bar{p}-\bar{p}^{}_{\T}\|_{L^2(\Omega)}^2
+
2E_{\lambda}^2,
\end{align}
where, in the last inequality, we have used the projection formula \eqref{projection_lambda}, the Lipschitz continuity of the operator $\Pi_{[-1,1]}$ and the a posteriori error estimate \eqref{def:lambda_estimator}.  
To conclude, we insert the estimate \eqref{errorpl2two} into the previous inequality and obtain that
\begin{multline}\label{eq:lambda_estimate_1}
\|\bar{\lambda}-\bar{\lambda}^{}_{\T}\|_{L^2(\Omega)}^2
\leq 
\displaystyle\frac{32}{(\alpha\beta)^2} \mathfrak{C}^{8}
\bigg[
\left\{2\mathfrak{C}^{8}+\frac{\alpha^2}{2}\right\}E_u^2  
\\
\displaystyle
+ 
\FF{C_1}^{2}\left\{2\mathfrak{C}^{4}+\frac{ \alpha^2}{4 \mathfrak{C}^4}\right\}E_y^2 
+
\FF{C_2}^{2}\left\{1+\frac{\alpha^{2}}{8 \mathfrak{C}^{8}}\right\}E_{p}^{2}
+\displaystyle\frac{(\alpha\beta)^2}{16\mathfrak{C}^{8}}E_\lambda^2
\bigg].
\end{multline}

\underline{Step 5.} The desired estimate follows from a collection of the estimates \eqref{finalerroru_linear}, \eqref{erroryl2two}, \eqref{errorpl2two} and \eqref{eq:lambda_estimate_1}. This concludes the proof.
\end{proof} 


The previous analysis allows us to derive the following result for a posteriori error estimation based on energy--type norms.

\begin{theorem}[global reliability of $\mathcal{E}$] 
\label{global_reliability}
Let $(\bar{y},\bar{p},\bar{u})\in H_0^1(\Omega)\times H_0^1(\Omega)\times \mathbb{U}_{ad}$ be the solution to the optimality system \eqref{eq:optimal_control_system}, and $(\bar{y}^{}_\mathscr{T},\bar{p}^{}_\mathscr{T},\bar{u}^{}_\mathscr{T})\in \mathbb{V}(\mathscr{T})\times \mathbb{V}(\mathscr{T})\times\mathbb{U}_{ad,0}(\mathscr{T})$ its numerical approximation obtained as the solution to \eqref{eq:fem_optimal_control}. Then, 
\begin{equation}
\label{bigestimator}
\VERT e \VERT_{\Omega}
\leq 
\mathcal{E},
\end{equation}
where $\VERT e \VERT_{\Omega}$ is defined as in \eqref{eq:energy_error} and
\begin{equation}	
\label{eq:mathcalE}
\mathcal{E}=\left(\sum_{K\in\mathscr{T}}\mathcal{E}_K^2\right)^{\frac{1}{2}},\quad \mathcal{E}_K^2:=\mathscr{C}_{st}\mathcal{E}_{y,K}^2+\mathscr{C}_{ad}\mathcal{E}_{p,K}^2
+\mathscr{C}_{ct}E_{u,K}^2+\mathscr{C}_{sg}E_{\lambda,K}^2.
\end{equation}
The constants $\mathscr{C}_{st},\mathscr{C}_{ad},\mathscr{C}_{ct}$ and $\mathscr{C}_{sg}$ are independent of the continuous and discrete optimal variables, the size of the elements of the mesh $\T$ and $\# \T$.
\end{theorem}
\begin{proof} 
\FF{The proof follows closely  the arguments developed in the proof of Theorem \ref{global_reliability_l2} upon using a Poincar\'e inequality. For brevity, we skip the details.}
\end{proof}

We now provide an a posteriori error estimation result when the variational discretization approach is used to approximate the optimal control problem \eqref{eq:weak_functional}--\eqref{eq:weak_state_eq}.

\begin{theorem}[global reliability of $\mathfrak{E}$] 
\label{global_reliability_variational}
Let $(\bar{y},\bar{p},\bar{u})\!\in H_0^1(\Omega)\times H_0^1(\Omega)\times \mathbb{U}_{ad}$ be the solution to the optimality system \eqref{eq:optimal_control_system} and $(\bar{y}^{}_\mathscr{T},\bar{p}^{}_\mathscr{T},\bar{u}^{}_\mathscr{T})\in \mathbb{V}(\mathscr{T})\times \mathbb{V}(\mathscr{T})\times\mathbb{U}_{ad}$ its numerical approximation obtained as the solution to \eqref{eq:variational_optimal_control}. If $\Omega$ is convex, then
\begin{align}
\label{bigestimator2_variational}
\| e \|_{\Omega}
 \leq
\mathfrak{E},
\end{align}
where $\| e \|_{\Omega}$ is defined as in \eqref{eq:L2_error} and
\begin{equation}
\label{eq:variational_E}
\mathfrak{E}=\left(\sum_{K\in\mathscr{T}}\mathfrak{E}_K^2\right)^{\frac{1}{2}},
\qquad
\mathfrak{E}_K^2=\mathsf{C}_{st}E_{y,K}^2+\mathsf{C}_{ad}E_{p,K}^2.
\end{equation}
The constants $\mathsf{C}_{st}$ and $\mathsf{C}_{ad}$ are independent of the continuous and discrete optimal variables, the size of the elements of the mesh $\T$ and $\# \T$.
\end{theorem}
\begin{proof}
\EO{The proof follows closely the arguments developed in the proof of Theorem \ref{global_reliability_l2} upon using that, in this case,  $\tilde{u} = \bar{u}_\T $ and $\tilde{\lambda} = \bar{\lambda}_\T$; see \eqref{eq:var_approac_equal}. For brevity, we skip the details.}
\end{proof}



\subsection{Error estimators for sparse PDE--constrained optimization: efficiency}

In what follows we examine the efficiency properties of the a posteriori error estimators $E$, $\mathcal{E}$ and $\mathfrak{E}$\ which are defined as in \eqref{eq:E}, \eqref{eq:mathcalE}, and \eqref{eq:variational_E}, respectively. To accomplish this task, we analyze each of their contributions separately. Before proceeding with such analyses we introduce the following notation: for an edge, triangle or tetrahedron $G$, let $\mathcal{V}(G)$ be the set of vertices of $G$.
We define, for each element $K\in\mathscr{T}$ and side $\gamma\in\Sides$, the standard element and edge bubble functions \cite{Verfurth2,Verfurth}
\begin{equation}\label{eq:bubbles_H1}
\beta^{}_{K}|^{}_{K}=
(d+1)^{(d+1)}\prod_{\textsc{v} \in \mathcal{V}(K)} \phi^{}_{\textsc{v}} 
,
\qquad
\beta^{}_{\gamma}|^{}_{K}=
d^{d} \prod_{\textsc{v} \in \mathcal{V}(\gamma)}\phi^{}_{\textsc{v}}|^{}_{K},
\qquad K \subset \Omega_{\gamma},
\end{equation}
respectively, where $\phi_{\textsc{v}}$ are the barycentric coordinates of $K$. We recall that $\Omega_{\gamma}$ corresponds to the patch composed of the two elements of $\T$ sharing $\gamma$.

We now present the following error equation associated to the state equation; it follows from the continuous state equation in \eqref{eq:optimal_control_system} and an application of an integration by parts formula:
\begin{multline}\label{eq:error_semi_state}
 \sum_{K\in\mathscr{T}}
\left(\bar{u}^{}_{\mathscr{T}}+ \Pi_{K}^{\ell}(f),v\right)^{}_{L^{2}(K)}
-
\sum_{\gamma\in\Sides}
\left([\![ \nabla \bar{y}^{}_\mathscr{T}\cdot \nu ]\!]_{\gamma},v\right)^{}_{L^{2}(\gamma)} \\ 
\displaystyle 
=
(\nabla(\bar{y}-\bar{y}^{}_{\mathscr{T}}),\nabla v)_{L^2(\Omega)} -(\bar{u}-\bar{u}^{}_{\mathscr{T}},v)^{}_{L^{2}(\Omega)}
-\sum_{K\in\mathscr{T}}
(f-\Pi_{K}^{\ell}(f),v)^{}_{L^{2}(K)},
\end{multline}
for all $v\in H_{0}^{1}(\Omega)$, where, for $K \in \T$ and $\ell \in \{0,1\}$, $\Pi_{K}^{\ell}(f)$ denotes the $L^{2}(K)$--orthogonal projection operator onto $\mathbb{P}_{\ell}(K)$.

On the other hand, similar arguments to the ones that led to \eqref{eq:error_semi_state} allow us to conclude the following error equation associated to the adjoint state equation: 
\begin{multline}
\label{eq:error_semi_adjoint}
\displaystyle
 \sum_{K\in\mathscr{T}}
\left(\bar{y}^{}_{\mathscr{T}} - \Pi_{K}^{\ell}(y^{}_{\Omega}),w\right)^{}_{L^{2}(K)}
-
\sum_{\gamma\in \Sides}
\left([\![ \nabla \bar{p}^{}_\mathscr{T}\cdot \nu ]\!]_{\gamma},w\right)^{}_{L^{2}(\gamma)} \\
\displaystyle \quad
=
(\nabla(\bar{p}-\bar{p}^{}_{\mathscr{T}}),\nabla w)_{L^2(\Omega)}
-(\bar{y}-\bar{y}^{}_{\mathscr{T}},w)^{}_{L^{2}(\Omega)}
+\sum_{K\in\mathscr{T}}
(y^{}_{\Omega}-\Pi_{K}^{\ell}(y^{}_{\Omega}),w)^{}_{L^{2}(K)},
\end{multline}
for all $w \in H_0^1(\Omega)$ and $\ell \in \{0,1\}$.

\subsubsection{Efficiency properties of $\mathcal{E}$}

We proceed on the basis of standard arguments, as the ones developed in \cite[Section 2.3]{MR1885308} and \cite[Section 1.4]{Verfurth}, to conclude the following estimates. First, for $K \in \T$, we consider $v=(\bar{u}^{}_{\mathscr{T}}+\Pi_{K}^\ell(f))\beta^{}_{K}$, with $\ell \in \{0,1\}$, in \eqref{eq:error_semi_state}. This yields the estimate
\begin{equation}
\label{eq:cota_residuo_state_H1_1}
h_{K}^{2}\|\bar{u}^{}_{\mathscr{T}}+\Pi_{K}^{\FF{\ell}}(f)\|_{L^{2}(K)}^{2} 
\lesssim 
|\bar{y}-\bar{y}^{}_{\mathscr{T}}|_{H^{1}(K)}^{2} 
+ 
h_{K}^{2}\left(
\|\bar{u}-\bar{u}^{}_{\mathscr{T}}\|_{L^{2}(K)}^{2} + 
\|f-\Pi_{K}^{\FF{\ell}}(f)\|_{L^{2}(K)}^{2}\right).
\end{equation}
Second, for $K \in \T$ and $\gamma \in \Sides_K$, we consider $v=[\![ \nabla \bar{y}^{}_\mathscr{T}\cdot \nu ]\!]_{\gamma}\beta_{\gamma}$ in \eqref{eq:error_semi_state} and conclude the estimate
\begin{multline}
\label{eq:cota_residuo_state_H2}
h_{K}\|[\![ \nabla \bar{y}^{}_\mathscr{T}\cdot \nu ]\!]_{\gamma}\|_{L^{2}(\gamma)}^{2}\\
\lesssim 
\sum_{K'\in\Omega_{\gamma}}
\Big(
|\bar{y}-\bar{y}^{}_{\mathscr{T}}|_{H^{1}(K')}^{2}
+ 
h_{K}^{2}
\left(
\|\bar{u}-\bar{u}^{}_{\mathscr{T}}\|_{L^{2}(K')}^{2} + 
\|f-\Pi_{K'}^{\FF{\ell}}(f)\|_{L^{2}(K')}^{2}\right)
\Big).
\end{multline}

We are now in position to derive the following local efficiency result. To accomplish this task, for $\FF{\kappa}\in \{0,1\}$, $g \in L^2(\Omega)$ and $\M \subset \T$, we define
\begin{equation}
 \textrm{osc}^{}_{\T,\FF{\kappa}}(g;\M):= \left( \sum_{K \in \M} h_K^{2(\FF{\kappa}+1)} \| g - \Pi_{K}^{\ell} (g) \|^2_{L^2(K)} \right)^{\frac{1}{2}},
\end{equation}
\FF{where $\ell \in \{0,1\}$.}

\begin{lemma}[local efficiency of $\mathcal{E}_{y,K}$ and $\mathcal{E}_{p,K} $]\label{H1-Lemma_y_p_eficiencia}
Let $(\bar{y},\bar{p},\bar{u})\in H_{0}^{1}(\Omega)\times H_{0}^{1}(\Omega) \times \mathbb{U}_{ad}$ be the solution to the optimality system \eqref{eq:optimal_control_system} and $(\bar{y}^{}_{\mathscr{T}},\bar{p}^{}_{\mathscr{T}},\bar{u}_{\mathscr{T}})\in \mathbb{V}(\T)\times \mathbb{V}(\T) \times \mathbb{U}_{ad,0}(\T)$ its numerical approximation obtained as the solution to \eqref{eq:fem_optimal_control}. Then, for $K \in \T$, the local error indicators $\mathcal{E}_{y,K}$ and $\mathcal{E}_{p,K}$, defined as in \eqref{def:estimators_y_and_p_H1} and \eqref{def:estimators_y_and_p_H12}, respectively, satisfy that
\begin{equation}
\label{eq:eff_EyK}
\mathcal{E}_{y,K} 
 \lesssim
|\bar{y}-\bar{y}^{}_{\mathscr{T}}|_{H^{1}(\Omega_{K})}
+ h_{K} \|\bar{u}-\bar{u}^{}_{\mathscr{T}}\|_{L^{2}(\Omega_{K})} + \mathrm{osc}_{\T,0}(f;\Omega_K),
\end{equation}
and
\begin{equation}
\label{eq:eff_Eyp}
\mathcal{E}_{p,K} 
 \lesssim 
|\bar{p}-\bar{p}^{}_{\mathscr{T}}|_{H^{1}(\Omega_{K})} 
+ h_{K} \|\bar{y}-\bar{y}^{}_{\mathscr{T}}\|_{L^{2}(\Omega_{K})} + \mathrm{osc}_{\T,0}(y^{}_{\Omega};\Omega_K),
\end{equation}
where $\Omega_K$ is defined as in \eqref{eq:OmegaK} and the hidden constants are independent of the optimal variables, their approximations, the size of the elements in the mesh $\T$, and $\# \T$.
\end{lemma}
\begin{proof}
Let $K \in \T$. We first control the term $h_K^2 \| \bar{u}_\T + f \|^2_{L^2(K)}$ in \eqref{def:estimators_y_and_p_H1}. A simple application of the triangle inequality yields
\begin{equation*}
\|\bar{u}^{}_\T + f\|_{L^2(K)} \leq
\|\bar{u}^{}_\T + \Pi_{K}^{\FF{\ell}}(f)\|_{L^2(K)}
+
\|f-\Pi_{K}^{\FF{\ell}}(f)\|_{L^2(K)}.
\end{equation*}
We thus apply the estimate \eqref{eq:cota_residuo_state_H1_1} to conclude that
\begin{equation}
\label{eq:aux_step}
 h_K^2 \left\|\bar{u}^{}_\T + f\right\|^2_{L^2(K)} \lesssim 
|\bar{y}-\bar{y}^{}_{\mathscr{T}}|^2_{H^{1}(K)}
+ h_{K}^2 \|\bar{u}-\bar{u}^{}_{\mathscr{T}}\|^2_{L^{2}(K)} + \mathrm{osc}_{\T,0}^2(f;K).
\end{equation}

Let $K \in \T$ and $\gamma \in \Sides$. The control of the term $h_{K}\|[\![ \nabla \bar{y}_\mathscr{T}\cdot \nu ]\!]_{\gamma}\|_{L^{2}(\gamma)}^{2}$, in \eqref{def:estimators_y_and_p_H1}, \EO{follows from \eqref{eq:cota_residuo_state_H2}. 
This bound and \eqref{eq:aux_step} yield \eqref{eq:eff_EyK}.}

The proof of \eqref{eq:eff_Eyp} follows similar arguments but based on the error equation \eqref{eq:error_semi_adjoint}. This concludes the proof.
\end{proof}

The next result gives the global efficiency property of the estimator $\mathcal{E}$.

\begin{theorem}[global efficiency of $\mathcal{E}$]\label{H1-theorem_y_p_eficiencia}
Let $(\bar{y},\bar{p},\bar{u})\!\in H_0^1(\Omega)\times H_0^1(\Omega)\times \mathbb{U}_{ad}$ be the solution to the optimality system \eqref{eq:optimal_control_system} and $(\bar{y}^{}_\mathscr{T},\bar{p}^{}_\mathscr{T},\bar{u}^{}_\mathscr{T})\in \mathbb{V}(\mathscr{T})\times \mathbb{V}(\mathscr{T})\times\mathbb{U}_{ad,0}(\mathscr{T})$ its numerical approximation obtained as the solution to \eqref{eq:fem_optimal_control}. Then, the error estimator $\mathcal{E}$, defined as in \eqref{eq:mathcalE}, 
satisfies that
\begin{equation}
\label{eq:global_eff_mathcalE_1}
\mathcal{E}
\lesssim
\VERT e \VERT_{\Omega} + \mathrm{osc}^{}_{\T,0}(f;\T) + \mathrm{osc}^{}_{\T,0}(y^{}_{\Omega};\T),
\end{equation} 
where $e = (e_y,e_p,e_u,e_{\lambda})$, $\VERT  \cdot \VERT_{\Omega}$ is defined as in \eqref{eq:energy_error} and the hidden constant is independent of the optimal variables, their approximations, the size of the elements in the mesh $\T$, and $\# \T$.
\end{theorem}
\begin{proof}
In view of the definition of $\mathcal{E}_y$, given by \eqref{def:estimators_y_and_p_H1}, the estimate \eqref{eq:eff_EyK} immediately yields
\begin{equation}
\label{eq:Ey_eff_global}
\mathcal{E}_{y} 
 \lesssim |\bar{y}-\bar{y}^{}_{\mathscr{T}}|_{H^{1}(\Omega)}
+ \|\bar{u}-\bar{u}^{}_{\mathscr{T}}\|_{L^{2}(\Omega)} + \mathrm{osc}_{\T,0}(f;\T),
\end{equation}
where we have used that $\Omega$ is bounded and the finite overlapping property of stars: each element $K$ is contained in at most $d+2$ patches $\Omega_{K'}$. 

On the basis of \eqref{def:estimators_y_and_p_H12}, similar arguments reveal that
\begin{equation}
\label{eq:Ep_eff_global}
\mathcal{E}_{p} 
 \lesssim |\bar{p}-\bar{p}^{}_{\mathscr{T}}|_{H^{1}(\Omega)}
+ \|\bar{y}-\bar{y}^{}_{\mathscr{T}}\|_{L^{2}(\Omega)} + \mathrm{osc}_{\T,0}(y^{}_{\Omega};\T).
\end{equation}

We now study the efficiency of the estimator $E_{\lambda}$, which is defined as in \eqref{def:lambda_estimator}. A trivial application of a triangle inequality yields
\begin{equation}
\label{eq:Elambda_eff_global1}
E_{\lambda} = \| \tilde{\lambda}-\bar{\lambda}_{\mathscr{T}} \|_{L^2(\Omega)}  
\leq
\|\tilde{\lambda}-\bar{\lambda}\|_{L^{2}(\Omega)}
+
\|\bar{\lambda}-\bar{\lambda}_{\mathscr{T}}\|_{L^{2}(\Omega)}.
\end{equation}
It thus suffices to bound $\|\tilde{\lambda}-\bar{\lambda}\|_{L^{2}(\Omega)}$. To accomplish this task, we invoke the projection formula \eqref{projection_lambda}, definition \eqref{eq:tilde{u}}, and the Lipschitz continuity of $\Pi^{}_{[-1,1]}$ to conclude that
\begin{equation}
\label{eq:Elambda_eff_global2}
 \|  \tilde \lambda - \bar \lambda  \|_{L^{2}(\Omega)} \leq \beta^{-1} \|  \bar p  - \bar p_{\T}  \|_{L^{2}(\Omega)}.
\end{equation}

The control of the estimator $E_{u}$, which is defined as in \eqref{def:control_estimator}, follows similar arguments. In fact, an application of a triangle inequality yields
\begin{equation}
\label{eq:Eu_eff_global1}
E_{u} = \|\tilde{u}-\bar{u}^{}_{\mathscr{T}}\|_{L^{2}(\Omega)} \leq \|\tilde{u}-\bar{u}\|_{L^{2}(\Omega)} + \|\bar{u}-\bar{u}^{}_{\mathscr{T}}\|_{L^{2}(\Omega)}.
\end{equation}
Since $\tilde u$ is defined as in \eqref{eq:tilde{u}} and $\Pi^{}_{[a,b]}$ is Lipschitz continuous, the projection formula \eqref{proju} and the estimate \eqref{eq:Elambda_eff_global2} imply that
\begin{multline}
\label{eq:Eu_eff_global2}
 \|\tilde{u}-\bar{u}\|_{L^{2}(\Omega)} \leq  \|  \alpha^{-1} (\bar p + \beta \bar \lambda ) -\alpha^{-1} (\bar p_{\T} + \beta \tilde \lambda ) \|_{L^{2}(\Omega)} 
 \\
 \leq  \alpha^{-1} \|  \bar p  - \bar p_{\T}  \|_{L^{2}(\Omega)} + \alpha^{-1} \beta  \|  \tilde \lambda  - \bar \lambda\|_{L^{2}(\Omega)} \leq 2 \alpha^{-1} \|  \bar p  - \bar p_{\T}  \|_{L^{2}(\Omega)}.
\end{multline}

The desired estimate \eqref{eq:global_eff_mathcalE_1} is thus a consequence of the estimates \eqref{eq:Ey_eff_global}--\eqref{eq:Eu_eff_global2} combined with the Poincar\'e inequality \eqref{eq:Poincare}. This concludes the proof.
\end{proof}

\subsubsection{Efficiency properties of $E$}
\label{subsubsec:E}
 
We begin by invoking the error equation \eqref{eq:error_semi_state}, associated to the state equation, with $v \in H_0^1(\Omega)$ such that $v|^{}_{K} \in C^2(K)$ for all $K \in \T$. Integration by parts allows us to conclude, \FF{for $\ell\in\{0,1\}$}, that
\begin{align}
\nonumber
\sum_{K\in\mathscr{T}}
\left(\bar{u}_{\mathscr{T}}+\Pi_K^{\FF{\ell}}(f),v\right)^{}_{L^{2}(K)}
-
\sum_{\gamma\in\Sides}
\left[
\left([\![ \nabla \bar{y}^{}_\mathscr{T}\cdot \nu ]\!]_{\gamma},v\right)^{}_{L^{2}(\gamma)} 
+ 
(\bar{y}-\bar{y}^{}_{\mathscr{T}},[\![ \nabla v\cdot \nu ]\!]_{\gamma})_{L^{2}(\gamma)}
\right]
\\ 
\label{eq:error_state_1}
= -
\sum_{K\in\mathscr{T}}
\left(
\!(\bar{y}-\bar{y}^{}_{\mathscr{T}},\Delta v)_{L^{2}(K)}
+ (\bar{u}-\bar{u}^{}_{\mathscr{T}},v)^{}_{L^{2}(K)}
+ (f-\Pi_K^{\FF{\ell}}(f),v)^{}_{L^{2}(K)} 
\right).
\end{align}
Let $K \in \T$. Consider $v = \delta_K:= \left(\bar{u}^{}_{\mathscr{T}}+\Pi_K^{\FF{\ell}}(f)\right)\beta_{K}^{2}$ in \eqref{eq:error_state_1}
and use that $\delta_K|^{}_{\gamma} \equiv 0$ and $\nabla \delta_K|^{}_{\gamma}\equiv 0$ for all $\gamma \in \Sides$ to conclude that
\begin{multline}
\label{eq:intermediate}
\|(\bar{u}^{}_{\mathscr{T}}+\Pi_K^{\FF{\ell}}(f))\beta_{K}\|_{L^2(K)}^{2}\\
 =
-(\bar{y}-\bar{y}^{}_{\mathscr{T}},\Delta \delta_K)_{L^{2}(K)} 
- (\bar{u}-\bar{u}^{}_{\mathscr{T}},\delta_K)_{L^{2}(K)} -(f-\Pi_K^{\FF{\ell}}(f),\delta_K)^{}_{L^{2}(K)}.
\end{multline}
With the previous \EO{identity} at hand, \EO{we invoke} properties of the bubble function $\beta_K$ to conclude that
\begin{multline}
\label{eq:cota_residuo_state}
h_{K}^{4}\|\bar{u}_{\mathscr{T}}+\Pi_K^{\FF{\ell}}(f)\|_{L^{2}(K)}^{2}  
\lesssim 
h_{K}^{4}\|(\bar{u}_{\mathscr{T}}+\Pi_K^{\FF{\ell}}(f))\beta_K\|_{L^{2}(K)}^{2}
\\ 
\lesssim
\|\bar{y}-\bar{y}^{}_{\mathscr{T}}\|_{L^{2}(K)}^{2} + h_{K}^{4}
\left(
\|\bar{u}-\bar{u}^{}_{\mathscr{T}}\|_{L^{2}(K)}^{2} + 
\|f-\Pi_K^{\FF{\ell}}(f)\|_{L^{2}(K)}^{2}
\right).
\end{multline}

Let $K \in \T$ and $\gamma \in \Sides_K$. We recall that the patch composed of the two elements of $\T$ sharing $\gamma$ is denoted by $\Omega_{\gamma} = K \cup K'$ with $K' \in \T$ and introduce the following edge bubble function
\begin{equation}\label{eq:burbuja_gama}
\zeta_{\gamma}|^{}_{\Omega_{\gamma}}=d^{4d}
\left(\prod_{\textsc{v}\in\mathcal{V}(\gamma)}\phi^{}_{\textsc{v}}|^{}_{K}\phi^{}_{\textsc{v}}|^{}_{K'}\right)^{2},
\end{equation}
where, for any ${\textsc{v}} \in \mathcal{V}(\gamma)$, $\phi_{\textsc{v}}|_K$ and $\phi_{\textsc{v}}|_{K'}$ are understood now as functions over $\Omega_{\gamma}$. We notice that $\zeta_{\gamma} \in \mathbb{P}_{4d}(\Omega_{\gamma})$, $\zeta_{\gamma} \in C^2(\Omega_{\gamma})$, and $\zeta_{\gamma} = 0$ on $\partial \Omega_{\gamma}$. In addition, we have that
\begin{equation}
\label{eq:delta_gamma}
 \nabla \zeta_{\gamma} = 0 \textrm{ on } \partial \Omega_{\gamma}, \quad 
 [\![\nabla \zeta^{}_{\gamma}\cdot\nu]\!]_{\gamma}= 0 \textrm{ on } \gamma.
\end{equation}

We thus consider $v = \delta_{\gamma}:= [\![ \nabla \bar{y}^{}_\mathscr{T}\cdot \nu ]\!]_{\gamma} \zeta_{\gamma}$ in \eqref{eq:error_state_1} and invoke \eqref{eq:delta_gamma} to obtain that
\begin{multline}
\nonumber
\sum_{K'\in \Omega_{\gamma}}
\left(\bar{u}_{\mathscr{T}}+\Pi_{K'}^{\FF{\ell}}(f), \delta_{\gamma} \right)^{}_{L^{2}(K')}
-
\left([\![ \nabla \bar{y}^{}_\mathscr{T}\cdot \nu ]\!]_{\gamma},\delta_{\gamma}\right)^{}_{L^{2}(\gamma)} 
\\
= -   
\sum_{K'\in\Omega_{\gamma}} \! \left(  (\bar{y}-\bar{y}^{}_{\mathscr{T}},\Delta \delta_{\gamma})_{L^{2}(K')}
+
(\bar{u}-\bar{u}^{}_{\mathscr{T}},\delta_{\gamma})^{}_{L^{2}(K')}
+(f-\Pi_{K'}^{\FF{\ell}}(f),\delta_{\gamma})^{}_{L^{2}(K')} \right).
\end{multline}  
\EO{We thus use standard arguments, the shape regularity property of the family $\{ \T \}$, and the estimate \eqref{eq:cota_residuo_state} to arrive at}
\begin{equation}
\label{eq:cota_residuo_state_L2}
h_K^3 \| [\![ \nabla \bar{y}^{}_\mathscr{T}\cdot \nu ]\!]_{\gamma}  \|_{L^{2}(\gamma)}^{2}
\lesssim 
\sum_{K'\in\Omega_{\gamma}}
\left(
\|\bar{y} -\bar{y}^{}_{\mathscr{T}} \|_{L^{2}(K')}^2 
+ h_K^4  \|\bar{u} -\bar{u}_{\mathscr{T}} \|^2_{L^{2}(K')} \right)
+  \mathrm{osc}^2_{\T,1}(f;\Omega_{\gamma}).
\end{equation}

On other hand, similar arguments to the ones that led to \eqref{eq:error_state_1} allow us to arrive at the following error equation associated to the adjoint state equation:
\begin{align*}
\sum_{K\in\mathscr{T}}
\left(\bar{y}^{}_{\mathscr{T}}-\Pi_K^{\FF{\ell}}(y^{}_\Omega),w\right)^{}_{L^{2}(K)}
-
\sum_{\gamma\in\Sides}
\Big[\left([\![ \nabla \bar{p}^{}_\mathscr{T}\cdot \nu ]\!]_{\gamma},w\right)^{}_{L^{2}(\gamma)} 
+ 
(\bar{p}-\bar{p}^{}_{\mathscr{T}},[\![ \nabla w\cdot \nu ]\!]_{\gamma})_{L^{2}(\gamma)}\Big] \\
= -
\sum_{K\in\mathscr{T}}
\Big(
\!(\bar{p}-\bar{p}^{}_{\mathscr{T}},\Delta w)_{L^{2}(K)}
+ (\bar{y}-\bar{y}^{}_{\mathscr{T}},w)^{}_{L^{2}(K)}
- (y^{}_\Omega-\Pi_K^{\FF{\ell}}(y^{}_\Omega),w)^{}_{L^{2}(K)}
\Big),
\end{align*}
for all $w \in H_0^1(\Omega)$ such that $v|_{K} \in C^2(K)$. Similar estimates to \eqref{eq:cota_residuo_state}--\eqref{eq:cota_residuo_state_L2} can be thus obtained.

We are now in position to derive the following local efficiency result.

\begin{theorem}[local efficiency of $E$]\label{Lemma_y_p_eficiencia}
Let $(\bar{y},\bar{p},\bar{u})\in H_{0}^{1}(\Omega)\times H_{0}^{1}(\Omega) \times \mathbb{U}_{ad}$ be the solution to the optimality system \eqref{eq:optimal_control_system} and $(\bar{y}^{}_{\mathscr{T}},\bar{p}^{}_{\mathscr{T}},\bar{u}_{\mathscr{T}})\in \mathbb{V}(\T)\times \mathbb{V}(\T) \times \mathbb{U}_{ad,1}(\T)$ its numerical approximation obtained as the solution to \eqref{eq:fem_optimal_control_1}. Then, for $K\in \T$, the local error indicators $E_{y,K}$, $E_{p,K}$, $E_{u,K}$, and $E_{\lambda,K}$, defined as in \eqref{def:estimators_y_and_p_L2}, \eqref{def:estimators_y_and_p_L22}, \eqref{def:control_estimator}, and \eqref{def:lambda_estimator}, respectively, satisfy that
\begin{equation}
\label{eq:local_eff_y_2}   
E_{y,K} \lesssim \|\bar{y} -\bar{y}^{}_{\mathscr{T}} \|_{L^{2}(\Omega_K)} 
+ h_K^2  \|\bar{u} -\bar{u}_{\mathscr{T}} \|_{L^{2}(\Omega_{K})} +  \mathrm{osc}_{\T,1}(f;\Omega_{K}),
\end{equation}
\begin{equation}
 \label{eq:local_eff_p_2}
E_{p,K} \lesssim \|\bar{p} -\bar{p}^{}_{\mathscr{T}} \|_{L^{2}(\Omega_K)} 
+ h_K^2  \|\bar{y} -\bar{y}^{}_{\mathscr{T}} \|_{L^{2}(\Omega_{K})} +  \mathrm{osc}_{\T,1}(y{}_{\Omega};\Omega_{K}),
\end{equation}
\begin{equation}
\label{eq:local_eff_u_2}
E_{u,K} \leq \|\bar{u}-\bar{u}^{}_\T \|_{L^2(K)} + 2\alpha^{-1}\|\bar{p}-\bar{p}^{}_\T \|_{L^2(K)},
\end{equation}
and 
\begin{equation}
\label{eq:local_eff_lambda_2}
E_{\lambda,K} \leq \|\bar{\lambda}-\bar{\lambda}_\T \|_{L^2(K)} + \beta^{-1}\|\bar{p}-\bar{p}^{}_\T \|_{L^2(K)},
\end{equation}
where $\Omega_K$ is defined as in \eqref{eq:OmegaK} and the hidden constants are independent of the optimal variables, their approximations, the size of the elements in the mesh $\T$ and $\# \T$.
\end{theorem}
\begin{proof}
The control of the interior residual in \eqref{def:estimators_y_and_p_L2} is a consequence of the triangle inequality and the estimate \eqref{eq:cota_residuo_state}. The jump or interelement residual in \eqref{def:estimators_y_and_p_L2} is bounded in \eqref{eq:cota_residuo_state_L2}. The collection of these estimates yield \eqref{eq:local_eff_y_2}. Similar arguments yield \eqref{eq:local_eff_p_2}.
The local efficiency estimates \eqref{eq:local_eff_u_2} and \eqref{eq:local_eff_lambda_2} correspond to local versions of the estimates \eqref{eq:Elambda_eff_global1}--\eqref{eq:Eu_eff_global2}.
\end{proof}   

\subsubsection{Efficiency properties of $\mathfrak{E}$}
The results of Section \ref{subsubsec:E} immediately imply the following local efficiency result for the error indicator $\mathfrak{E}_K$, which is defined as in \eqref{eq:variational_E}.

\begin{theorem}[local efficiency of $\mathfrak{E}$]\label{Lemma_y_p_eficiencia_variational}
Let $(\bar{y},\bar{p},\bar{u})\in H_{0}^{1}(\Omega)\times H_{0}^{1}(\Omega) \times \mathbb{U}_{ad}$ be the solution to the optimality system \eqref{eq:optimal_control_system} and $(\bar{y}^{}_{\mathscr{T}},\bar{p}^{}_{\mathscr{T}},\bar{u}_{\mathscr{T}})\in \mathbb{V}(\T)\times \mathbb{V}(\T) \times \mathbb{U}_{ad}$ its numerical approximation obtained as the solution to \eqref{eq:variational_optimal_control}. Then, for $K\in \T$, the local error indicators $E_{y,K}$ and $E_{p,K}$, defined as in \eqref{def:estimators_y_and_p_L2} and \eqref{def:estimators_y_and_p_L22}, satisfy that
\begin{align}
E_{y,K} &\lesssim \|\bar{y} -\bar{y}^{}_{\mathscr{T}} \|_{L^{2}(\Omega_K)} 
+ h_K^2  \|\bar{u} -\bar{u}_{\mathscr{T}} \|_{L^{2}(\Omega_{K})} +  \mathrm{osc}_{\T,1}(f;\Omega_{K}), \label{eq:local_eff_y_variational}  
\\
E_{p,K} &\lesssim \|\bar{p} -\bar{p}^{}_{\mathscr{T}} \|_{L^{2}(\Omega_K)} 
+ h_K^2  \|\bar{y} -\bar{y}^{}_{\mathscr{T}} \|_{L^{2}(\Omega_{K})} +  \mathrm{osc}_{\T,1}(y{}_{\Omega};\Omega_{K}),\label{eq:local_eff_p_variational}
\end{align}
where $\Omega_K$ is defined as in \eqref{eq:OmegaK} and the hidden constants are independent of the optimal variables, their approximations, the size of the elements in the mesh $\T$ and $\# \T$.
\end{theorem}
\begin{proof}
The proof of the estimates \eqref{eq:local_eff_y_variational} and \eqref{eq:local_eff_p_variational} can be found in \eqref{eq:local_eff_y_2} and \eqref{eq:local_eff_p_2}, respectively.
\end{proof}


\section{Numerical Results}\label{section6}

In this section we conduct a series of numerical examples that illustrate the performance of the \EO{devised error estimators}. In Example 2 below, we go beyond the presented analysis and perform a numerical experiment where we violate the assumption of the convexity of the domain; \EO{the latter being needed to prove the results of} Theorems \ref{global_reliability_l2} and \ref{global_reliability_variational}. All the numerical experiments have been carried out with the help of a code that we implemented using \texttt{C++}. \EO{All matrices have been assembled exactly. The right hand sides as well as the approximation
errors are computed by a quadrature formula which is exact for polynomials of degree 19.} 
The global linear systems were solved using the multifrontal massively parallel sparse direct solver (MUMPS) \cite{MUMPS1,MUMPS2}.


For a given partition $\mathscr{T}$, we seek $(\bar{y}^{}_\mathscr{T},\bar{p}^{}_\mathscr{T},\bar{u}^{}_\mathscr{T})\,\in \mathbb{V}(\mathscr{T})\times\mathbb{V}(\mathscr{T})\times\mathbb{U}_{ad}(\T)$ that solves the discrete optimality system \eqref{eq:fem_optimal_control} if $\mathbb{U}_{ad}(\T)=\mathbb{U}_{ad,0}(\T)$, \eqref{eq:fem_optimal_control_1} if $\mathbb{U}_{ad}(\T)=\mathbb{U}_{ad,1}(\T)$, and \eqref{eq:variational_optimal_control} if $\mathbb{U}_{ad}(\T)=\mathbb{U}_{ad}$. The nonlinear systems obtained when $\mathbb{U}_{ad}(\T)=\mathbb{U}_{ad,0}(\T)$ and $\mathbb{U}_{ad}(\T)=\mathbb{U}_{ad,1}(\T)$ are solved by using the Newton--type primal--dual active set strategy of \citep[Section 4]{MR2556849}. The nonlinear system \EO{associated to} the variational discretization approach is solved by using \EO{an adaptation of the semi-smooth Newton method described in \citep[Section 6]{Brett} in conjunction with the characterization of the optimal control $\bar{u}$ given in \cite[equation (4.4)]{MR2556849}; see Remark \ref{variational_remark}. Once the discrete solutions $(\bar{y}^{}_\mathscr{T},\bar{p}^{}_\mathscr{T},\bar{u}^{}_\mathscr{T})$ are obtained, we calculate the error indicators $E_K$ or $\mathcal{E}_K$, defined by \eqref{eq:E} and \eqref{eq:mathcalE}, respectively, to drive the adaptive procedure described in \textbf{Algorithm 1}, or calculate $\mathfrak{E}_K$, defined by \eqref{eq:variational_E}, to drive the adaptive procedure described in \textbf{Algorithm 2}.}
%

\begin{table}[!h]
\begin{flushleft}
\begin{tabular}{p{12cm}} 
\textbf{Algorithm 1:  Adaptive Primal-Dual Active Set Algorithm.}
\vspace{0.15cm}\\
\hline
\vspace*{0.1cm}
\textbf{Input:} Initial mesh $\mathscr{T}_{0}$, desired state $y^{}_{\Omega}$, constraints $a$ and $b$, regularization parameter $\alpha$, sparsity parameter $\beta$ and external source $f$.
\\
\textbf{Set:} $i=0$.
\\
\textbf{Active set strategy:}
\\
$\boldsymbol{1:}$ Choose an initial guess for the adjoint variable $p_{\T}^{0}\in\mathbb{V}(\T)$.
\\
$\boldsymbol{2:}$ Compute $[\bar{y}^{}_\mathscr{T},\bar{p}^{}_\mathscr{T},\bar{u}^{}_\mathscr{T},\bar{\lambda}^{}_\mathscr{T}]=\textbf{Active-Set}[\mathscr{T},  p_{\T}^{0}, \alpha, \beta, a, b, y^{}_{\Omega}, f]$, which implements the active set strategy of \citep[Algorithm 2]{MR2556849}. In this step, the characterizations, given in Lemmas \ref{lem:projection_Uad0} and \ref{lem:projection_Uad1}, for the discrete variables $\bar{u}^{}_\T$ and $\bar{\lambda}^{}_\T$, are used.
\\
\textbf{A posteriori error estimation:}
\\
$\boldsymbol{3:}$ For each $K\in\mathscr{T}$ compute the local error indicator $E_K$ ($\mathcal{E}_K$) given in \eqref{eq:E} (\eqref{eq:mathcalE}).
\\
$\boldsymbol{4:}$ Mark an element $K$ for refinement if $E_K> \displaystyle\frac{1}{2}\max_{K'\in \mathscr{T}}E_{K'}$ \bigg($\mathcal{E}_K> \displaystyle\frac{1}{2}\max_{K'\in \mathscr{T}}\mathcal{E}_{K'}$\bigg).
\\
$\boldsymbol{5:}$ From step $\boldsymbol{4}$, construct a new mesh, using a longest edge bisection algorithm. Set $i \leftarrow i + 1$, and go to step $\boldsymbol{1}$.
\\
\hline
\end{tabular}
\vspace{-0.3cm}
\end{flushleft}
\end{table}
\begin{table}[!h]
\begin{flushleft}
\begin{tabular}{p{12cm}} 
\textbf{Algorithm 2:  Adaptive Semi-Smooth Newton Algorithm.}
\vspace{0.15cm}\\
\hline
\vspace*{0.1cm}
\textbf{Input:} Initial mesh $\mathscr{T}_0$, desired state $y^{}_{\Omega}$, constraints $a$ and $b$, regularization parameter $\alpha$, sparsity parameter $\beta$ and external source $f$.
\\
\textbf{Set:} $i=0$.
\\
\textbf{Active set strategy:}
\\
$\boldsymbol{1}$ Choose initial guesses $y_\T^0\in \mathbb{V}(\T)$ and $p_{\T}^{0}\in\mathbb{V}(\T)$, for the state and adjoint variables, respectively.
\\
$\boldsymbol{2}$ Compute $[\bar{y}^{}_\mathscr{T},\bar{p}^{}_\mathscr{T},\bar{u}^{}_\mathscr{T},\bar{\lambda}^{}_\mathscr{T}]=\textbf{Semi-Smooth}[\mathscr{T},  y_{\T}^{0}, p_{\T}^{0}, \alpha, \beta, a, b, y^{}_{\Omega}, f]$, which implements an adaption of the semi-smooth \EO{Newton strategy of \citep[Algorithm 1]{Brett} in conjunction with the characterization of the optimal control $\bar{u}_\T$ given in \cite[equation (4.4)]{MR2556849}}.
\\
\textbf{A posteriori error estimation:}
\\
$\boldsymbol{3}$ For each $K\in\mathscr{T}$ compute the local error indicator $\mathfrak{E}_K$ given in \eqref{eq:variational_E}. 
\\
$\boldsymbol{4}$ \FF{Mark an element $K$ for refinement if $\mathfrak{E}_K> \displaystyle\frac{1}{2}\max_{K'\in \mathscr{T}}\mathfrak{E}_{K'}$}.
\\
$\boldsymbol{5}$ From step $\boldsymbol{4}$, construct a new mesh, using a longest edge bisection algorithm. Set $i \leftarrow i + 1$, and go to step $\boldsymbol{1}$.
\\
\hline
\end{tabular}
\vspace{-0.3cm}
\end{flushleft}
\end{table}

For the numerical results, we define the total numbers of degrees of freedom as
\begin{equation*}
\left.
\begin{array}{rcl}
\textsf{Ndof}_0 & := & 2 \:\text{dim}(\mathbb{V}(\mathscr{T}))+\#\mathscr{T}, ~
\quad\text{if}\quad \mathbb{U}_{ad}(\T)=\mathbb{U}_{ad,0}(\T),\\
\textsf{Ndof}_1 & := & 3 \:\text{dim}(\mathbb{V}(\mathscr{T})), ~\quad\text{if}\quad \mathbb{U}_{ad}(\T) =\mathbb{U}_{ad,1}(\T),\\
\textsf{Ndof}_2 & := & 2 \:\text{dim}(\mathbb{V}(\mathscr{T})), ~ \quad\text{if} \quad\mathbb{U}_{ad}(\T) = \mathbb{U}_{ad}.
\end{array}
\right.
\end{equation*}

We recall that  $e = (e^{}_y,e^{}_p,e^{}_u,e^{}_{\lambda})$ and that
\begin{align*}
 \VERT e \VERT_{\Omega}^2 &= |  e_y |^2_{H^1(\Omega)} + | e_p |^2_{H^1(\Omega)} 
 + \| e_u \|^2_{L^2(\Omega)} + \| e_{\lambda} \|^2_{L^2(\Omega)},
 \\
 \| e \|_{\Omega}^2 &= \| e_y \|^2_{L^2(\Omega)} + \| e_p \|^2_{L^2(\Omega)} \
 + \| e_u \|^2_{L^2(\Omega)} + \| e_{\lambda} \|^2_{L^2(\Omega)}.
\end{align*}


The initial meshes for our numerical examples are shown in Figure \ref{fig:initial_meshes}.

\begin{figure}[!h]
\centering
\begin{minipage}{0.22\textwidth}\centering
\includegraphics[width=2cm,height=2cm,scale=0.5]{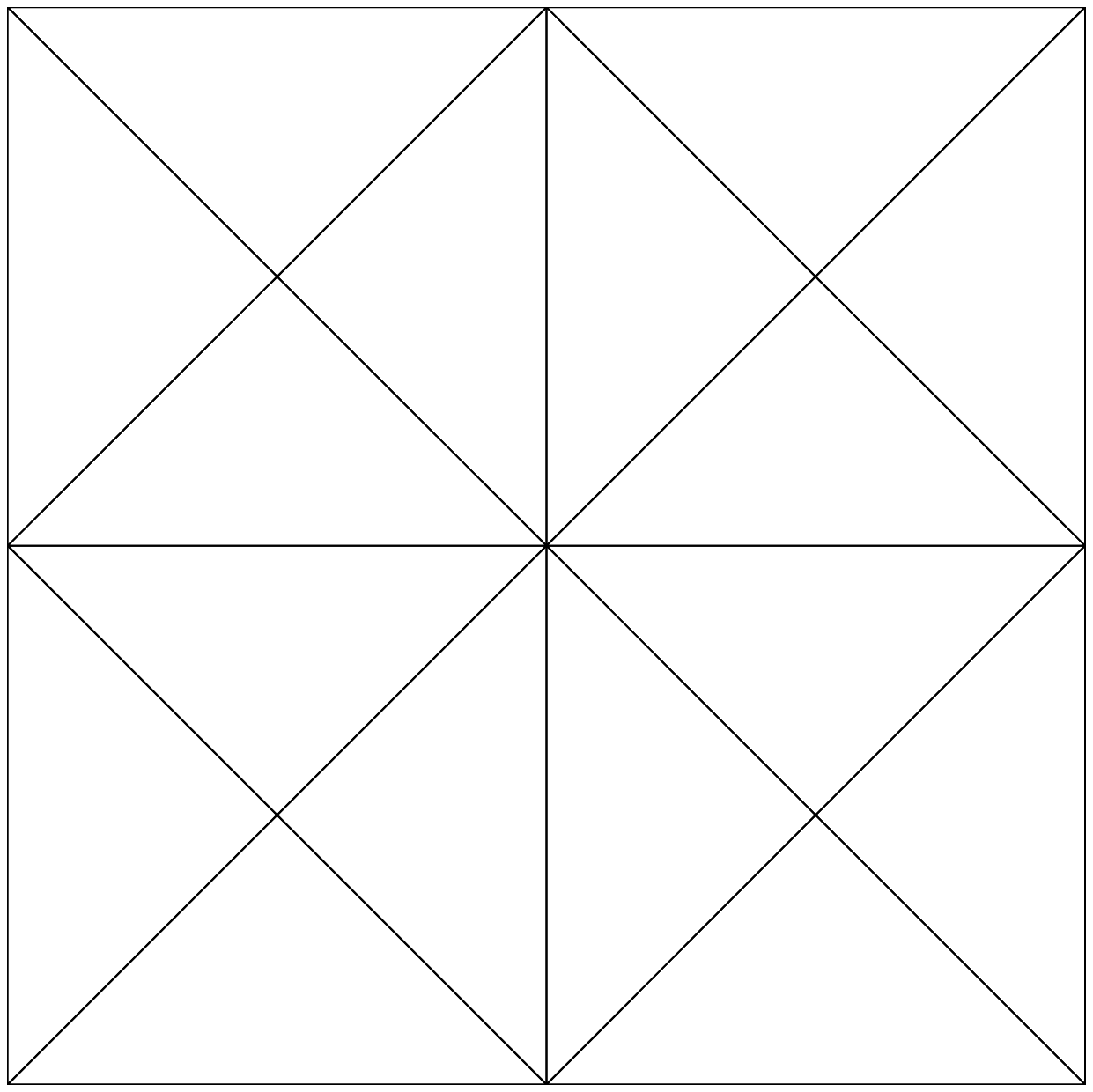}\\
\end{minipage}
\begin{minipage}{0.22\textwidth}\centering
\includegraphics[width=2cm,height=2cm,scale=0.5]{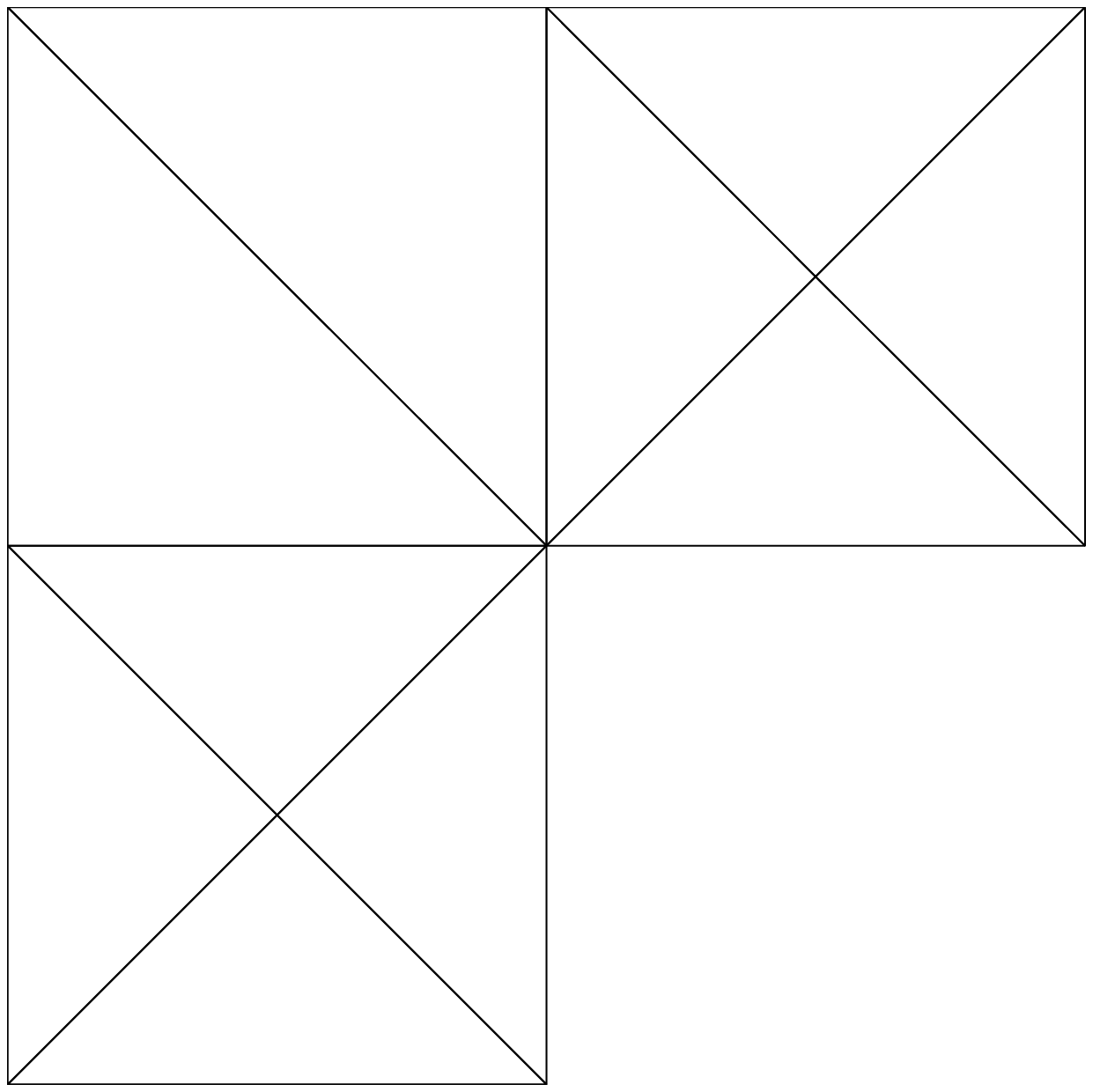}\\
\end{minipage}
\caption{The initial meshes used when the domain $\Omega$ is a square (Example 1) and a two--dimensional $L$--shaped domain (Example 2).}
\label{fig:initial_meshes}
\end{figure}

\FF{
\begin{remark}[Semi--smooth Newton algorithm]\label{variational_remark}
To solve the nonlinear system \eqref{eq:variational_optimal_control} associated to the variational discretization approach we proceed on the basis of \citep[Section  6]{Brett}: First, we define the operator $\mathcal{F}: \mathbb{V}(\T) \times \mathbb{V}(\T)  \rightarrow \mathbb{V}(\T)^{\star} \times \mathbb{V}(\T)^{\star}$ 
as
\begin{equation}\label{eq:discrete_var}
\mathcal{F}({y}_\T,{p}_\T)(v_\T,w_\T)
:=
\begin{bmatrix}
(\nabla {y}_\T, \nabla v_\T)_{L^2(\Omega)} -\left({u}_\T+f,v_\T\right)_{L^2(\Omega)} \\
(\nabla w_\T, \nabla {p}_\T)_{L^2(\Omega)} -({y}_\T-y_\Omega, w_\T)_{L^2(\Omega)}
\end{bmatrix}
\end{equation}
for $v_\T,w_\T\in \mathbb{V}(\T)$, where $u_\T=\Pi_{[a,b]}\left(-\alpha^{-1}(p_\T+\beta \lambda_\T)\right)$ and the discrete subgradient $\lambda_\T=\Pi_{[-1,1]}\left(-\beta^{-1} p_\T\right)$. Second, we notice that the discrete pair $(\bar{y}_\T,\bar{p}_\T)\in \mathbb{V}(\T)\times \mathbb{V}(\T)$ solves \eqref{eq:variational_optimal_control} if and only if 
\begin{equation}\label{eq:semi_smooth_0}
\mathcal{F}({y}_\T,{p}_\T)(v_\T,w_\T)=0
\end{equation}
for all $v_\T,w_\T\in \mathbb{V}(\T)$; see also \citep[equation (6.2)]{Brett} and \citep[equation(6.2)]{CHW:12}. Third, to solve \eqref{eq:semi_smooth_0}, we proceed by using a semismooth Newton method; see, for instance, \citep{MR1786137,MR1972217}. To accomplish this task, we observe that the control term $u_\T$ in \eqref{eq:discrete_var} can be characterized in view of the identity of \cite[equation (4.4)]{MR2556849}:
\begin{multline*}
{u}_\T= \alpha^{-1}\left \{ \max(0,-{p}_\T - \beta) +\min(0,-{p}_\T+\beta)\right. \\
\left.-\max(0, -{p}_\T -\beta -\alpha b) -\min(0,-{p}_\T+ \beta -\alpha a)\right\}.
\end{multline*}
In view of the terms $\max(0,\cdot)$ and $\min(0,\cdot)$, we notice that $\mathcal{F}({y}_\T,{p}_\T)$ is non--differentiable in the Fr\'echet sense. In spite of this fact, a generalized Newton method can be used; a semi-smooth Newton method. This amounts to apply the standard Newton method but considering now the following derivatives:
\begin{equation}\label{eq:semi_smooth_2}
\max(0,\chi)'
:=
\begin{cases}
1\quad \text{ if } \chi \geq 0, \\
0\quad \text{ if } \chi < 0,
\end{cases}
\quad
\text{ and }
\quad
\min(0,\chi)'
:=
\begin{cases}
1\quad \text{ if } \chi \leq 0, \\
0\quad \text{ if } \chi > 0.
\end{cases}
\end{equation}
Then, taking initial guesses $y^0_\T,p^0_\T\in \mathbb{V}(\T)$ and $n\in\mathbb{N}$, we consider the following semismooth Newton iteration
\begin{equation*}
\begin{bmatrix}
y_\T^{n+1} \\
p_\T^{n+1}
\end{bmatrix}
=
\begin{bmatrix}
y_\T^{n} \\
p_\T^{n}
\end{bmatrix}
+
\begin{bmatrix}
\delta y_\T \\
\delta p_\T
\end{bmatrix},
\end{equation*}
where the incremental term $(\delta y_\T, \delta p_\T)\in \mathbb{V}(\T)\times \mathbb{V}(\T)$ solves 
\begin{equation*}
[\mathcal{F}'(y_\T^n,p_\T^n)(\delta y_\T, \delta p_\T)](v_\T,w_\T)=-\mathcal{F}(y_\T^n,p_\T^n)(v_\T,w_\T),
\end{equation*}
for all $v_\T, w_\T\in \mathbb{V}(\T)$, with
\begin{equation*}
[\mathcal{F}'(y_\T^n,p_\T^n)(\delta y_\T, \delta p_\T)](v_\T,w_\T)
=
\begin{bmatrix}
(\nabla \delta {y}_\T, \nabla v_\T)_{L^2(\Omega)} -\left({\xi}_\T(p_\T^n) \delta p_\T,v_\T\right)_{L^2(\Omega)} \\
(\nabla w_\T, \nabla \delta {p}_\T)_{L^2(\Omega)} -(\delta {y}_\T, w_\T)_{L^2(\Omega)}
\end{bmatrix},
\end{equation*}
and
\begin{multline}\label{eq:semi_smooth_3}
{\xi}_\T(p_\T^n):=\alpha^{-1}\left[\max(0,-{p}_\T^n - \beta)' +\min(0,-{p}_\T^n+\beta)'\right. \\
\left.-\max(0, -{p}_\T^n -\beta -\alpha b)' -\min(0,-{p}_\T^n+ \beta -\alpha a)'\right].
\end{multline}
The derivatives in \eqref{eq:semi_smooth_3} are defined in \eqref{eq:semi_smooth_2}. For further details we wefer the reader to \cite[Section 6]{Brett}.
\end{remark}
}


\EO{
\begin{remark}[optimal experimental rates of convergence]
We state that approximation errors or error estimators exhibit optimal experimental rates of convergence
to refer that they achieve maximal decay rate in terms of approximation within corresponding norms and  discrete spaces \cite{NSV:09}.
\end{remark}
} 

We now provide two numerical experiments. In \FF{both examples we consider problems} where an exact solution can be obtained: we fix the optimal state and adjoint state variables and compute the exact optimal control, its associated subgradient, the desired state $y^{}_{\Omega}$ and the source term $f$, by invoking the projection formulas \eqref{proju} and \eqref{projection_lambda} and the state and adjoint equations \eqref{eq:weak_state_eq} and \eqref{eq:adjoint_equation}, respectively.
~\\

\textbf{Example 1:} We set $\Omega=(0,1)^2$, $a=-3$, and $b=3$. The exact optimal state and adjoint state are given by
\begin{equation*}
\bar{y}=x^{}_1 x^{}_2(x^{}_1-1)(x^{}_2-1)\arctan((x_1-0.5)/0.01), \qquad
\bar{p}=\FF{20 x_1 x_2}(1-x_1)(1-x_2).
\end{equation*}
The purpose of this example is to investigate the performance of the a posteriori error estimators when varying the parameters $\alpha$ and $\beta$. First, we investigate the effect of \EO{diminishing} the regularization parameter $\alpha$ by considering
\begin{equation}\label{eq:varying_alpha}
\beta=7\cdot 10^{-1}\quad\textrm{and}\quad
\FF{\alpha\in \left \{10^{0},10^{-1},10^{-2},10^{-3} \right \}}.
\end{equation}
Second, we investigate the effect of \EO{diminishing} the sparsity parameter $\beta$ by considering
\begin{equation}\label{eq:varying_beta}
\alpha=10^{-3}\quad\textrm{and}\quad
\FF{\beta\in \left \{10^{0},10^{-1},10^{-2},10^{-3} \right \}}.
\end{equation}
\FF{Third, we set $\alpha=10^{-3}$ and $\beta=10^{0}$ and study the effectivity indices associated to the piecewise constant, the piecewise linear, and the variational discretization schemes which are given by $\mathcal{E}/\VERT e \VERT_\Omega$, $E/\|e\|_\Omega$, and $\mathfrak{E}/\|e\|_\Omega$, respectively. }
%
\FF{In Figures \ref{ex_1}, \ref{ex_2}, and \ref{ex_3} we present, \EO{for the piecewise constant, the piecewise linear, and the variational discretization schemes, respectively, the experimental rates of convergence for the involved posteriori error estimator, its individual contributions, the correspondying total approximation error and its individual contributions. We observe, for all the values of the parameters $\alpha$ and $\beta$ considered in \eqref{eq:varying_alpha} and \eqref{eq:varying_beta}, optimal experimental rates of convergence.
We set $\alpha=10^{-3}$ and $\beta=10^{0}$ and present, in Figure \ref{ex_eff}, effectivity indices.} We observe that, when the total number of degrees of freedom increases, the effectivity index its stabilized around the value of $5.2$ for the piecewise constant discretization, around the value of $1.1$ for the piecewise linear discretization, and around the value of $6.3$ for the variational discretization; this shows the accuracy of each proposed a posteriori error estimator with respect to its associated total error.}

~\\ 
\FF{\textbf{Example 2:} We set $\Omega=(-1,1)^2\setminus[0,1)\times(-1,0]$, $a=-0.6$, $b=1$, $\alpha=10^{-3}$, and $\beta=0.2$. The exact optimal state and adjoint state are given, in polar coordinates $(\rho,\omega)$ with $\omega\in[0,3\pi/2]$, by
\begin{align}\label{eq:sol_y_p_L}
\begin{split}
\bar{y}&=0.2\sin(\pi(\rho\sin(\omega)+1)/2)\sin(\pi(\rho \cos(\omega)+1)/2)\rho^{2/3}\sin(2\omega/3), 
\\
\bar{p}&=0.5\cos(\pi\rho\sin(\omega)/2)\sin(\pi(\rho\cos(\omega)+1)/2)\rho^{2/3}\sin(2\omega/3).
\end{split}
\end{align}
The purpose of this example is to investigate the performance of the devised a posteriori error estimators in a non--convex domain. \FF{We recall that the convexity of the domain is an assumption in the statements of Theorems \ref{global_reliability_l2} and \ref{global_reliability_variational}.}

\EO{Notice that} the state $\bar{y}$ and the adjoint state $\bar{p}$ given in \eqref{eq:sol_y_p_L} exhibit reduced regularity properties: \EO{$\bar{y}, \bar{p} \notin H^2(\Omega)$. Consequently, optimal experimental rates of convergence cannot be expected for the error approximation of the state and adjoint state variables; see, for instance, \citep[Corollary 5.1]{NSV:09}.}
In Figure \ref{ex_4}, we present for each discretization scheme, \EO{the experimental decay of the total approximation error and total error estimator within uniform}
(A.1)--(C.1) and adaptive refinement (A.2)--(C.2). \EO{As theory predicts, we observe,} for uniform refinement, reduced experimental rates of convergence. 
Nevertheless, when whichever of the proposed adaptive schemes is used, we recover optimal experimental rates of convergence for \EO{the involved approximation errors.}
We also observe that, for adaptive refinement, the total error estimator behaves similarly to the total approximation error for all the discretization schemes. \EO{In fact,}
in the last row of Figure \ref{ex_4} we present 
effectivity indices for each discretization scheme: \EO{it can be observed that} when the total number of degrees of freedom increases,  the effectivity index its stabilized around the value of $3.7$ for the piecewise constant discretization, around the value of $0.7$ for the piecewise linear discretization, and around the value of $0.3$ for the variational discretization. \EO{Let us consider $\alpha=10^{-3}$ and $\beta=2\cdot 10^{-1}$. In Figure \ref{ex_4_2} we present, for the piecewise constant, the piecewise linear, and the variational discretization schemes, experimental rates of convergence for each contribution of the error estimator and approximation error, for uniform and adaptive refinement. We observe, for uniform refinement, that for each discretization scheme there are contributions of the total approximation error and}
contributions of the total error estimator that exhibit a reduced experimental decay. Nevertheless, for adaptive refinement, we observe for each discretization scheme that all the contributions of the total approximation error as well as all the contributions of the total error estimator exhibit optimal experimental rates of convergence. 
%
}

   
From the presented numerical examples several general conclusions can be drawn: 
\begin{itemize}

\item[$\bullet$] \FF{All the total approximation errors, as well as each of their contributions, associated to each discretization scheme exhibit optimal experimental rates of convergence for all the values of the parameters $\alpha$ and $\beta$ considered in the experiments that we have performed. This suggests that the AFEMs described in \textbf{Algorithms 1} and \textbf{2} outperform the FEMs described in Section \ref{section4}.}

\item[$\bullet$] We observe that, even when the assumption of convexity of $\Omega$ is violated, the devised \FF{AFEMs} deliver optimal experimental rates of convergence \FF{for all the involved approximation errors and their contributions}; see Figure \ref{ex_4_2}.



\end{itemize}
\textbf{Acknowledgment.} We would like to thank R. Rankin for a fruitful discussion regarding the bubble function used in the efficiency analysis of Section \ref{subsubsec:E}.

\begin{figure}[!h]
\centering
\begin{minipage}{0.23\textwidth}\centering
\psfrag{total-error-varying-alpha}{\quad \quad \normalsize{Total error $\VERT e \VERT_{\Omega}$ varying $\alpha$}}
\psfrag{alpha-0}{\small{$\alpha=10^0$}}
\psfrag{alpha-1}{\small{$\alpha=10^{-1}$}}
\psfrag{alpha-2}{\small{$\alpha=10^{-2}$}}
\psfrag{alpha-3}{\small{$\alpha=10^{-3}$}}
\psfrag{alpha-4}{\small{$\alpha=10^{-4}$}}
\psfrag{alpha-5}{\small{$\alpha=10^{-5}$}}
\psfrag{ndof}{\footnotesize{$\textrm{Ndof}^{-1/2}_0$}}
\psfrag{Ndofs}{\normalsize{$\textrm{Ndof}^{}_0$}}
\includegraphics[trim={0 0 0 0},clip,width=3.0cm,height=2.8cm,scale=0.6]{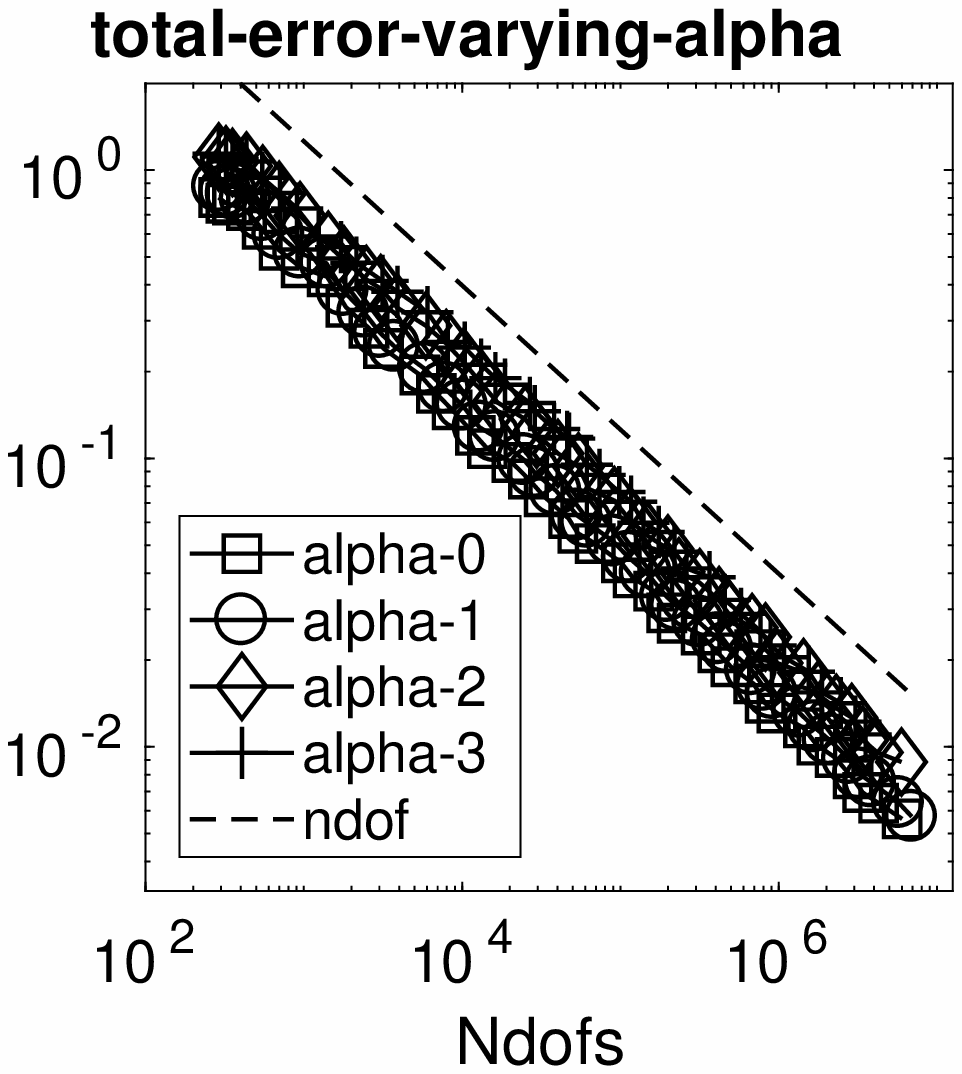}\\
\psfrag{ndof}{\footnotesize{$\textrm{Ndof}^{-1/2}_0$}}
\tiny{(A.1)} \\
\psfrag{total-error-varying-alpha}{\quad \quad \normalsize{State error $|e_y|_{H^1(\Omega)}$ varying $\alpha$}}
\psfrag{ndof}{\footnotesize{$\textrm{Ndof}^{-1/2}_0$}}
\includegraphics[trim={0 0 0 0},clip,width=3.0cm,height=2.8cm,scale=0.6]{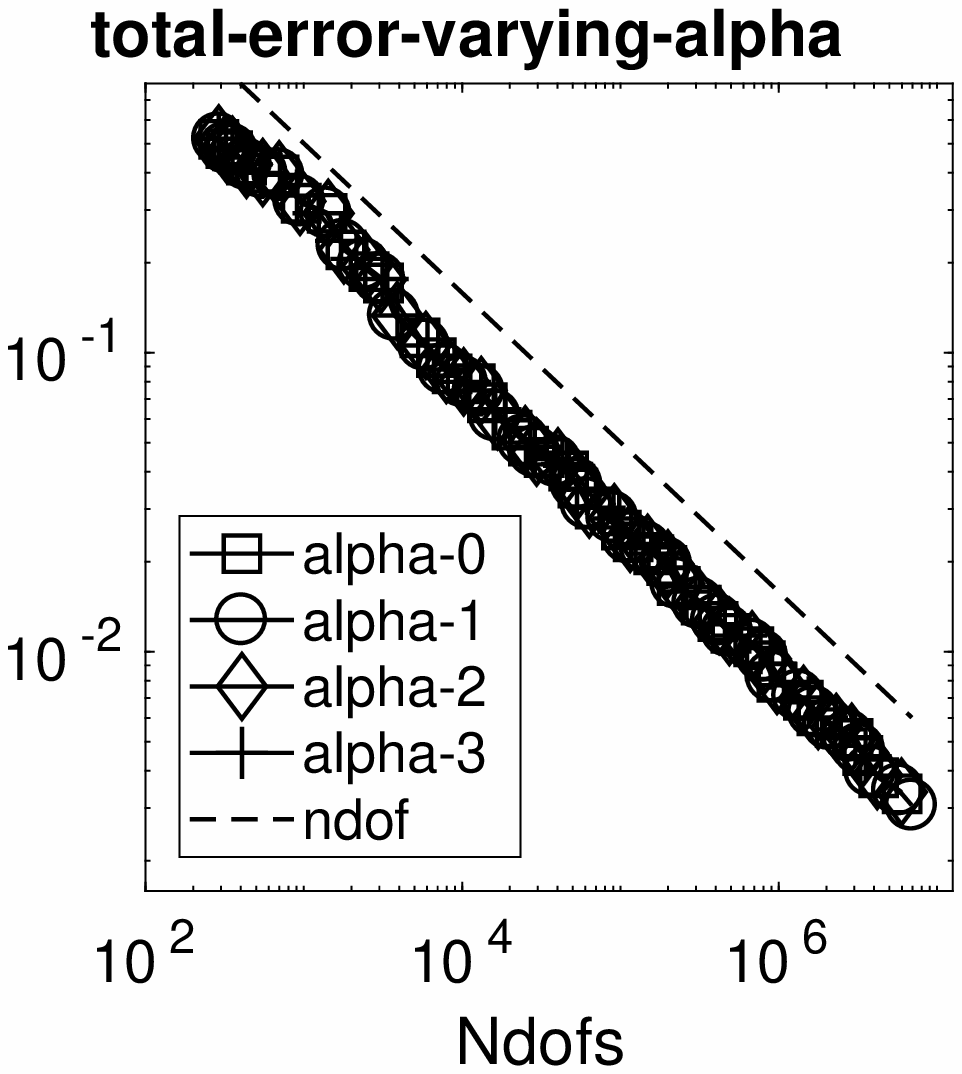}\\
\psfrag{ndof}{\footnotesize{$\textrm{Ndof}^{-1/2}_0$}}
\tiny{(A.2)} \\
\psfrag{total-error-varying-alpha}{\quad \quad \normalsize{Adjoint error $|e_p|_{H^1(\Omega)}$ varying $\alpha$}}
\includegraphics[trim={0 0 0 0},clip,width=3.0cm,height=2.8cm,scale=0.6]{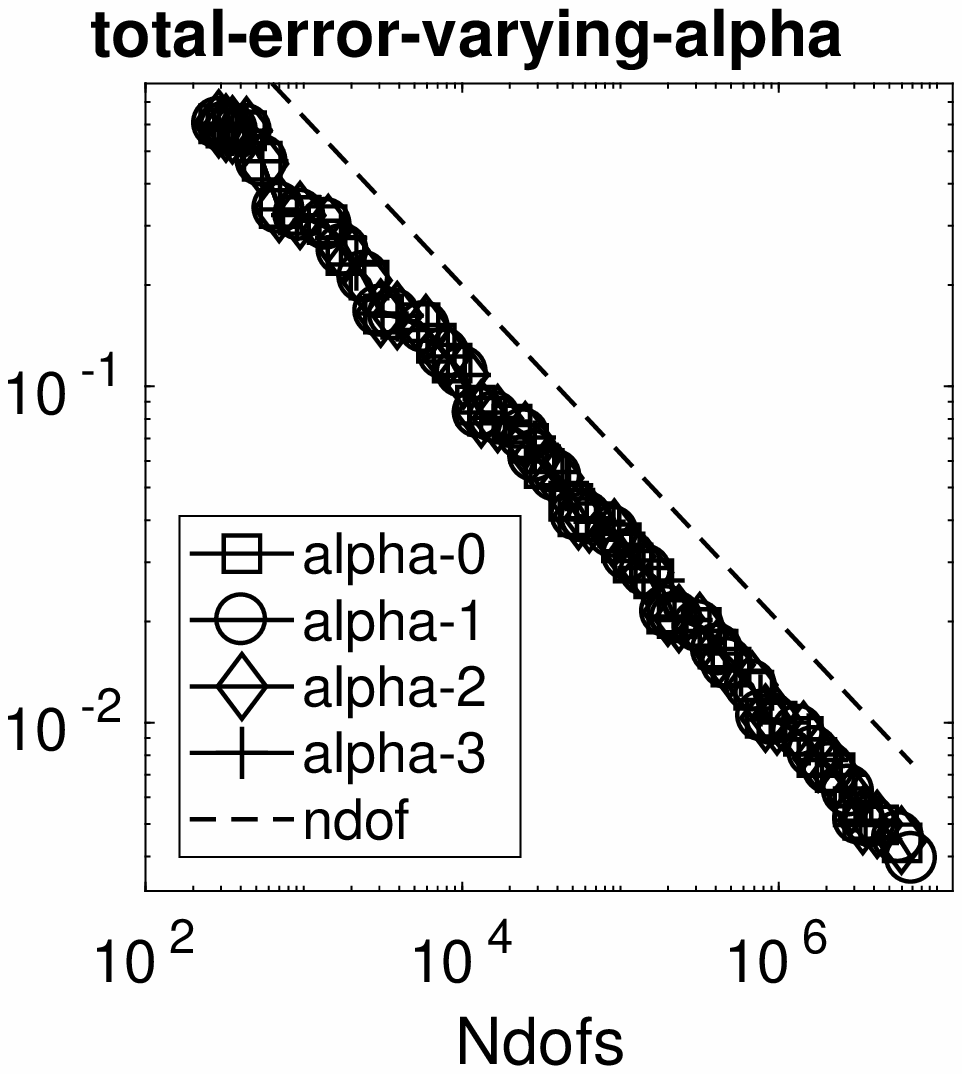}\\
\tiny{(A.3)}\\
\psfrag{total-error-varying-alpha}{\quad \quad \normalsize{Control error $\|e_u\|_{L^2(\Omega)}$ varying $\alpha$}}
\includegraphics[trim={0 0 0 0},clip,width=3.0cm,height=2.8cm,scale=0.6]{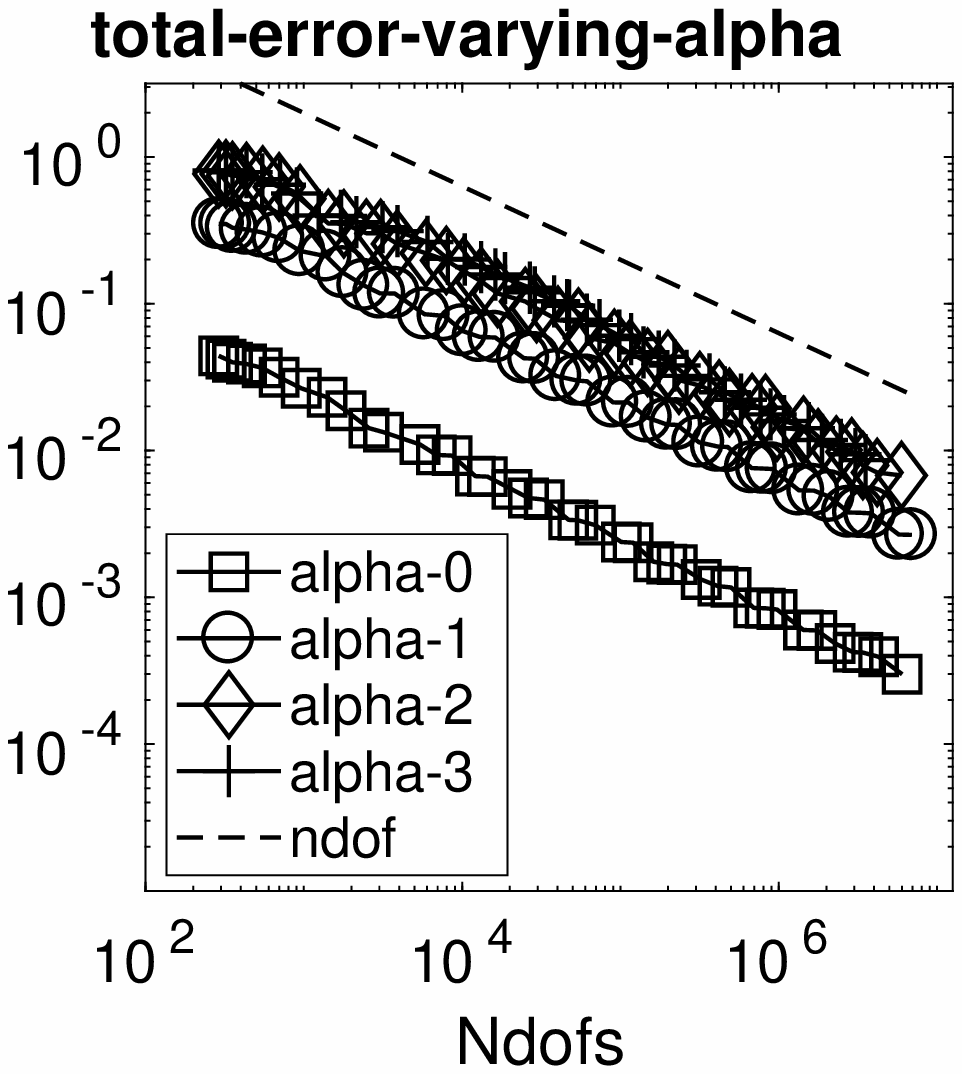}\\
\tiny{(A.4)} \\
\psfrag{total-error-varying-alpha}{\normalsize{Subgradient error $\|e_{\lambda}\|_{L^2(\Omega)}$ varying $\alpha$}}
\includegraphics[trim={0 0 0 0},clip,width=3.0cm,height=2.8cm,scale=0.6]{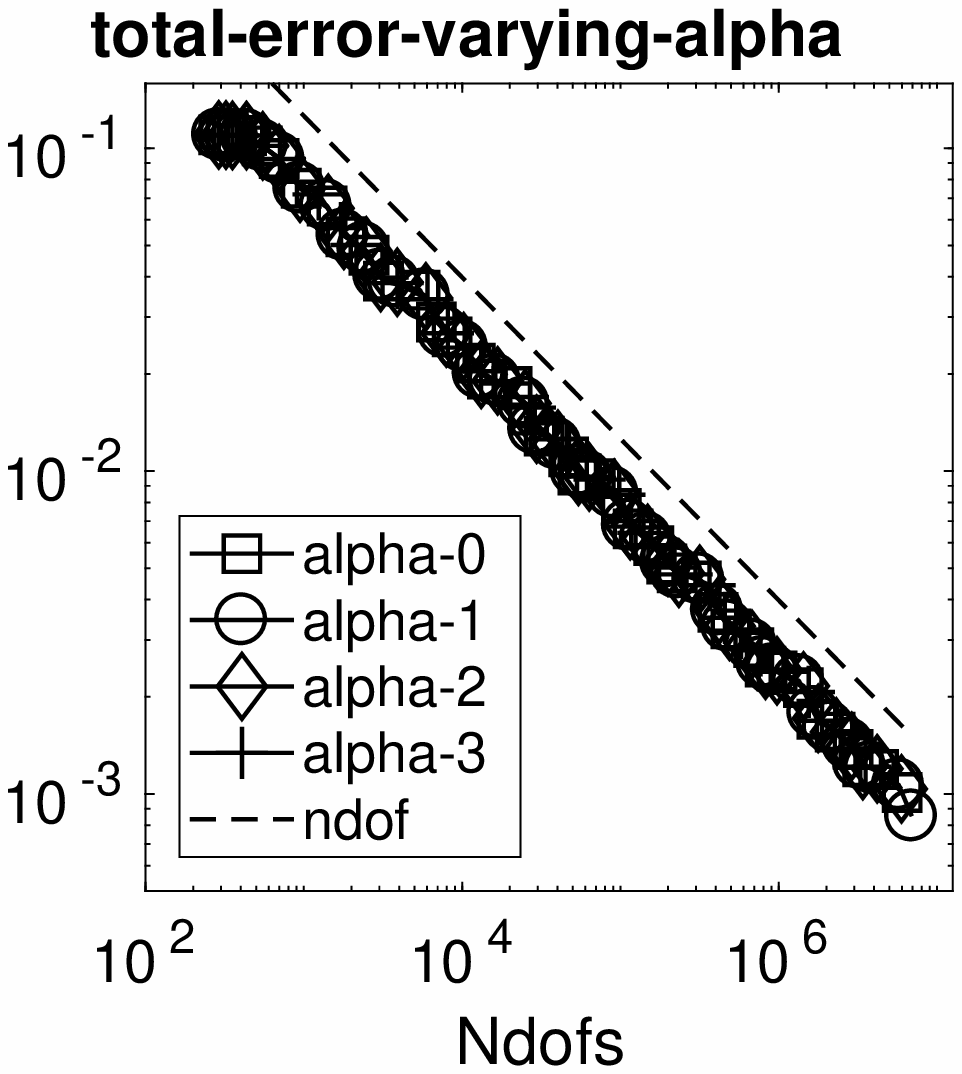}\\
\tiny{(A.5)}
\end{minipage}
\begin{minipage}{0.23\textwidth}\centering
\psfrag{total-error-varying-alpha}{\quad \quad \normalsize{Total estimator $\mathcal{E}$ varying $\alpha$}}
\psfrag{alpha-0}{\small{$\alpha=10^0$}}
\psfrag{alpha-1}{\small{$\alpha=10^{-1}$}}
\psfrag{alpha-2}{\small{$\alpha=10^{-2}$}}
\psfrag{alpha-3}{\small{$\alpha=10^{-3}$}}
\psfrag{alpha-4}{\small{$\alpha=10^{-4}$}}
\psfrag{alpha-5}{\small{$\alpha=10^{-5}$}}
\psfrag{ndof}{\footnotesize{$\textrm{Ndof}^{-1/2}_0$}}
\psfrag{Ndofs}{\normalsize{$\textrm{Ndof}^{}_0$}}
\includegraphics[trim={0 0 0 0},clip,width=3.0cm,height=2.8cm,scale=0.6]{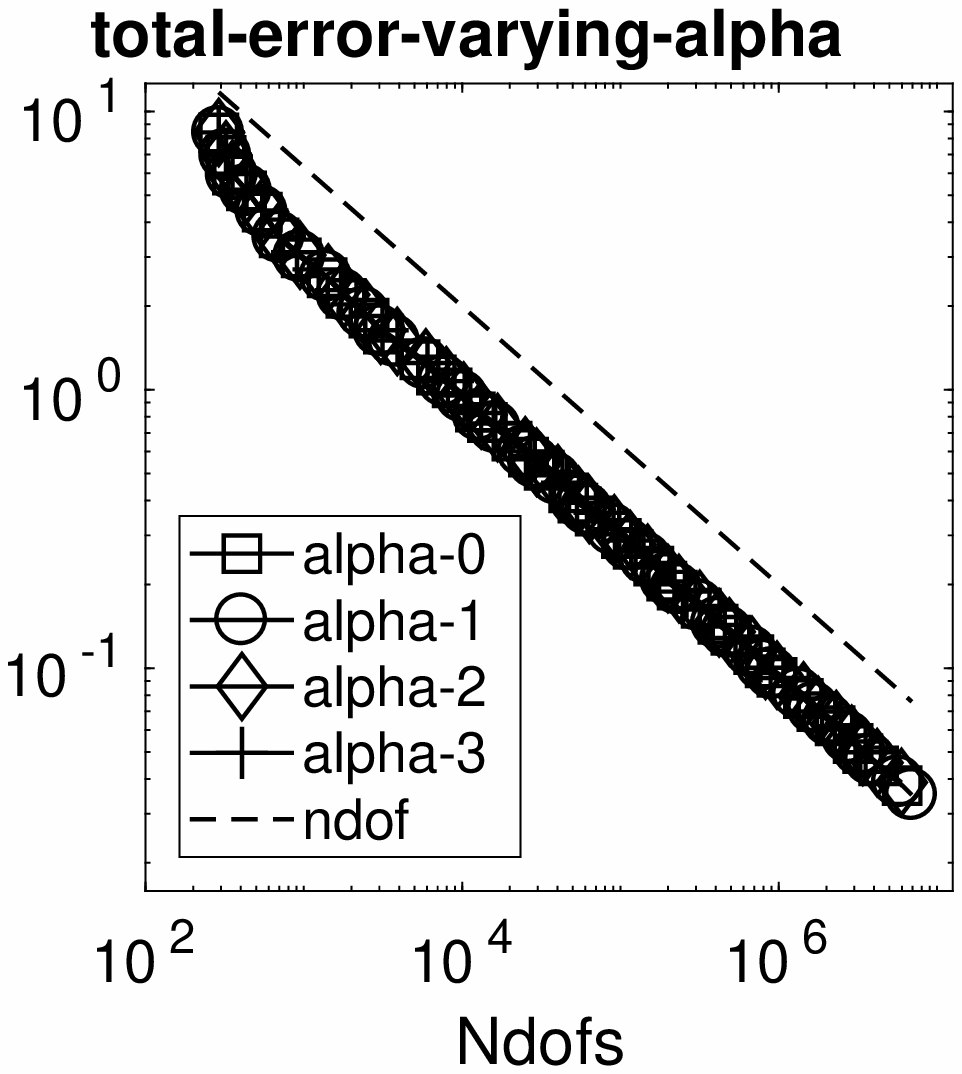}\\
\tiny{(B.1)}\\
\psfrag{total-error-varying-alpha}{\quad \normalsize{Individual contribution $\mathcal{E}_{y}$ varying $\alpha$}}
\includegraphics[trim={0 0 0 0},clip,width=3.0cm,height=2.8cm,scale=0.6]{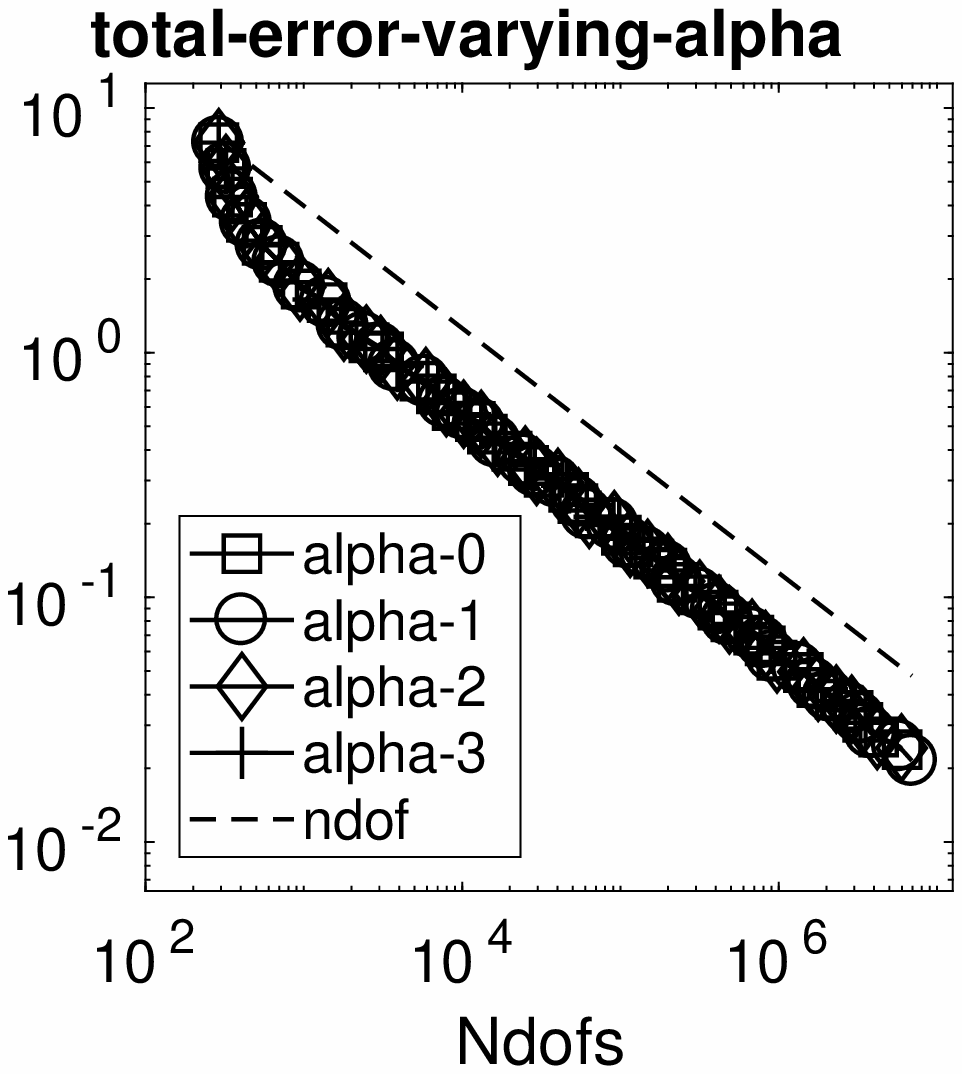}\\
\tiny{(B.2)}\\
\psfrag{total-error-varying-alpha}{\quad \normalsize{Individual contribution $\mathcal{E}_{p}$ varying $\alpha$}}
\includegraphics[trim={0 0 0 0},clip,width=3.0cm,height=2.8cm,scale=0.6]{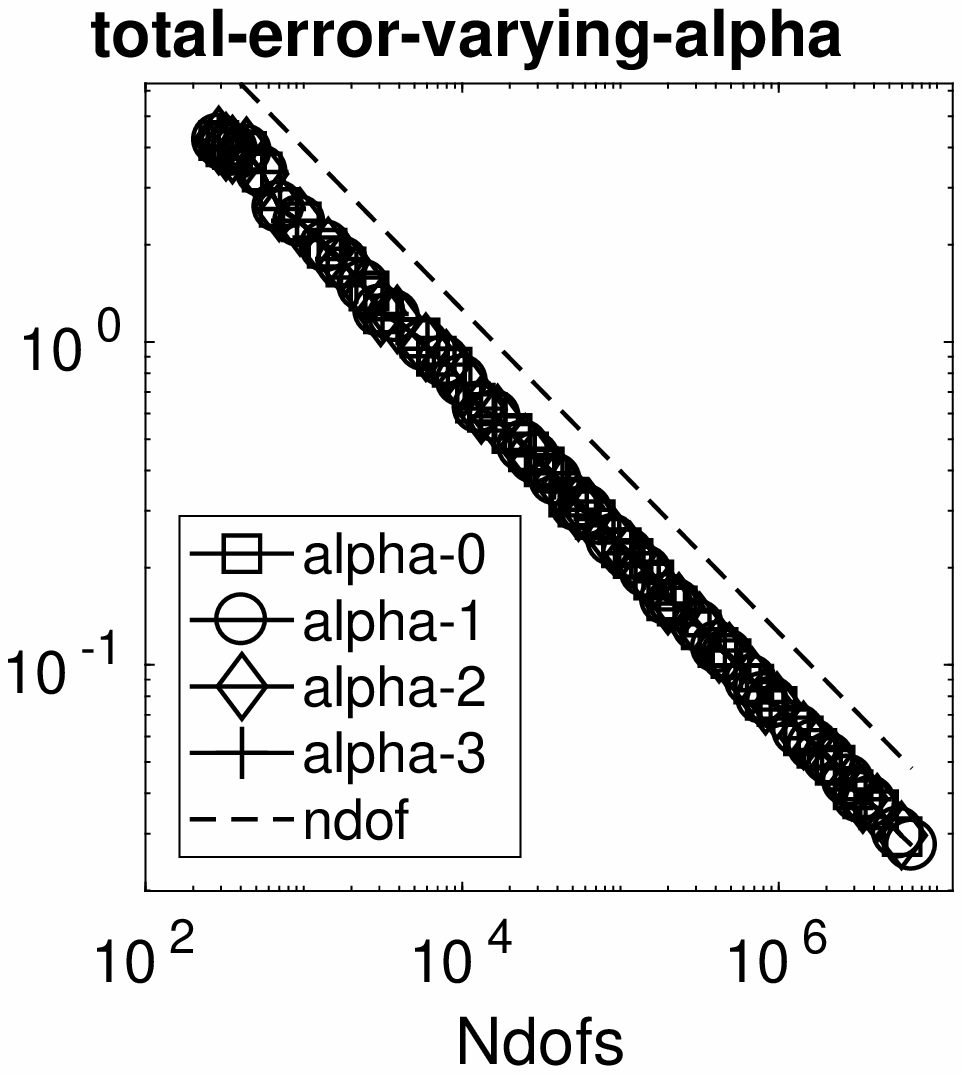}\\
\tiny{(B.3)}\\
\psfrag{total-error-varying-alpha}{\quad \normalsize{Individual contribution $E_{u}$ varying $\alpha$}}
\includegraphics[trim={0 0 0 0},clip,width=3.0cm,height=2.8cm,scale=0.6]{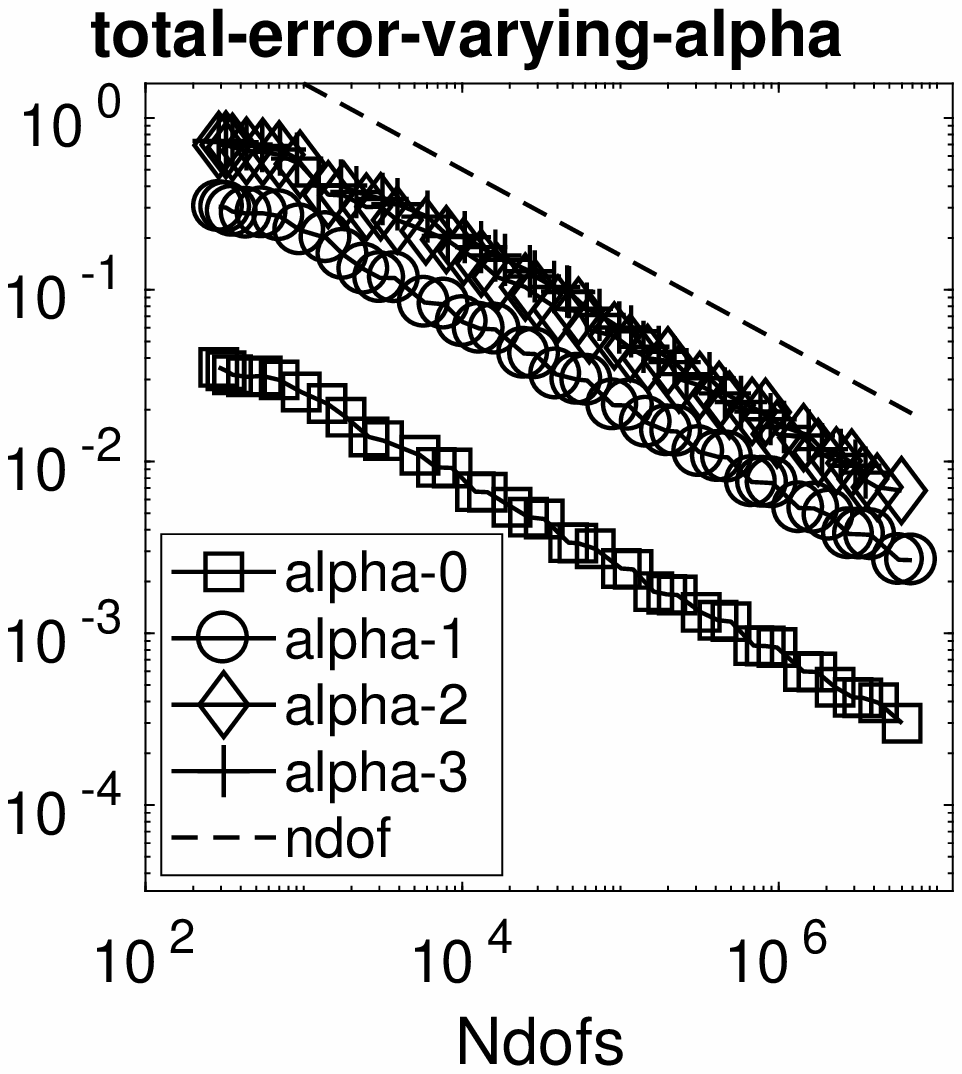}\\
\tiny{(B.4)}\\
\psfrag{total-error-varying-alpha}{\normalsize{Individual contribution $E_{\lambda}$ varying $\alpha$}}
\includegraphics[trim={0 0 0 0},clip,width=3.0cm,height=2.8cm,scale=0.6]{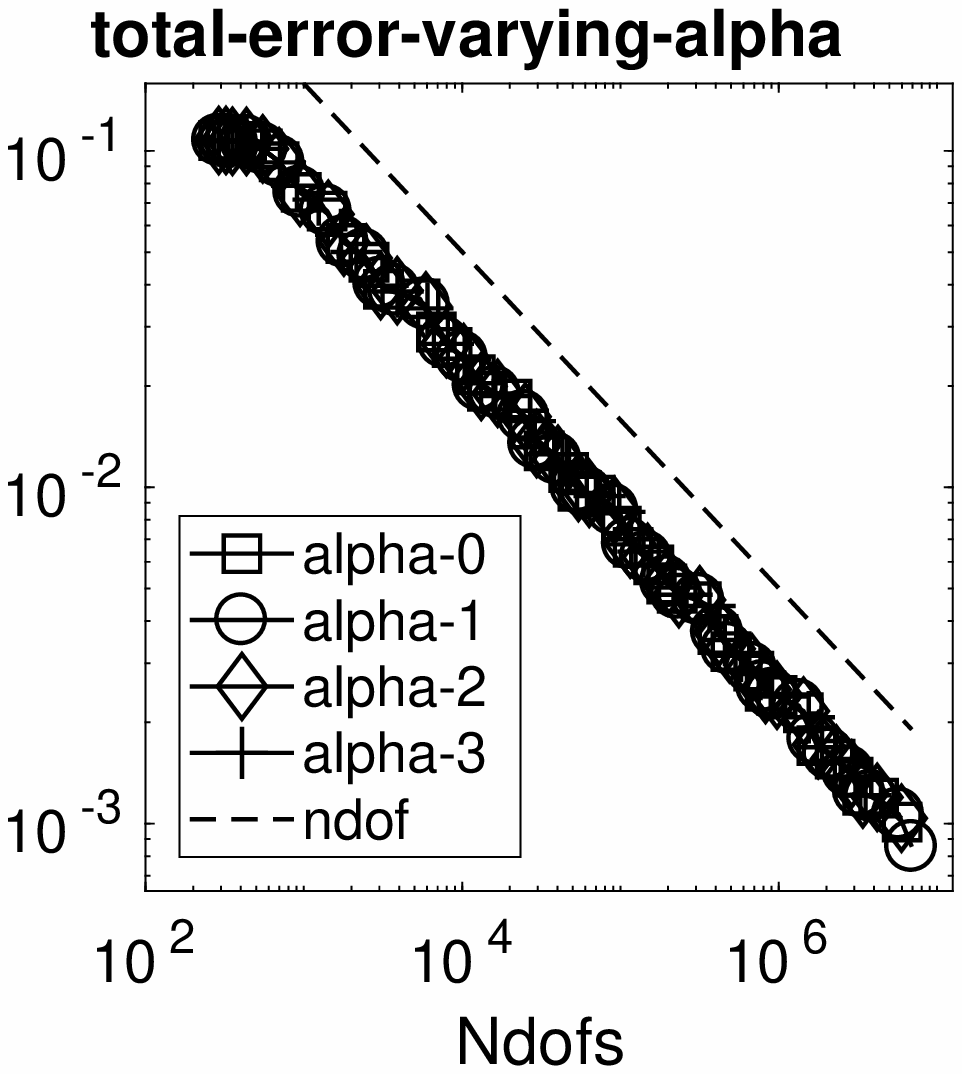}\\
\tiny{(B.5)}
\end{minipage}
\begin{minipage}{0.23\textwidth}\centering
\psfrag{total-error-varying-beta}{\quad \quad \normalsize{Total error $\VERT e\VERT_{\Omega}$ varying $\beta$}}
\psfrag{betaa-0}{\small{$\beta=10^0$}}
\psfrag{betaa-1}{\small{$\beta=10^{-1}$}}
\psfrag{betaa-2}{\small{$\beta=10^{-2}$}}
\psfrag{betaa-3}{\small{$\beta=10^{-3}$}}
\psfrag{betaa-4}{\small{$\beta=10^{-4}$}}
\psfrag{betaa-5}{\small{$\beta=10^{-5}$}}
\psfrag{ndof}{\footnotesize{$\textrm{Ndof}^{-1/2}_0$}}
\psfrag{Ndofs}{\normalsize{$\textrm{Ndof}^{}_0$}}
\includegraphics[trim={0 0 0 0},clip,width=3.0cm,height=2.8cm,scale=0.6]{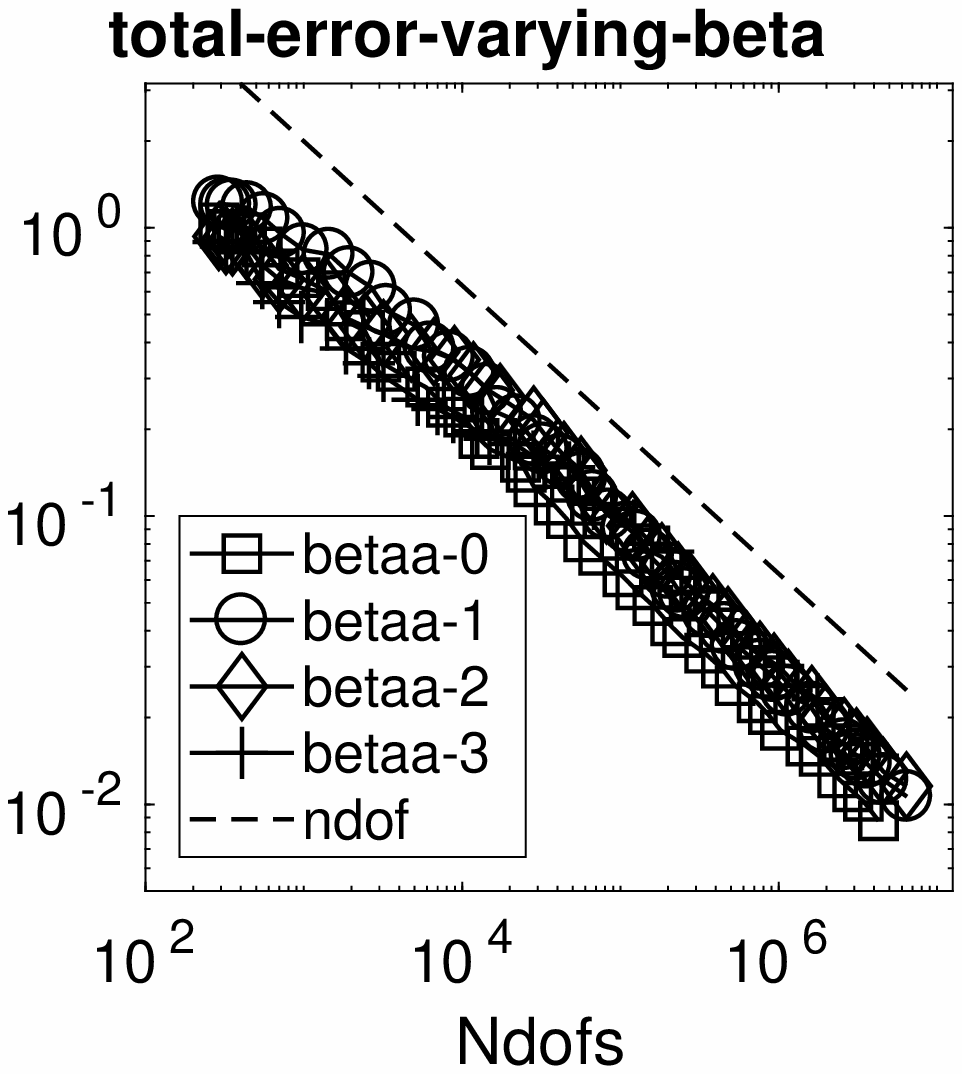}\\
\tiny{(C.1)}\\
\psfrag{total-error-varying-beta}{\quad \quad \normalsize{State error $|e_{y}|_{H^{1}(\Omega)}$ varying $\beta$}}
\includegraphics[trim={0 0 0 0},clip,width=3.0cm,height=2.8cm,scale=0.6]{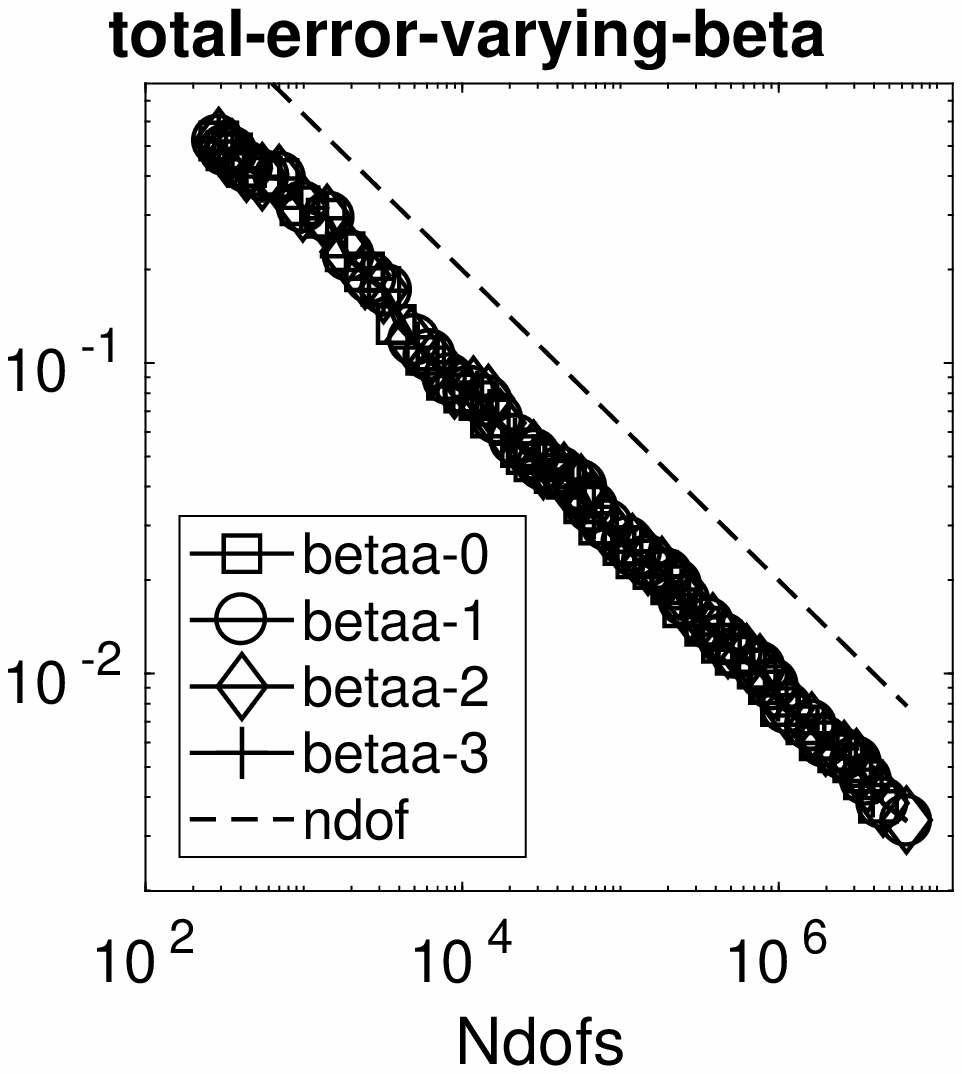}\\
\tiny{(C.2)}\\
\psfrag{total-error-varying-beta}{\quad \quad \normalsize{Adjoint error $|e_{p}|_{H^{1}(\Omega)}$ varying $\beta$}}
\includegraphics[trim={0 0 0 0},clip,width=3.0cm,height=2.8cm,scale=0.6]{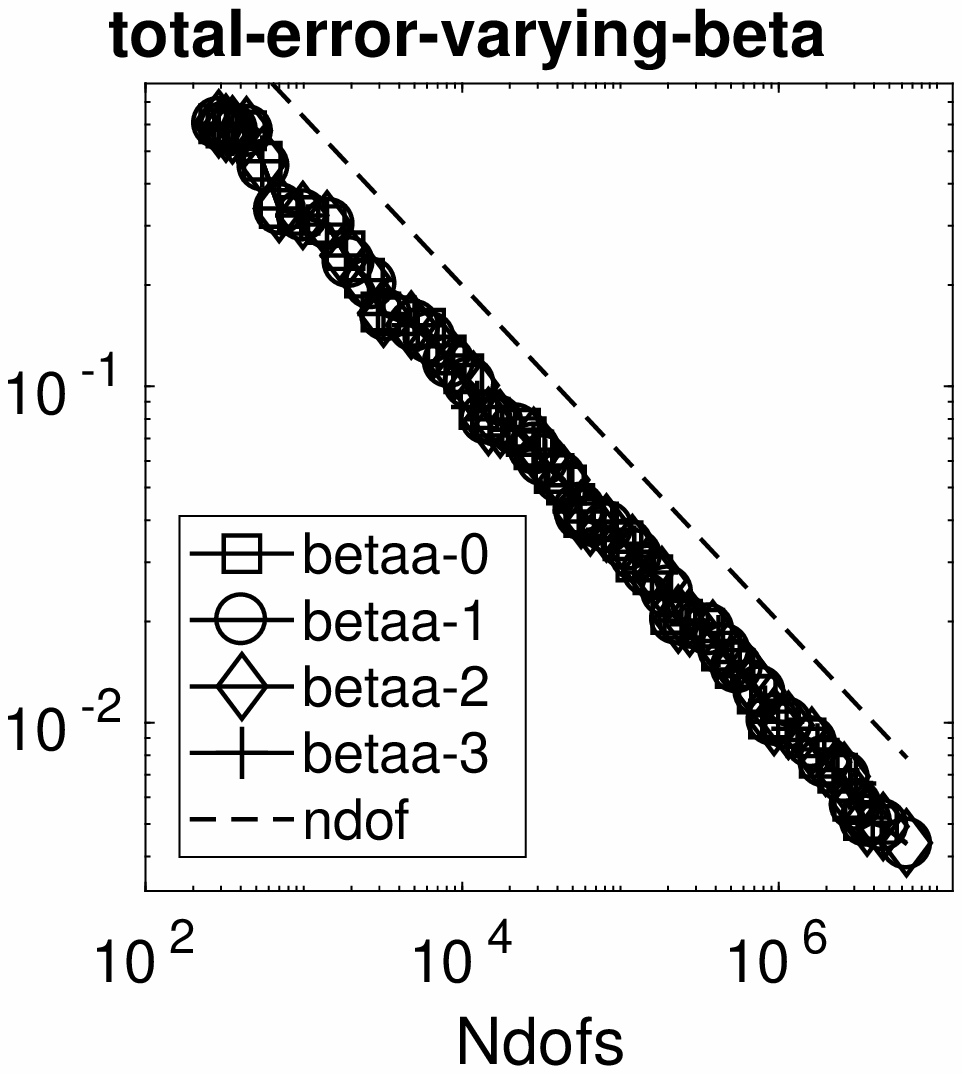}\\
\tiny{(C.3)}\\
\psfrag{total-error-varying-beta}{\quad \quad \normalsize{Control error $\|e_{u}\|_{L^{2}(\Omega)}$ varying $\beta$}}
\includegraphics[trim={0 0 0 0},clip,width=3.0cm,height=2.8cm,scale=0.6]{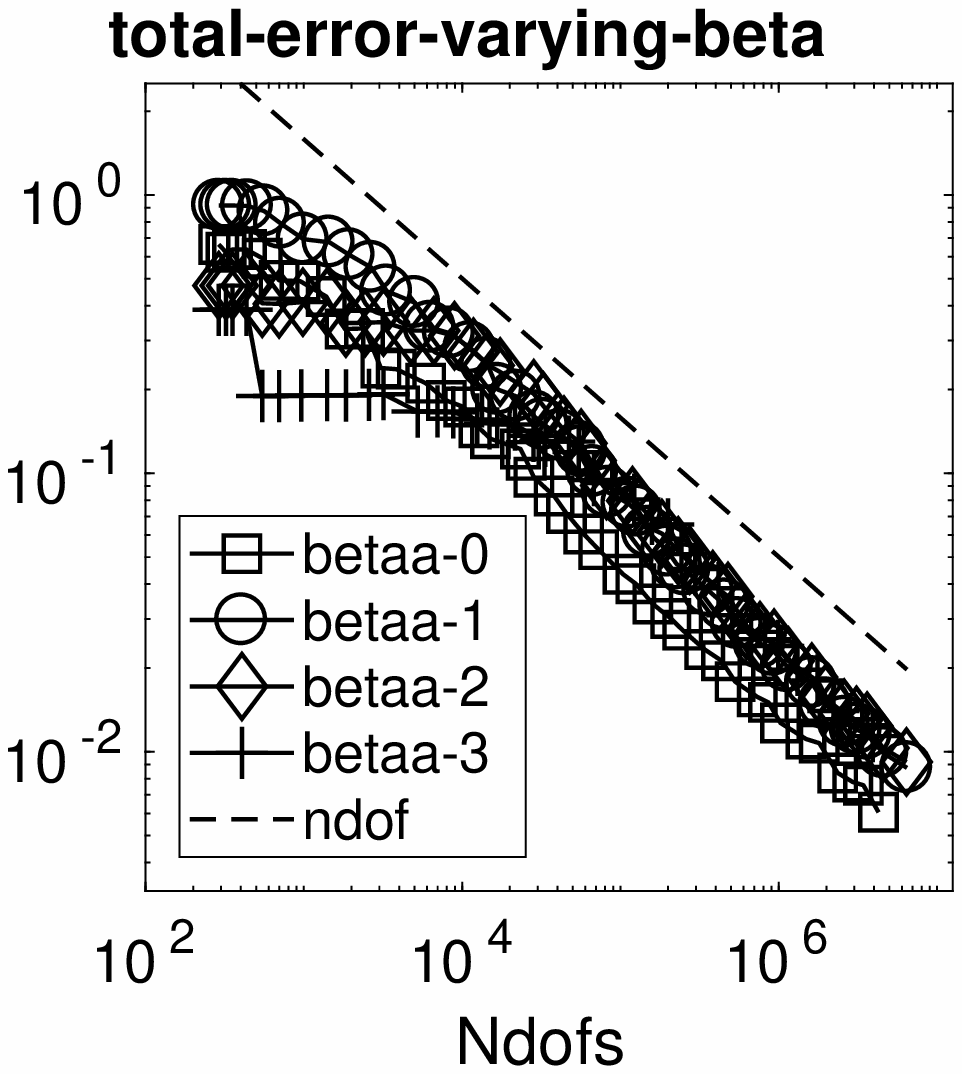}\\
\tiny{(C.4)}\\
\psfrag{total-error-varying-beta}{\normalsize{Subgradient error $\|e_{\lambda}\|_{L^{2}(\Omega)}$ varying $\beta$}}
\includegraphics[trim={0 0 0 0},clip,width=3.0cm,height=2.8cm,scale=0.6]{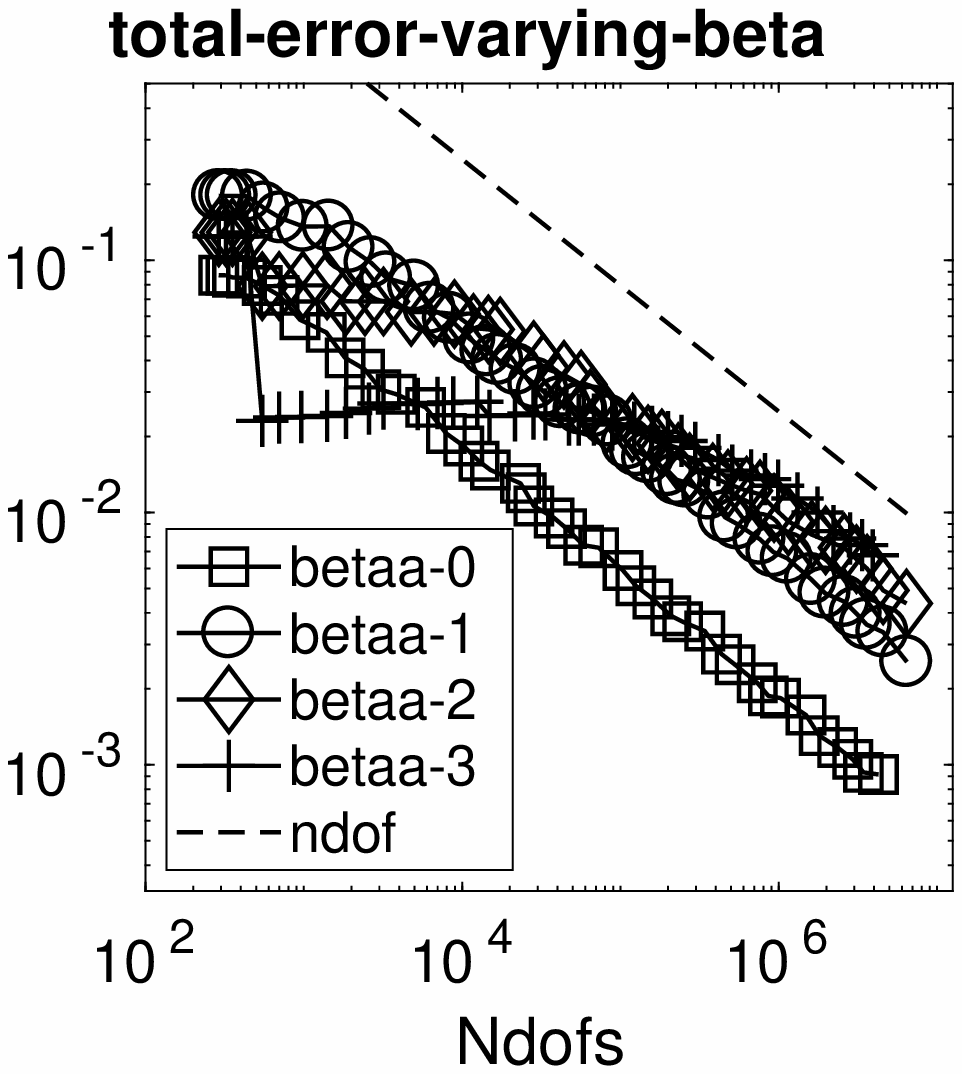}\\
\tiny{(C.5)}
\end{minipage}
\begin{minipage}{0.23\textwidth}\centering
\psfrag{total-error-varying-beta}{\quad \quad \normalsize{Total estimator $\mathcal{E}$ varying $\beta$}}
\psfrag{betaa-0}{\small{$\beta=10^0$}}
\psfrag{betaa-1}{\small{$\beta=10^{-1}$}}
\psfrag{betaa-2}{\small{$\beta=10^{-2}$}}
\psfrag{betaa-3}{\small{$\beta=10^{-3}$}}
\psfrag{betaa-4}{\small{$\beta=10^{-4}$}}
\psfrag{betaa-5}{\small{$\beta=10^{-5}$}}
\psfrag{ndof}{\footnotesize{$\textrm{Ndof}^{-1/2}_0$}}
\psfrag{Ndofs}{\normalsize{$\textrm{Ndof}^{}_0$}}
\includegraphics[trim={0 0 0 0},clip,width=3.0cm,height=2.8cm,scale=0.6]{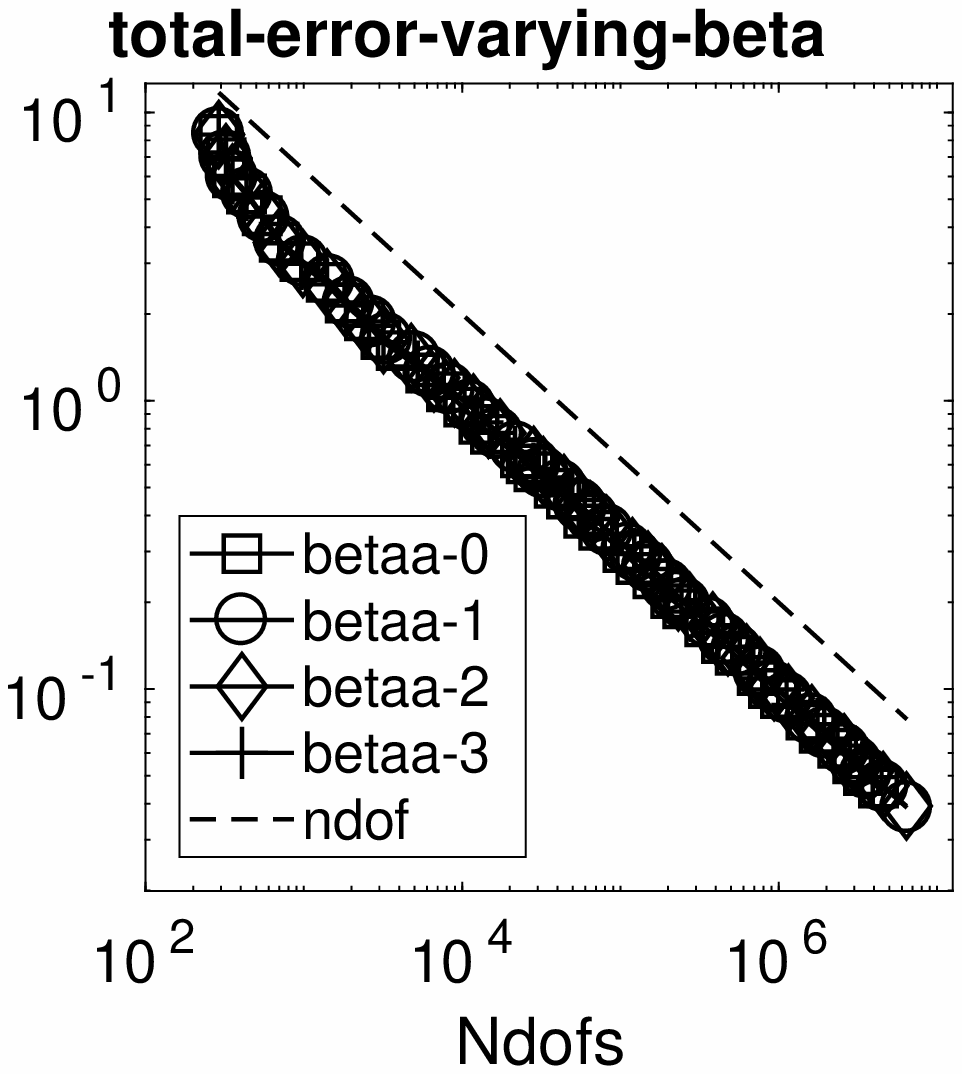}\\
\tiny{(D.1)}\\
\psfrag{total-error-varying-beta}{\normalsize{Individual contribution $\mathcal{E}_{y}$ varying $\beta$}}
\includegraphics[trim={0 0 0 0},clip,width=3.0cm,height=2.8cm,scale=0.6]{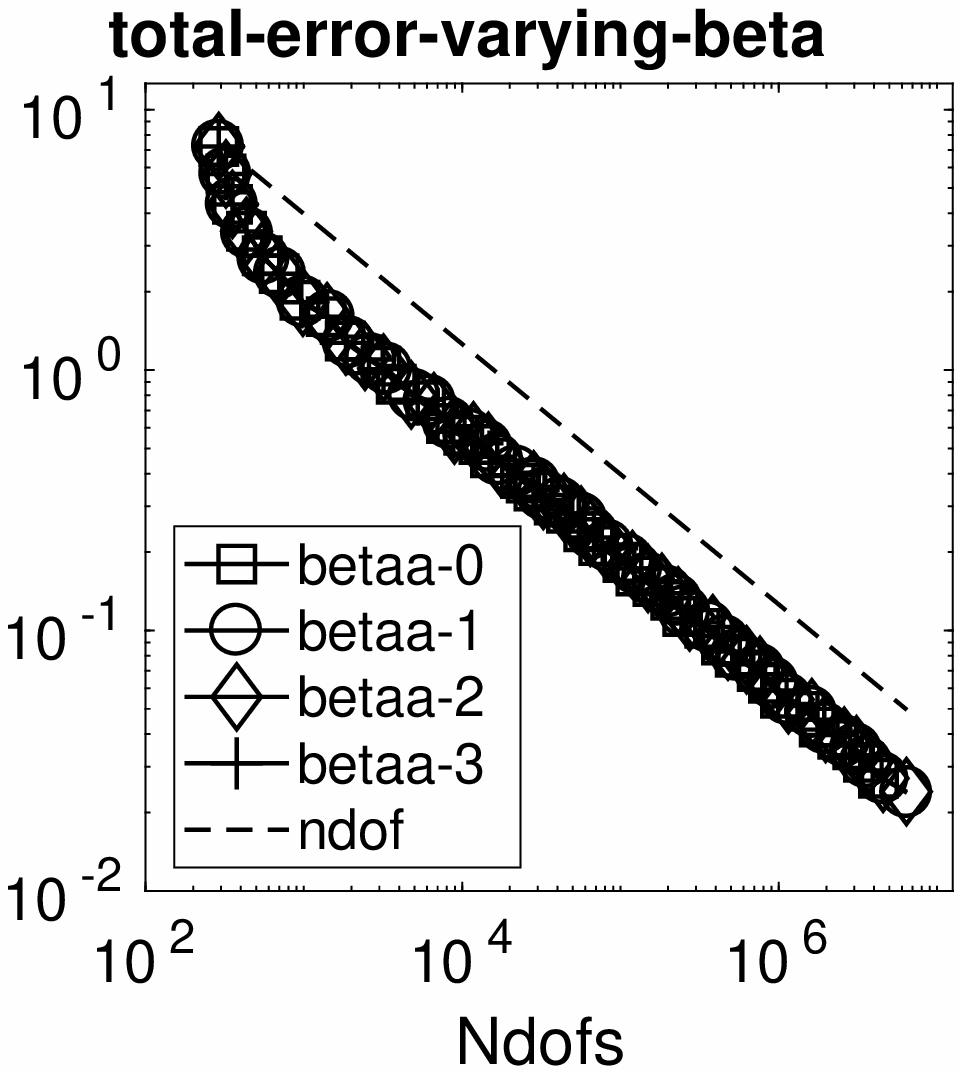}\\
\tiny{(D.2)}\\
\psfrag{total-error-varying-beta}{\normalsize{Individual contribution $\mathcal{E}_{p}$ varying $\beta$}}
\includegraphics[trim={0 0 0 0},clip,width=3.0cm,height=2.8cm,scale=0.6]{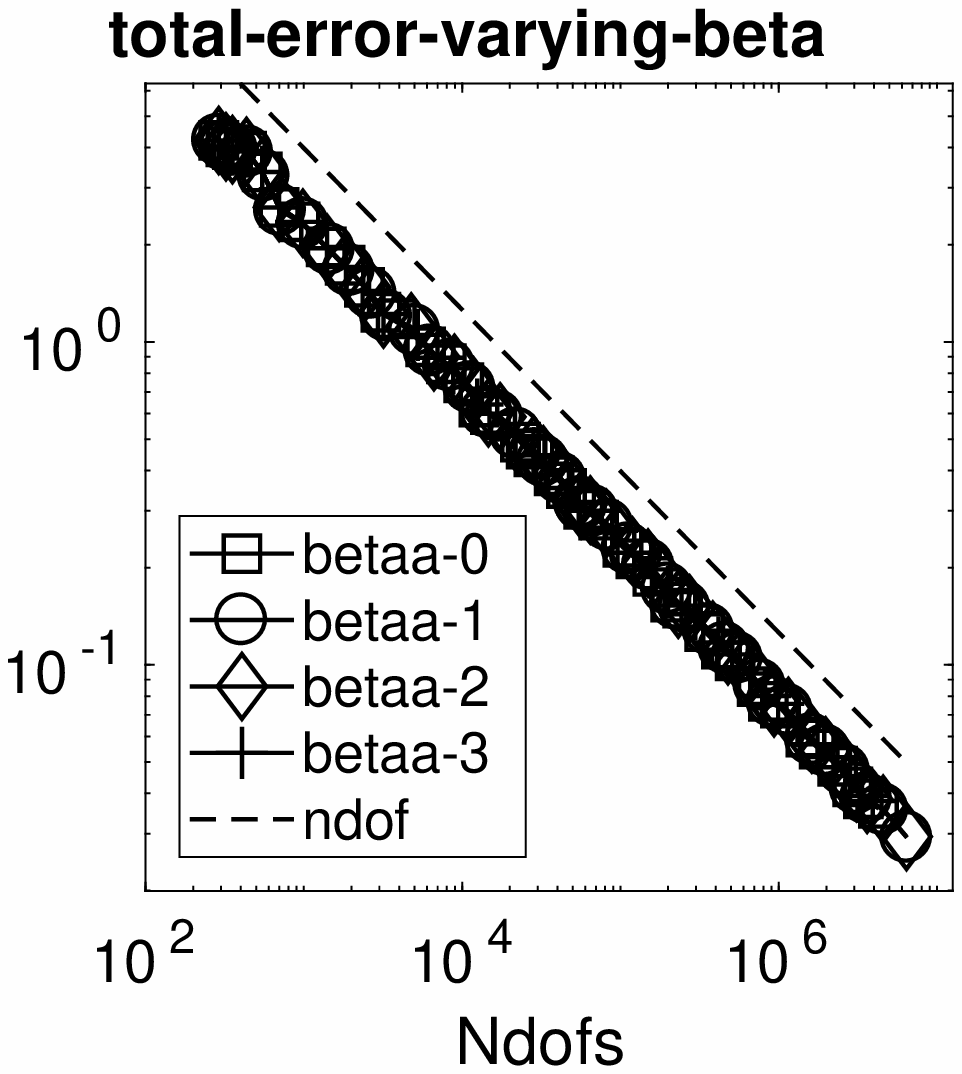}\\
\tiny{(D.3)}\\
\psfrag{total-error-varying-beta}{\normalsize{Individual contribution $E_{u}$ varying $\beta$}}
\includegraphics[trim={0 0 0 0},clip,width=3.0cm,height=2.8cm,scale=0.6]{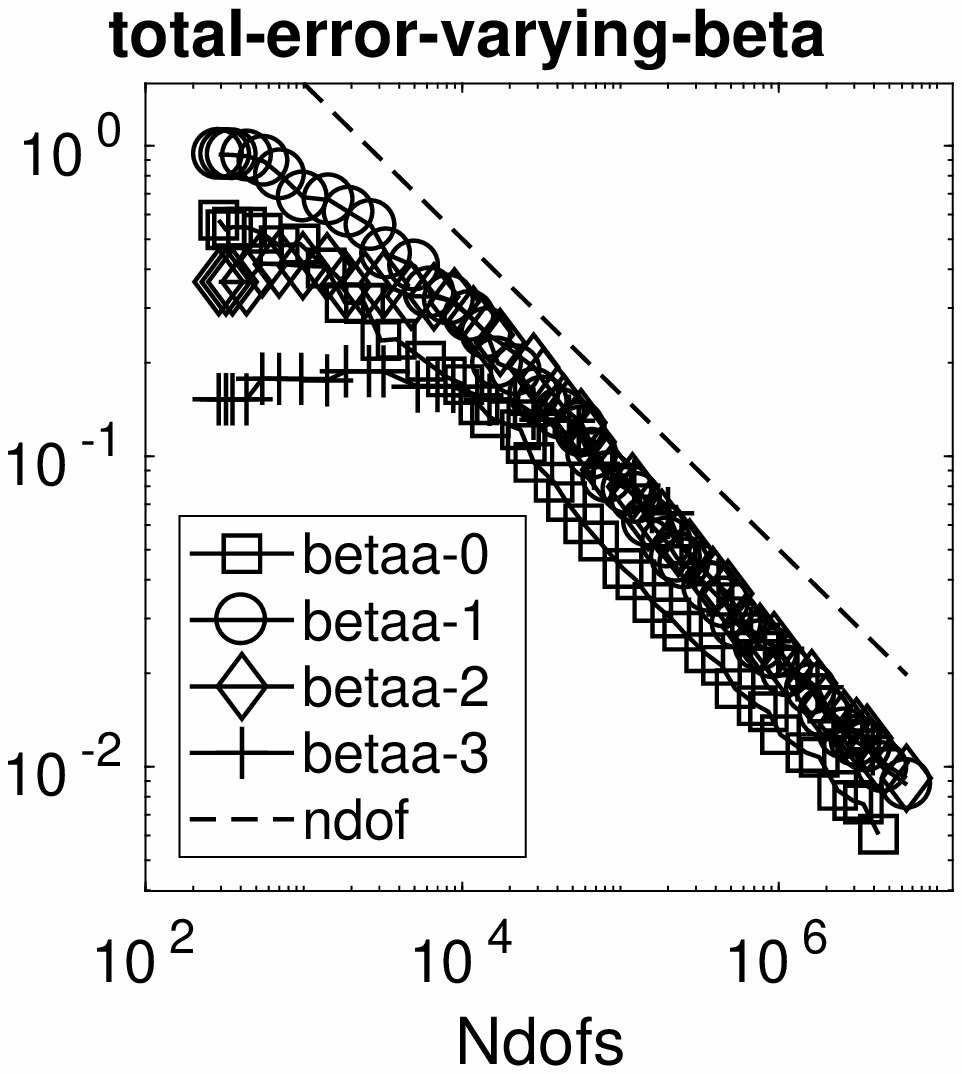}\\
\tiny{(D.4)}\\
\psfrag{total-error-varying-beta}{\normalsize{Individual contribution $E_{\lambda}$ varying $\beta$}}
\includegraphics[trim={0 0 0 0},clip,width=3.0cm,height=2.8cm,scale=0.6]{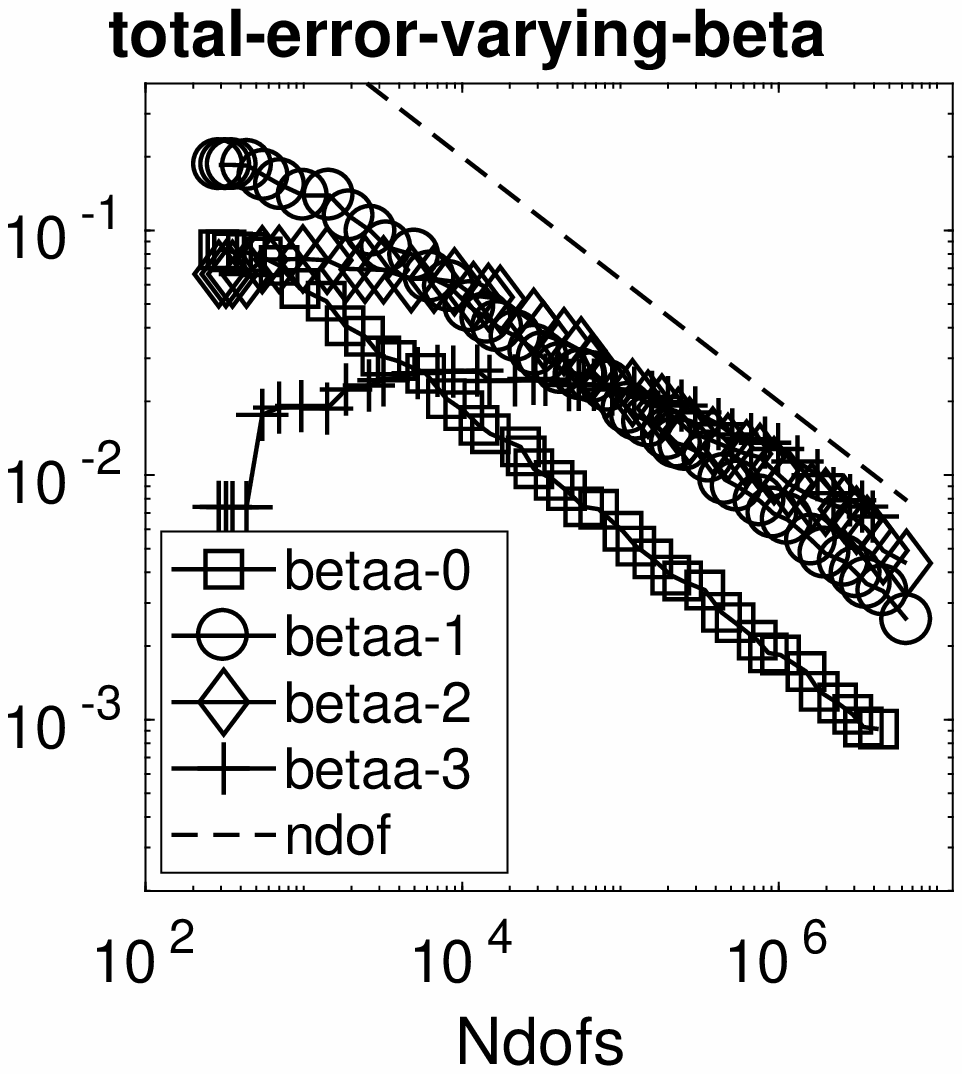}\\
\tiny{(D.5)}
\end{minipage}
\caption{Example 1. \FF{Experimental rates of convergence for the piecewise constant discretization scheme described in Section \ref{subsec:pcd}. In (A.1)--(A.5) and (B.1)--(B.5) we have considered $\beta=7\cdot 10^{-1}$ and $\alpha \in \{10^{0},10^{-1},10^{-2},10^{-3}\}$ while in (C.1)--(C.5) and (D.1)--(D.5) we have considered $\alpha=10^{-3}$ and $\beta \in \{10^{0},10^{-1},10^{-2},10^{-3}\}$}.}
\label{ex_1}
\end{figure}


\begin{figure}[!h]
\centering
\begin{minipage}{0.23\textwidth}\centering
\psfrag{total-error-varying-alpha}{\quad \quad \normalsize{Total error $\| e \|_{\Omega}$ varying $\alpha$}}
\psfrag{alpha-0}{\small{$\alpha=10^0$}}
\psfrag{alpha-1}{\small{$\alpha=10^{-1}$}}
\psfrag{alpha-2}{\small{$\alpha=10^{-2}$}}
\psfrag{alpha-3}{\small{$\alpha=10^{-3}$}}
\psfrag{alpha-4}{\small{$\alpha=10^{-4}$}}
\psfrag{alpha-5}{\small{$\alpha=10^{-5}$}}
\psfrag{Ndofs}{\normalsize{$\textrm{Ndof}^{}_1$}}
\psfrag{ndof}{\footnotesize{$\textrm{Ndof}^{-1/2}_1$}}
\psfrag{ndof2}{\footnotesize{$\textrm{Ndof}^{-1}_1$}}
\includegraphics[trim={0 0 0 0},clip,width=3.0cm,height=2.8cm,scale=0.6]{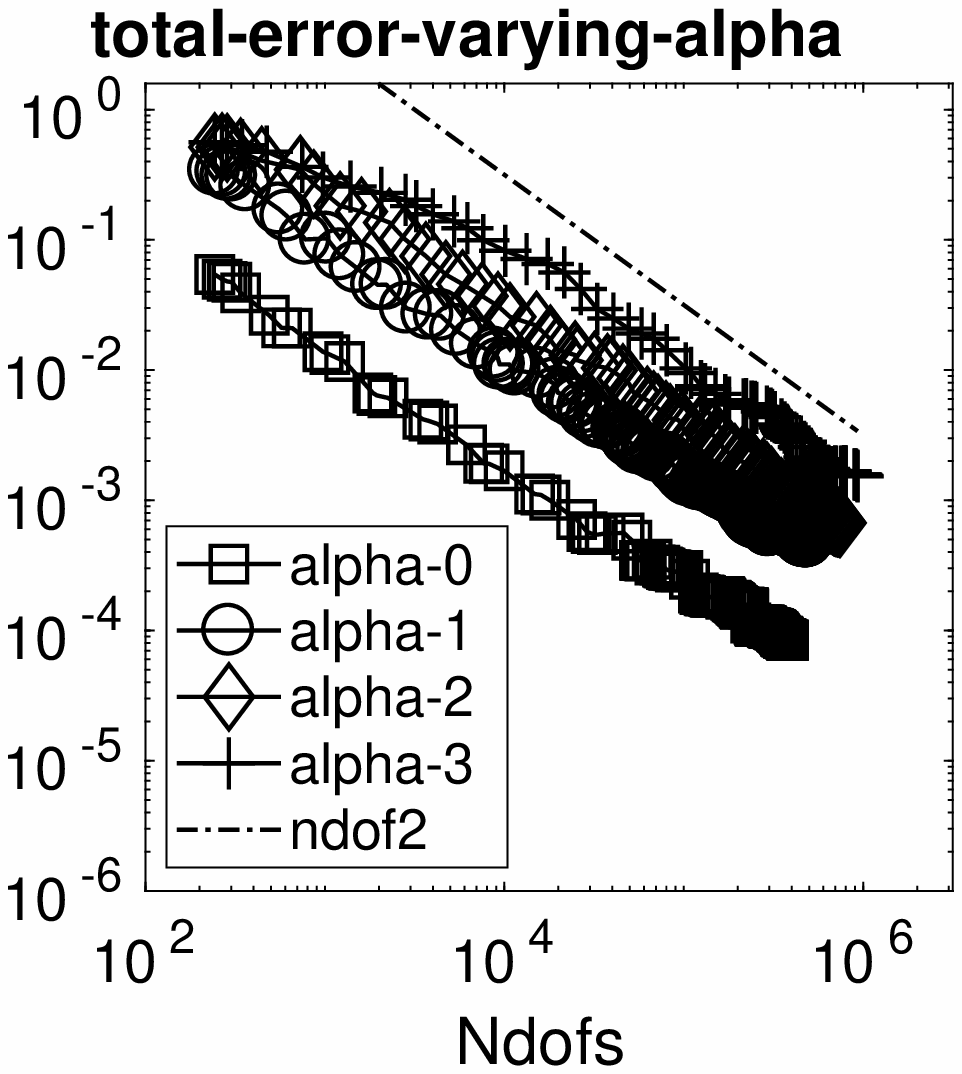}\\
\tiny{(A.1)} \\
\psfrag{total-error-varying-alpha}{\quad \quad \normalsize{State error $\|e_y\|_{L^2(\Omega)}$ varying $\alpha$}}
\psfrag{ndof}{$\footnotesize{\textrm{Ndof}^{-1/2}_1}$}
\includegraphics[trim={0 0 0 0},clip,width=3.0cm,height=2.8cm,scale=0.6]{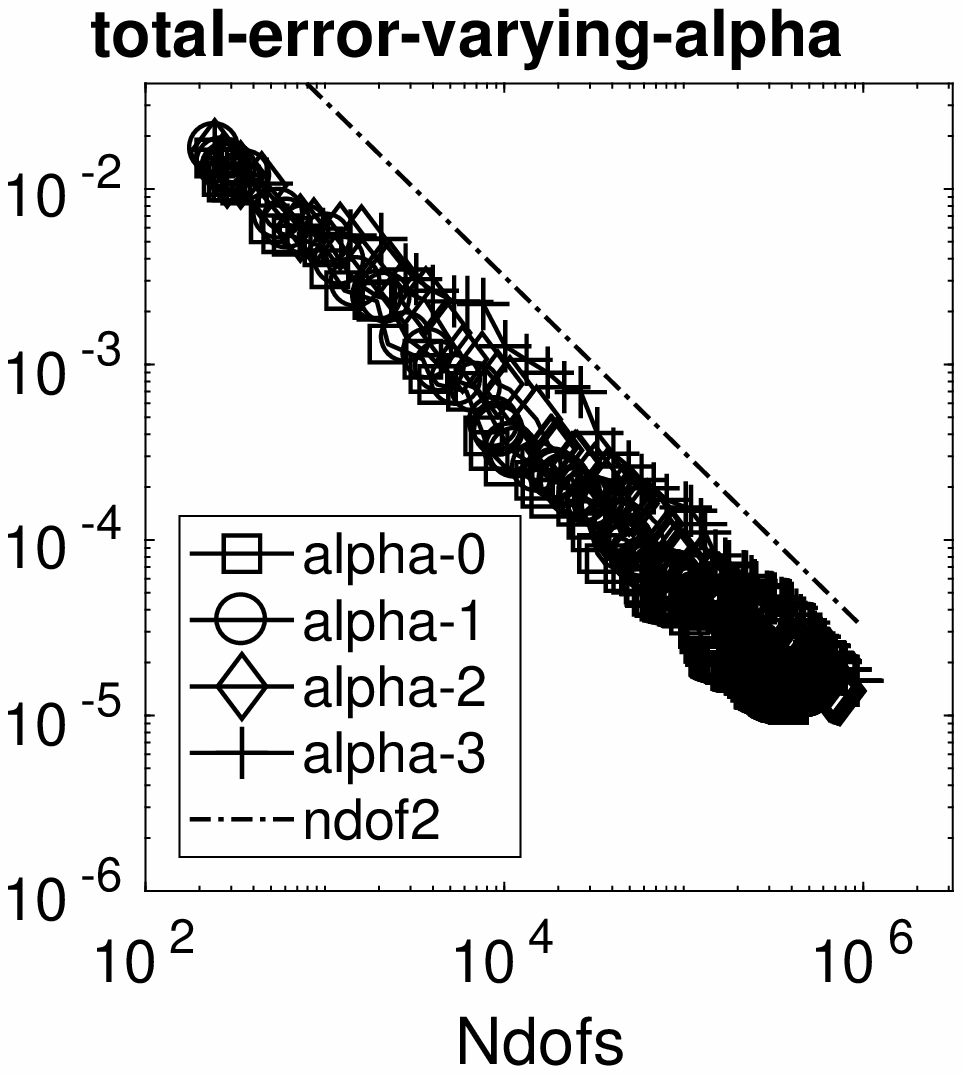}\\
\tiny{(A.2)} \\
\psfrag{total-error-varying-alpha}{\quad \quad \normalsize{Adjoint error $\|e_p\|_{L^2(\Omega)}$ varying $\alpha$}}
\includegraphics[trim={0 0 0 0},clip,width=3.0cm,height=2.8cm,scale=0.6]{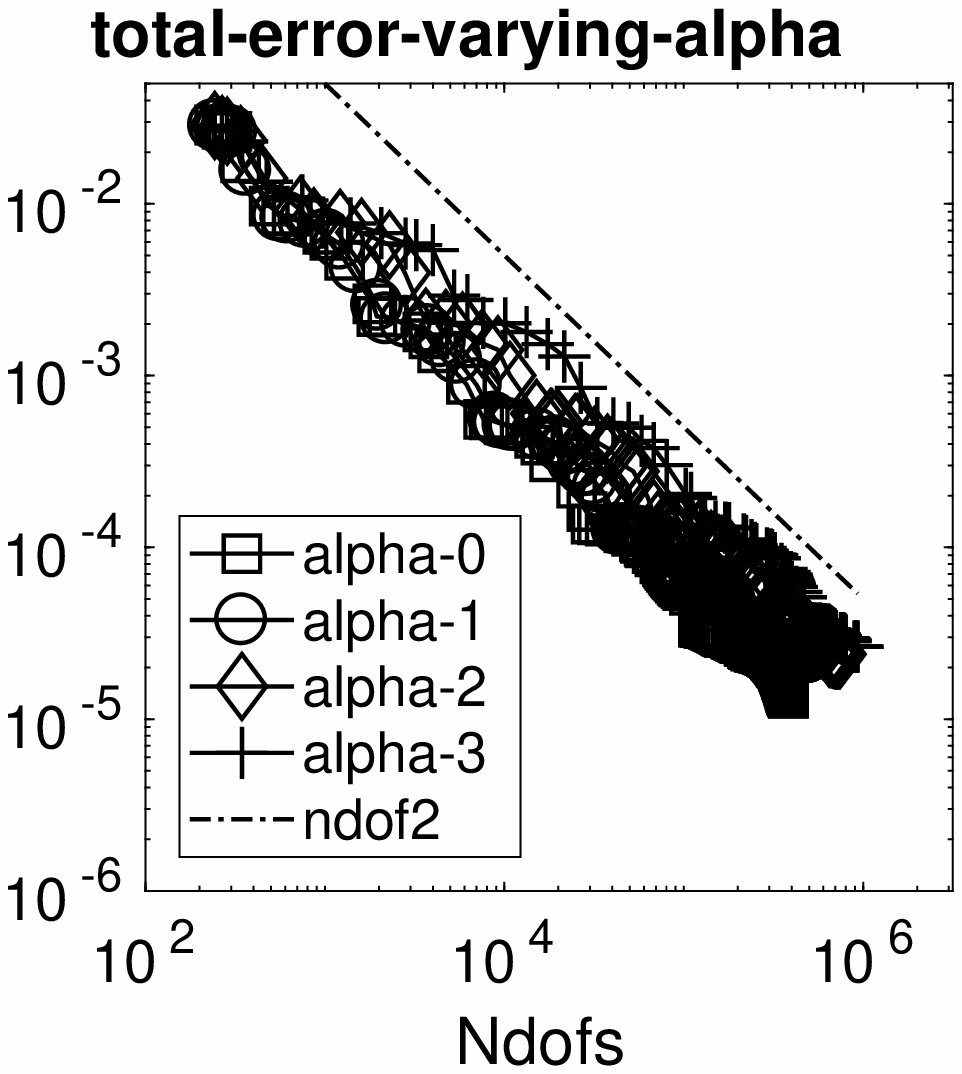}\\
\tiny{(A.3)}\\
\psfrag{total-error-varying-alpha}{\quad \quad \normalsize{Control error $\|e_u\|_{L^2(\Omega)}$ varying $\alpha$}}
\includegraphics[trim={0 0 0 0},clip,width=3.0cm,height=2.8cm,scale=0.6]{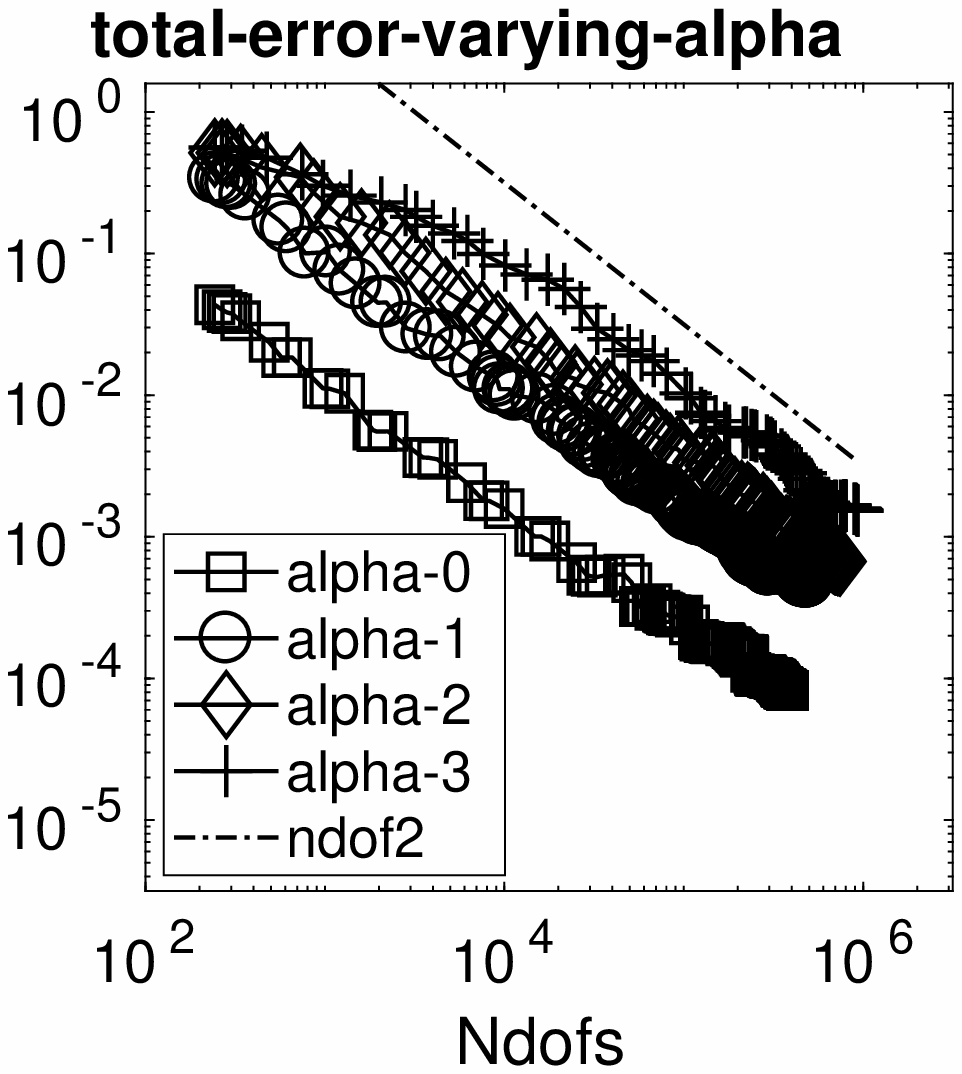}\\
\tiny{(A.4)} \\
\psfrag{total-error-varying-alpha}{\normalsize{Subgradient error $\|e_{\lambda}\|_{L^2(\Omega)}$ varying $\alpha$}}
\includegraphics[trim={0 0 0 0},clip,width=3.0cm,height=2.8cm,scale=0.6]{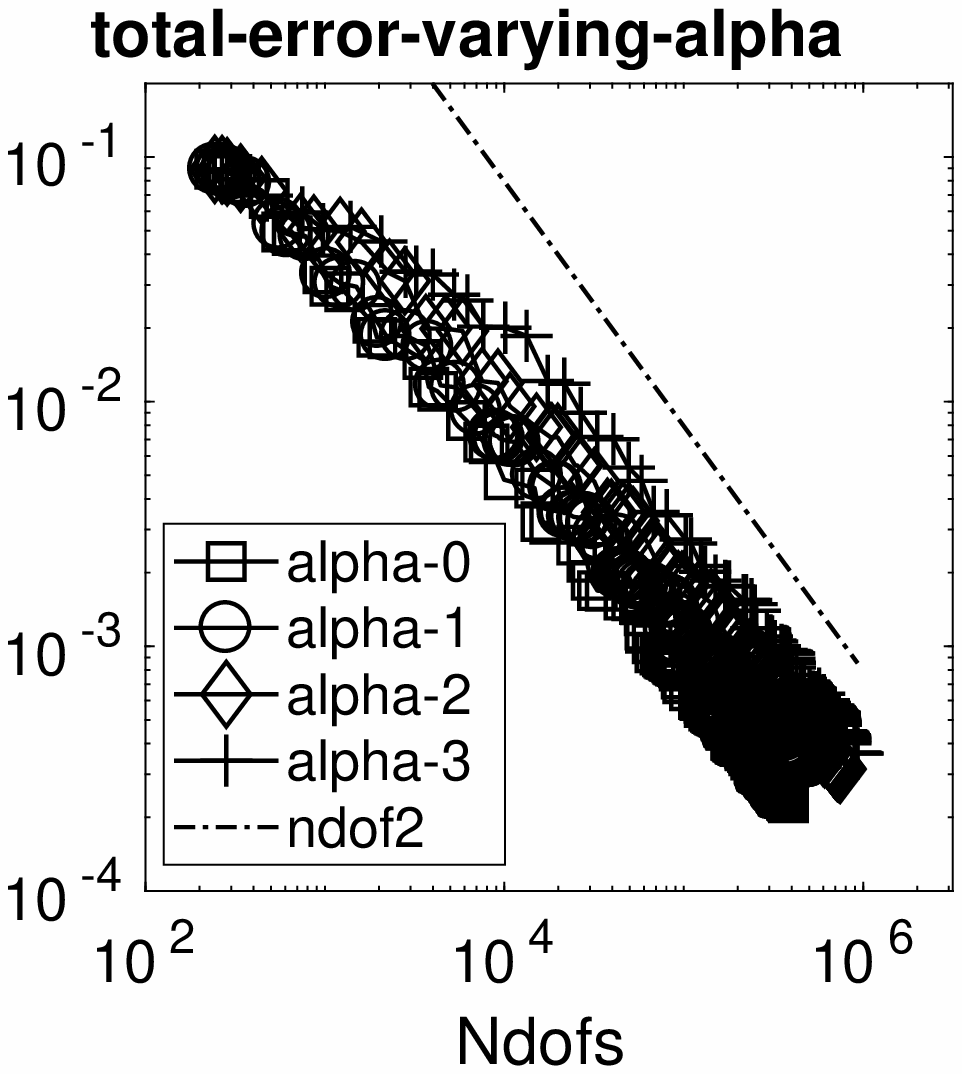}\\
\tiny{(A.5)}
\end{minipage}
\begin{minipage}{0.23\textwidth}\centering
\psfrag{total-error-varying-alpha}{\quad \quad \normalsize{Total estimator $E$ varying $\alpha$}}
\psfrag{alpha-0}{\small{$\alpha=10^0$}}
\psfrag{alpha-1}{\small{$\alpha=10^{-1}$}}
\psfrag{alpha-2}{\small{$\alpha=10^{-2}$}}
\psfrag{alpha-3}{\small{$\alpha=10^{-3}$}}
\psfrag{alpha-4}{\small{$\alpha=10^{-4}$}}
\psfrag{alpha-5}{\small{$\alpha=10^{-5}$}}
\psfrag{Ndofs}{\normalsize{$\textrm{Ndof}^{}_1$}}
\psfrag{ndof}{$\footnotesize{\textrm{Ndof}^{-1/2}_1}$}
\psfrag{ndof2}{\footnotesize{$\textrm{Ndof}^{-1}_1$}}
\includegraphics[trim={0 0 0 0},clip,width=3.0cm,height=2.8cm,scale=0.6]{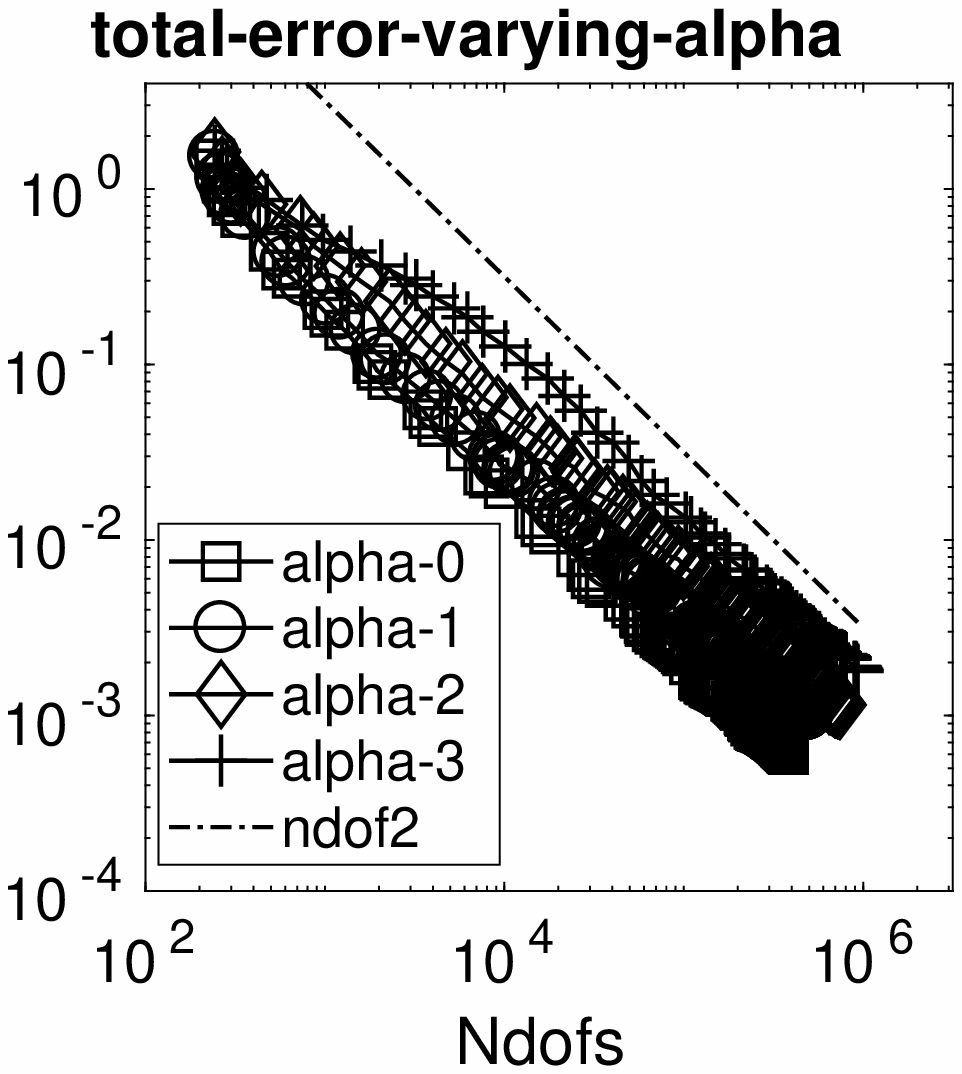}\\
\tiny{(B.1)}\\
\psfrag{total-error-varying-alpha}{\quad \normalsize{Individual contribution ${E}_{y}$ varying $\alpha$}}
\includegraphics[trim={0 0 0 0},clip,width=3.0cm,height=2.8cm,scale=0.6]{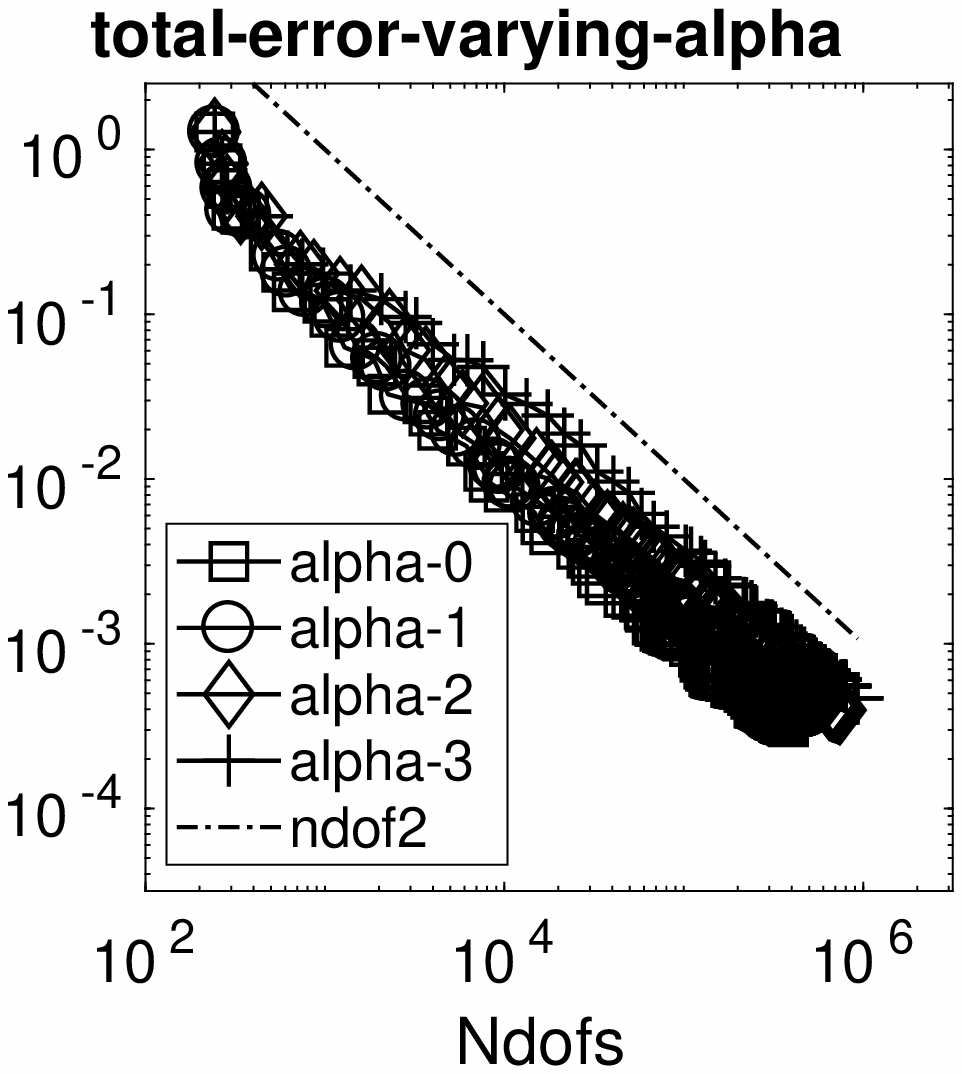}\\
\tiny{(B.2)}\\
\psfrag{total-error-varying-alpha}{\quad \normalsize{Individual contribution ${E}_{p}$ varying $\alpha$}}
\includegraphics[trim={0 0 0 0},clip,width=3.0cm,height=2.8cm,scale=0.6]{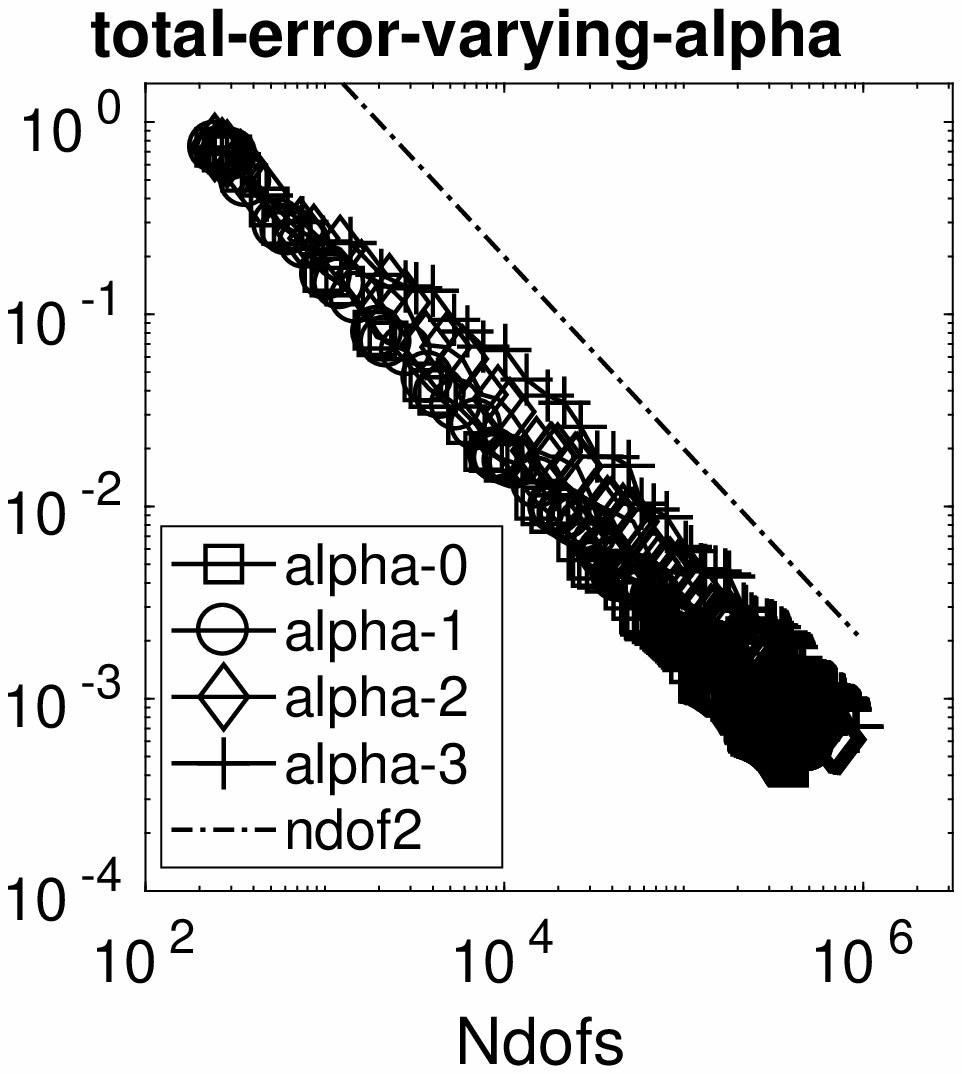}\\
\tiny{(B.3)}\\
\psfrag{total-error-varying-alpha}{\quad \normalsize{Individual contribution $E_{u}$ varying $\alpha$}}
\includegraphics[trim={0 0 0 0},clip,width=3.0cm,height=2.8cm,scale=0.6]{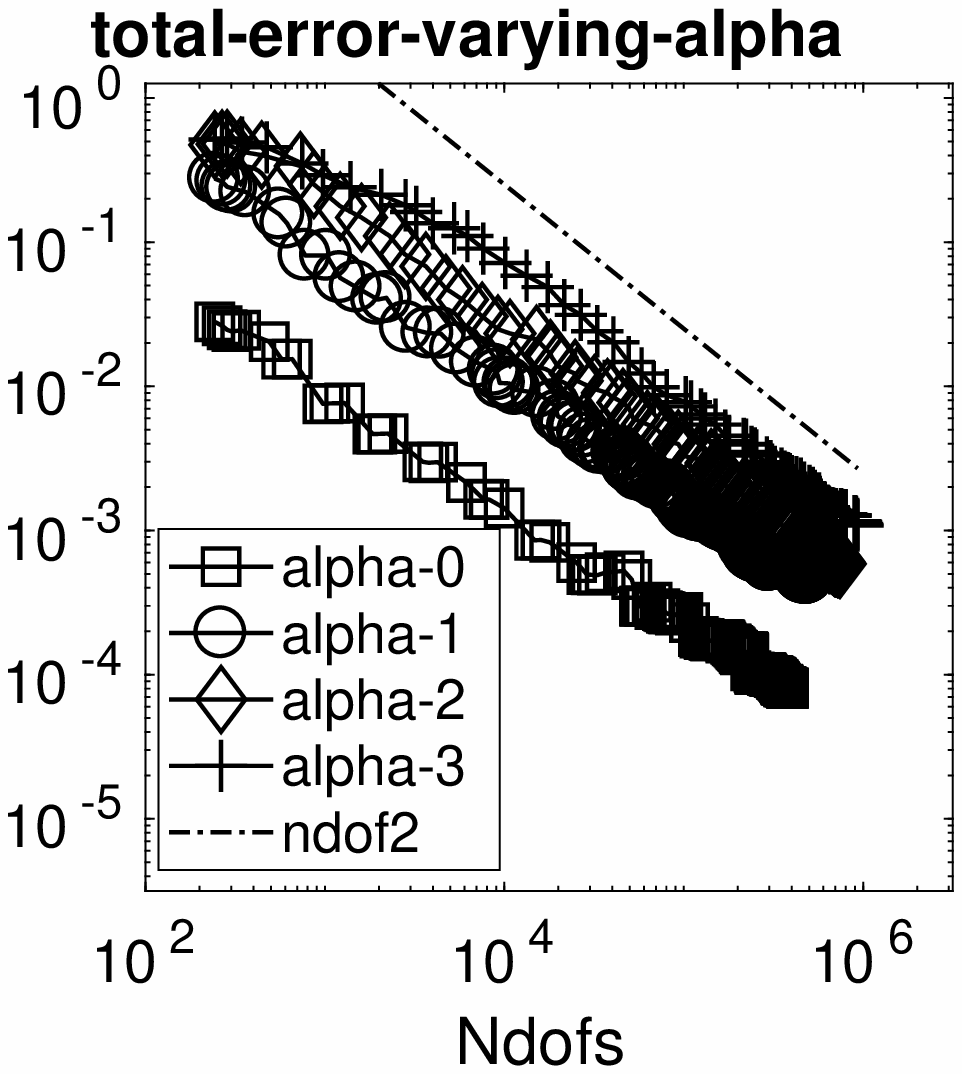}\\
\tiny{(B.4)}\\
\psfrag{total-error-varying-alpha}{\quad \normalsize{Individual contribution $E_{\lambda}$ varying $\alpha$}}
\includegraphics[trim={0 0 0 0},clip,width=3.0cm,height=2.8cm,scale=0.6]{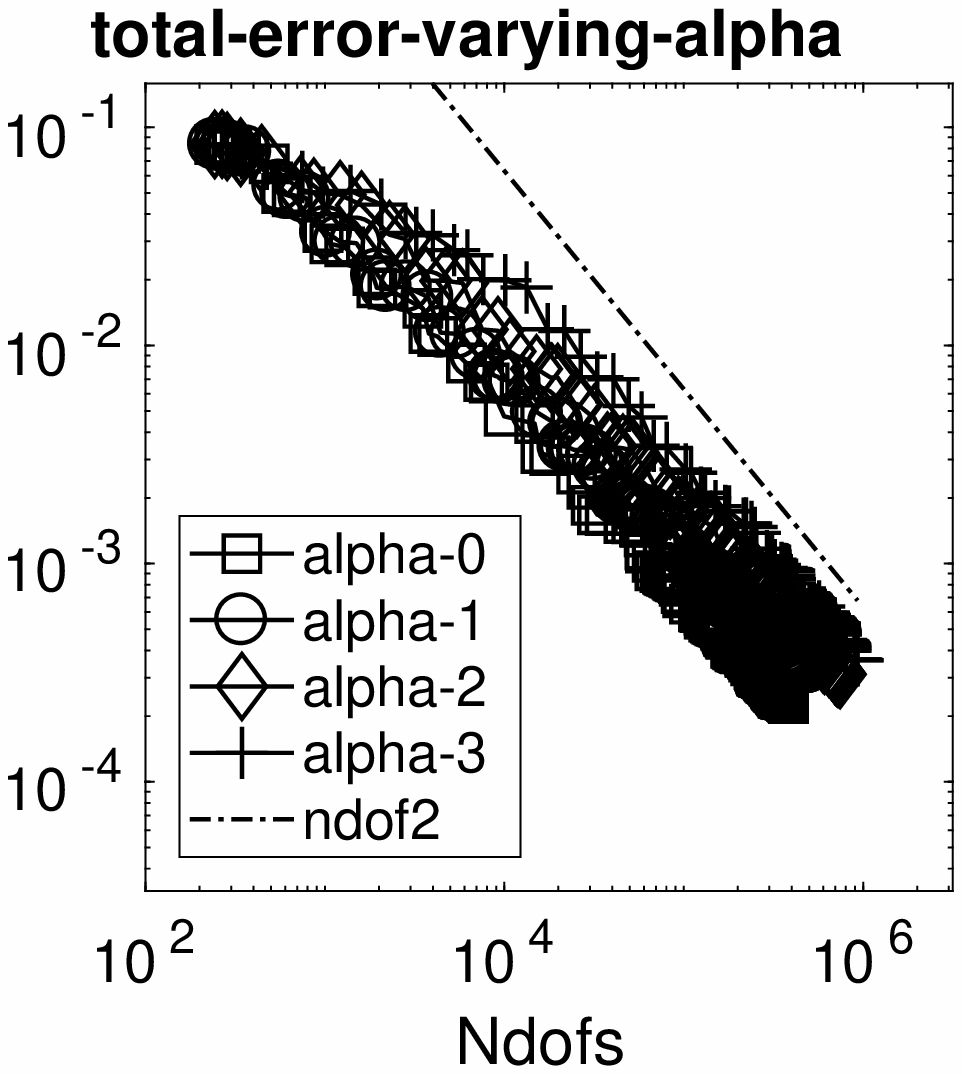}\\
\tiny{(B.5)}
\end{minipage}
\begin{minipage}{0.23\textwidth}\centering
\psfrag{total-error-varying-beta}{\quad \quad \normalsize{Total error $\|e\|_{\Omega}$ varying $\beta$}}
\psfrag{betaa-0}{\small{$\beta=10^0$}}
\psfrag{betaa-1}{\small{$\beta=10^{-1}$}}
\psfrag{betaa-2}{\small{$\beta=10^{-2}$}}
\psfrag{betaa-3}{\small{$\beta=10^{-3}$}}
\psfrag{betaa-4}{\small{$\beta=10^{-4}$}}
\psfrag{betaa-5}{\small{$\beta=10^{-5}$}}
\psfrag{Ndofs}{\normalsize{$\textrm{Ndof}^{}_1$}}
\psfrag{ndof}{$\footnotesize{\textrm{Ndof}^{-1/2}_1}$}
\psfrag{ndof2}{\footnotesize{$\textrm{Ndof}^{-1}_1$}}
\includegraphics[trim={0 0 0 0},clip,width=3.0cm,height=2.8cm,scale=0.6]{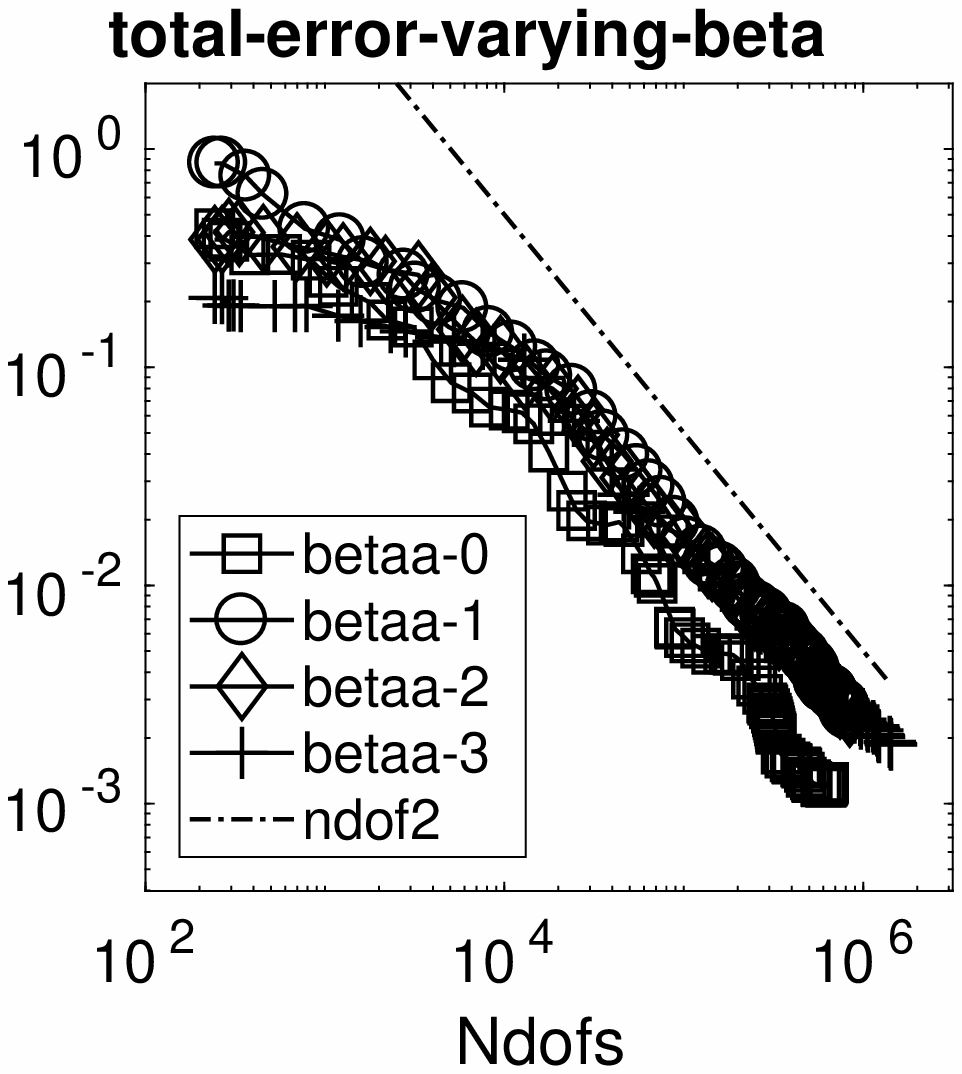}\\
\tiny{(C.1)}\\
\psfrag{total-error-varying-beta}{\quad \normalsize{State error $\|e_{y}\|_{L^{2}(\Omega)}$ varying $\beta$}}
\includegraphics[trim={0 0 0 0},clip,width=3.0cm,height=2.8cm,scale=0.6]{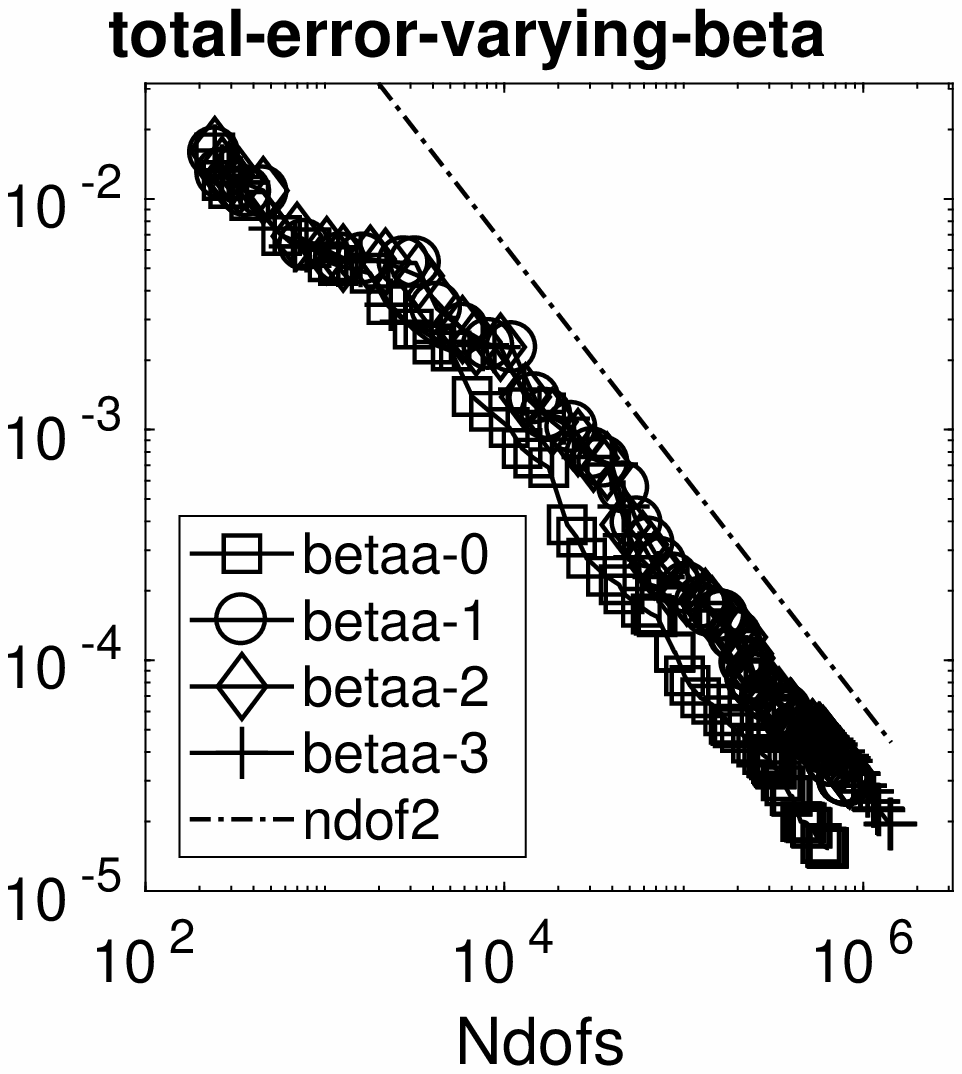}\\
\tiny{(C.2)}\\
\psfrag{total-error-varying-beta}{\quad \normalsize{Adjoint error $\|e_{p}\|_{L^{2}(\Omega)}$ varying $\beta$}}
\includegraphics[trim={0 0 0 0},clip,width=3.0cm,height=2.8cm,scale=0.6]{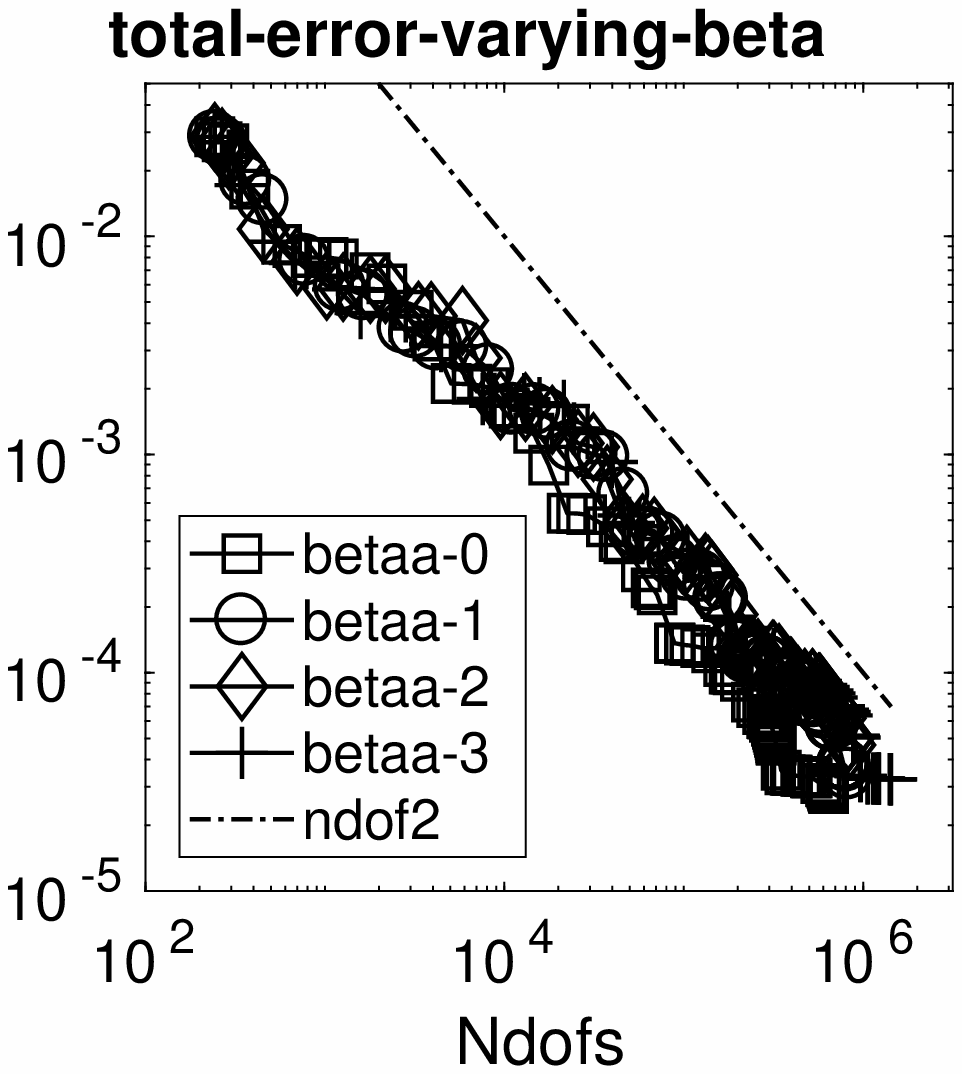}\\
\tiny{(C.3)}\\
\psfrag{total-error-varying-beta}{\quad \normalsize{Control error $\|e_{u}\|_{L^{2}(\Omega)}$ varying $\beta$}}
\includegraphics[trim={0 0 0 0},clip,width=3.0cm,height=2.8cm,scale=0.6]{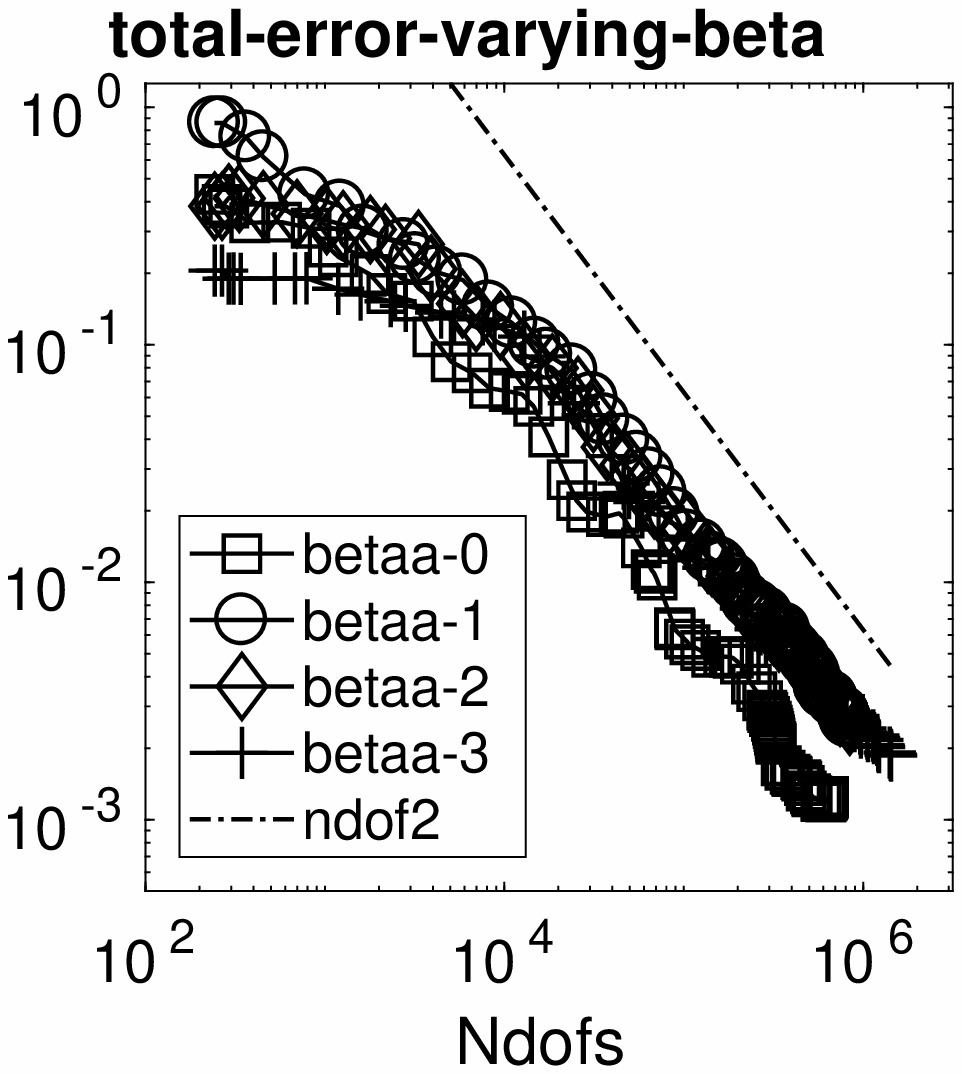}\\
\tiny{(C.4)}\\
\psfrag{total-error-varying-beta}{\normalsize{Subgradient error $\|e_{\lambda}\|_{L^{2}(\Omega)}$ varying $\beta$}}
\includegraphics[trim={0 0 0 0},clip,width=3.0cm,height=2.8cm,scale=0.6]{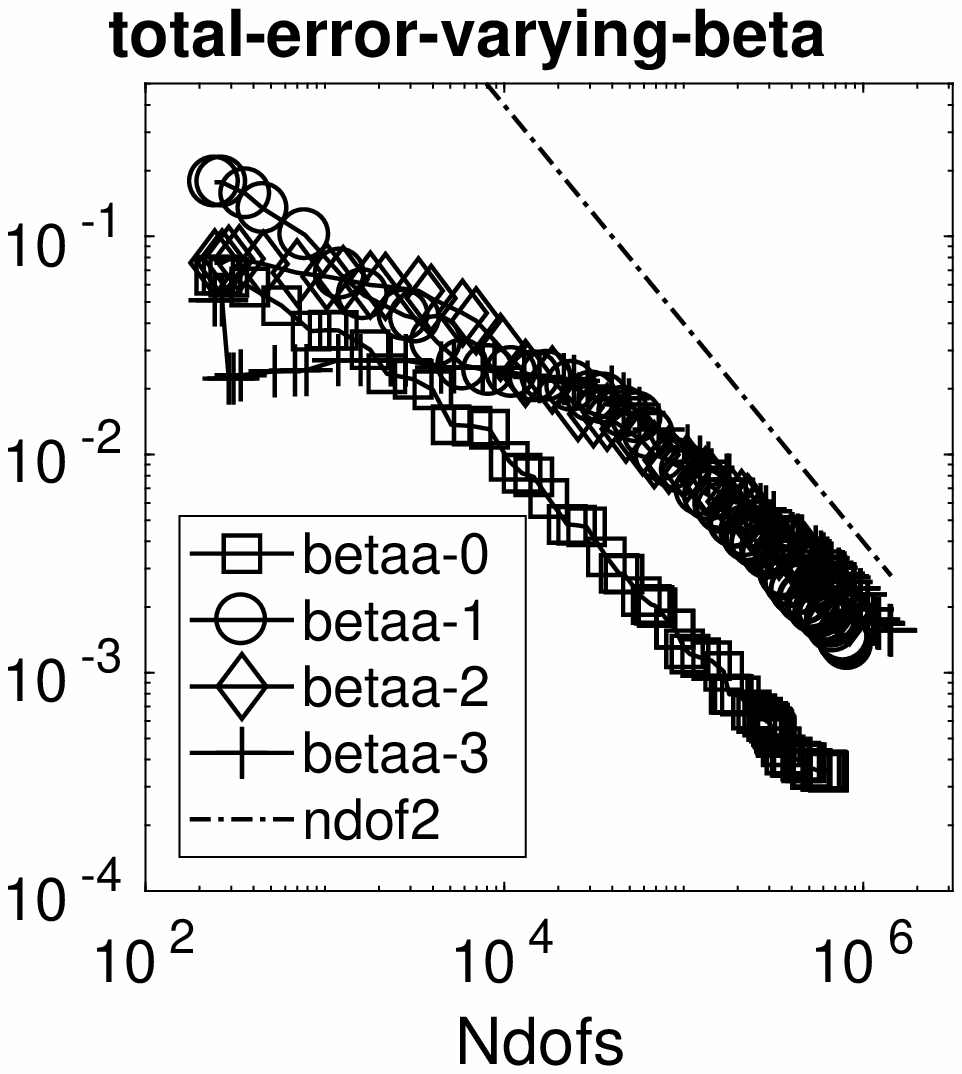}\\
\tiny{(C.5)}
\end{minipage}
\begin{minipage}{0.23\textwidth}\centering
\psfrag{total-error-varying-beta}{\quad \quad \normalsize{Total estimator $E$ varying $\beta$}}
\psfrag{betaa-0}{\small{$\beta=10^0$}}
\psfrag{betaa-1}{\small{$\beta=10^{-1}$}}
\psfrag{betaa-2}{\small{$\beta=10^{-2}$}}
\psfrag{betaa-3}{\small{$\beta=10^{-3}$}}
\psfrag{betaa-4}{\small{$\beta=10^{-4}$}}
\psfrag{betaa-5}{\small{$\beta=10^{-5}$}}
\psfrag{Ndofs}{\normalsize{$\textrm{Ndof}^{}_1$}}
\psfrag{ndof}{$\footnotesize{\textrm{Ndof}^{-1/2}_1}$}
\psfrag{ndof2}{\footnotesize{$\textrm{Ndof}^{-1}_1$}}
\includegraphics[trim={0 0 0 0},clip,width=3.0cm,height=2.8cm,scale=0.6]{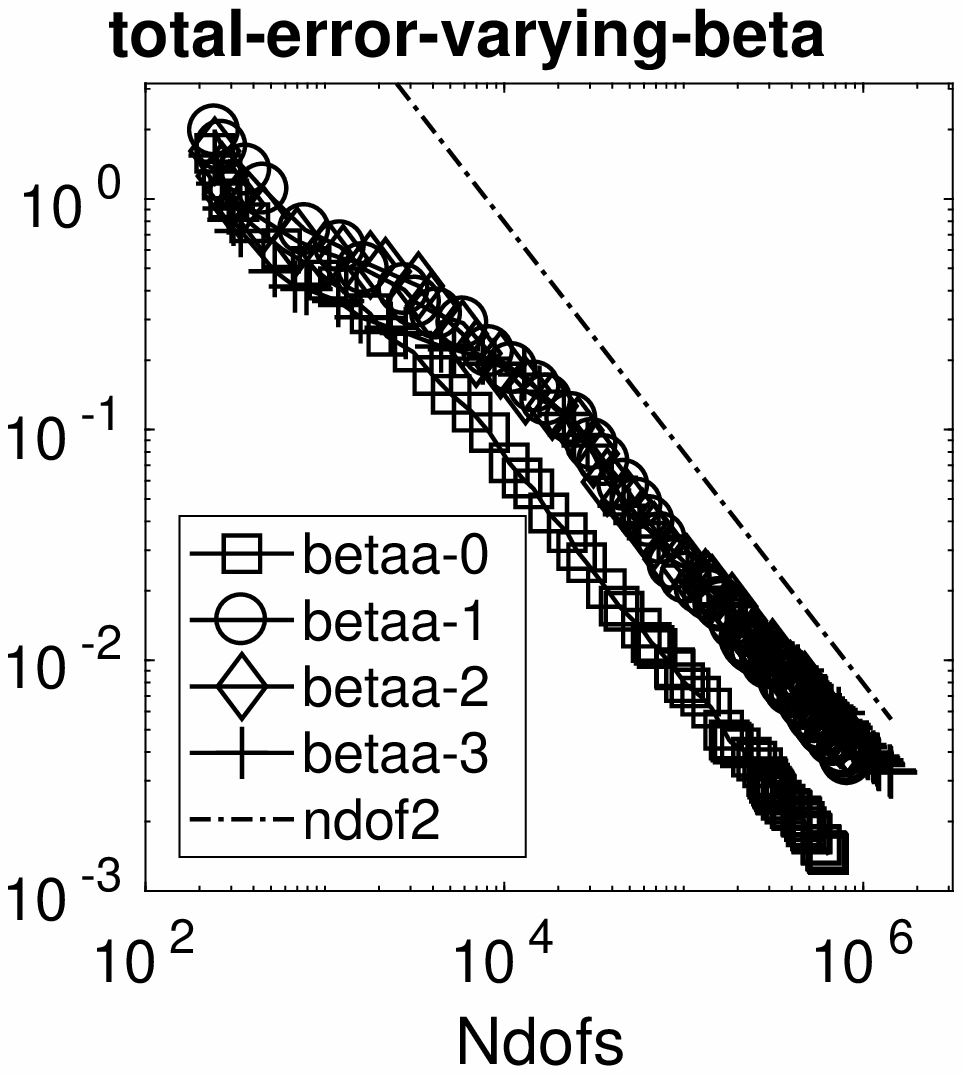}\\
\tiny{(D.1)}\\
\psfrag{total-error-varying-beta}{\normalsize{Individual contribution ${E}_{y}$ varying $\beta$}}
\includegraphics[trim={0 0 0 0},clip,width=3.0cm,height=2.8cm,scale=0.6]{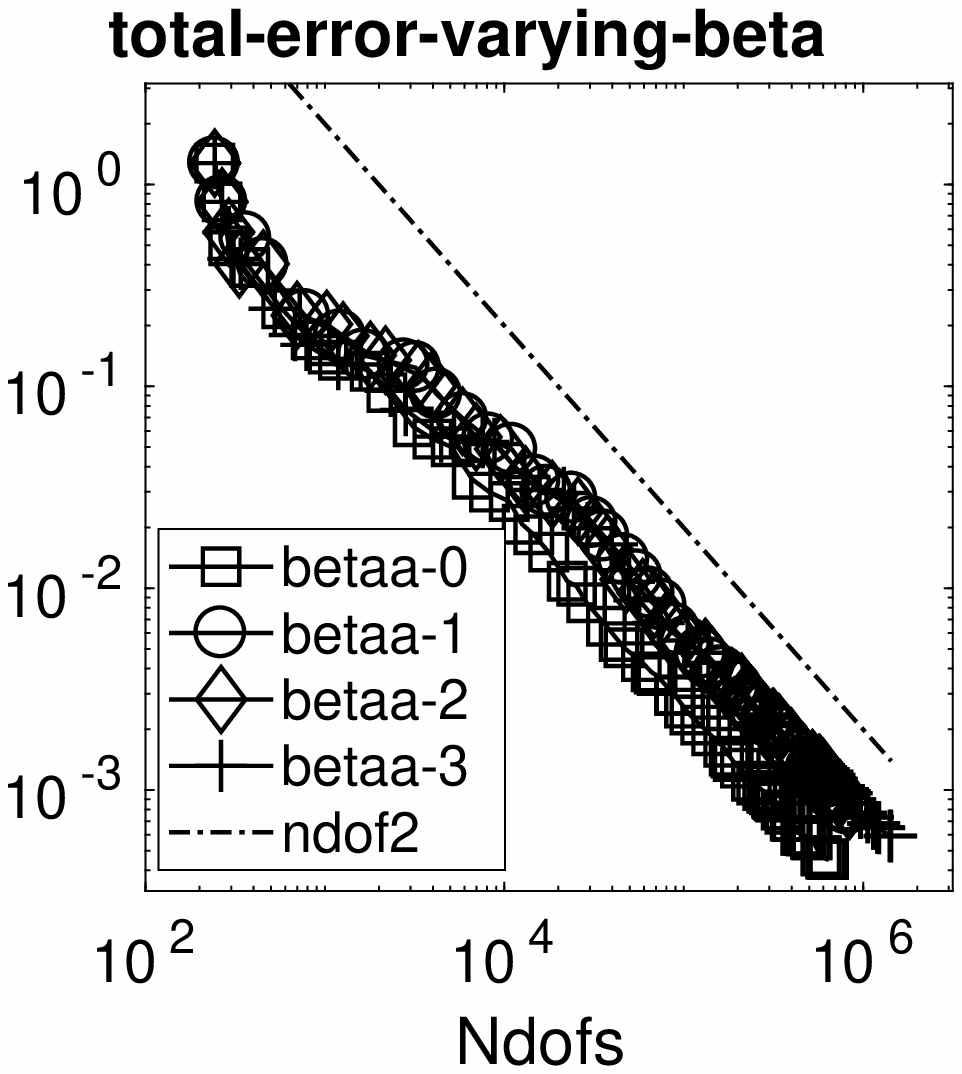}\\
\tiny{(D.2)}\\
\psfrag{total-error-varying-beta}{\normalsize{Individual contribution ${E}_{p}$ varying $\beta$}}
\includegraphics[trim={0 0 0 0},clip,width=3.0cm,height=2.8cm,scale=0.6]{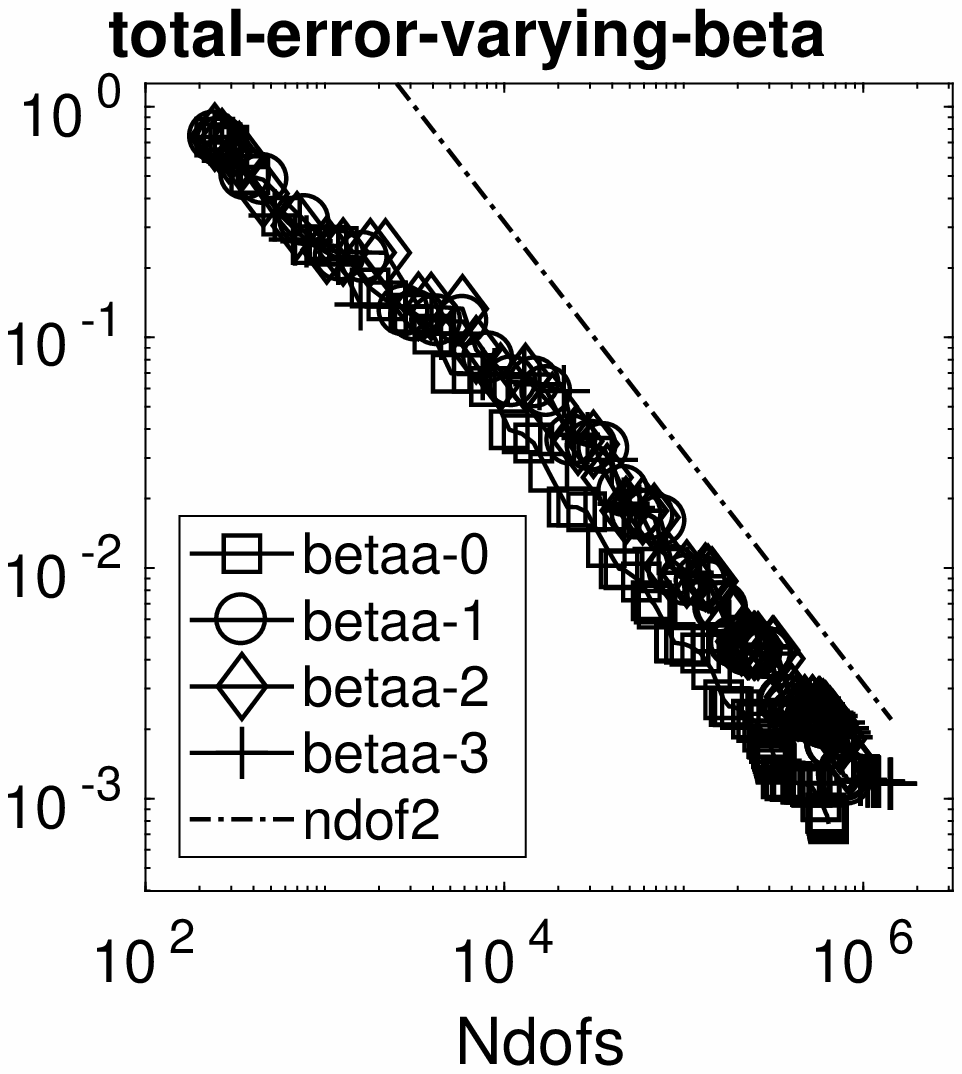}\\
\tiny{(D.3)}\\
\psfrag{total-error-varying-beta}{\normalsize{Individual contribution $E_{u}$ varying $\beta$}}
\includegraphics[trim={0 0 0 0},clip,width=3.0cm,height=2.8cm,scale=0.6]{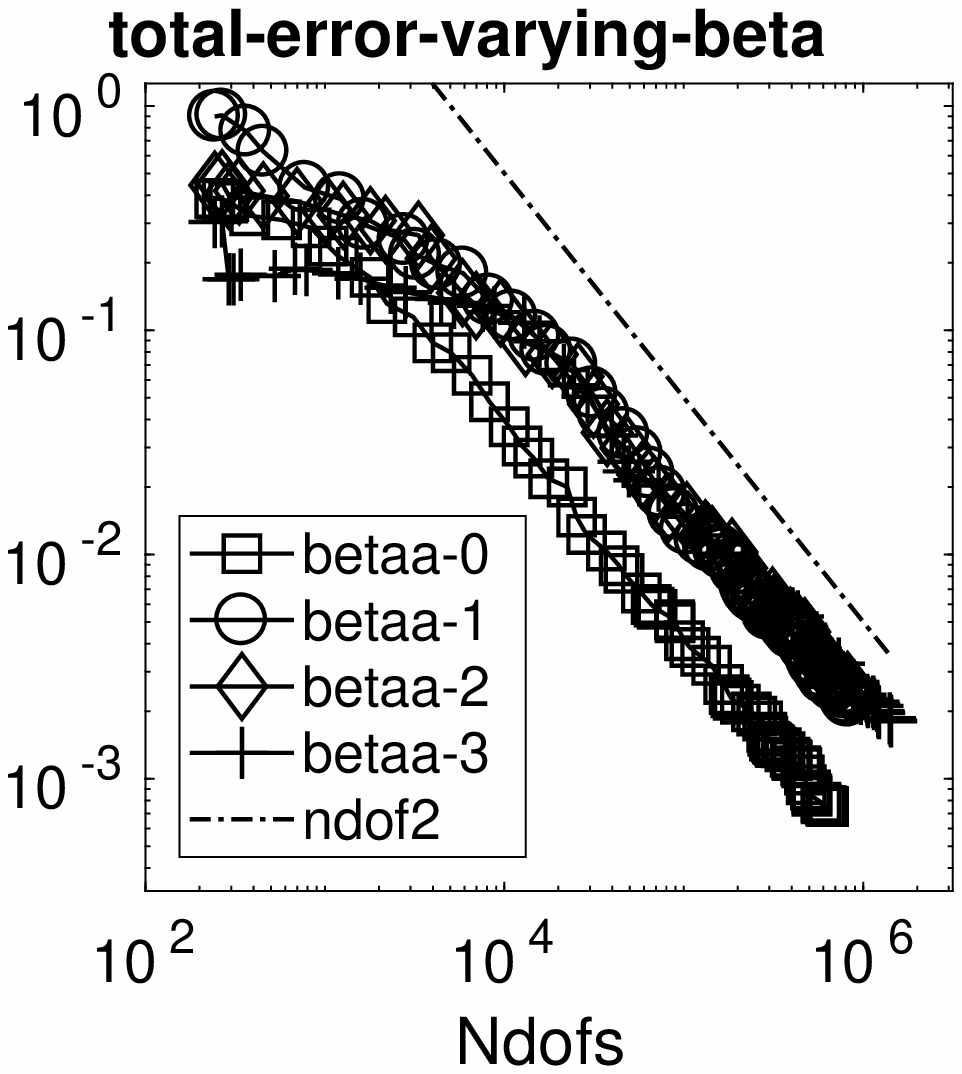}\\
\tiny{(D.4)}\\
\psfrag{total-error-varying-beta}{\normalsize{Individual contribution $E_{\lambda}$ varying $\beta$}}
\includegraphics[trim={0 0 0 0},clip,width=3.0cm,height=2.8cm,scale=0.6]{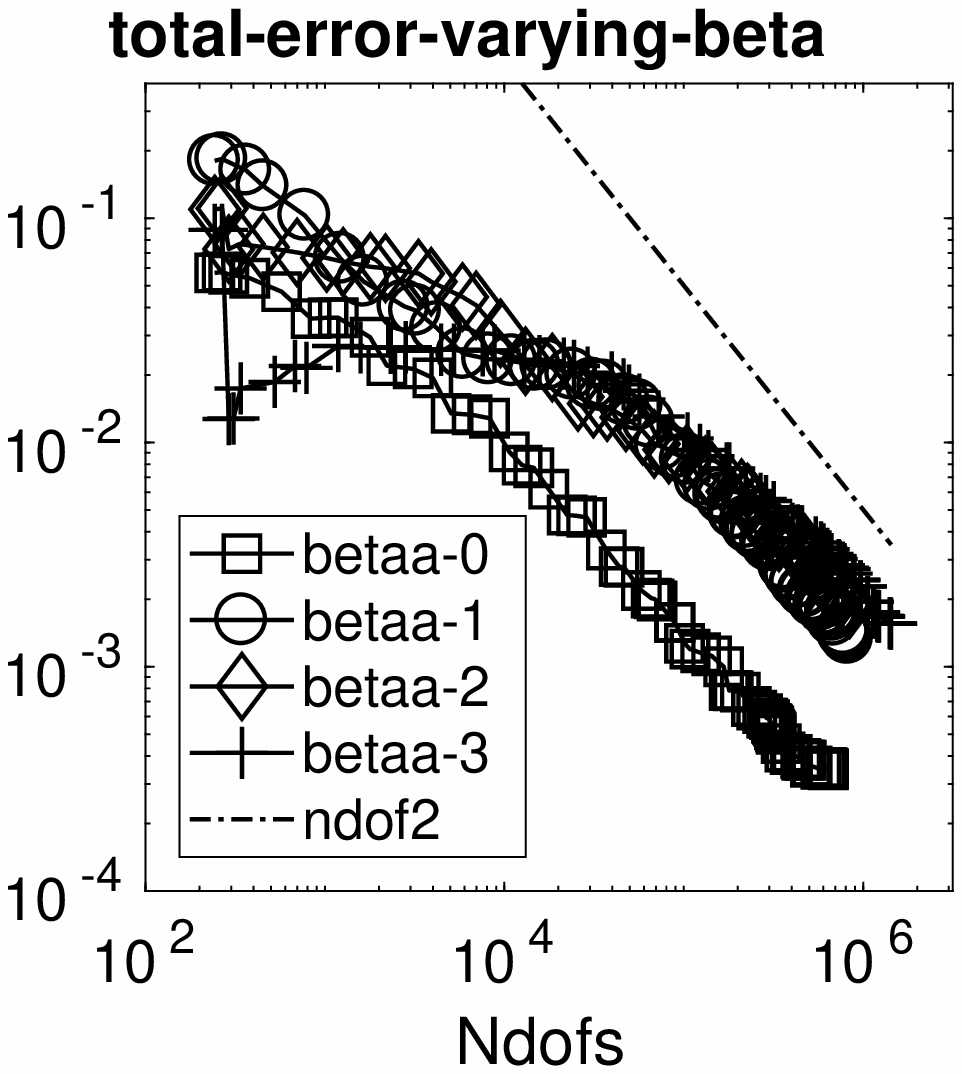}\\
\tiny{(D.5)}
\end{minipage}
\caption{Example 1. \FF{Experimental rates of convergence for the piecewise linear discretization scheme described in Section \ref{subsec:pld}. In (A.1)--(A.5) and (B.1)--(B.5) we have considered $\beta=7\cdot 10^{-1}$ and $\alpha \in \{10^{0},10^{-1},10^{-2},10^{-3}\}$ while in (C.1)--(C.5) and (D.1)--(D.5) we have considered $\alpha=10^{-3}$ and $\beta \in \{10^{0},10^{-1},10^{-2},10^{-3}\}$}.}
\label{ex_2}
\end{figure}


\begin{figure}[!h]
\centering
\begin{minipage}{0.23\textwidth}\centering
\psfrag{total-error-varying-alpha}{\quad \quad \normalsize{Total error $\| e \|_{\Omega}$ varying $\alpha$}}
\psfrag{alpha-0}{\small{$\alpha=10^0$}}
\psfrag{alpha-1}{\small{$\alpha=10^{-1}$}}
\psfrag{alpha-2}{\small{$\alpha=10^{-2}$}}
\psfrag{alpha-3}{\small{$\alpha=10^{-3}$}}
\psfrag{alpha-4}{\small{$\alpha=10^{-4}$}}
\psfrag{alpha-5}{\small{$\alpha=10^{-5}$}}
\psfrag{Ndofs}{\normalsize{$\textrm{Ndof}^{}_2$}}
\psfrag{ndof}{{\footnotesize{\FF{$\textrm{Ndof}^{-1}_2$}}}}
\psfrag{ndof2}{{\footnotesize{$\textrm{Ndof}^{-1}_2$}}}
\includegraphics[trim={0 0 0 0},clip,width=2.9cm,height=2.8cm,scale=0.6]{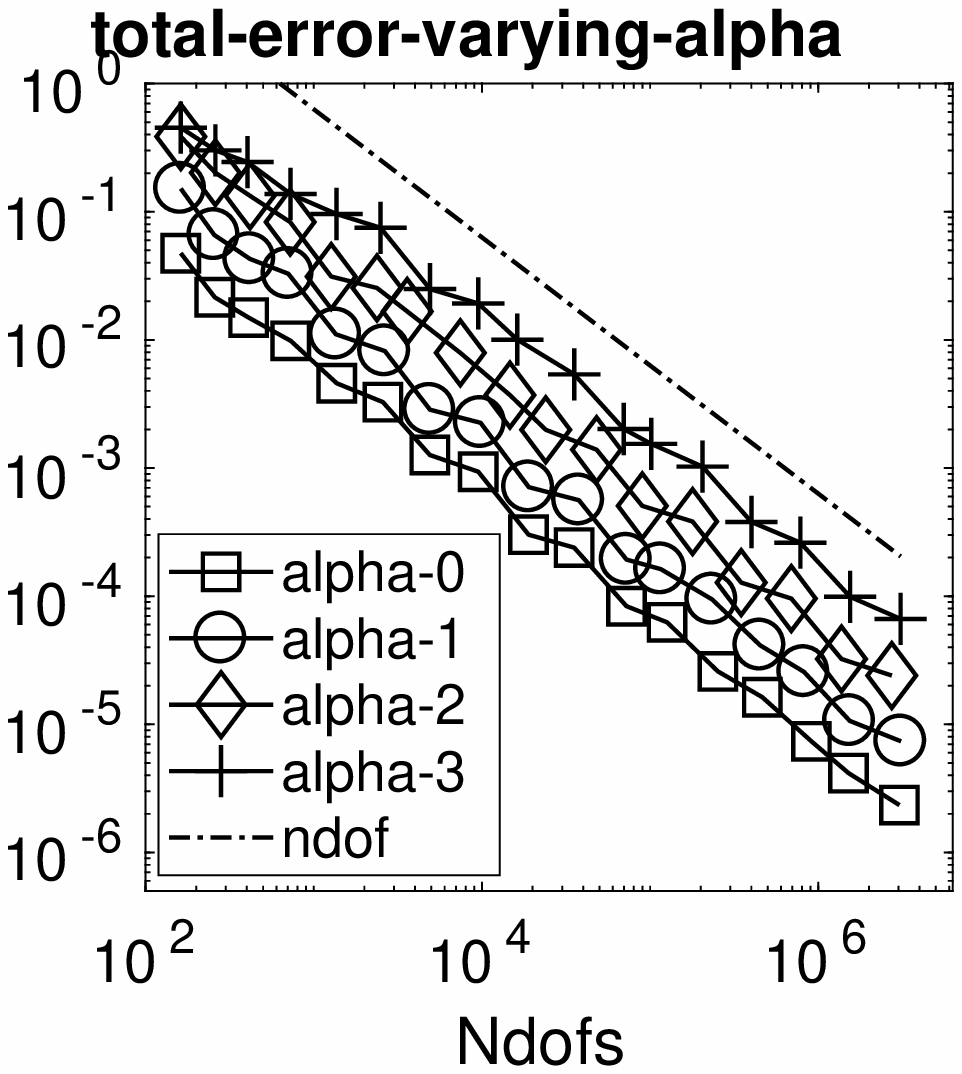}\\
\tiny{(A.1)} \\
\psfrag{total-error-varying-alpha}{\quad \quad \normalsize{State error $\|e_y\|_{L^2(\Omega)}$ varying $\alpha$}}
\includegraphics[trim={0 0 0 0},clip,width=2.9cm,height=2.8cm,scale=0.6]{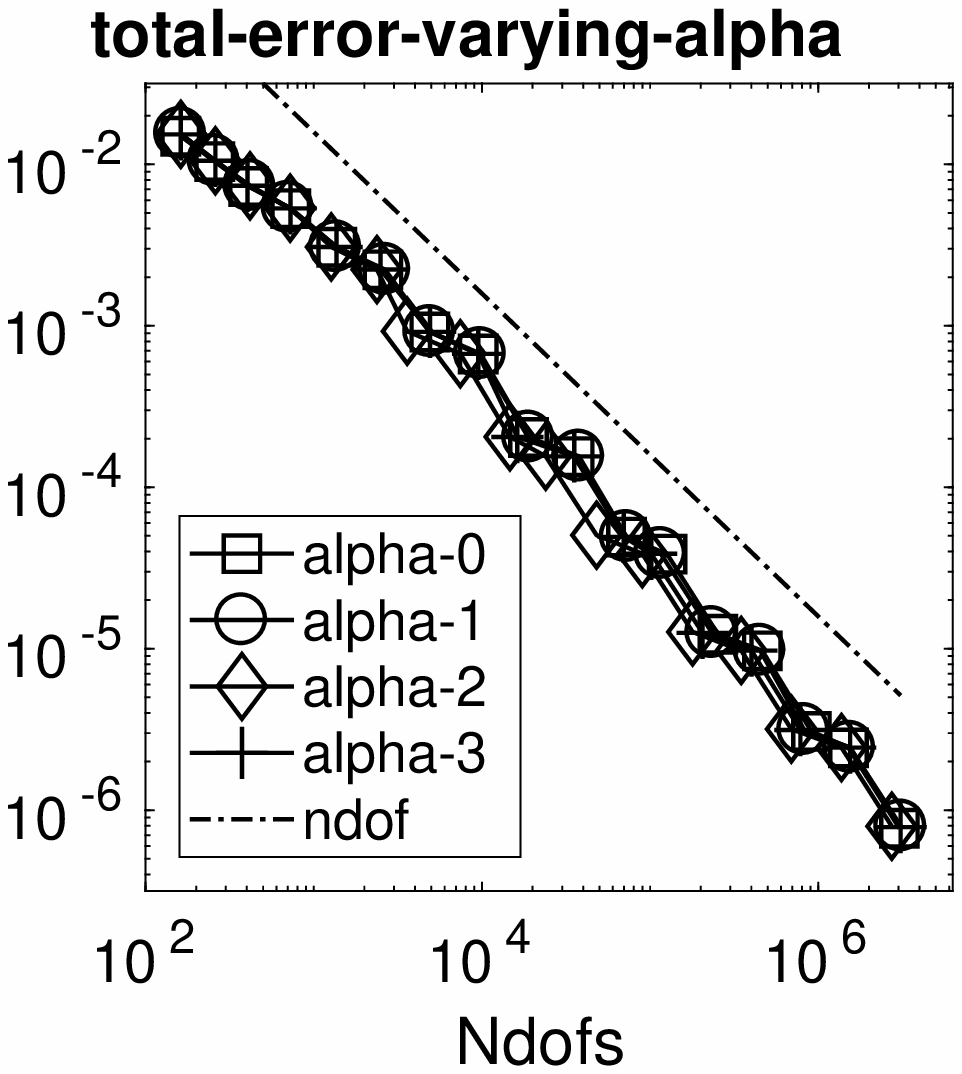}\\
\tiny{(A.2)} \\
\psfrag{total-error-varying-alpha}{\quad \quad \normalsize{Adjoint error $\|e_p\|_{L^2(\Omega)}$ varying $\alpha$}}
\includegraphics[trim={0 0 0 0},clip,width=2.9cm,height=2.8cm,scale=0.6]{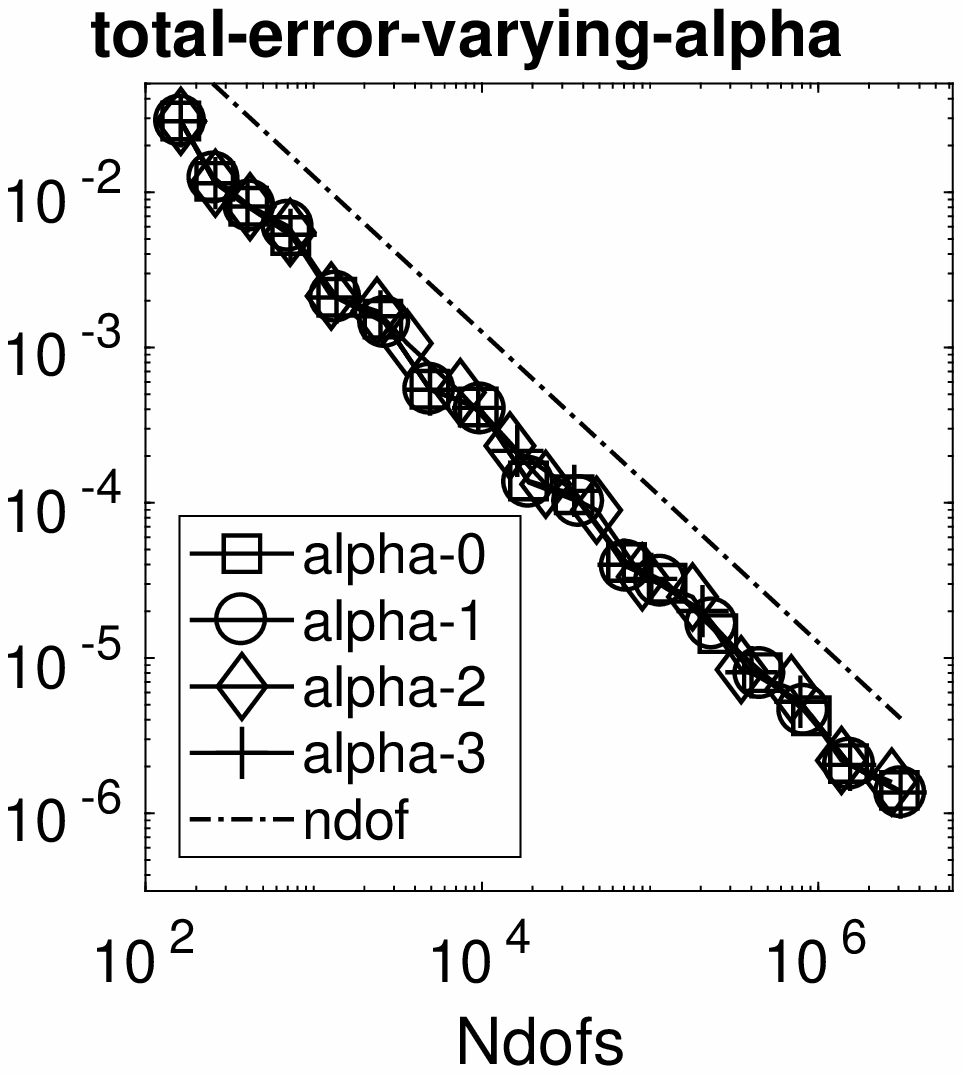}\\
\tiny{(A.3)}\\
\psfrag{total-error-varying-alpha}{\quad \quad \normalsize{Control error $\|e_u\|_{L^2(\Omega)}$ varying $\alpha$}}
\includegraphics[trim={0 0 0 0},clip,width=2.9cm,height=2.8cm,scale=0.6]{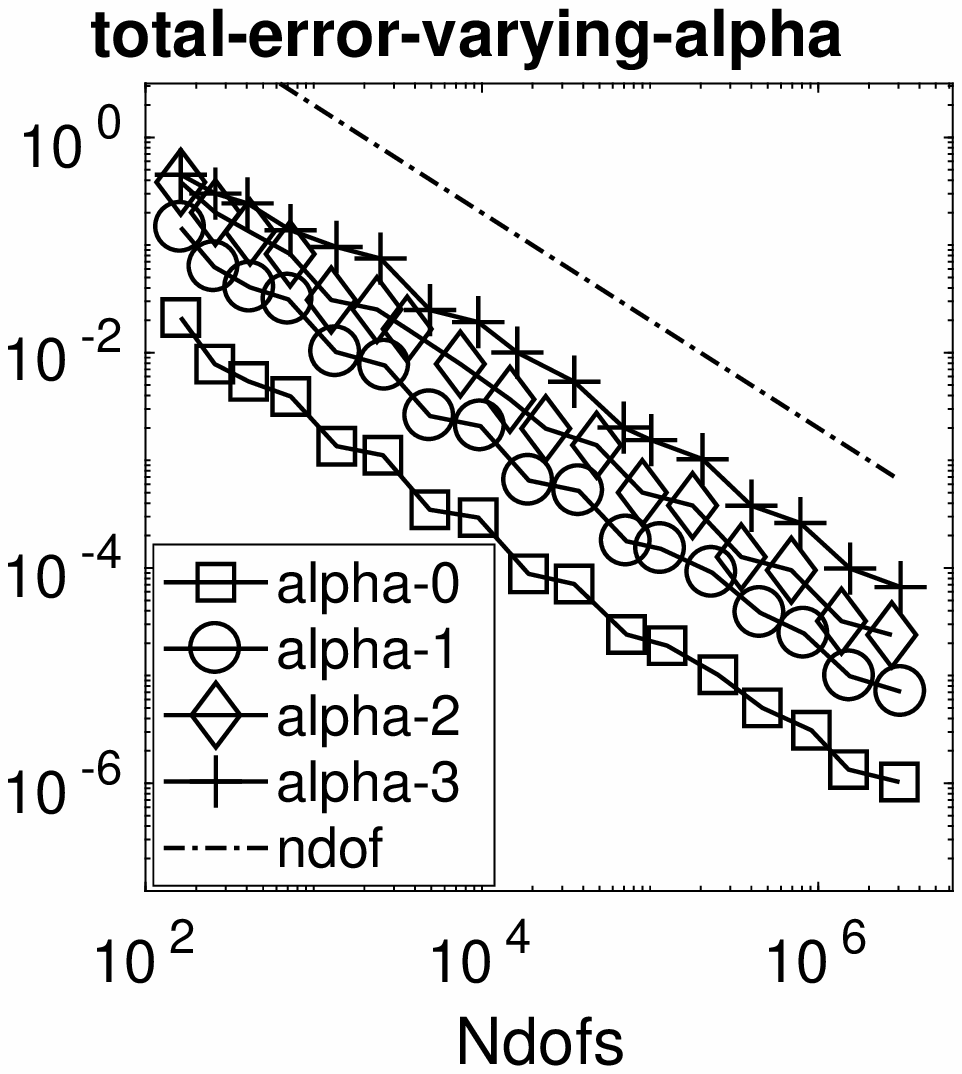}\\
\tiny{(A.4)} \\
\psfrag{total-error-varying-alpha}{\normalsize{Subgradient error $\|e_\lambda\|_{L^2(\Omega)}$ varying $\alpha$}}
\includegraphics[trim={0 0 0 0},clip,width=2.9cm,height=2.8cm,scale=0.6]{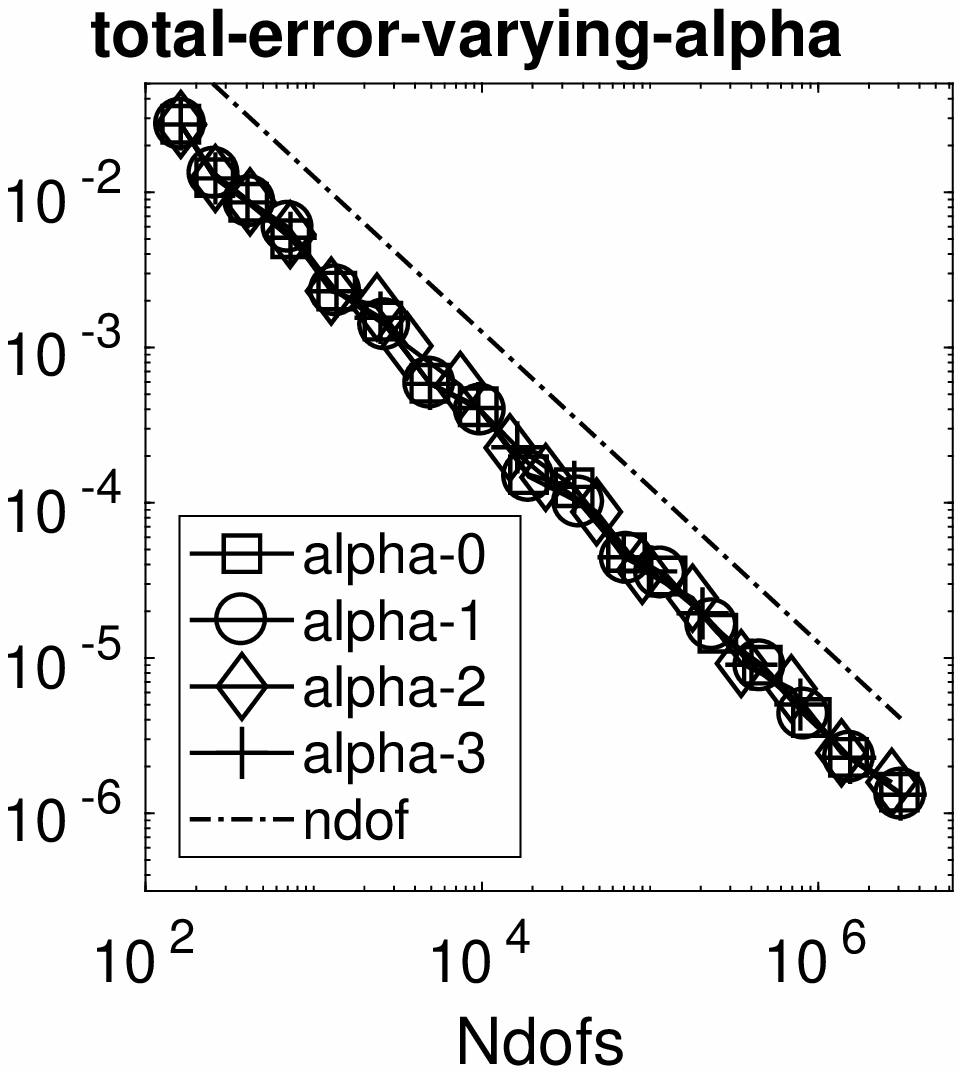}\\
\tiny{(A.5)}
\end{minipage}
\begin{minipage}{0.23\textwidth}\centering
\psfrag{total-error-varying-alpha}{\quad \quad \normalsize{Total estimator $\mathfrak{E}$ varying $\alpha$}}
\psfrag{alpha-0}{\small{$\alpha=10^0$}}
\psfrag{alpha-1}{\small{$\alpha=10^{-1}$}}
\psfrag{alpha-2}{\small{$\alpha=10^{-2}$}}
\psfrag{alpha-3}{\small{$\alpha=10^{-3}$}}
\psfrag{alpha-4}{\small{$\alpha=10^{-4}$}}
\psfrag{alpha-5}{\small{$\alpha=10^{-5}$}}
\psfrag{Ndofs}{\normalsize{$\textrm{Ndof}^{}_2$}}
\psfrag{ndof}{{\footnotesize{\FF{$\textrm{Ndof}^{-1}_2$}}}}
\psfrag{ndof2}{{\footnotesize{$\textrm{Ndof}^{-1}_2$}}}
\includegraphics[trim={0 0 0 0},clip,width=2.9cm,height=2.8cm,scale=0.6]{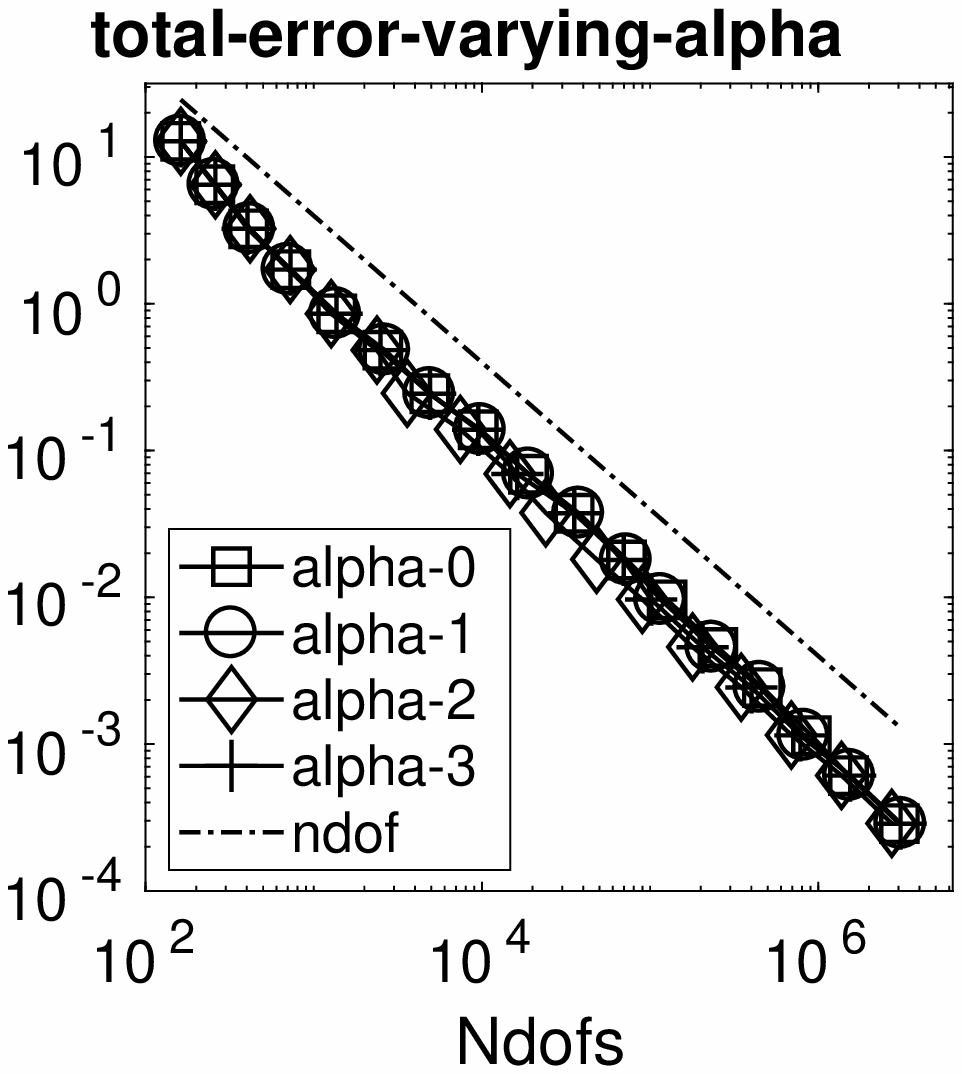}\\
\tiny{(B.1)}\\
\psfrag{total-error-varying-alpha}{\quad \normalsize{Individual contribution ${E}_{y}$ varying $\alpha$}}
\includegraphics[trim={0 0 0 0},clip,width=2.9cm,height=2.8cm,scale=0.6]{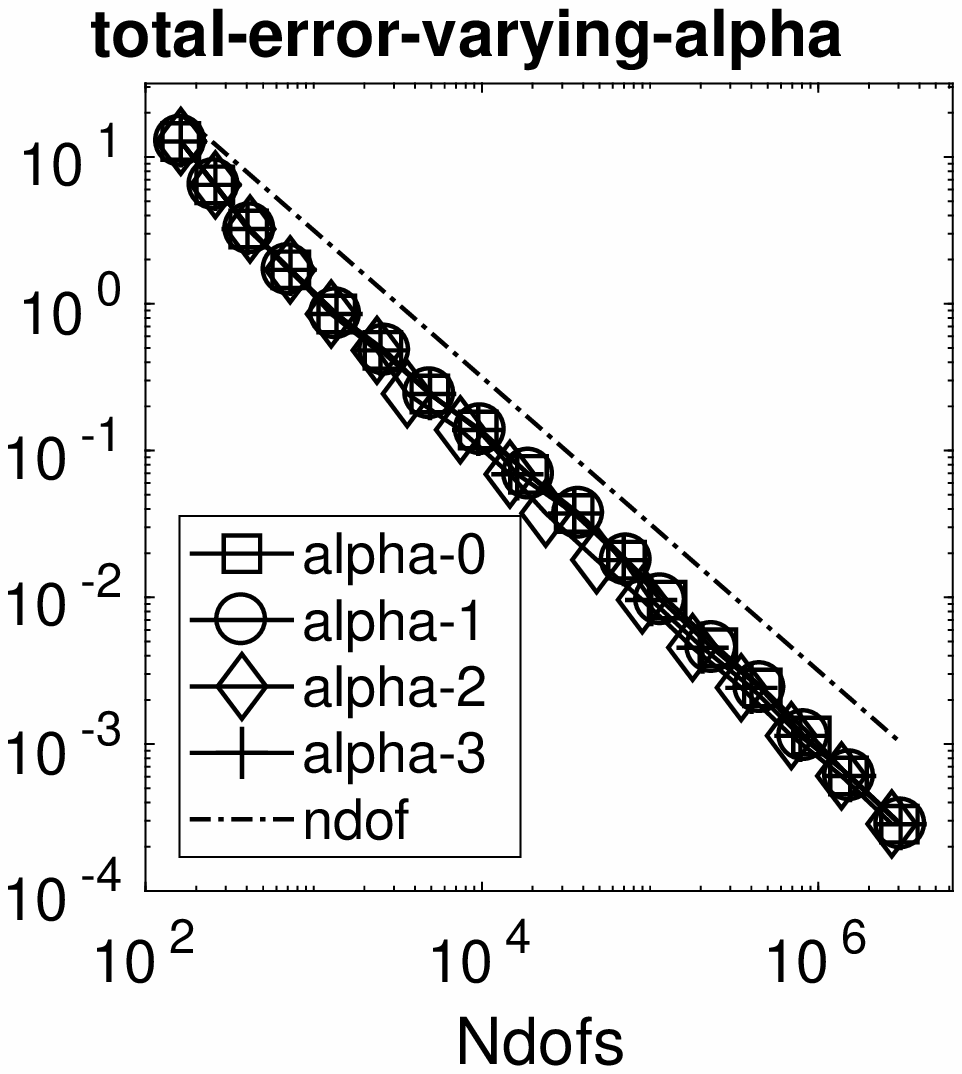}\\
\tiny{(B.2)}\\
\psfrag{total-error-varying-alpha}{\quad \normalsize{Individual contribution ${E}_{p}$ varying $\alpha$}}
\includegraphics[trim={0 0 0 0},clip,width=2.9cm,height=2.8cm,scale=0.6]{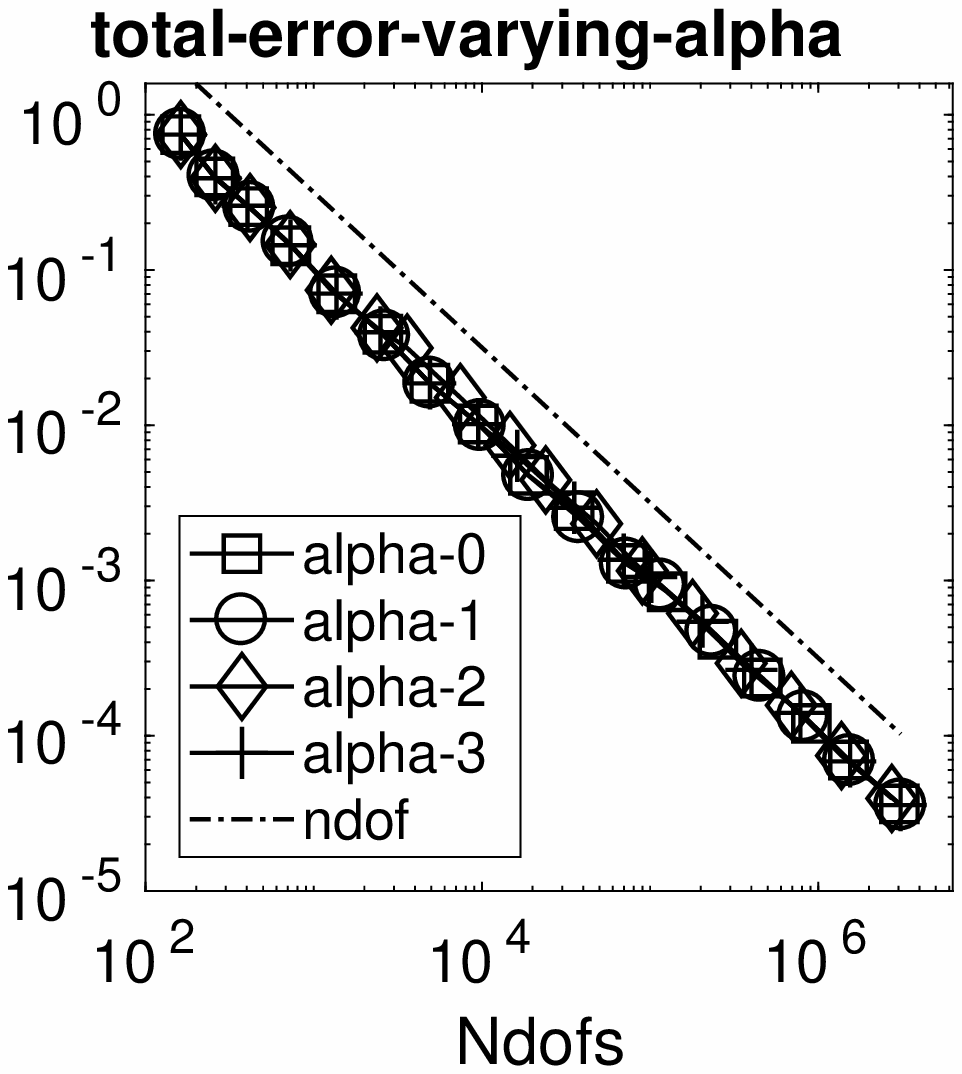}\\
\tiny{(B.3)}\\
\vspace{24.5em}
\end{minipage}
\begin{minipage}{0.23\textwidth}\centering
\psfrag{total-error-varying-beta}{\quad \quad \normalsize{Total error $\| e\|_{\Omega}$ varying $\beta$}}
\psfrag{betaa-0}{\small{$\beta=10^0$}}
\psfrag{betaa-1}{\small{$\beta=10^{-1}$}}
\psfrag{betaa-2}{\small{$\beta=10^{-2}$}}
\psfrag{betaa-3}{\small{$\beta=10^{-3}$}}
\psfrag{betaa-4}{\small{$\beta=10^{-4}$}}
\psfrag{betaa-5}{\small{$\beta=10^{-5}$}}
\psfrag{Ndofs}{\normalsize{$\textrm{Ndof}^{}_2$}}
\psfrag{ndof}{{\footnotesize{\FF{$\textrm{Ndof}^{-1}_2$}}}}
\psfrag{ndof2}{$\footnotesize{\textrm{Ndof}^{-1}_2}$}
\includegraphics[trim={0 0 0 0},clip,width=2.9cm,height=2.8cm,scale=0.6]{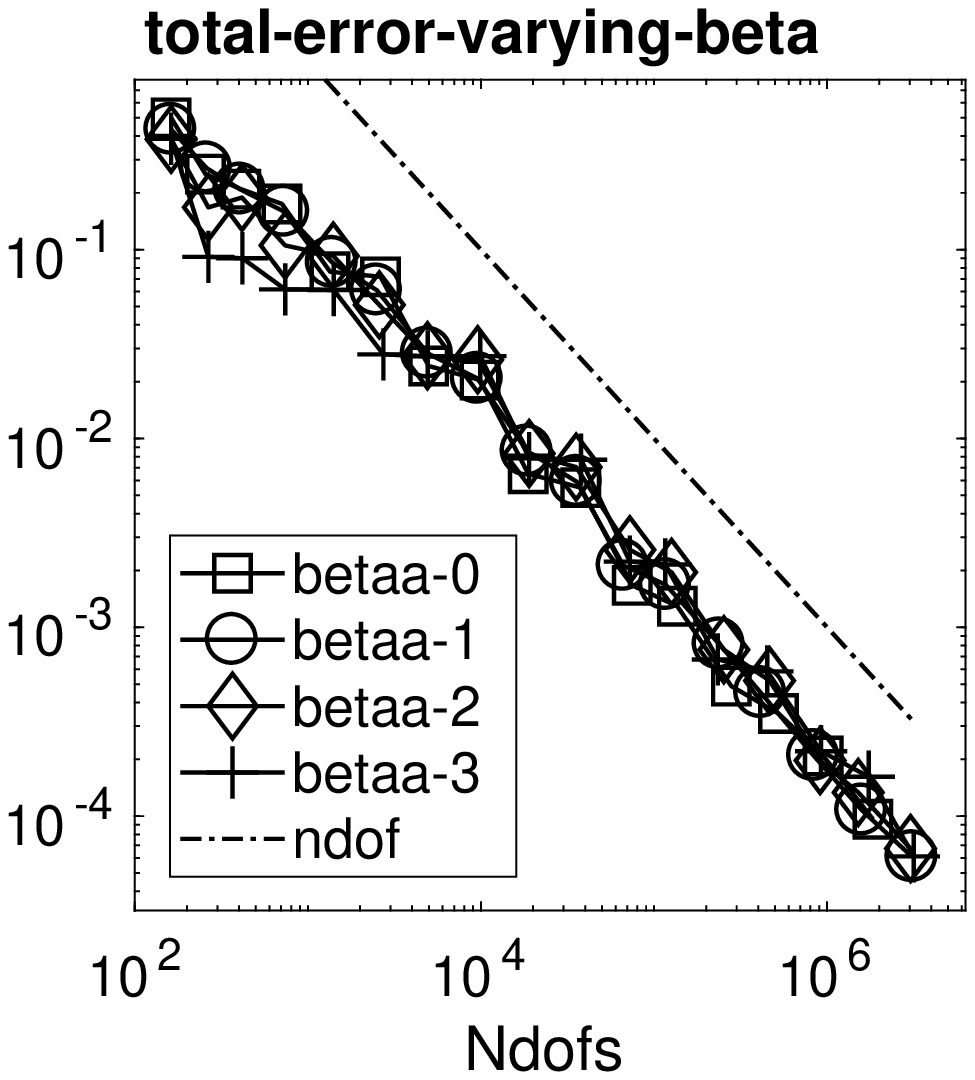}\\
\tiny{(C.1)}\\
\psfrag{total-error-varying-beta}{\quad \normalsize{State error $\|e_{y}\|_{L^{2}(\Omega)}$ varying $\beta$}}
\includegraphics[trim={0 0 0 0},clip,width=2.9cm,height=2.8cm,scale=0.6]{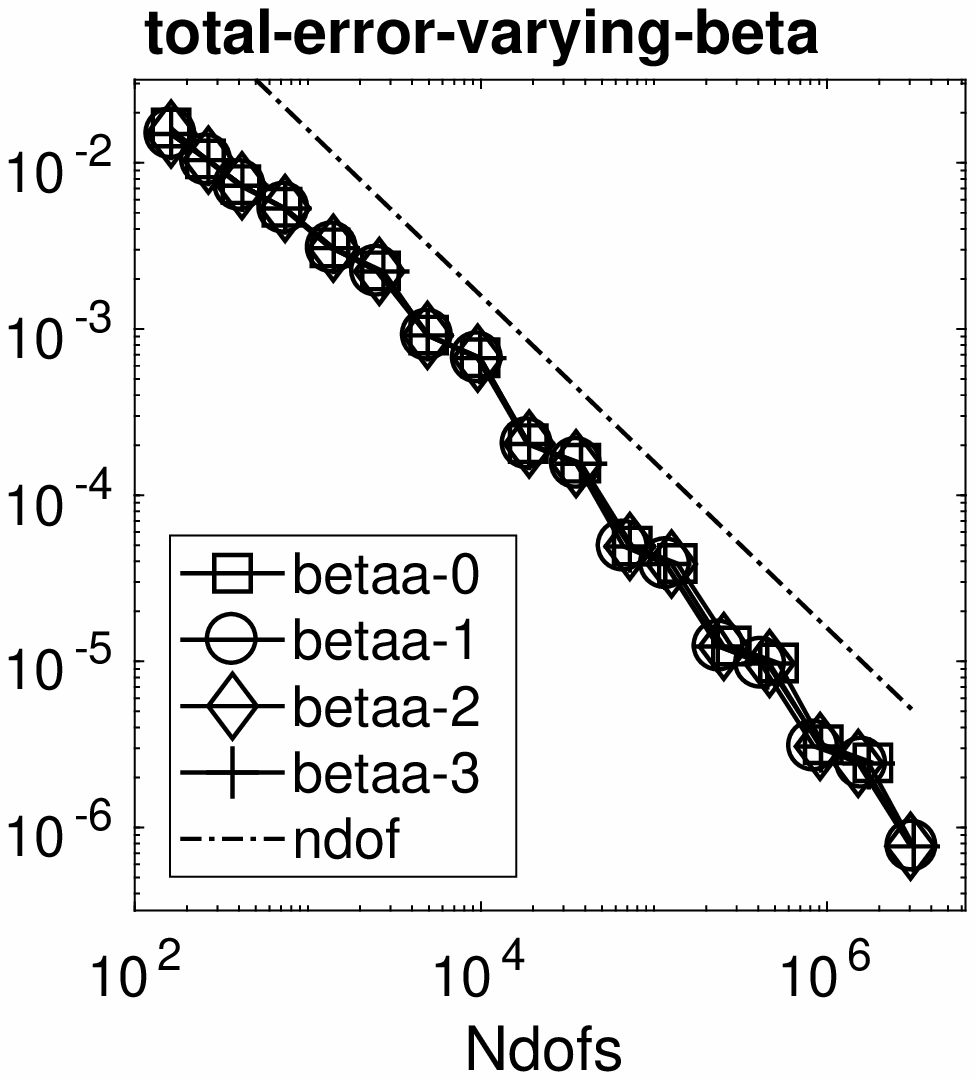}\\
\tiny{(C.2)}\\
\psfrag{total-error-varying-beta}{\quad \normalsize{Adjoint error $\|e_{p}\|_{L^{2}(\Omega)}$ varying $\beta$}}
\includegraphics[trim={0 0 0 0},clip,width=2.9cm,height=2.8cm,scale=0.6]{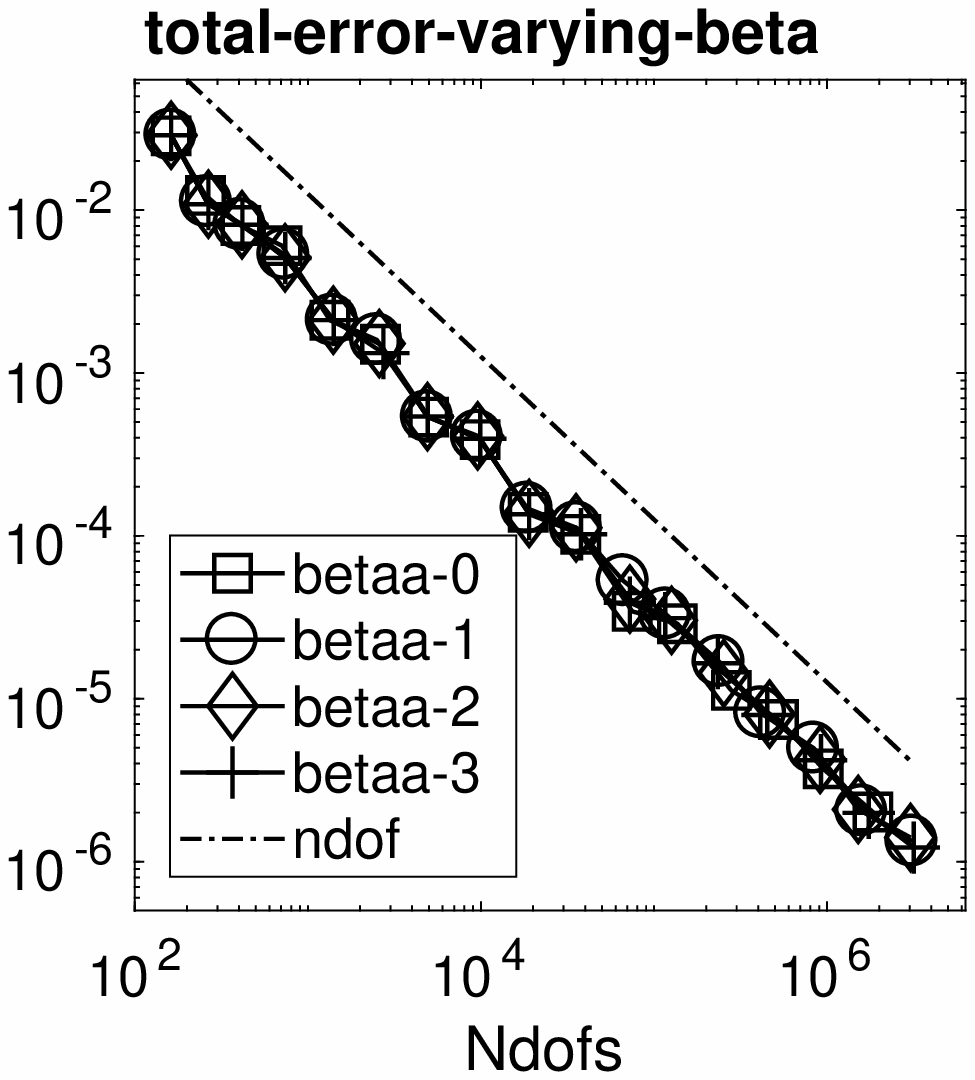}\\
\tiny{(C.3)}\\
\psfrag{total-error-varying-beta}{\normalsize{Control error $\|e_u\|_{L^{2}(\Omega)}$ varying $\beta$}}
\includegraphics[trim={0 0 0 0},clip,width=2.9cm,height=2.8cm,scale=0.6]{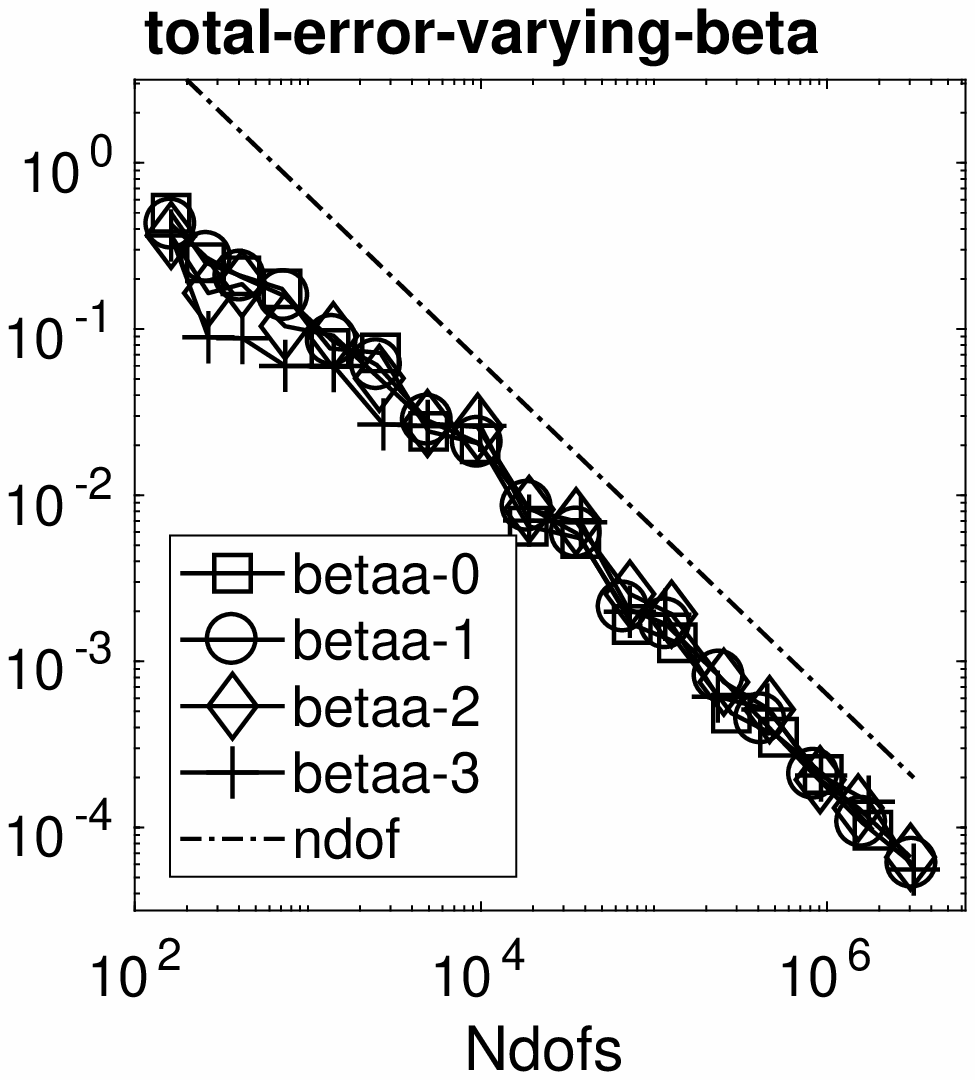}\\
\tiny{(C.4)}\\
\psfrag{total-error-varying-beta}{\hspace{-0.2cm}\normalsize{Subgradient error $\|e_\lambda\|_{L^{2}(\Omega)}$ varying $\beta$}}
\includegraphics[trim={0 0 0 0},clip,width=2.9cm,height=2.8cm,scale=0.6]{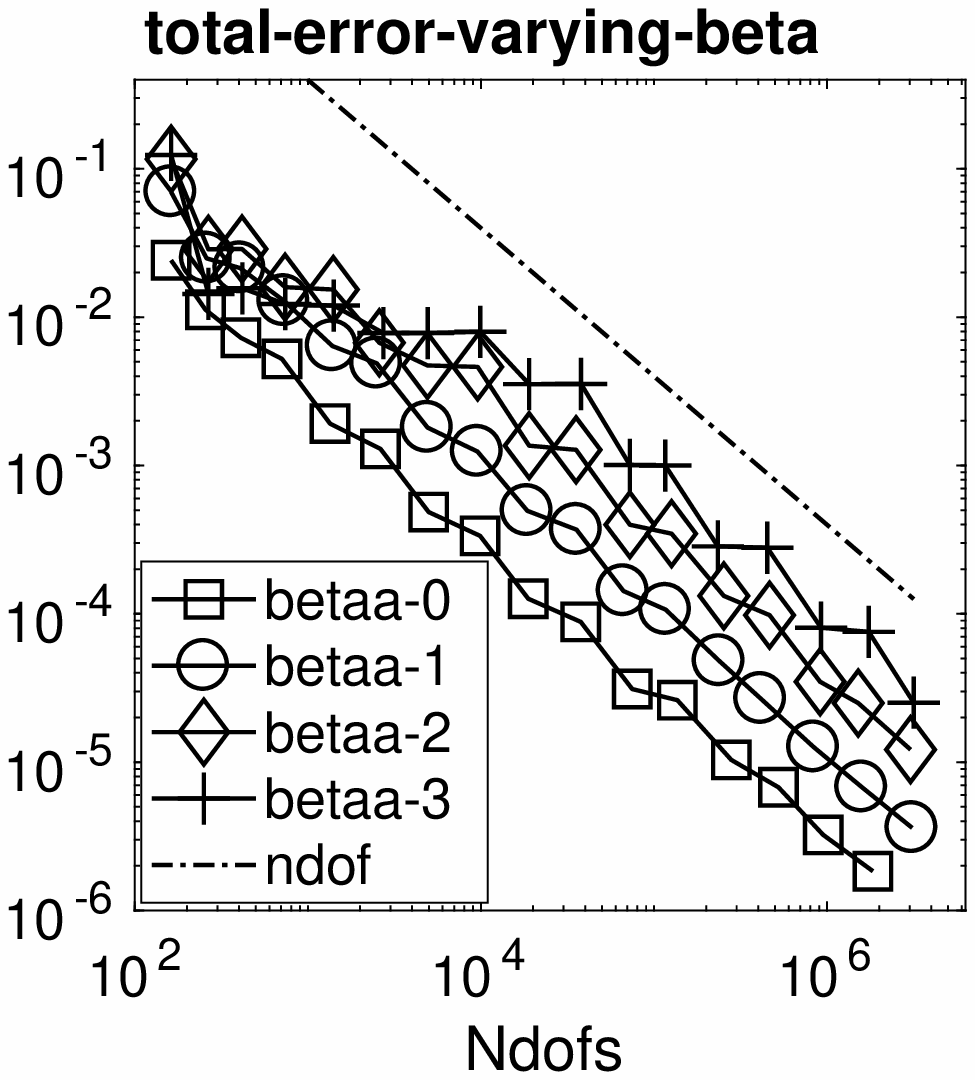}\\
\tiny{(C.5)}
\end{minipage}
\begin{minipage}{0.23\textwidth}\centering
\psfrag{total-error-varying-beta}{\quad \quad \normalsize{Total estimator $\mathfrak{E}$ varying $\beta$}}
\psfrag{betaa-0}{\small{$\beta=10^0$}}
\psfrag{betaa-1}{\small{$\beta=10^{-1}$}}
\psfrag{betaa-2}{\small{$\beta=10^{-2}$}}
\psfrag{betaa-3}{\small{$\beta=10^{-3}$}}
\psfrag{betaa-4}{\small{$\beta=10^{-4}$}}
\psfrag{betaa-5}{\small{$\beta=10^{-5}$}}
\psfrag{Ndofs}{\normalsize{$\textrm{Ndof}^{}_2$}}
\psfrag{ndof}{{\footnotesize{\FF{$\textrm{Ndof}^{-1}_2$}}}}
\psfrag{ndof2}{$\footnotesize{\textrm{Ndof}^{-1}_2}$}
\includegraphics[trim={0 0 0 0},clip,width=2.9cm,height=2.8cm,scale=0.6]{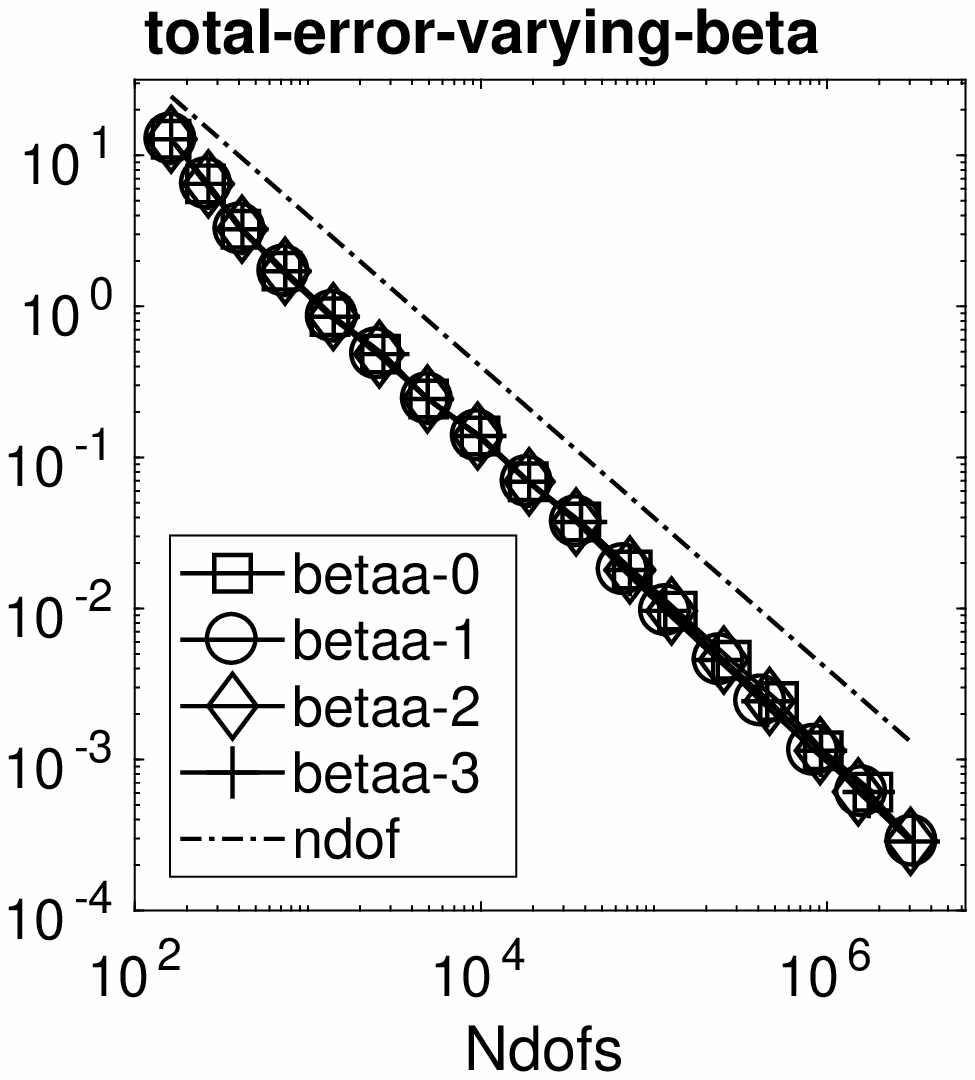}\\
\tiny{(D.1)}\\
\psfrag{total-error-varying-beta}{\normalsize{Individual contribution ${E}_{y}$ varying $\beta$}}
\includegraphics[trim={0 0 0 0},clip,width=2.9cm,height=2.8cm,scale=0.6]{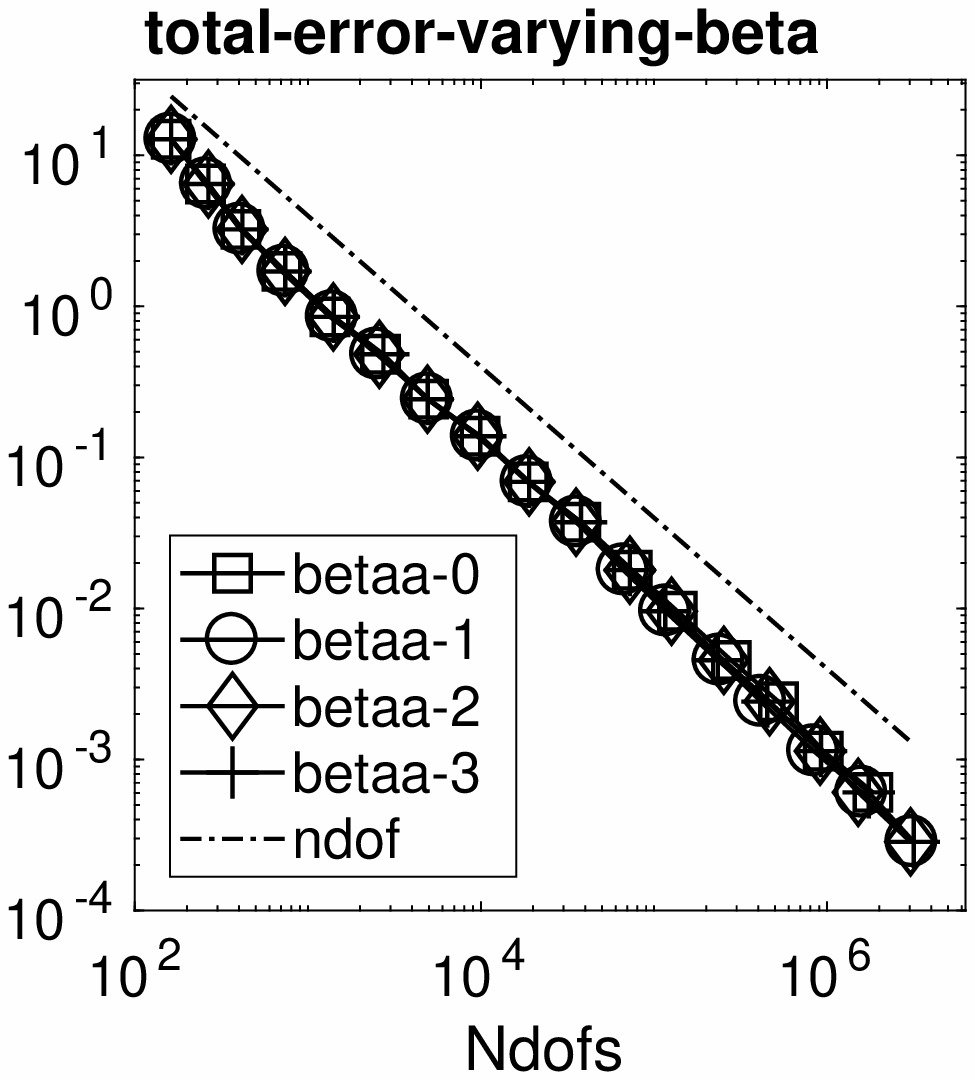}\\
\tiny{(D.2)}\\
\psfrag{total-error-varying-beta}{\normalsize{Individual contribution ${E}_{p}$ varying $\beta$}}
\includegraphics[trim={0 0 0 0},clip,width=2.9cm,height=2.8cm,scale=0.6]{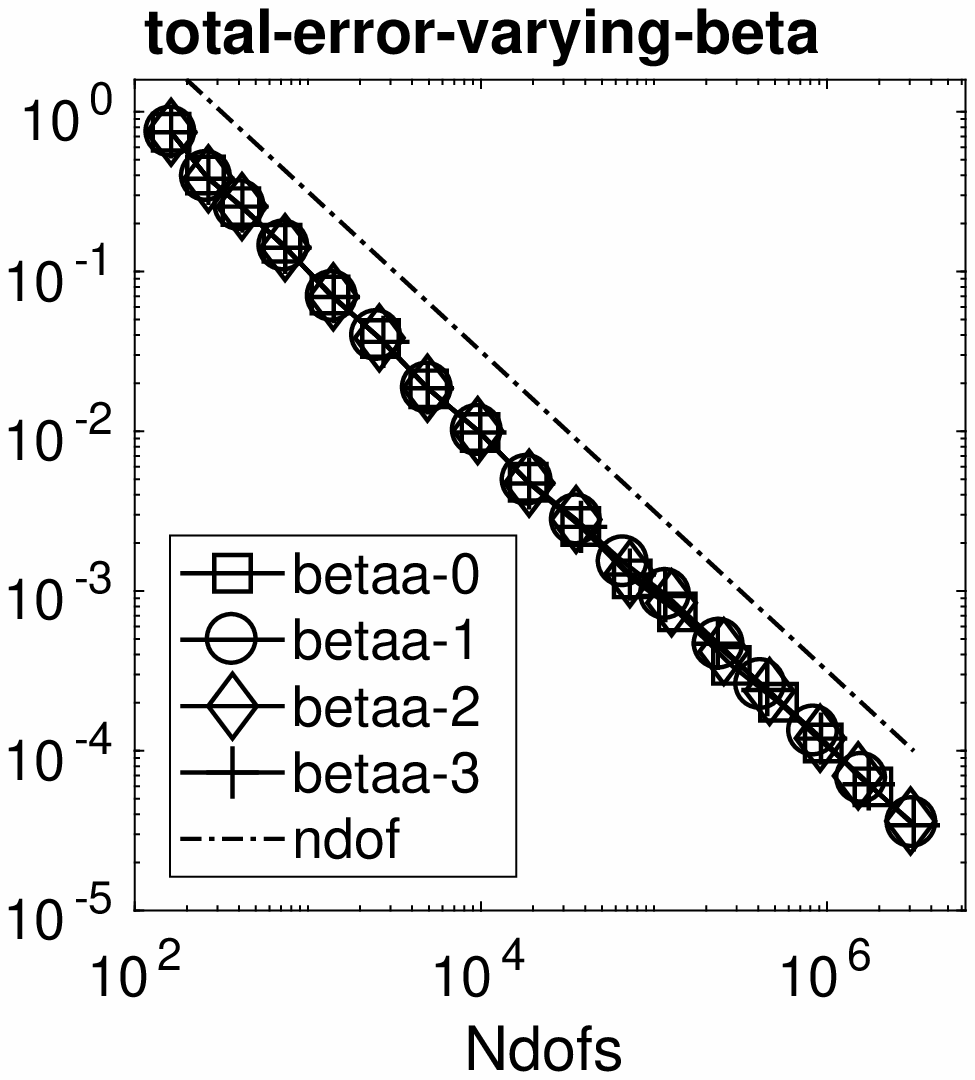}\\
\tiny{(D.3)}\\
\vspace{24.5em}
\end{minipage}
\caption{Example 1. \FF{Experimental rates of convergence for the variational discretization scheme of Section \ref{subsec:vd}. In (A.1)--(A.5) and (B.1)--(B.3) we have considered $\beta=7\cdot 10^{-1}$ and $\alpha \in \{10^{0},10^{-1},10^{-2},10^{-3}\}$ while in (C.1)--(C.5) and (D.1)--(D.3) we have considered $\alpha=10^{-3}$ and $\beta \in \{10^{0},10^{-1},10^{-2},10^{-3}\}$}.}
\label{ex_3}
\end{figure}


\begin{figure}[!h]
\centering
\begin{minipage}{0.3\textwidth}\centering
\psfrag{eff-index-varying-beta}{\quad \normalsize{Effectiviy index $\mathcal{E}/\VERT e \VERT_{\Omega}$}}
\psfrag{eff-index}{\normalsize{$\mathcal{E}/\VERT e \VERT_{\Omega}$}}
\psfrag{Ndofs}{\normalsize{$\textrm{Ndof}^{}_0$}}
\psfrag{ndof}{{\footnotesize{\FF{$\textrm{Ndof}^{-1}_2$}}}}
\psfrag{ndof2}{{\footnotesize{$\textrm{Ndof}^{-1}_2$}}}
\includegraphics[trim={0 0 0 0},clip,width=3.7cm,height=3.5cm,scale=0.6]{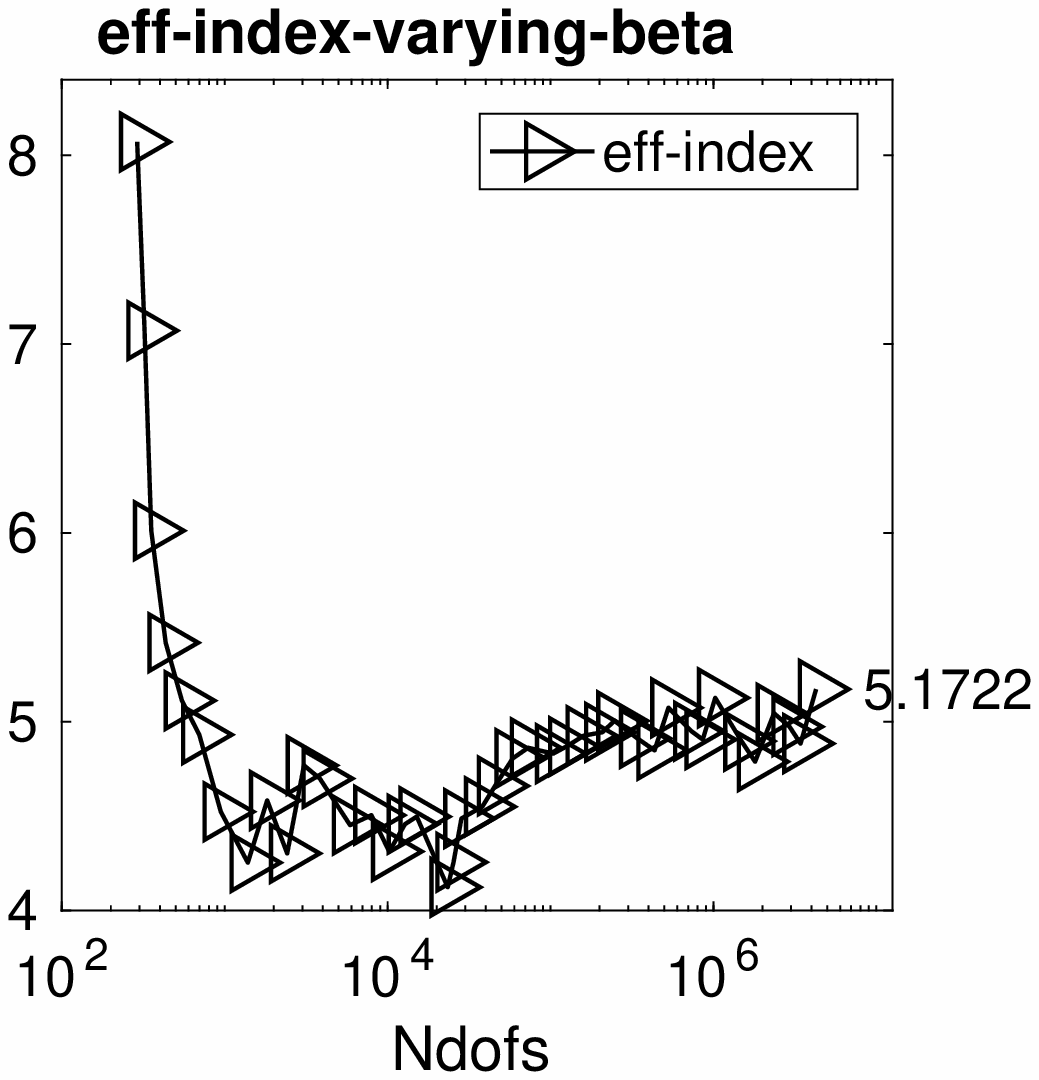}\\
\tiny{(A.1)}
\end{minipage}
\begin{minipage}{0.3\textwidth}\centering
\centering
\psfrag{eff-index-varying-alpha}{\quad \normalsize{Effectiviy index $E/\| e \|_{\Omega}$}}
\psfrag{eff-index}{\normalsize{$E/\| e \|_{\Omega}$}}
\psfrag{Ndofs}{\normalsize{$\textrm{Ndof}^{}_1$}}
\psfrag{ndof}{{\footnotesize{\FF{$\textrm{Ndof}^{-1}_2$}}}}
\psfrag{ndof2}{{\footnotesize{$\textrm{Ndof}^{-1}_2$}}}
\includegraphics[trim={0 0 0 0},clip,width=3.7cm,height=3.5cm,scale=0.6]{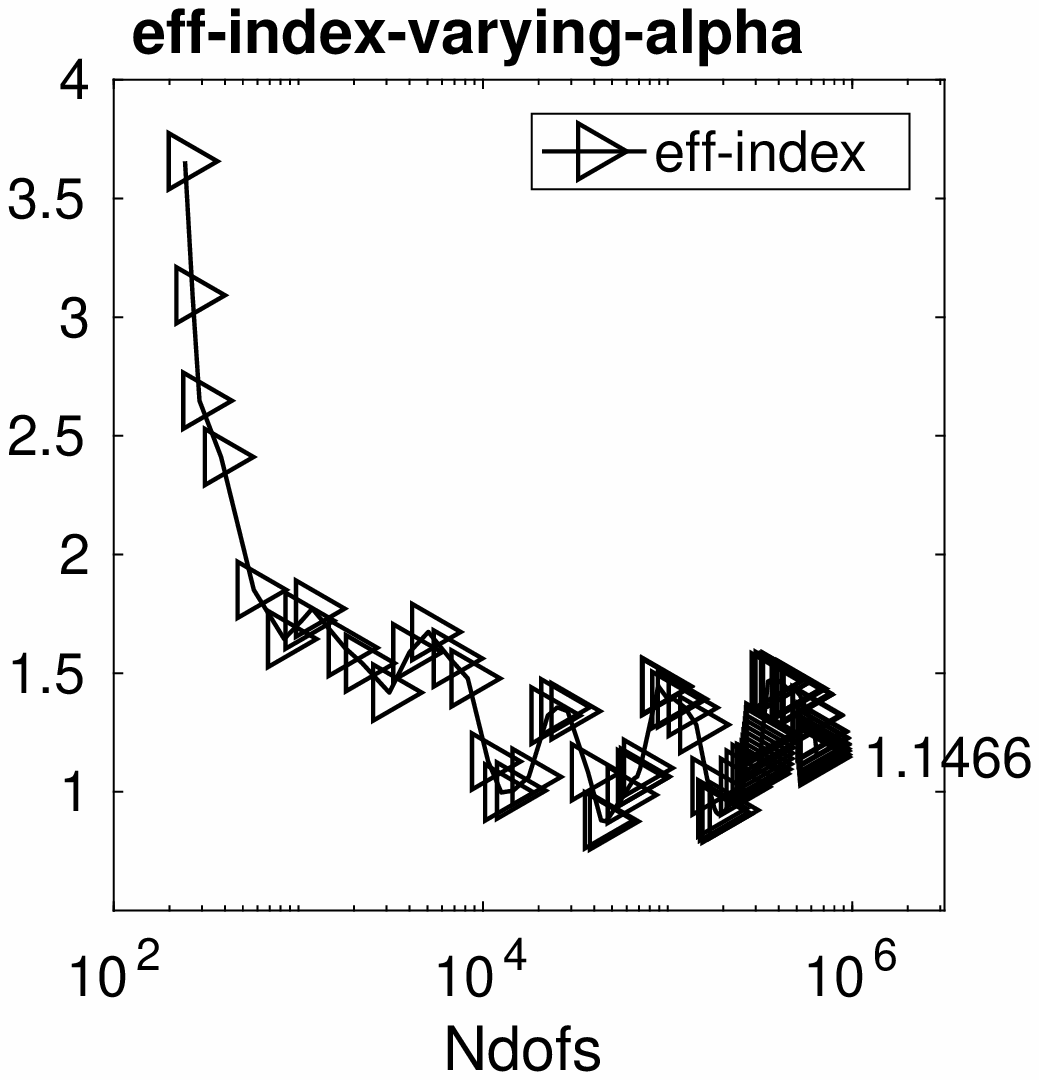}\\
\tiny{(B.1)}
\end{minipage}
\begin{minipage}{0.3\textwidth}\centering
\centering
\psfrag{eff-index-varying-beta}{\quad \normalsize{Effectiviy index $\mathfrak{E}/\| e \|_{\Omega}$}}
\psfrag{eff-index}{\normalsize{$\mathfrak{E}/\| e \|_{\Omega}$}}
\psfrag{Ndofs}{\normalsize{$\textrm{Ndof}^{}_2$}}
\psfrag{ndof}{{\footnotesize{\FF{$\textrm{Ndof}^{-1}_2$}}}}
\psfrag{ndof2}{{\footnotesize{$\textrm{Ndof}^{-1}_2$}}}
\includegraphics[trim={0 0 0 0},clip,width=3.7cm,height=3.5cm,scale=0.6]{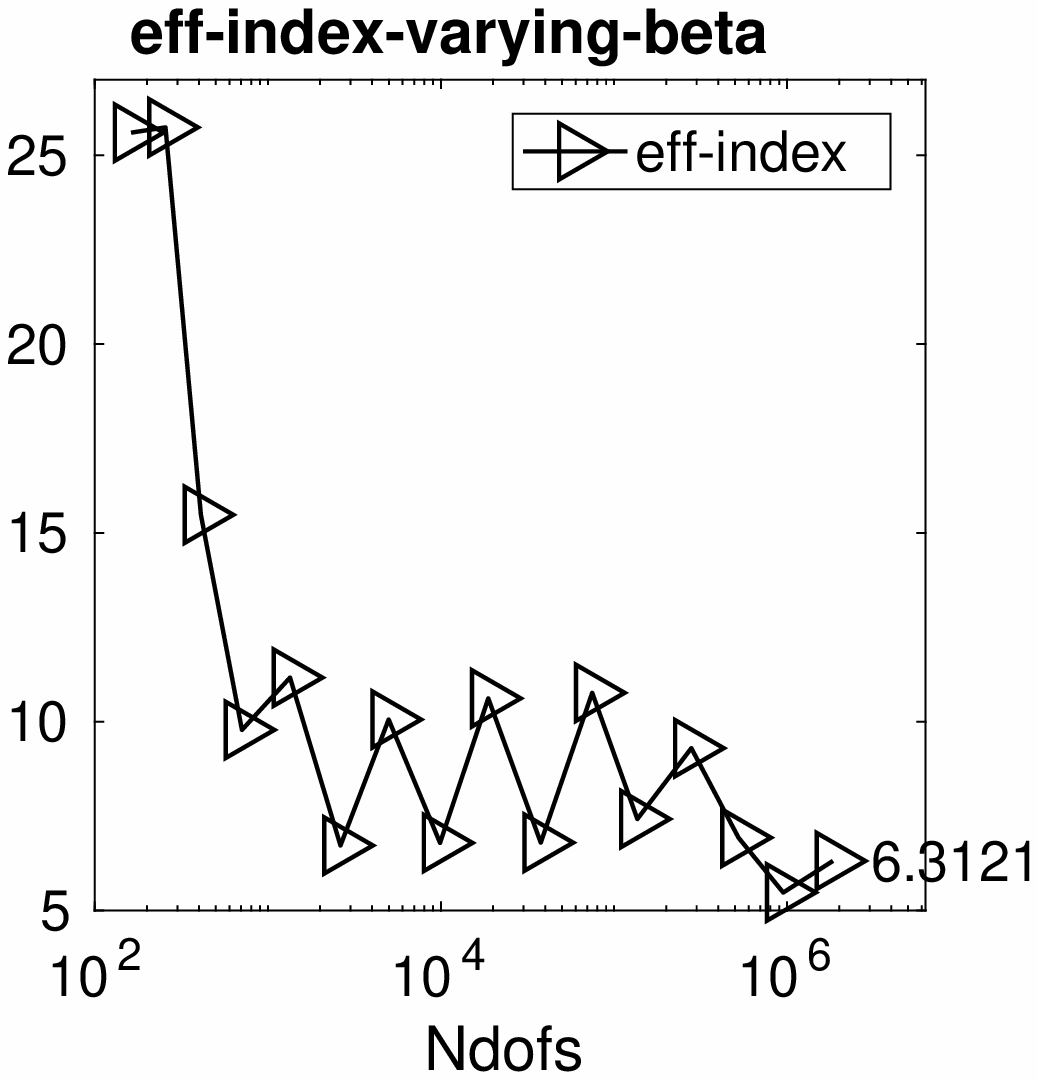}\\
\tiny{(C.1)}
\end{minipage}
\caption{Example 1. \FF{Effectivity indices for the piecewise constant (A.1), piecewise linear (B.1), and variational discretization scheme (C.1). In each case we have considered $\alpha=10^{-3}$ and $\beta=10^{0}$.}}
\label{ex_eff}
\end{figure}

\begin{figure}[!h]
\centering
\begin{minipage}[c]{0.28\textwidth}\centering
Piecewise constant:\\
\tiny{Uniform refinement.}\\
\psfrag{error t11}{\normalsize{$\VERT e \VERT_\Omega$}}
\psfrag{esti t}{\normalsize{$\mathcal{E}$}}
\psfrag{eff-index}{\normalsize{$\mathcal{E}/\VERT e \VERT_\Omega$}}
\psfrag{Ndofs}{\normalsize{$\textrm{Ndof}^{}_0$}}
\psfrag{ndof}{$\footnotesize{\textrm{Ndof}^{-1/2}_0}$}
\psfrag{aprox P1P1P0-SPARSE-L2}{$\mathbb{V}(\mathscr{T})\times \mathbb{V}(\mathscr{T})\times \mathbb{U}_{ad,0}(\mathscr{T})$}
\psfrag{aprox P1P1P0-SPARSE-L2}{$\mathbb{V}(\mathscr{T})\times \mathbb{V}(\mathscr{T})\times \mathbb{U}_{ad,0}(\mathscr{T})$}
\psfrag{ndof}{$\footnotesize{\textrm{Ndof}^{-1/3}_0}$}
\psfrag{eff-index}{\normalsize{$\mathcal{E}/\VERT e \VERT_\Omega$}}
\includegraphics[trim={0 0 0 0},clip,width=3.6cm,height=3.3cm,scale=0.6]{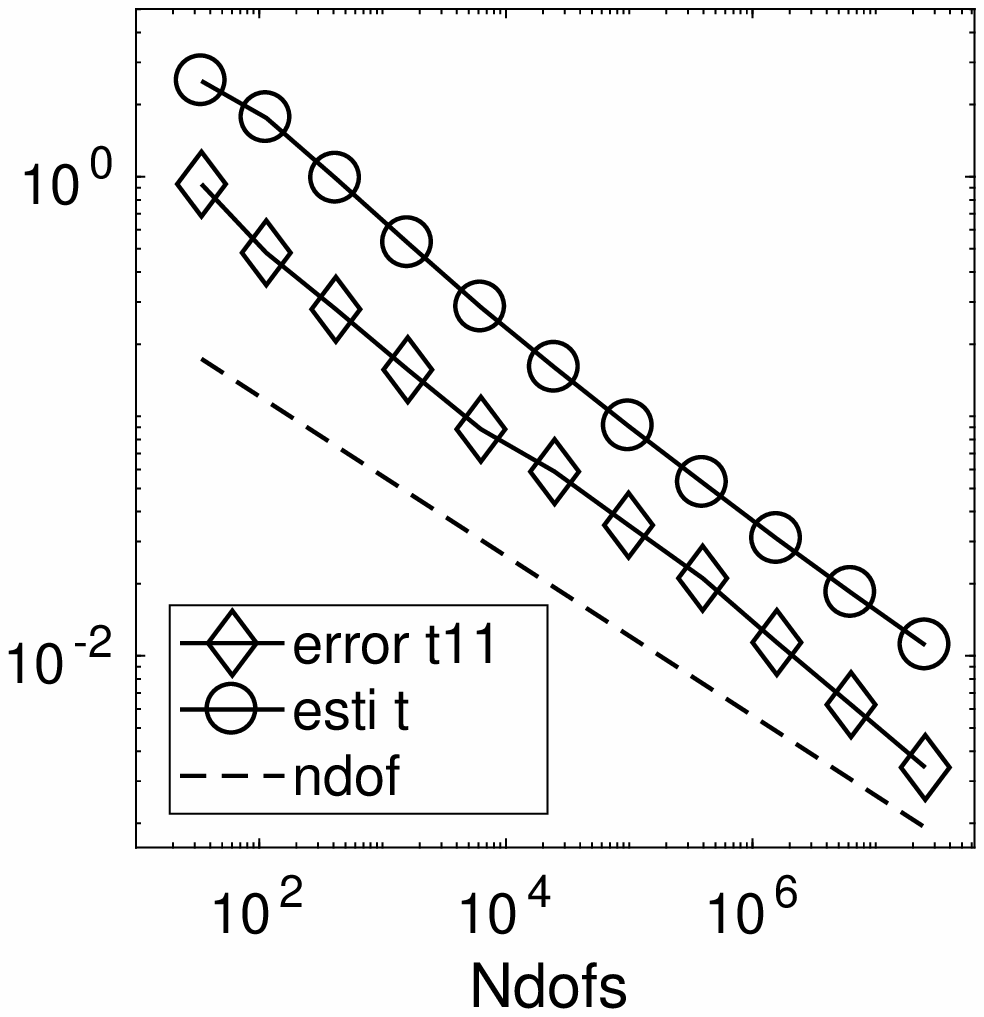}\\
\tiny{(A.1)}~\\~\\
\tiny{Adaptive refinement.}\\
\psfrag{ndof}{$\footnotesize{\textrm{Ndof}^{-1/2}_0}$}
\includegraphics[trim={0 0 0 0},clip,width=3.6cm,height=3.3cm,scale=0.6]{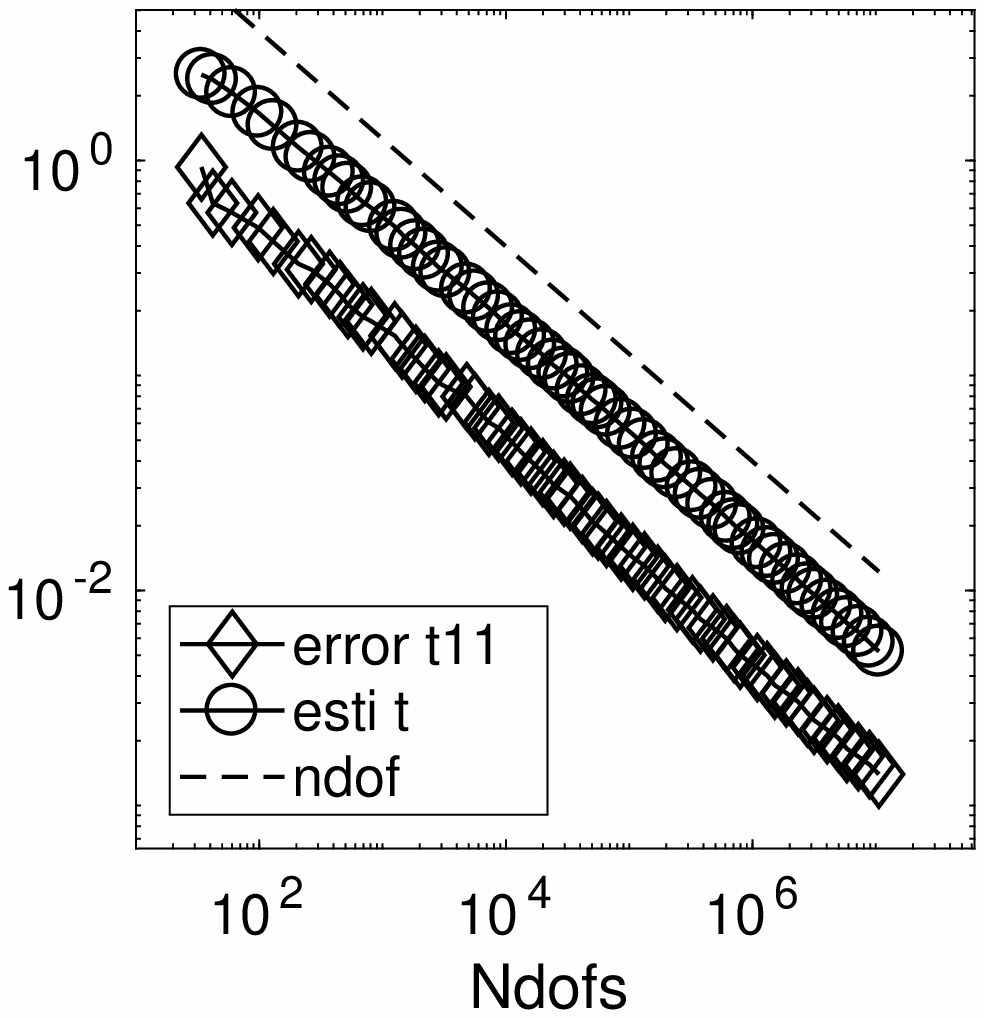}\\
\tiny{(A.2)}\\~\\
\includegraphics[trim={0 0 0 0},clip,width=3.7cm,height=3.3cm,scale=0.6]{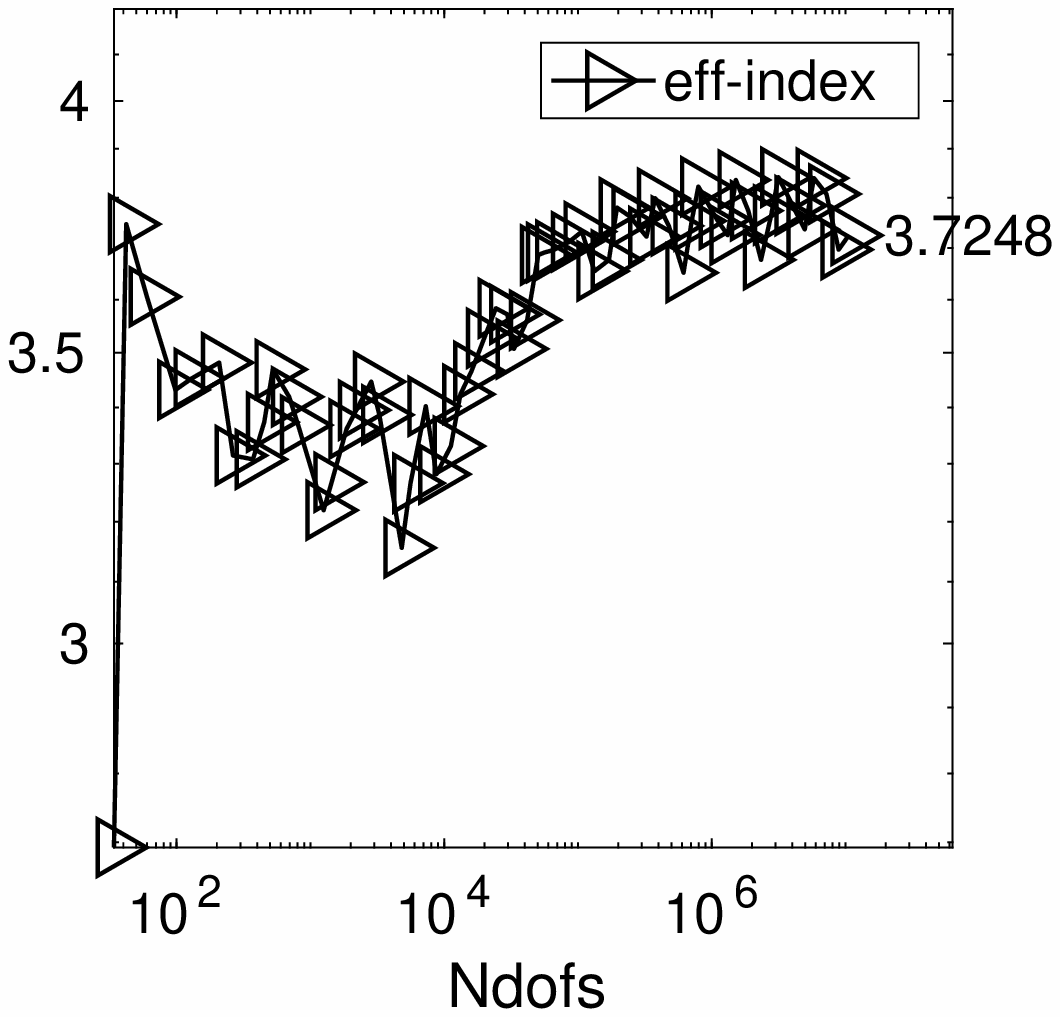}\\\tiny{(A.3)}
\end{minipage}
\begin{minipage}[c]{0.28\textwidth}\centering
Piecewise linear:\\
\tiny{Uniform refinement.}\\
\psfrag{error t11}{\normalsize{$\| e \|_\Omega$}}
\psfrag{esti t}{\normalsize{$E$}}
\psfrag{Ndofs}{\normalsize{$\textrm{Ndof}^{}_0$}}
\psfrag{Ndofs}{\normalsize{$\textrm{Ndof}_1$}}
\psfrag{ndof}{\footnotesize{$\textrm{Ndof}^{-1/2}_1$}}
\psfrag{ndof}{$\footnotesize{\textrm{Ndof}^{-2/3}_1}$}
\psfrag{ndof2}{$\footnotesize{\textrm{Ndof}^{-1}_1}$}
\psfrag{aprox P1P1P1-SPARSE-l2}{$\mathbb{V}(\mathscr{T})\times \mathbb{V}(\mathscr{T})\times \mathbb{U}_{ad,1}(\mathscr{T})$}
\psfrag{eff-index}{\normalsize{${E}/ \| e \|_\Omega$}}
\includegraphics[trim={0 0 0 0},clip,width=3.7cm,height=3.3cm,scale=0.6]{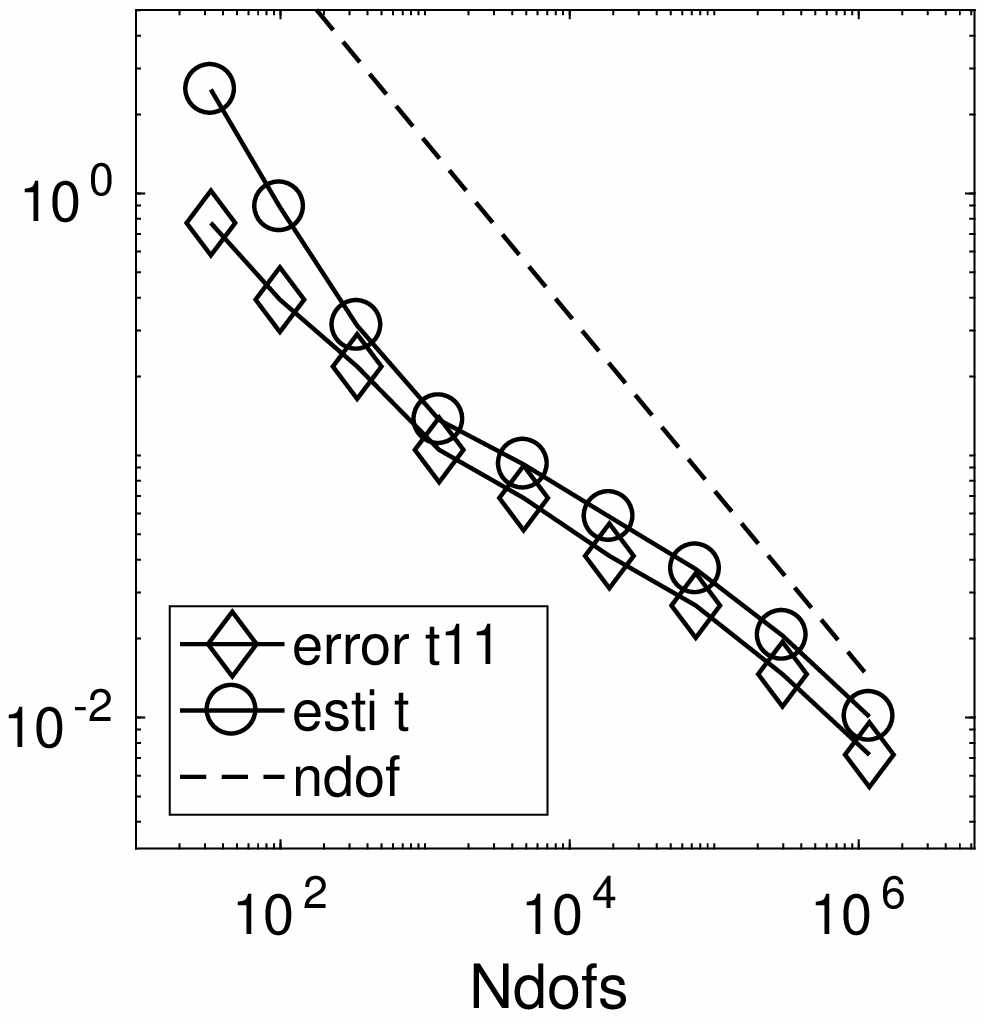}\\
\tiny{(B.1)}\\~\\
\tiny{Adaptive refinement.}\\
\psfrag{ndof}{$\footnotesize{\textrm{Ndof}^{-1}_1}$}
\includegraphics[trim={0 0 0 0},clip,width=3.7cm,height=3.3cm,scale=0.6]{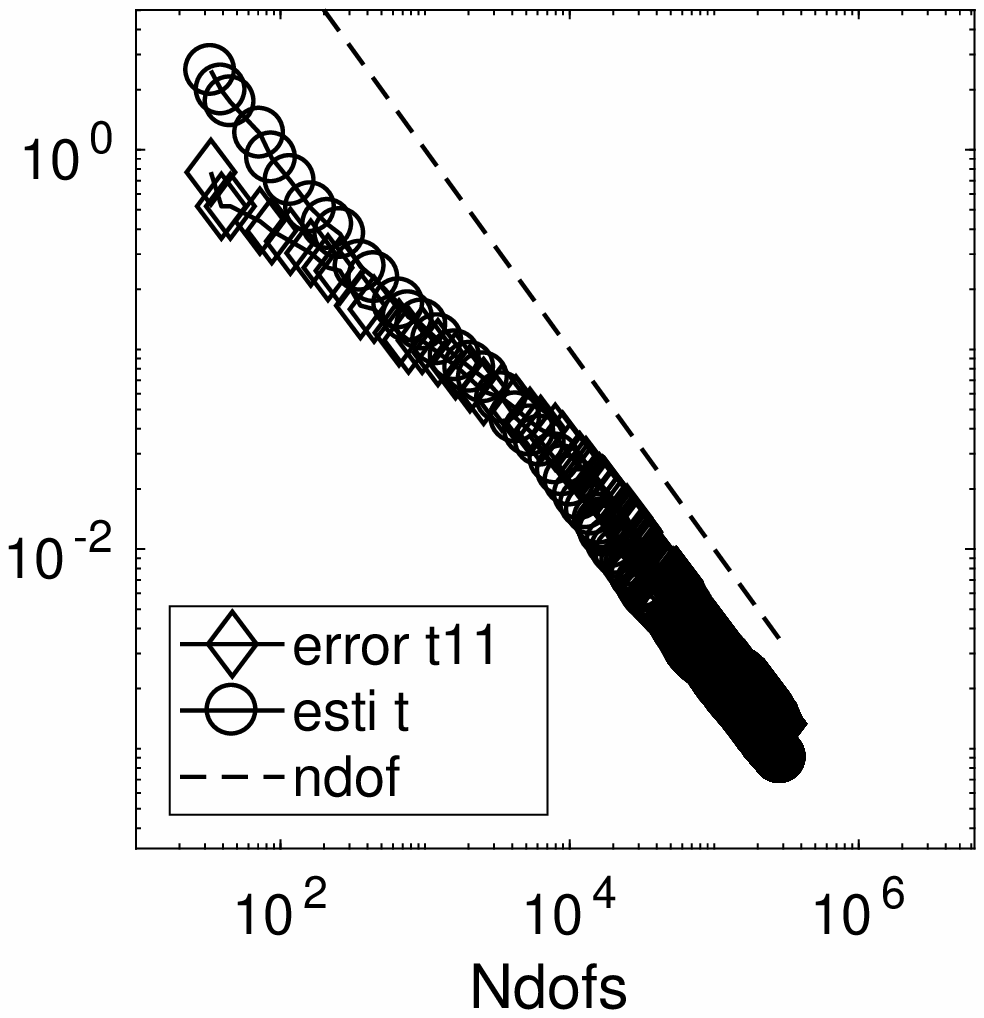}\\
\tiny{(B.2)}\\~\\
\includegraphics[trim={0 0 0 0},clip,width=3.6cm,height=3.3cm,scale=0.6]{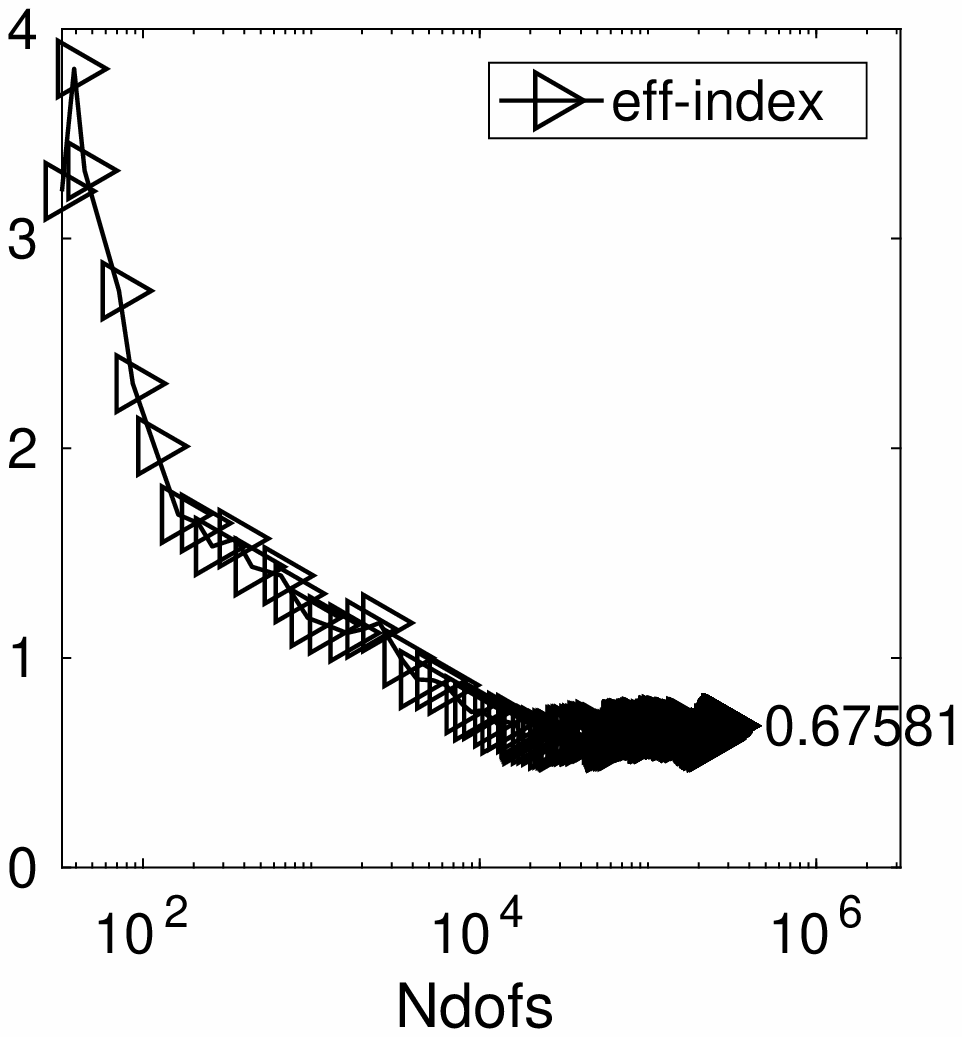}\\\tiny{(B.3)}
\end{minipage}
\begin{minipage}[c]{0.34\textwidth}\centering
Variational discretization:\\
\tiny{Uniform refinement.}\\
\psfrag{error t11}{\normalsize{$\| e \|_\Omega$}}
\psfrag{esti t}{\normalsize{$\mathfrak{E}$}}
\psfrag{Ndofs}{\normalsize{$\textrm{Ndof}_2$}}
\psfrag{ndof}{$\footnotesize{\textrm{Ndof}^{-2/3}_2}$}
\psfrag{aprox P1P1P1-SPARSE-l2}{$\mathbb{V}(\mathscr{T})\times \mathbb{V}(\mathscr{T})\times \mathbb{U}_{ad}$}
\psfrag{eff-index}{\normalsize{$\mathfrak{E}/\| e \|_\Omega$}}
\includegraphics[trim={0 0 0 0},clip,width=3.7cm,height=3.3cm,scale=0.6]{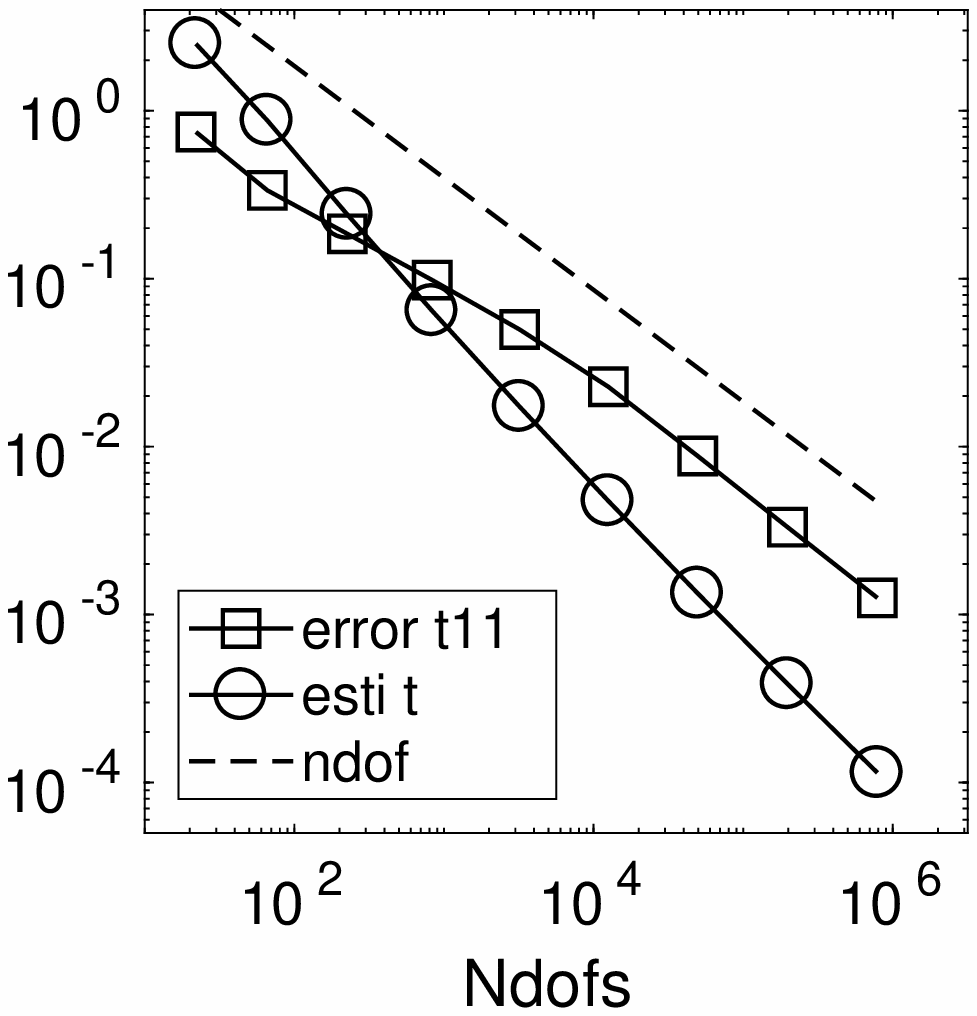}\\
\tiny{(C.1)}\\~\\
\tiny{Adaptive refinement.}\\
\psfrag{ndof}{$\footnotesize{\textrm{Ndof}^{-1}_2}$}
\includegraphics[trim={0 0 0 0},clip,width=3.7cm,height=3.3cm,scale=0.6]{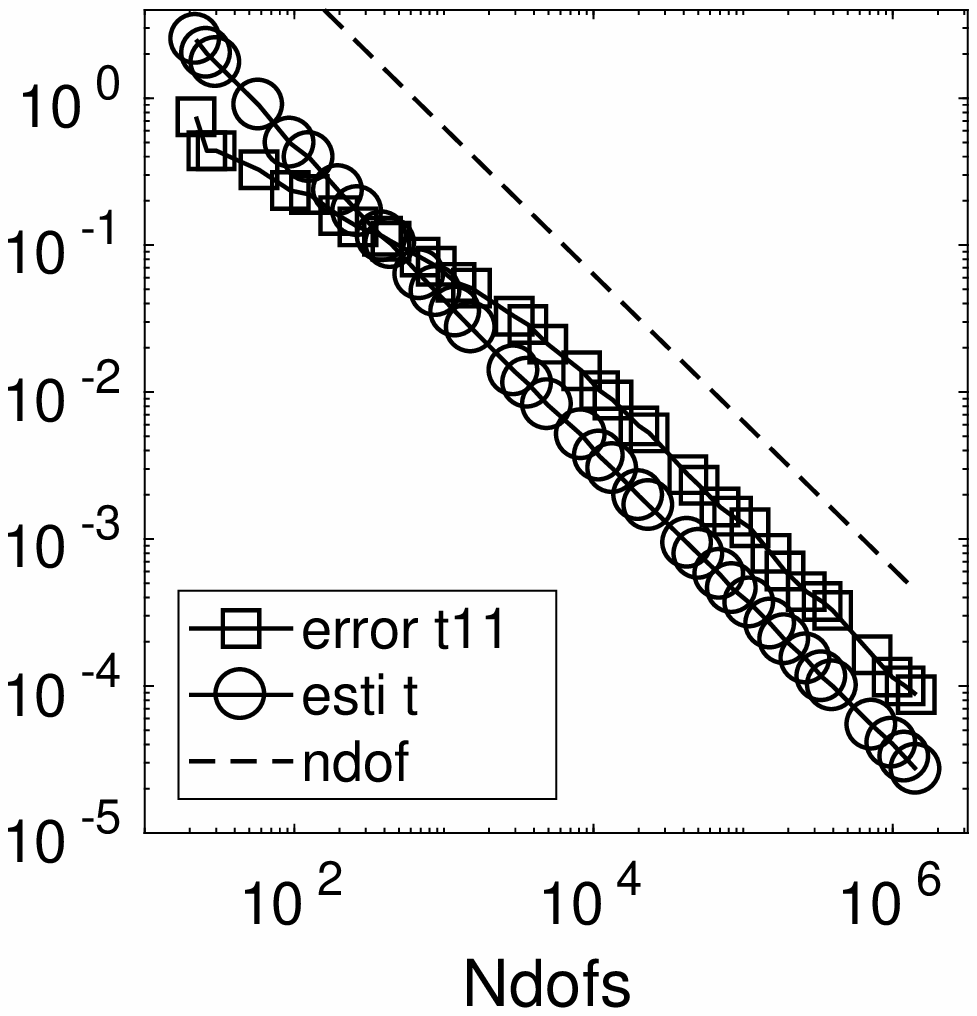}\\
\tiny{(C.2)}\\~\\
\includegraphics[trim={0 0 0 0},clip,width=3.8cm,height=3.3cm,scale=0.6]{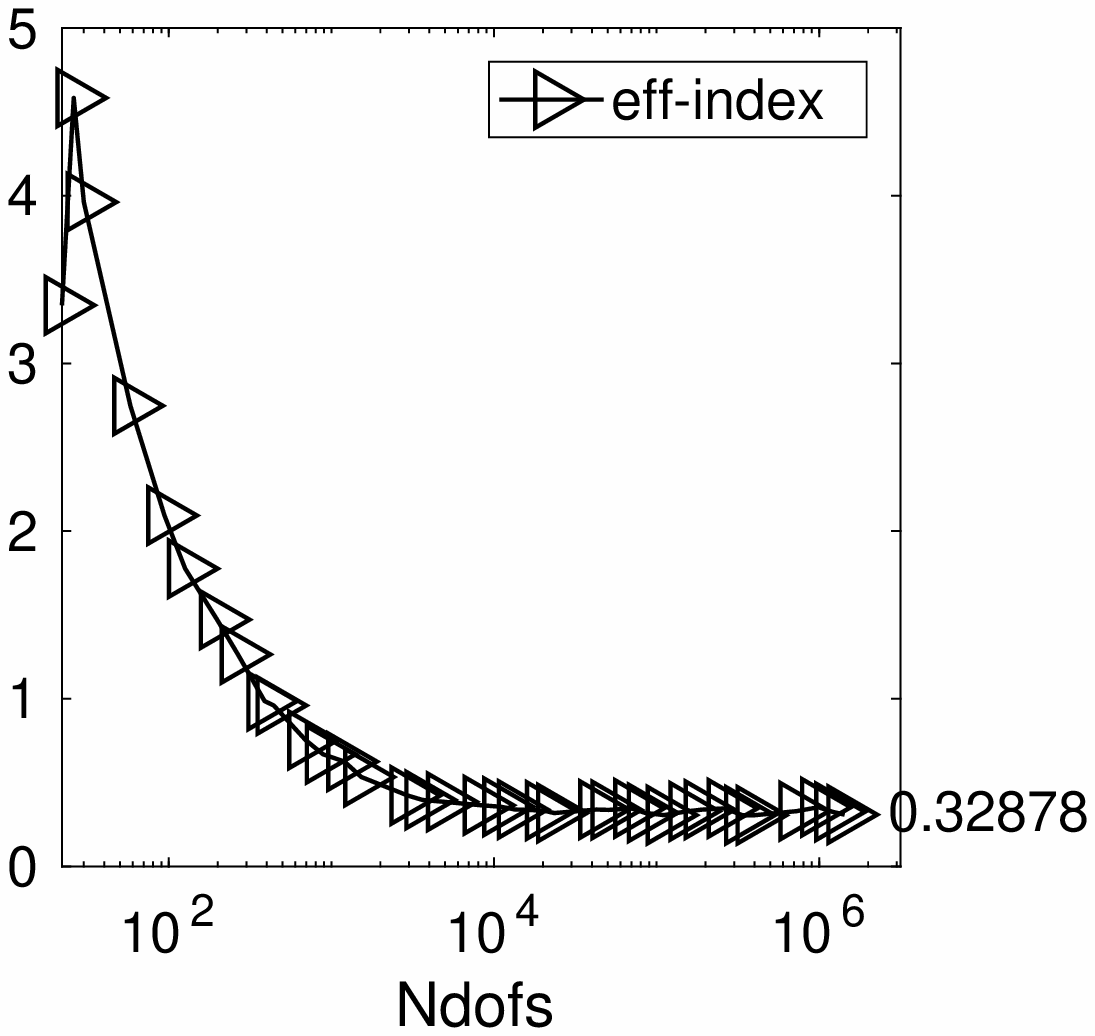}\\\tiny{(C.3)}
\end{minipage}
\caption{Example 2.  \FF{Experimental rates of convergence for the total approximation error and error estimator for uniform refinement (A.1) and adaptive refinement (A.2), and effectivity index $\mathcal{E}/\VERT e \VERT_\Omega$ (A.3), for the piecewise constant discretization. Experimental rates of convergence for the total approximation error and error estimator for uniform refinement (B.1) and adaptive refinement (B.2), and effectivity index ${E}/\| e \|_\Omega$ (B.3), for the piecewise linear discretization. Experimental rates of convergence for the total approximation error and error estimator for uniform refinement (C.1) and adaptive refinement (C.2), and effectivity index $\mathfrak{E}/\| e \|_\Omega$ (C.3), for the variational discretization. In each case we have considered $\alpha=10^{-3}$ and $\beta =2\cdot 10^{-1}$.}}
\label{ex_4}
\end{figure}


\begin{figure}[!h]
\centering
\begin{minipage}[c]{0.28\textwidth}\centering
Piecewise constant:\\
\tiny{Uniform refinement.}\\
\psfrag{error y11}{\normalsize{$|e_y|_{H^1(\Omega)}$}}
\psfrag{error p}{\normalsize{$|e_p|_{H^1(\Omega)}$}}
\psfrag{error u}{\normalsize{$\|e_u\|_{L^2(\Omega)}$}}
\psfrag{error l}{\normalsize{$\|e_\lambda\|_{L^2(\Omega)}$}}
\psfrag{esti y11}{\normalsize{$\mathcal{E}_y$}}
\psfrag{esti p}{\normalsize{$\mathcal{E}_p$}}
\psfrag{esti u}{\normalsize{${E}_u$}}
\psfrag{esti l}{\normalsize{$E_\lambda$}}
\psfrag{Ndofs}{\normalsize{$\textrm{Ndof}^{}_0$}}
\psfrag{ndof}{$\footnotesize{\textrm{Ndof}^{-1/2}_0}$}
\psfrag{ndof3}{$\footnotesize{\textrm{Ndof}^{-1/3}_0}$}
\psfrag{aprox P1P1P0-SPARSE-L2}{$\mathbb{V}(\mathscr{T})\times \mathbb{V}(\mathscr{T})\times \mathbb{U}_{ad,0}(\mathscr{T})$}
\psfrag{aprox P1P1P0-SPARSE-L2}{$\mathbb{V}(\mathscr{T})\times \mathbb{V}(\mathscr{T})\times \mathbb{U}_{ad,0}(\mathscr{T})$}
\psfrag{ndof}{$\footnotesize{\textrm{Ndof}^{-1/3}_0}$}
\psfrag{eff-index}{\normalsize{$\mathcal{E}/\VERT e \VERT_\Omega$}}
\includegraphics[trim={0 0 0 0},clip,width=3.5cm,height=3.0cm,scale=0.6]{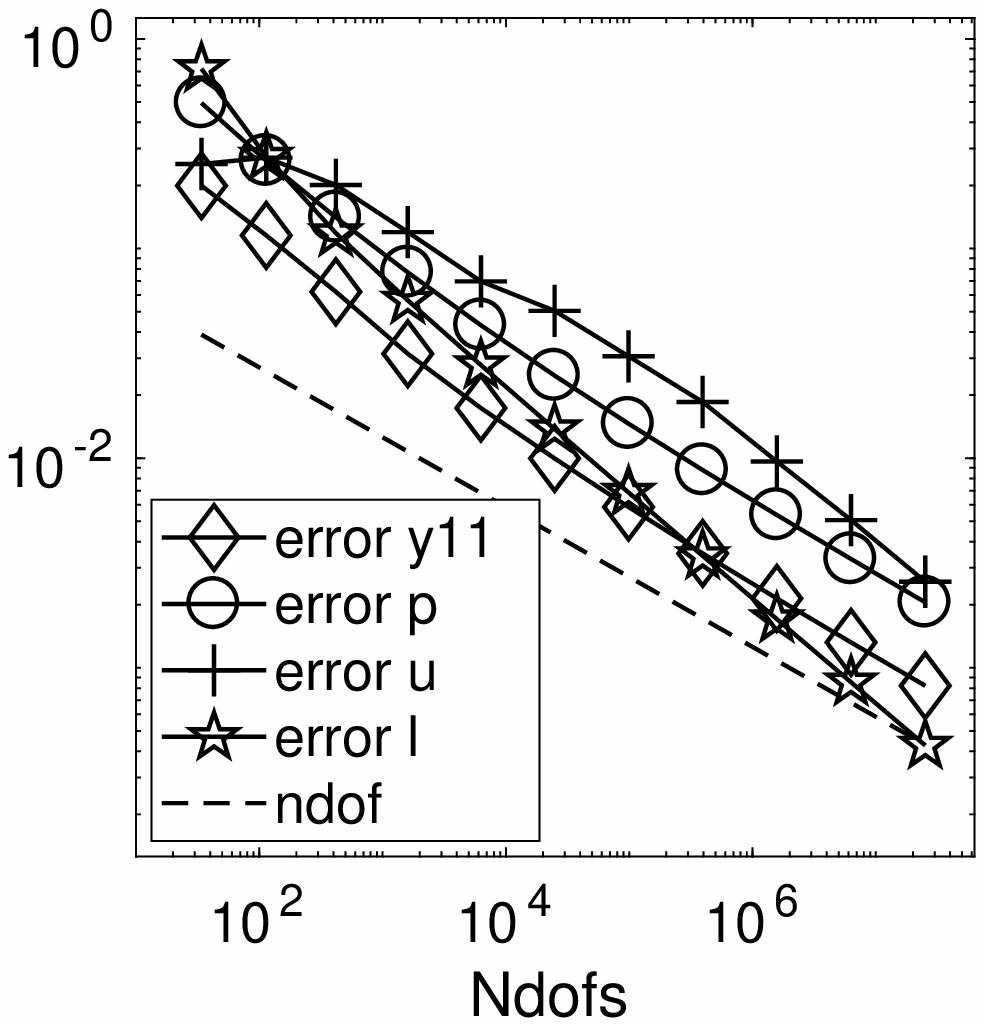}\\
\tiny{(A.1)}~\\~
\tiny{Uniform refinement.}\\
\psfrag{ndof}{$\footnotesize{\textrm{Ndof}^{-1/2}_0}$}
\includegraphics[trim={0 0 0 0},clip,width=3.5cm,height=3.0cm,scale=0.6]{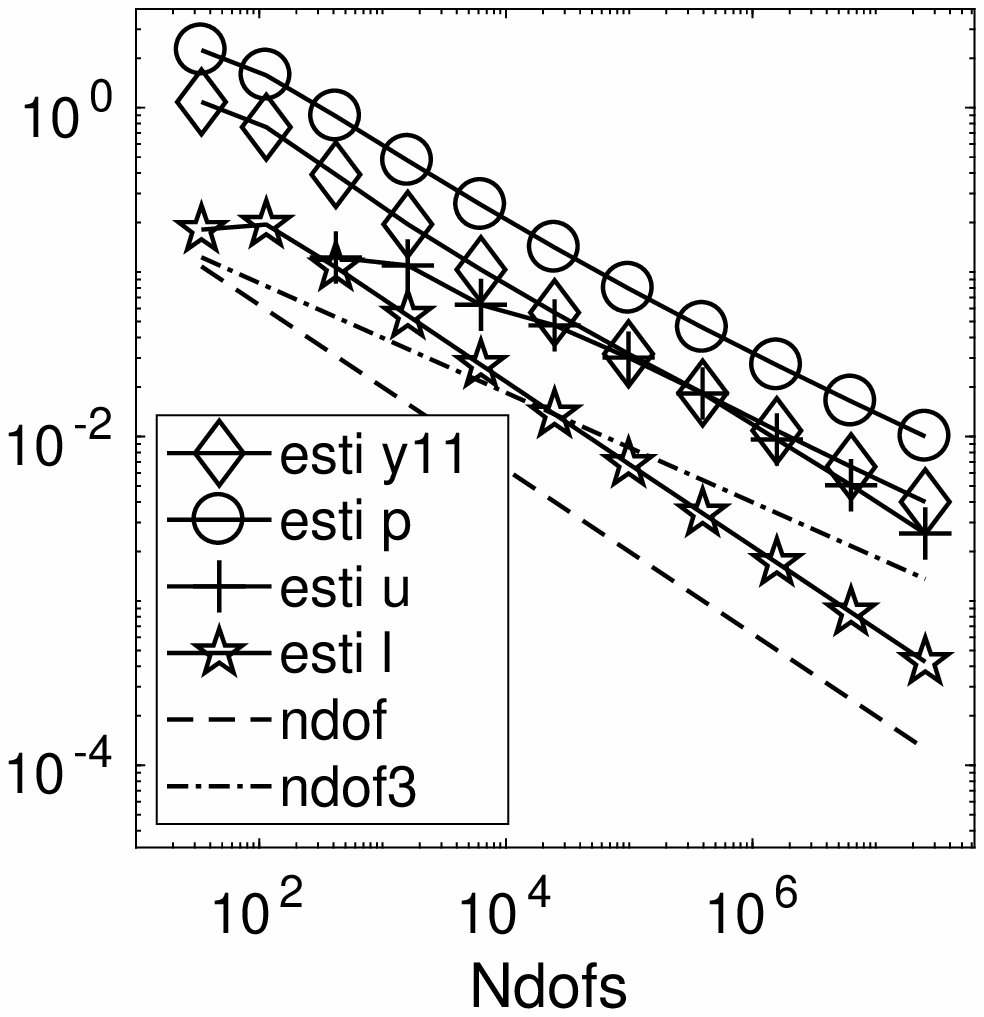}\\
\tiny{(A.2)}~\\~
\tiny{Adaptive refinement.}\\
\includegraphics[trim={0 0 0 0},clip,width=3.5cm,height=3.0cm,scale=0.6]{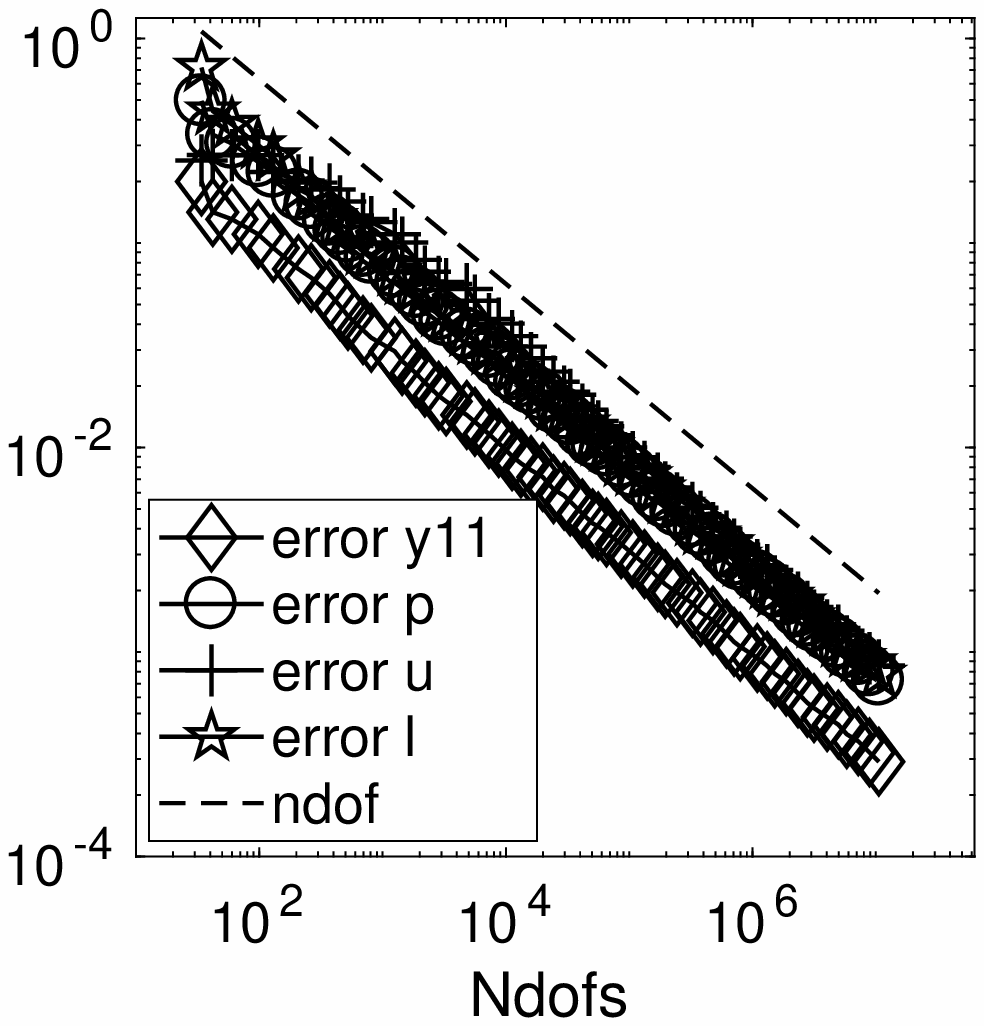}\\
\tiny{(A.3)}\\~
\tiny{Adaptive refinement.}\\
\includegraphics[trim={0 0 0 0},clip,width=3.5cm,height=3.0cm,scale=0.6]{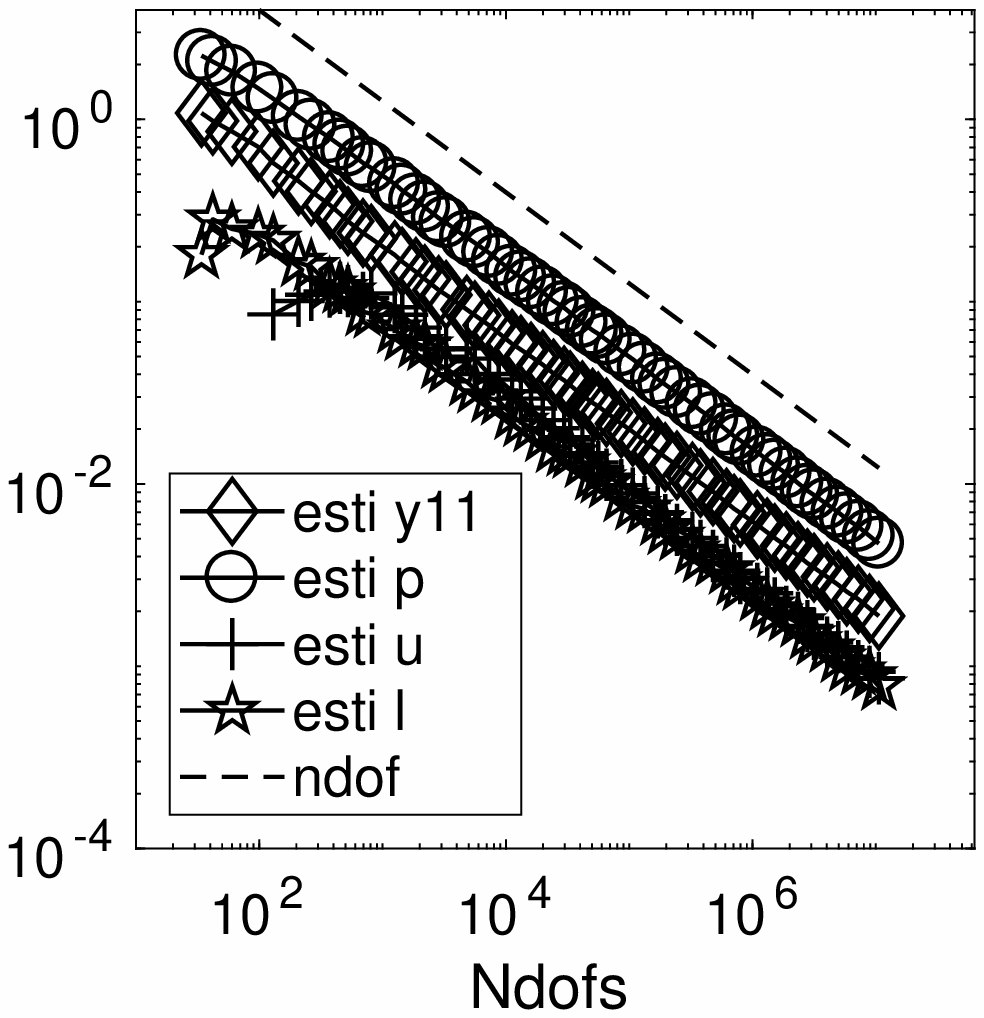}\\
\tiny{(A.4)}
\end{minipage}
\begin{minipage}[c]{0.28\textwidth}\centering
Piecewise linear:\\
\tiny{Uniform refinement.}\\
\psfrag{est-t11}{\normalsize{${E}$}}
\psfrag{error y11}{\normalsize{$\|e_y\|_{L^2(\Omega)}$}}
\psfrag{error p}{\normalsize{$\|e_p\|_{L^2(\Omega)}$}}
\psfrag{error u}{\normalsize{$\|e_u\|_{L^2(\Omega)}$}}
\psfrag{error l}{\normalsize{$\|e_\lambda\|_{L^2(\Omega)}$}}
\psfrag{esti y11}{\normalsize{${E}_y$}}
\psfrag{esti p}{\normalsize{${E}_p$}}
\psfrag{esti u}{\normalsize{${E}_u$}}
\psfrag{esti l}{\normalsize{$E_\lambda$}}
\psfrag{Ndofs}{\normalsize{$\textrm{Ndof}_1$}}
\psfrag{ndof}{\footnotesize{$\textrm{Ndof}^{-1/2}_1$}}
\psfrag{ndof}{$\footnotesize{\textrm{Ndof}^{-2/3}_1}$}
\psfrag{ndof2}{$\footnotesize{\textrm{Ndof}^{-1}_1}$}
\psfrag{aprox P1P1P1-SPARSE-l2}{$\mathbb{V}(\mathscr{T})\times \mathbb{V}(\mathscr{T})\times \mathbb{U}_{ad,1}(\mathscr{T})$}
\psfrag{eff-index}{\normalsize{${E}/ \| e \|_\Omega$}}
\includegraphics[trim={0 0 0 0},clip,width=3.5cm,height=3.0cm,scale=0.6]{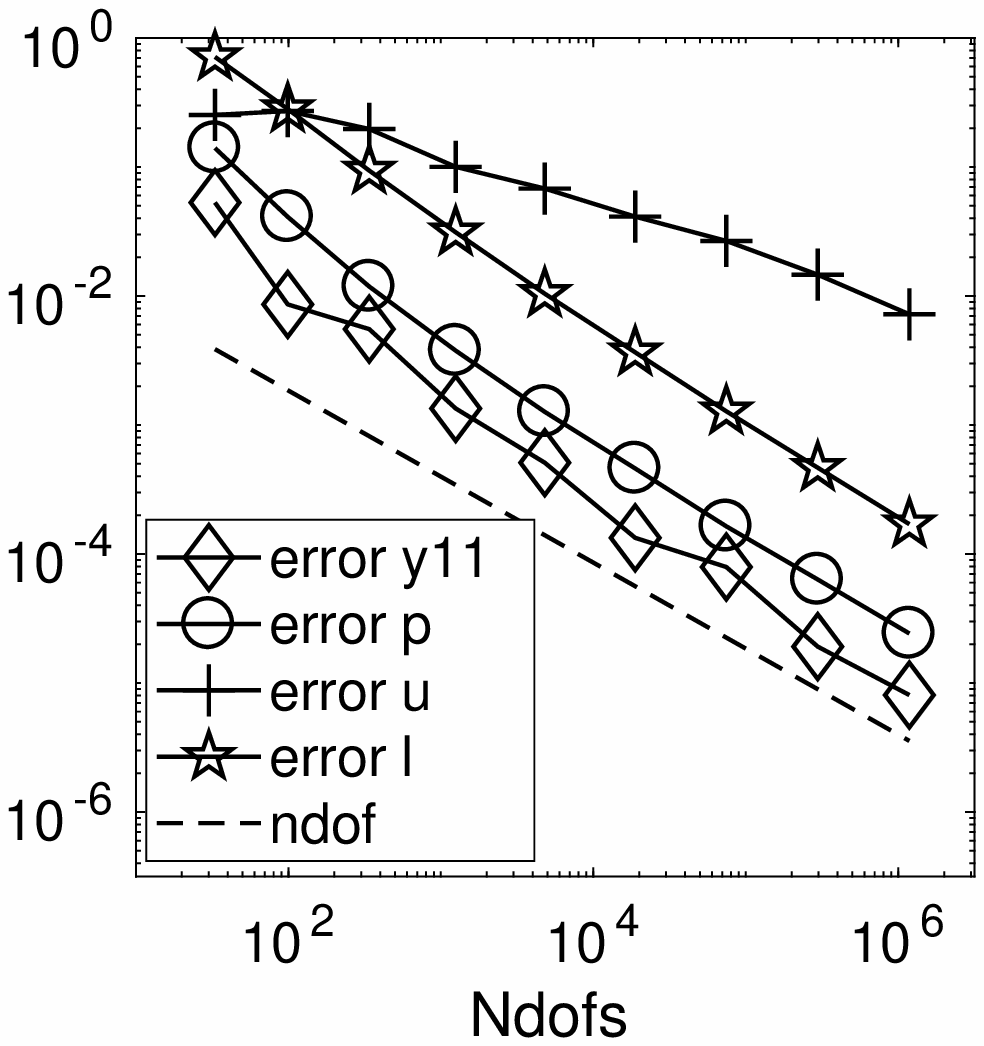}\\
\tiny{(B.1)}~\\~
\tiny{Uniform refinement.}\\
\psfrag{ndof}{$\footnotesize{\textrm{Ndof}^{-2/3}_1}$}
\includegraphics[trim={0 0 0 0},clip,width=3.5cm,height=3.0cm,scale=0.6]{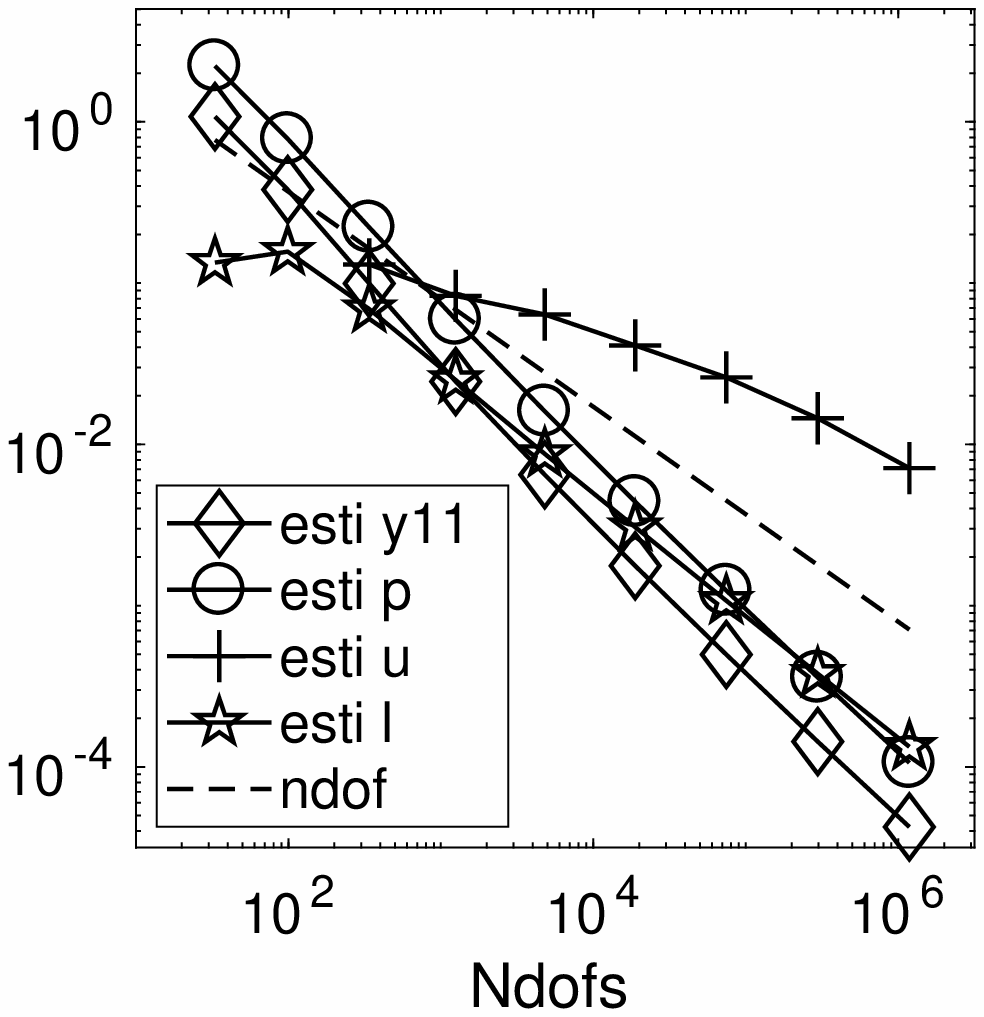}\\
\tiny{(B.2)}~\\~
\tiny{Adaptive refinement.}\\
\psfrag{ndof}{$\footnotesize{\textrm{Ndof}^{-1}_1}$}
\includegraphics[trim={0 0 0 0},clip,width=3.5cm,height=3.0cm,scale=0.6]{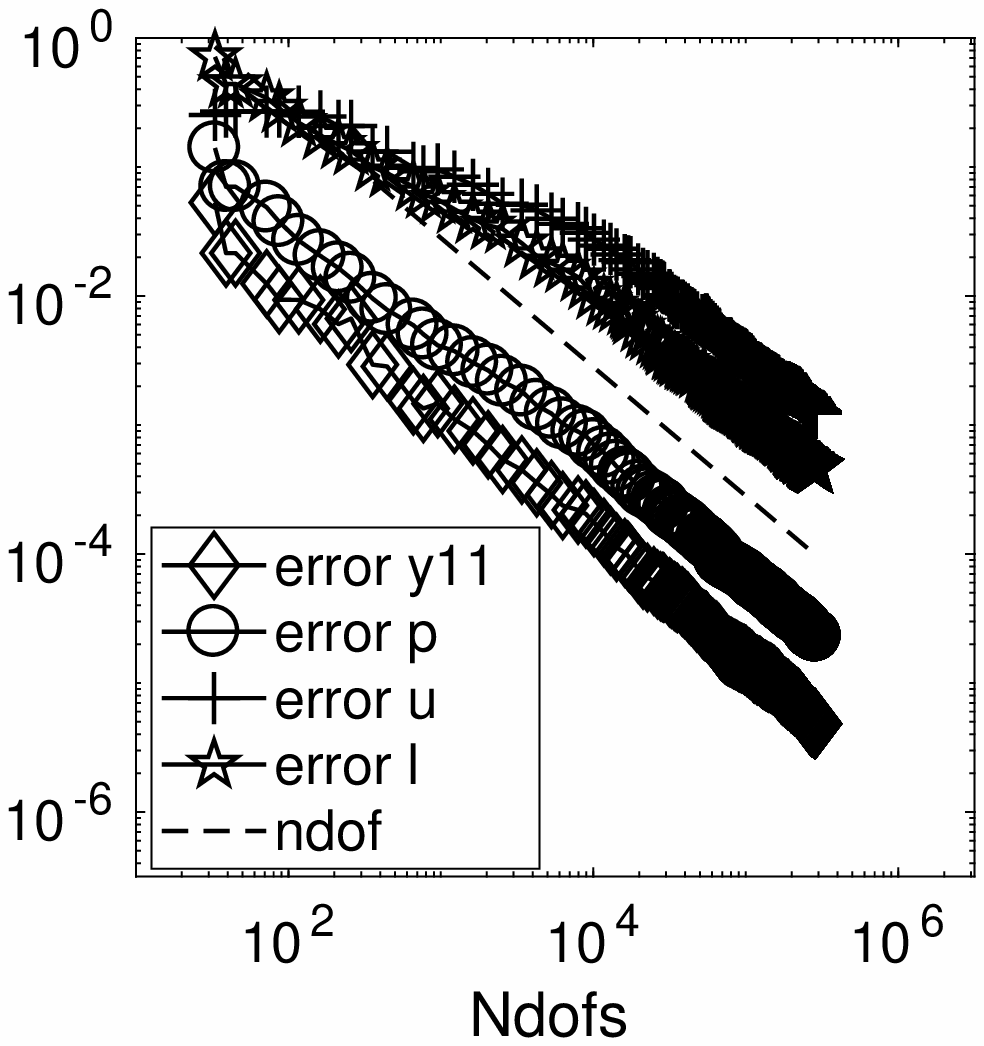}\\
\tiny{(B.3)}\\~
\tiny{Adaptive refinement.}\\
\psfrag{ndof}{$\footnotesize{\textrm{Ndof}^{-1}_1}$}
\includegraphics[trim={0 0 0 0},clip,width=3.5cm,height=3.0cm,scale=0.6]{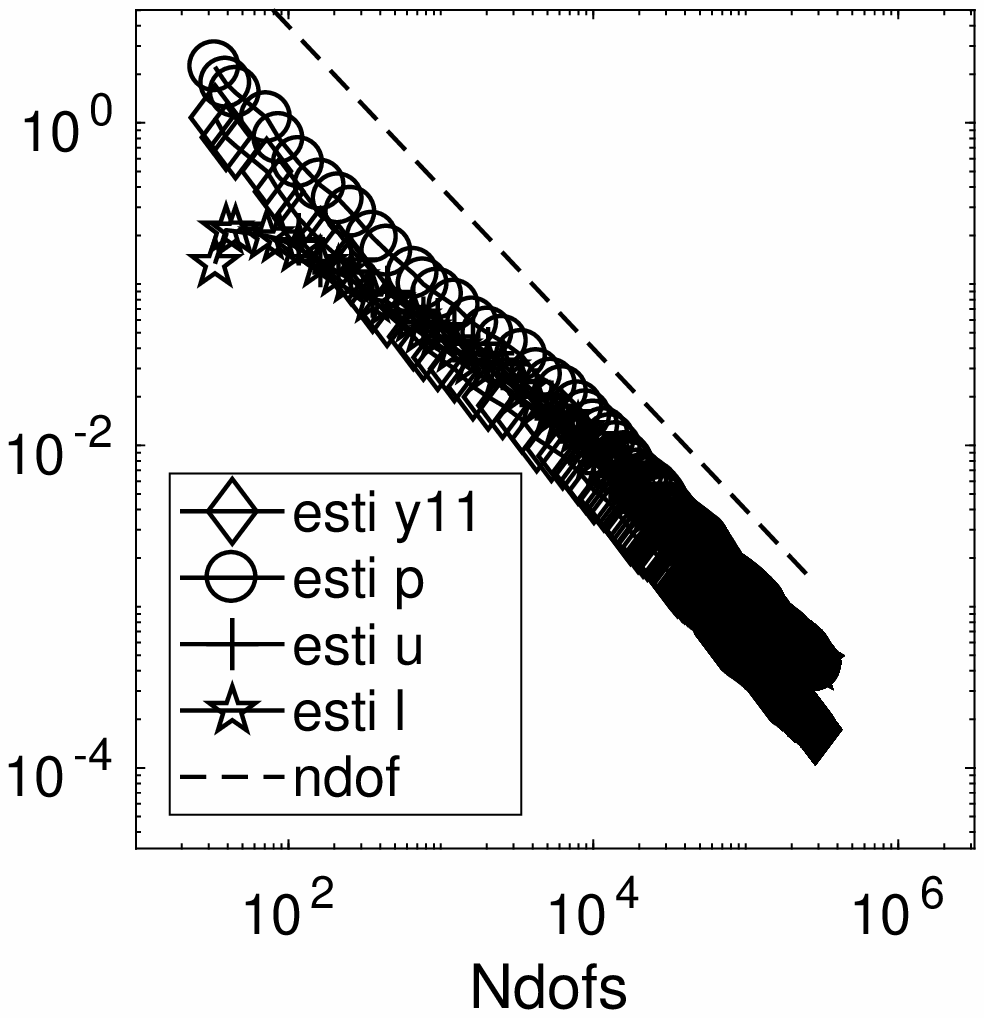}\\
\tiny{(B.4)}
\end{minipage}
\begin{minipage}[c]{0.34\textwidth}\centering
Variational discretization:\\
\tiny{Uniform refinement.}\\
\psfrag{est-t}{\normalsize{$\mathfrak{E}$}}
\psfrag{error y11}{\normalsize{$\|e_y\|_{L^2(\Omega)}$}}
\psfrag{error p}{\normalsize{$\|e_p\|_{L^2(\Omega)}$}}
\psfrag{error u}{\normalsize{$\|e_u\|_{L^2(\Omega)}$}}
\psfrag{error l}{\normalsize{$\|e_\lambda\|_{L^2(\Omega)}$}}
\psfrag{esti y11}{\normalsize{${E}_y$}}
\psfrag{esti p}{\normalsize{${E}_p$}}
\psfrag{Ndofs}{\normalsize{$\textrm{Ndof}_2$}}
\psfrag{ndof}{$\footnotesize{\textrm{Ndof}^{-2/3}_2}$}
\psfrag{aprox P1P1P1-SPARSE-l2}{$\mathbb{V}(\mathscr{T})\times \mathbb{V}(\mathscr{T})\times \mathbb{U}_{ad}$}
\psfrag{eff-index}{\normalsize{$\mathfrak{E}/\| e \|_\Omega$}}
\includegraphics[trim={0 0 0 0},clip,width=3.5cm,height=3.0cm,scale=0.6]{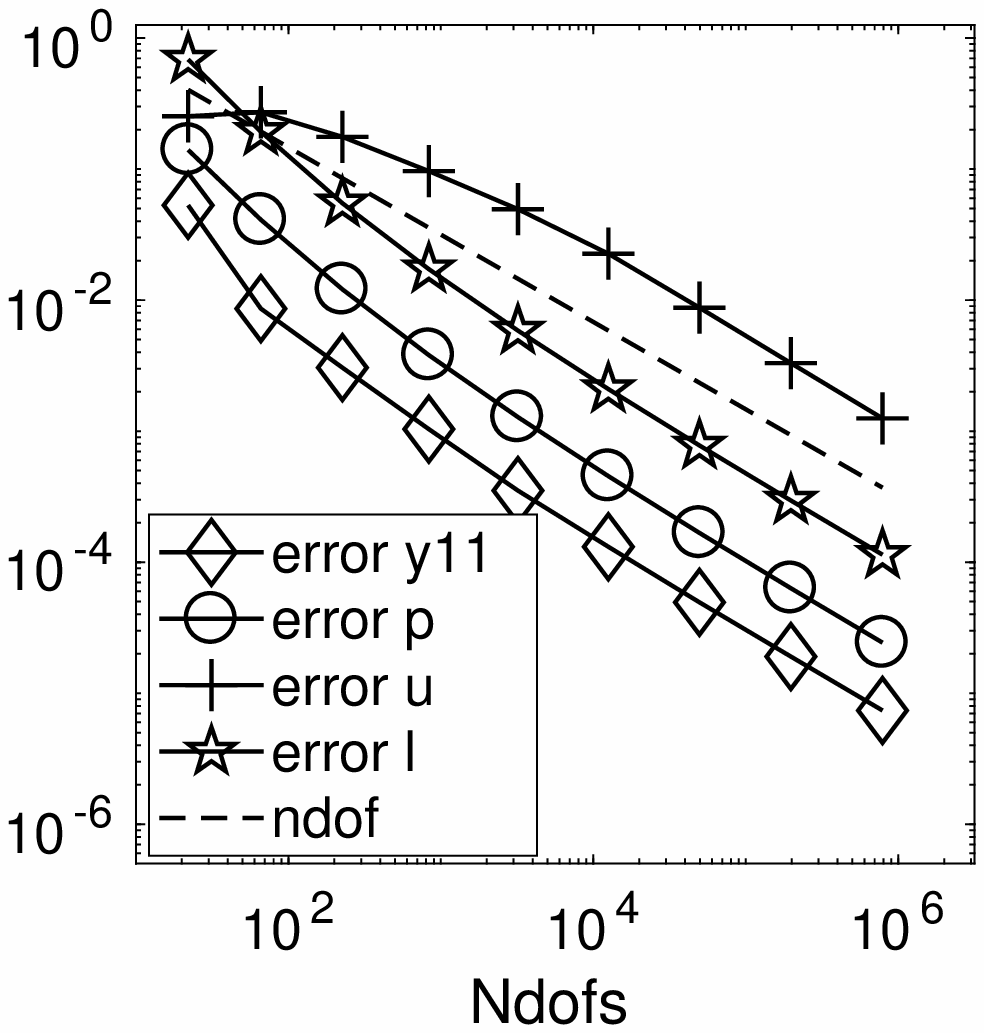}\\
\tiny{(C.1)}~\\~
\tiny{Uniform refinement.}\\
\psfrag{ndof}{$\footnotesize{\textrm{Ndof}^{-2/3}_2}$}
\psfrag{error y11}{\normalsize{${E}_y$}}
\psfrag{error p}{\normalsize{${E}_p$}}
\includegraphics[trim={0 0 0 0},clip,width=3.5cm,height=3.0cm,scale=0.6]{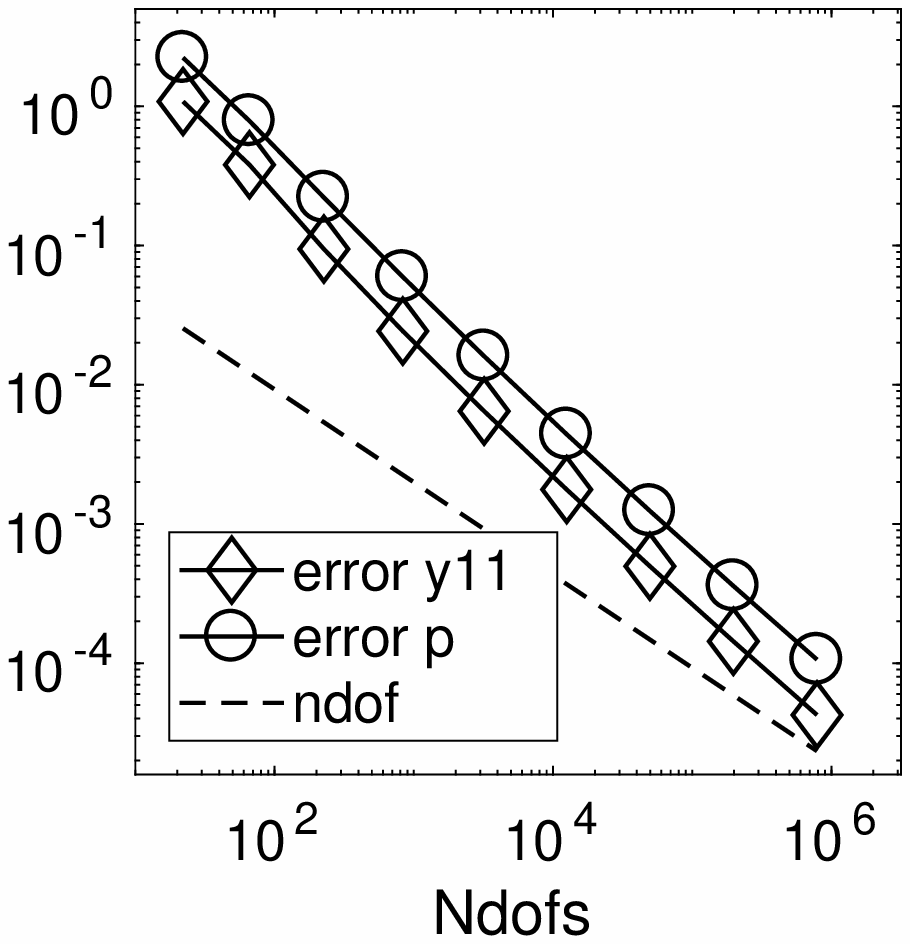}\\
\tiny{(C.2)}~\\~
\tiny{Adaptive refinement.}\\
\psfrag{ndof}{$\footnotesize{\textrm{Ndof}^{-1}_2}$}
\psfrag{error y11}{\normalsize{$\|e_y\|_{L^2(\Omega)}$}}
\psfrag{error p}{\normalsize{$\|e_p\|_{L^2(\Omega)}$}}
\includegraphics[trim={0 0 0 0},clip,width=3.5cm,height=3.0cm,scale=0.6]{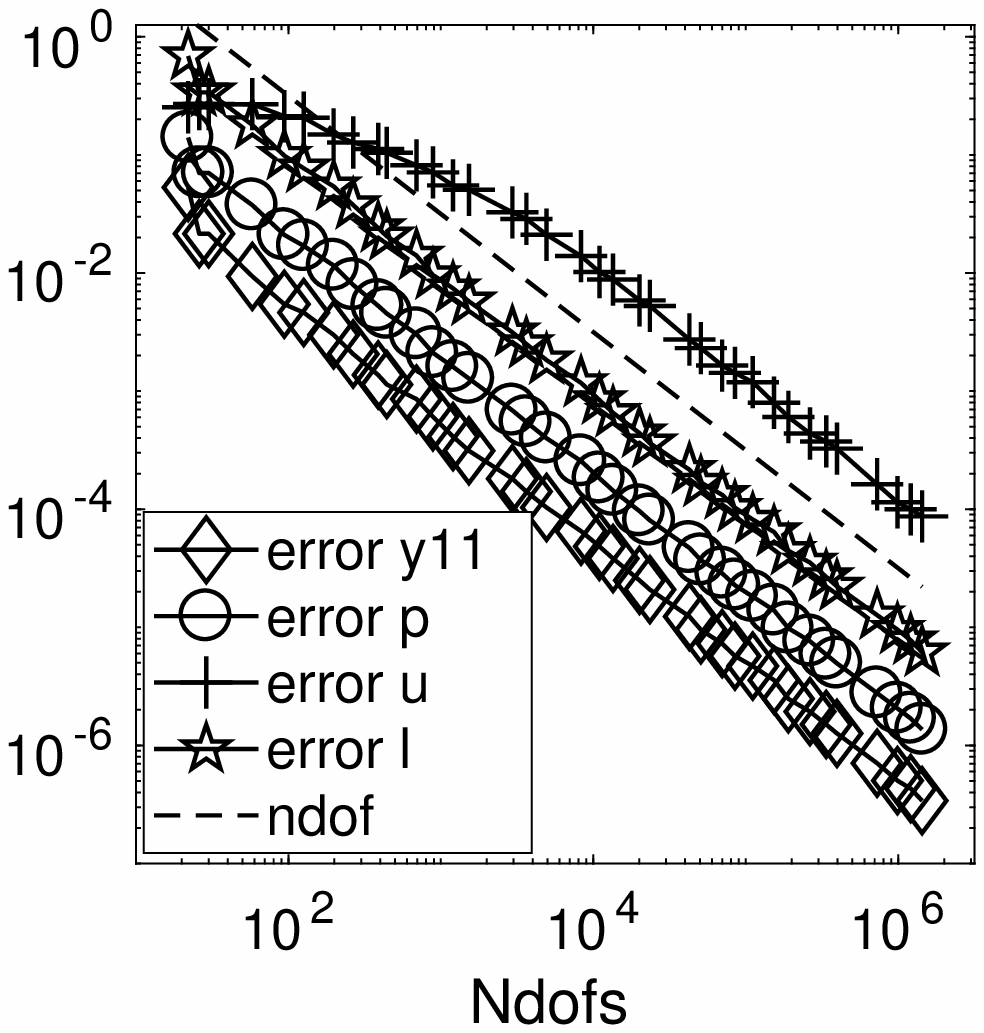}\\
\tiny{(C.3)}\\~
\tiny{Adaptive refinement.}\\
\psfrag{error y11}{\normalsize{${E}_y$}}
\psfrag{error p}{\normalsize{${E}_p$}}
\psfrag{ndof}{$\footnotesize{\textrm{Ndof}^{-1}_2}$}
\includegraphics[trim={0 0 0 0},clip,width=3.5cm,height=3.0cm,scale=0.6]{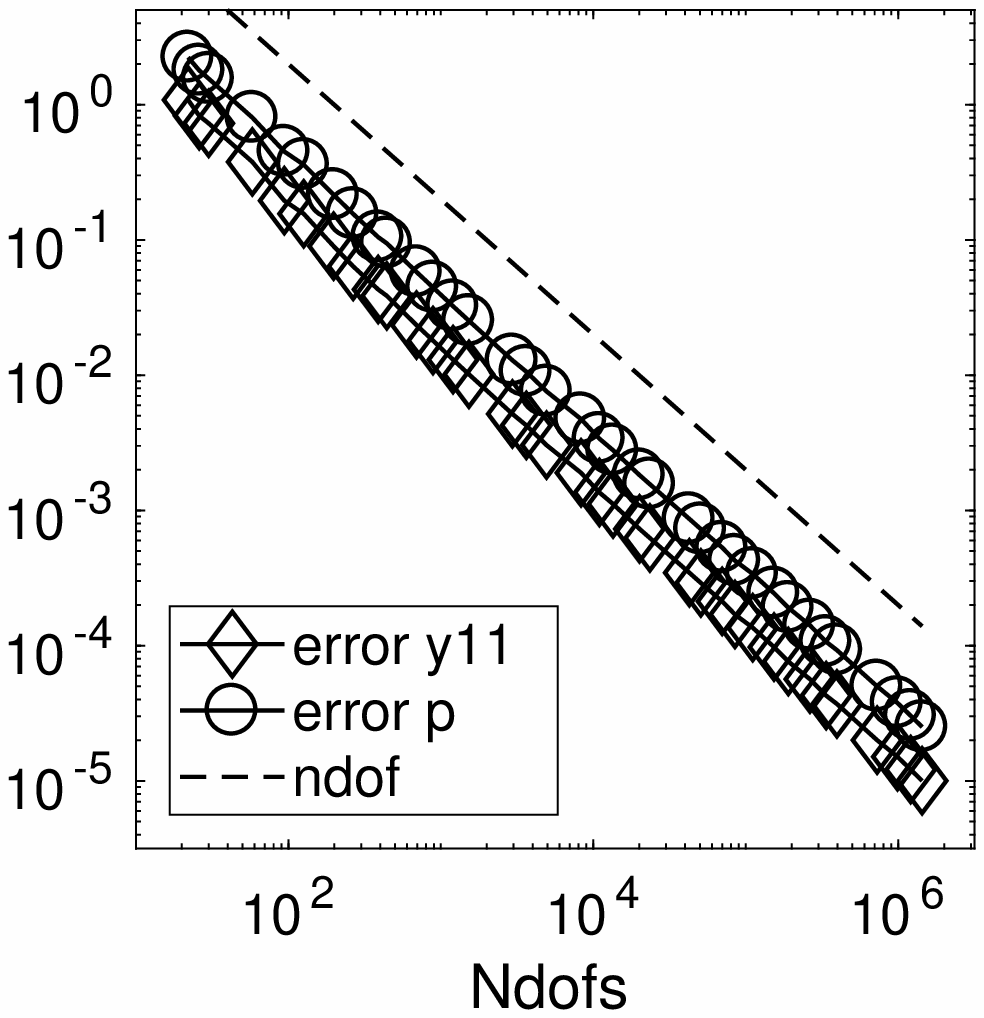}\\
\tiny{(C.4)}\end{minipage}
\caption{Example 2.  \FF{Experimental rates of convergence for each contribution of the total approximation error $\VERT e \VERT_\Omega$ and error estimator $\mathcal{E}$ for uniform and adaptive refinement, for the piecewise constant discretization (left column). Experimental rates of convergence for each contribution of the total approximation error $\| e \|_\Omega$ and error estimator ${E}$ for uniform and adaptive refinement, for the piecewise linear discretization (center column).
Experimental rates of convergence for each contribution of the total approximation error $\| e \|_\Omega$ and error estimator $\mathfrak{E}$ for uniform and adaptive refinement, for the variational discretization (right column). In each case
we have considered $\alpha=10^{-3}$ and $\beta =2\cdot 10^{-1}$.}}
\label{ex_4_2}
\end{figure}

\bibliographystyle{siam}
\footnotesize
\bibliography{bi}

\begin{thebibliography}{10}

\bibitem{MR1885308}
{\sc M.~Ainsworth and J.~T. Oden}, {\em A posteriori error estimation in finite
  element analysis}, Pure and Applied Mathematics (New York),
  Wiley-Interscience [John Wiley \& Sons], New York, 2000.

\bibitem{MUMPS1}
{\sc P.~R. Amestoy, I.~S. Duff, and J.-Y. L'Excellent}, {\em Multifrontal
  parallel distributed symmetric and unsymmetric solvers}, Comput. Methods in
  Appl. Mech. Eng., 184 (2000), pp.~501 -- 520.

\bibitem{MUMPS2}
{\sc P.~R. Amestoy, I.~S. Duff, J.-Y. L'Excellent, and J.~Koster}, {\em A fully
  asynchronous multifrontal solver using distributed dynamic scheduling}, SIAM
  J. Matrix Anal. Appl., 23 (2001), pp.~15--41 (electronic).

\bibitem{BBMRV}
{\sc R.~Becker, M.~Braack, D.~Meidner, R.~Rannacher, and B.~Vexler}, {\em
  Adaptive finite element methods for {PDE}-constrained optimal control
  problems}, in Reactive Flows, Diffusion and Transport, Springer, 2007.

\bibitem{Brett}
{\sc C.~Brett, A.~S.~Dedner, and C.~M.~Elliott}, {\em Optimal control of
  elliptic pdes at points}, IMA Journal of Numerical Analysis, 36 (2015),
  p.~1015–1050.

\bibitem{Casas2017}
{\sc E.~Casas}, {\em A review on sparse solutions in optimal control of partial
  differential equations}, SeMA Journal, 74 (2017), pp.~319--344.

\bibitem{CHW:12b}
{\sc E.~Casas, R.~Herzog, and G.~Wachsmuth}, {\em Approximation of sparse
  controls in semilinear equations by piecewise linear functions}, Numerische
  Mathematik, 122 (2012), pp.~645--669.

\bibitem{CHW:12}
{\sc E.~Casas, R.~Herzog, and G.~Wachsmuth}, {\em Optimality conditions and
  error analysis of semilinear elliptic control problems with {$L^1$} cost
  functional}, SIAM J. Optim., 22 (2012), pp.~795--820.

\bibitem{MR3601024}
\leavevmode\vrule height 2pt depth -1.6pt width 23pt, {\em Analysis of
  spatio-temporally sparse optimal control problems of semilinear parabolic
  equations}, ESAIM Control Optim. Calc. Var., 23 (2017), pp.~263--295.

\bibitem{MR3612174}
{\sc E.~Casas and K.~Kunisch}, {\em Stabilization by sparse controls for a
  class of semilinear parabolic equations}, SIAM J. Control Optim., 55 (2017),
  pp.~512--532.

\bibitem{MR1786137}
{\sc X.~Chen, Z.~Nashed, and L.~Qi}, {\em Smoothing methods and semismooth
  methods for nondifferentiable operator equations}, SIAM J. Numer. Anal., 38
  (2000), pp.~1200--1216.

\bibitem{CiarletBook}
{\sc P.~G. Ciarlet}, {\em The finite element method for elliptic problems},
  SIAM, Philadelphia, PA, 2002.

\bibitem{MR1058436}
{\sc F.~H. Clarke}, {\em Optimization and nonsmooth analysis}, vol.~5 of
  Classics in Applied Mathematics, Society for Industrial and Applied
  Mathematics (SIAM), Philadelphia, PA, second~ed., 1990.

\bibitem{Guermond-Ern}
{\sc A.~Ern and J.-L. Guermond}, {\em Theory and practice of finite elements},
  vol.~159 of Applied Mathematical Sciences, Springer-Verlag, New York, 2004.

\bibitem{HH:08}
{\sc M.~Hinterm\"{u}ller and R.~H.~W. Hoppe}, {\em Goal-oriented adaptivity in
  control constrained optimal control of partial differential equations}, SIAM
  J. Control Optim., 47 (2008), pp.~1721--1743.

\bibitem{HHIK}
{\sc M.~Hinterm\"{u}ller, R.~H.~W. Hoppe, Y.~Iliash, and M.~Kieweg}, {\em An a
  posteriori error analysis of adaptive finite element methods for distributed
  elliptic control problems with control constraints}, ESAIM: Control Optim.
  Calc. of Var., 14 (2008), pp.~540--560.

\bibitem{hinze}
{\sc M.~Hinze}, {\em \textit{A variational discretization concept in control
  constrained optimization: the linear-quadratic case}}, Computational
  Optimization and Applications, 30 (2005), pp.~45--63.

\bibitem{KRS}
{\sc K.~Kohls, A.~R\"{o}sch, and K.~G. Siebert}, {\em A posteriori error
  analysis of optimal control problems with control constraints}, SIAM J.
  Control Optim., 52 (2014), pp.~1832--1861.

\bibitem{LiuYan}
{\sc W.~Liu and N.~Yan}, {\em A posteriori error estimates for distributed
  convex optimal control problems}, Adv. in Comput. Math., 15 (2001),
  pp.~285--309.

\bibitem{MRW:15}
{\sc C.~Meyer, A.~Rademacher, and W.~Wollner}, {\em Adaptive optimal control of
  the obstacle problem}, SIAM J. Sci. Comput., 37 (2015), pp.~918--945.

\bibitem{NSV:09}
{\sc R.~H. Nochetto, K.~G. Siebert, and A.~Veeser}, {\em Theory of adaptive
  finite element methods: an introduction}, in Multiscale, nonlinear and
  adaptive approximation, Springer, 2009.

\bibitem{NV}
{\sc R.~H. Nochetto and A.~Veeser}, {\em Primer of adaptive finite element
  methods}, in Multiscale and Adaptivity: Modeling, Numerics and Applications,
  CIME Lectures, Springer, 2011.

\bibitem{OS2}
{\sc E.~Ot\'arola and A.~J. Salgado}, {\em Sparse optimal control for
  fractional diffusion}, Comput. Methods Appl. Math., 18 (2018), pp.~95--110.

\bibitem{non}
{\sc W.~Schirotzek}, {\em Nonsmooth analysis}, Universitext, Springer, Berlin,
  2007.

\bibitem{SW:16}
{\sc R.~Schneider and G.~Wachsmuth}, {\em A posteriori error estimation for
  control-constrained, linear-quadratic optimal control problems}, SIAM Journal
  on Numerical Analysis, 54 (2016), pp.~1169--1192.

\bibitem{MR2556849}
{\sc G.~Stadler}, {\em Elliptic optimal control problems with {$L^1$}-control
  cost and applications for the placement of control devices}, Comput. Optim.
  Appl., 44 (2009), pp.~159--181.

\bibitem{Troltzsch}
{\sc F.~Tr{\"o}ltzsch}, {\em Optimal control of partial differential
  equations}, vol.~112 of Graduate Studies in Mathematics, American
  Mathematical Society, Providence, RI, 2010.
\newblock Theory, methods and applications, Translated from the 2005 German
  original by J{\"u}rgen Sprekels.

\bibitem{MR1972217}
{\sc M.~Ulbrich}, {\em Semismooth {N}ewton methods for operator equations in
  function spaces}, SIAM J. Optim., 13 (2002), pp.~805--842 (2003).

\bibitem{Verfurth2}
{\sc R.~Verf\"urth}, {\em A posteriori error estimators for the {S}tokes
  equations}, Numer. Math., 55 (1989), pp.~309--325.

\bibitem{Verfurth}
{\sc R.~Verf\"urth}, {\em A posteriori error estimation techniques for finite
  element methods}, Numerical Mathematics and Scientific Computation, Oxford
  University Press, Oxford, 2013.

\bibitem{VW:08}
{\sc B.~Vexler and W.~Wollner}, {\em Adaptive finite elements for elliptic
  optimization problems with control constraints}, SIAM J. Control Optim., 47
  (2008), pp.~509--534.

\bibitem{WW:11}
{\sc G.~Wachsmuth and D.~Wachsmuth}, {\em Convergence and regularization
  results for optimal control problems with sparsity functional}, ESAIM Control
  Optim. Calc. Var., 17 (2011), pp.~858--886.

\end{thebibliography}

\end{document}